\documentclass[11pt]{amsart}

\usepackage{fancyhdr}
\usepackage{cite}

\usepackage{etex}
\usepackage{etoolbox}

\usepackage{longtable}

\usepackage{dsfont}

\usepackage{booktabs} 

\usepackage{colortbl}

\usepackage{enumitem}

\usepackage{tikz}
\usetikzlibrary{calc,patterns,intersections}
\usepackage{tikz-cd}

\usepackage[fixamsmath]{mathtools}
\mathtoolsset{centercolon}
\usepackage[margin=1.27in]{geometry}                % See geometry.pdf to learn the layout options. There are lots.
\usepackage{graphicx}
\usepackage{amssymb}
\usepackage{amsthm}
\usepackage{epstopdf}
\usepackage{mathrsfs}
\usepackage{multirow}
\usepackage{empheq}
\usepackage{color}
\usepackage[normalem]{ulem}
\newtheorem{theorem}{Theorem}%[section]
\newtheorem{lemma}[theorem]{Lemma}
\newtheorem{proposition}[theorem]{Proposition}
\newtheorem{corollary}[theorem]{Corollary}

\theoremstyle{remark}
\newtheorem*{remark}{Remark}

\theoremstyle{remark}
\newtheorem*{remarks}{Remarks}

\theoremstyle{definition}
\newtheorem*{definition}{Definition}

\theoremstyle{definition}
\newtheorem*{example}{Example}

\theoremstyle{definition}
\newtheorem*{examples}{Examples}

\makeatletter
\patchcmd{\@afterheading}%
    {\clubpenalty \@M}{\clubpenalties 3 \@M \@M 0}{}{}
\patchcmd{\@afterheading}%
    {\clubpenalty \@clubpenalty}{\clubpenalties 2 \@clubpenalty 0}{}{}
\makeatother 

\makeatletter
\def\thm@space@setup{%
\thm@preskip=0.8em
\thm@postskip=\thm@preskip % or whatever, if you don't want them to be equal
}
\makeatother

\newcommand{\IU}{I\nhp(\nhp\mathbf{u}\hspace{-0.2pt})} % I(u) in normal script, tighter
\newcommand{\Iu}{\scriptscriptstyle{I\nhp(\nhp\mathbf{u}\hspace{-0.25pt})}} % I(u) in very small script, for superscripts

\newcommand{\Io}{{\scriptscriptstyle I_0}}

\newcommand{\0}{{\scriptstyle 0'}} % makes the script smaller
\newcommand{\1}{{\scriptstyle 1'}} 

\newcommand{\A}{{\scriptscriptstyle A}} % makes the script smaller
\newcommand{\B}{{\scriptscriptstyle B}} 
\newcommand{\C}{{\scriptscriptstyle C}} 
\newcommand{\D}{{\scriptscriptstyle D}}

\newcommand{\I}{{\scriptscriptstyle I}} % makes the script smaller

\newcommand{\pt}{\hspace{1pt}} % adds a small space of 1 pt
\newcommand{\hp}{\hspace{0.5pt}} % adds a small space of 1/2 pt

\newcommand{\npt}{\hspace{-1pt}} % subtracts a small space of 1 pt
\newcommand{\nhp}{\hspace{-0.5pt}} % subtracts a small space of 1/2 pt

\newcommand{\Li}{\hp\mbox{Li}\hp}

\renewcommand{\Re}{\hp\mbox{Re}\pt} % redefine symbol for the real and imaginary parts
\renewcommand{\Im }{\hp\mbox{Im}\hp}

\newcommand{\psitilde}{\raisebox{0.3pt}{\smash{$\tilde{\smash{\mathrlap{\psi}} \vphantom{\rule[6.5pt]{0pt}{0pt}} \hphantom{\hspace{10.5pt}}}$}} \hspace{-4.1pt}}

\newcommand{\psitildeindx}{\raisebox{-0.8pt}{\smash{$\scriptstyle\tilde{\smash{\mathrlap{\psi}} \vphantom{\rule[4.7pt]{0pt}{0pt}} \hphantom{\hspace{8pt}}}$}} \hspace{-2.4pt}}

\newcommand{\plus}{\mathord{\begin{tikzpicture}[baseline=0ex, line width=0.3, scale=0.08]
\draw (1.2,0) -- (1.2,0.8); 	\draw (1.2,1.2) -- (1.2,2);
\draw (0.8,0) -- (0.8,0.8); 	\draw (0.8,1.2) -- (0.8,2);
\draw (0,1.2) -- (0.8,1.2); 	\draw (1.2,1.2) -- (2,1.2);
\draw (0,0.8) -- (0.8,0.8); 	\draw (1.2,0.8) -- (2,0.8);
\draw (0,0.8) -- (0,1.2);
\draw (2,0.8) -- (2,1.2);
\draw (0.8,0) -- (1.2,0);
\draw (0.8,2) -- (1.2,2);
\end{tikzpicture}}}

\newcommand{\minus}{\mathord{\begin{tikzpicture}[baseline=0ex, line width=0.3, scale=0.08]
\draw (0,1.2) -- (2,1.2); 
\draw (0,0.8) -- (2,0.8);
\draw (0,1.2) -- (0,0.8);
\draw (2,1.2) -- (2,0.8);
\end{tikzpicture}}}

\newcommand{\noarrow}{\mathord{\begin{tikzpicture}[baseline=-0.03ex, line width=0.3, scale=0.1]
\draw (0,1) -- (9.54,1);
\end{tikzpicture}}}

\newcommand{\Hom}{\text{Hom}}
\newcommand{\Ext}{\text{Ext}}

\newcommand\pgfmathsinandcos[3]{%
  \pgfmathsetmacro#1{sin(#3)}%
  \pgfmathsetmacro#2{cos(#3)}%
}
\newcommand\LongitudePlane[3][current plane]{%
  \pgfmathsinandcos\sinEl\cosEl{#2} % elevation
  \pgfmathsinandcos\sint\cost{#3} % azimuth
  \tikzset{#1/.style={cm={\cost,\sint*\sinEl,0,\cosEl,(0,0)}}}
}
\newcommand\LatitudePlane[3][current plane]{%
  \pgfmathsinandcos\sinEl\cosEl{#2} % elevation
  \pgfmathsinandcos\sint\cost{#3} % latitude
  \pgfmathsetmacro\yshift{\cosEl*\sint}
  \tikzset{#1/.style={cm={\cost,0,0,\cost*\sinEl,(0,\yshift)}}} %
}
\newcommand\DrawLongitudeCircle[2][1]{
  \LongitudePlane{\angEl}{#2}
  \tikzset{current plane/.prefix style={scale=#1}}
  \pgfmathsetmacro\angVis{atan(sin(#2)*cos(\angEl)/sin(\angEl))} %
  \draw[current plane] (\angVis:1) arc (\angVis:\angVis+180:1);
  \draw[current plane,dashed] (\angVis-180:1) arc (\angVis-180:\angVis:1);
}
\newcommand\DrawLatitudeCircle[2][2]{
  \LatitudePlane{\angEl}{#2}
  \tikzset{current plane/.prefix style={scale=#1}}
  \pgfmathsetmacro\sinVis{sin(#2)/cos(#2)*sin(\angEl)/cos(\angEl)}
  \pgfmathsetmacro\angVis{asin(min(1,max(\sinVis,-1)))}
  \draw[current plane] (\angVis:1) arc (\angVis:-\angVis-180:1);
  \draw[current plane,dashed] (180-\angVis:1) arc (180-\angVis:\angVis:1);
}

\makeatletter
\pretocmd{\chapter}{\addtocontents{toc}{\protect\addvspace{15\p@}}}{}{}
\pretocmd{\section}{\addtocontents{toc}{\protect\addvspace{5\p@}}}{}{}
\pretocmd{\subsection}{\addtocontents{toc}{\protect\addvspace{1\p@}}}{}{}
\makeatother
\makeatletter
\renewcommand{\tocsection}[3]{%
  \indentlabel{\@ifnotempty{#2}{\bfseries\ignorespaces#1 #2\quad}}\bfseries#3}
\renewcommand{\tocsubsection}[3]{%
  \indentlabel{\@ifnotempty{#2}{\ignorespaces#1 #2\quad}}#3}
\newcommand\@dotsep{4.5}
\def\@tocline#1#2#3#4#5#6#7{\relax
  \ifnum #1>\c@tocdepth % then omit
  \else
    \par \addpenalty\@secpenalty\addvspace{#2}%
    \begingroup \hyphenpenalty\@M
    \@ifempty{#4}{%
      \@tempdima\csname r@tocindent\number#1\endcsname\relax
    }{%
      \@tempdima#4\relax
    }%
    \parindent\z@ \leftskip#3\relax \advance\leftskip\@tempdima\relax
    \rightskip\@pnumwidth plus1em \parfillskip-\@pnumwidth
    #5\leavevmode\hskip-\@tempdima{#6}\nobreak
    \leaders\hbox{$\m@th\mkern \@dotsep mu\hbox{.}\mkern \@dotsep mu$}\hfill
    \nobreak
    \hbox to\@pnumwidth{\@tocpagenum{\ifnum#1=1\bfseries\fi#7}}\par% <-- \bfseries for \section page
    \nobreak
    \endgroup
  \fi}
\AtBeginDocument{%
\expandafter\renewcommand\csname r@tocindent0\endcsname{0pt}
}
\def\l@subsection{\@tocline{2}{0pt}{2.5pc}{5pc}{}}
\makeatother

\title{Twisted hyperk\"ahler symmetries and hyperholomorphic line bundles}
\author{Radu A. Iona\c{s}}
\date{}                                           % Activate to display a given date or no date

\newcommand{\Address}{{% additional braces for segregating the font style
  \bigskip\bigskip 
  \sffamily
            Radu A. Iona\c{s} \\[3pt]
\indent C.\,N. Yang Institute for Theoretical Physics \\
\indent Stony Brook University, Stony Brook, NY 11794, U.S.A.

}}

\begin{document}

\begin{abstract}
In this paper we propose and investigate in full generality new notions of (continuous, non-isometric) symmetry on hyperk\"ahler spaces. These can be grouped into two categories, corresponding to the two basic types of continuous hyperk\"ahler isometries which they deform: tri-Hamiltonian isometries, on one hand, and rotational isometries, on the other. The first category of deformations gives rise to Killing spinors and generate what are known as hidden hyperk\"ahler symmetries. The second category gives rise to hyperholomorphic line bundles over the hyperk\"ahler manifolds on which they are defined and, by way of the Atiyah-Ward correspondence, to holomorphic line bundles over their twistor spaces endowed with meromorphic connections, generalizing similar structures found in the purely rotational case by Haydys and Hitchin. Examples of hyperk\"ahler metrics with this type of symmetry include the c-map metrics on cotangent bundles of affine special K\"ahler manifolds with generic prepotential function, and the hyperk\"ahler constructions on the total spaces of certain integrable systems proposed by Gaiotto, Moore and Neitzke in connection with the wall-crossing formulas of Kontsevich and Soibelman, to which our investigations add a new layer of geometric understanding.
\end{abstract}

\maketitle

\thispagestyle{fancy}\rhead{YITP-SB-17-52}

\vspace{10pt}

\tableofcontents

\section{Introduction}

Killing vector fields on hyperk\"ahler manifolds come in two flavors:
\begin{itemize}
\setlength{\itemsep}{1pt}

\item[1.] \textit{tri-Hamiltonian} Killing vector fields, whose Lie actions separately preserve each one of the three elements of a standard global frame of the bundle of hyperk\"ahler symplectic forms, that is,
\begin{align}
& \mathcal{L}_X\omega_1 = 0 
&& \mathcal{L}_X\omega_2 = 0 
&& \mathcal{L}_X\omega_3 = 0  \mathrlap{\hp ;} \\
\intertext{\item[2.] \textit{rotational} Killing vector fields, which preserve only one hyperk\"ahler symplectic form while rotating the transversal ones. For example} 
& \mathcal{L}_X \omega_1 = \omega_2
&& \mathcal{L}_X \omega_2 = - \, \omega_1 \hspace{-12pt}
&& \mathcal{L}_X \omega_3 = 0 \rlap{.}
\end{align}
\end{itemize}

%\noindent \textcolor{blue}{Both of these types of actions extend naturally to holomorphic actions on the twistor space of the hyperk\"ahler manifold on which they are defined. }

Other non-discrete notions of symmetries on hyperk\"ahler spaces besides the ones generated by Killing vector fields also exist in the literature. A generalization of the concept of continuous tri-Hamiltonian symmetry under the name of \textit{hidden symmetries} was proposed by Dunajski and Mason, with tri-Hamiltonian Killing vector fields replaced by \textit{Killing spinors} \cite{MR1785432, MR2006758}. A twistor space version of Dunajski and Mason's construction has been explored independently in the physical language of projective superspace by Lindstr\"om and Ro\v{c}ek \cite{MR929144, Lindstrom:2008gs}. Closely related to this is also Bielawski's notion of \textit{twistor group action} \cite{MR1848654}. On the other hand, in the rotational Killing case, Haydys \cite{MR2394039} has shown that the presence of such a symmetry implies automatically the existence of a hyperholomorphic line bundle over the hyperk\"ahler manifold endowed with a hyperhermitian connection, that is, a connection whose curvature 2-form is of \mbox{$(1,1)$} type with respect to all hyperk\"ahler complex structures of the manifold simultaneously (here we will refer to a 2-form with this property as being \textit{of hyper \mbox{$(1,1)$} type}; in four dimensions, this is equivalent to requiring that the form is \textit{self-dual}). On the twistor space\,---\,Hitchin then shows in \cite{MR3116317}\,---\,by the hyperk\"ahler version of the Atiyah-Ward correspondence, one has a corresponding holomorphic line bundle, trivial on twistor lines, equipped with a meromorphic connection with simple poles on the twistor fibers over 
%\mbox{$\zeta = 0$} and $\infty$ 
the antipodal pair of points of the twistor sphere corresponding to the two complex structures preserved by the action (in the example above, $I_3$ and $-I_3$) 
and globally-defined residues uniquely determined by the action. Hyperholomorphic line bundles with similar characteristics have also been encountered by Neitzke in \cite{Neitzke:2011za} in relation to a highly non-trivial deformation of the so-called semi-flat metric proposed previously by Gaiotto, Moore and Neitzke in \cite{MR2672801}, although, puzzlingly, in a context \textit{lacking} any rotational symmetries (the initial semi-flat metric, however, does have one such symmetry). Other examples exhibiting this feature were later studied by Korman \cite{MR3625762}. 

In this paper we develop a geometric framework in which these seemingly very different extensions of the two fundamental notions of continuous Killing symmetry can be understood in a unified manner as manifestations of the presence of what we will call \textit{twisted hyperk\"ahler symmetries}. To define these, it helps if we reformulate first the above action conditions in twistor space language. Consider an open covering of the sphere of hyperk\"ahler complex structures\,---\,\textit{i.e.}~of $\smash{ S^2 = \{x_1 I_1 + x_2 I_2 + x_3 I_3 \, | \, x_1,x_2,x_3 \in \mathbb{R}, \, x_1^2 + x_2^2 + x_3^2 = 1 \} }$\,---\,with two open sets obtained by removing the points on the sphere corresponding to $-I_3$~and~$I_3$, respectively. In what follows we will refer to the elements of this cover as the \textit{polar} regions (\textit{northern} and \textit{southern}), and to their intersection as the \textit{tropical} region of the twistor sphere. Regarding the sphere as a complex projective line, let also $\zeta$ be a complex affine parameter for the northern chart of the holomorphic coordinate atlas associated to this covering, chosen such that the two removed points are labeled by \mbox{$\zeta = 0$} and \mbox{$\zeta = \infty$}, respectively. On the twistor space, one has a globally-defined holomorphic 2-form supported on the twistor fibers and twisted by the $\mathcal{O}(2)$ bundle over the twistor projective line, whose tropical component takes in a standard local trivialization the form
\begin{equation}
\omega(\zeta) = \frac{\omega_+}{\zeta\ } + \omega_0 + \zeta \hp \omega_-
\end{equation}
where, by definition,
$\smash{ \omega_{\pm} = \pm \pt \frac{1}{2}(\omega_1 \pm i \hp \omega_2) }$ and
\mbox{$\omega_0 = \omega_3 $}. The two Killing conditions can then be equivalently expressed in a condensed form as follows:
\begin{itemize}
\setlength{\itemsep}{0pt}

\item[1.] in the tri-Hamiltonian case:
\begin{equation}
\mathcal{L}_X\omega(\zeta) = 0 \mathrlap{\hp ;}
\end{equation}

\item[2.] in the rotational case:
\begin{equation}
\Big( \! - i \hp \zeta \frac{\partial}{\partial \zeta} + \mathcal{L}_{X} \Big) \omega(\zeta) = 0 \rlap{.}
\end{equation}
\end{itemize}

\noindent The Lie derivatives are assumed to be taken fiberwise. Furthermore, one can show that both actions extend naturally to holomorphic actions on the twistor space. 

The \textit{twisted} versions of these actions will be defined for any finite integer \mbox{$j > 1$} by the conditions which result from replacing in these formulas the generating vector field $X$ with an $\mathcal{O}(2j-2)$-twisted vector field
\begin{equation} 
X(\zeta) = \sum_{n =1-j}^{j-1} X_n \hp \zeta^{-n}
\end{equation}
while still preserving the fiberwise assumption about the action of the Lie derivatives. In other words, we allow the vector field action to depend on the twistor fibers in a holomorphic way controlled by a complex line bundle of finite positive even degree over the twistor $\smash{ \mathbb{CP}^1 }$. We will call these twisted actions \textit{trans-tri-Hamiltonian} and \textit{trans-rotational}, respectively. 

At first sight such generalizations look problematic. For one thing, even though the resulting actions admit natural lifts to the twistor space, these are not holomorphic any more. Worse, one can show that the generators of such an action, should they exist, are not uniquely defined. However, a closer analysis reveals that while there are no obviously preserved quantities\,---\,like a metric, a complex structure, and so on\,---\,in the usual sense one associates to a symmetry, these conditions give nevertheless rise to interesting holomorphic objects on the twistor space. 

Thus, in the trans-tri-Hamiltonian case, assuming that the hyperk\"ahler manifold has vanishing first cohomology group, one obtains a globally-defined holomorphic section of the (pullback) $\mathcal{O}(2j)$ bundle over the twistor space, encoding the components of a symmetric Killing spinor. Extrapolating this picture to \mbox{$j=1$}, we retrieve the tri-Hamiltonian case proper, with the spinor morphing into an $\mathbb{R}^3$-vector-valued function, the image of the hyperk\"ahler moment map of the untwisted action. On the other hand, in the trans-rotational case what we get is a holomorphic line bundle over the twistor space,
%(that is, provided that a certain hyper $(1,1)$ form is integral; otherwise, we get a \textit{holomorphic Lie algebroid extension})
trivial on  twistor lines, equipped with a meromorphic connection with poles of order $j$ on the twistor fibers over \mbox{$\zeta = 0$} and $\infty$ and globally-defined residues of orders 1 through $j$ uniquely determined by the action. On the hyperk\"ahler manifold itself we have a corresponding hyperholomorphic line bundle endowed with a hyperhermitian connection. Its curvature 2-form is part of a triplet of hyper \mbox{$(1,1)$} forms naturally induced on the manifold by the twisted action, with the other two elements of the triplet having vanishing cohomology classes. (One should note that the existence of these line bundles is in fact also predicated on the cohomologically non-trivial hyper $(1,1)$ form satisfying an integrality condition; otherwise, for instance, instead of the holomorphic line bundle over the twistor space we get a \textit{holomorphic Lie algebroid extension}.) Clearly, these structures are generalizations of the ones found by  Hitchin and Haydys in the case of a purely rotational $S^1$-action, to which they reduce when we specialize to \mbox{$j=1$}.

The main body of the paper is organized into seven sections. In \textbf{section~\ref{sec:HK-sp}} we develop a new, twistor-optimal algebraic approach to describing the family of Cauchy-Riemann problems associated to the sphere of complex structures of a hyperk\"ahler manifold. Given a standard two-set complex coordinate chart on the sphere, viewed as $\smash{ \mathbb{CP}^1 }$, we notice that generic complex subspace projectors admit three different factorizations\,---\,depending on whether the corresponding hyperk\"ahler complex structure belongs to the northern, southern, or tropical region of the sphere\,---\,with two algebraic factors, one depending holomorphically and the other depending anti-holomorphically on the local complex affine coordinate. This prompts us to associate to each of the two polar regions a family of idempotent, and to the tropical region one of nilpotent endomorphisms of the tangent bundle of the hyperk\"ahler manifold holomorphically parametrized by the local complex affine coordinate, with the polar endomorphisms interchanged by antipodal conjugation and the tropical one left invariant. In terms of these holomorphically-parametrized families of endomorphisms, the hyperk\"ahler Cauchy-Riemann equations admit a natural formulation, specified in Proposition~\ref{CR_M}.

In the first part of \textbf{section~\ref{sec:Tw-sp}} we review basic facts about twistor spaces of hyperk\"ahler \pagebreak manifolds through the lens of the formalism developed in the previous section. This is intended mostly as a practice exercise and can be safely skipped by readers familiar with the subject. The formulation of the Cauchy-Riemann equations on the twistor space given in Proposition~\ref{CR_Z} underscores why this formalism is optimally suited for the twistor approach. The remaining part of the section is dedicated to proving an important technical result, namely Lemma~\ref{fiberwise-lemma}, which will play a crucial role in the proof of Lemma~\ref{A_V_mero} from section~\ref{sec:Trans-rot} concerning the meromorphicity of the connection associated to a trans-rotational action.

In \textbf{section~\ref{sec:Hyp-1,1}} we study the general properties of \emph{closed} hyper \mbox{$(1,1)$} forms on hyperk\"ahler manifolds and of their K\"ahler potentials. By the local $\partial\bar{\partial}$\pt-\hp lemma, closed hyper $(1,1)$ forms can be locally derived from a potential in each hyperk\"ahler complex structure. We call such potentials \textit{hyperpotentials} with respect to the corresponding complex structure. This notion plays a prominent role in our investigations. A set of five criterions for a local function to be a hyperpotential is collected in Proposition~\ref{criterion_2}. A remarkable feature of hyperpotentials is that they always come in infinite families. A recursive argument based on the \mbox{$\bar{\partial}$\pt-\hp Poincar\'e} lemma shows that any given hyperpotential generates automatically, although not uniquely, a \textit{chain of hyper\-potentials}\,---\,that is, an infinite sequence of hyperpotentials with respect to the \textit{same} hyperk\"ahler complex structure, related by certain first-order recursion relations (\mbox{\S\,\ref{ssec:rec_chains}}). As a consequence, one has an analogue of the local $\partial\bar{\partial}$\pt-\hp lemma for closed hyper \mbox{$(1,1)$} forms locally associating to any such form and to any hyperk\"ahler complex structure a corresponding chain of hyperpotentials (Lemma~\ref{exp_ddbar_lem}). Given a closed hyper \mbox{$(1,1)$} form, and fixing a complex structure which we label conventionally by \mbox{$\zeta = 0$}, then the chain of hyperpotentials with respect to this complex structure can be used to construct hyperpotentials with respect to other complex structures labeled by some small enough $\zeta$ in the form of an infinite series with one part holomorphic and another one anti-holomorphic in~$\zeta$; conversely, any hyperpotential for the closed hyper \mbox{$(1,1)$} form with respect to a complex structure sufficiently close to the one labeled by \mbox{$\zeta = 0$} admits a decomposition of this type (Proposition~ \ref{lulu}). Chains of hyperpotentials are also intimately related to holomorphic functions on the twistor space. The Laurent coefficients of the $\zeta$-\pt expansion of a twistor space holomorphic function always form a chain of hyperpotentials with respect to the  complex structure labeled by \mbox{$\zeta = 0$}; conversely, given a chain of hyperpotentials with respect to the complex structure labeled by \mbox{$\zeta = 0$}, a $\zeta$-\pt series with these coefficients defines a holomorphic function on the twistor space domain on which it converges, if such a domain exists (Proposition~\ref{hol-hyper(1,1)}).

In \textbf{section~\ref{sec:Tri-Ham}} we show how the presence of a symmetric Killing spinor on a simply-connected hyperk\"ahler manifold implies the existence of a trans-tri-Hamiltonian action. In a departure from the register of our previous discussions we frame the first part of our considerations in the language of Salamon's \mbox{$E \otimes\npt H$} formalism, which we review in \mbox{\S\,\ref{ssec:EH}}. The first thing to observe is that any valence \mbox{$(0,2j)$} Killing tensor, which is how a symmetric Killing spinor is termed in this approach, automatically gives rise to a valence \mbox{$(1,2j-1)$}  Killing tensor (Proposition~\ref{higher_K_tens}). These two tensors satisfy two important properties, which we formulate in Lemma~\ref{pic} and Lemma~\ref{poc}. To make the junction with the concepts developed in the previous section we choose a certain frame in which the valence \mbox{$(0,2j)$} Killing tensor can be equivalently viewed as a \mbox{$2j+1$\pt-\pt component} function, and the valence \mbox{$(1,2j-1)$} one as a \mbox{$2j-1$\pt-\pt component} vector field. The two properties can then be interpreted to mean that the components of the function form what we call a bounded chain of hyperpotentials (Proposition~\ref{Kill_tens_chain}), and that they furthermore play the role of generalized moment map functions to the \mbox{$\mathcal{O}(2j-2)$-twisted} trans-tri-Hamiltonian action generated by the multi-component vector field. The corresponding twistor space picture is given in Theorem~\ref{Kill_spins_tw_sp}. Independently from this, we also show that the vector fields defining a trans-tri-Ha\-mil\-to\-ni\-an action are not unique\,---\,rather, they form an equivalence class, with a canonical representative (Lemma~\ref{canon_pres}). The results of this section are not needed to understand the remainder of the paper; readers interested in other aspects may skip it without losing the thread of the arguments.

In \textbf{section~\ref{sec:Trans-rot}} we define and study twisted hyperk\"ahler actions of trans-rotational type. The first notion we define in this pursuit is, however, that of a \textit{quasi-rotational vector field}, formalizing a certain deformation of the usual definition of a rotational vector field. Quasi-rotational vector fields have the remarkable property that they naturally give rise on the hyperk\"ahler manifold to three closed hyper \mbox{$(1,1)$} forms (Proposition~\ref{quasi_rot_hyperpot}). Similarly to trans-tri-Hamiltonian actions, the generators of trans-rotational actions are not uniquely defined, but rather form an equivalence class with a canonical representative (Lemma~\ref{canon_pres_rot}). The middle generator in the canonical presentation, in particular, is a quasi-rotational vector field (Proposition~\ref{middle_q-rot}). Thus, trans-rotational actions always give rise to three closed hyper \mbox{$(1,1)$} forms, two of which can be shown now to have trivial cohomology classes. Incidentally, one should note that, unlike trans-tri-Hamiltonian actions but similarly to purely rotational ones, trans-rotational actions single out a privileged (conjugated) pair of hyperk\"ahler complex structures, which in this paper we conventionally label with \mbox{$\zeta = 0$} and \mbox{$\zeta = \infty$}. In particular, this gives us a natural geography for the twistor sphere, with notions such as the poles and the equator clearly defined. The existence of the closed hyper \mbox{$(1,1)$} forms associated to a trans-rotational action allows us to use in its study the hyperpotential machinery developed in section~\ref{sec:Hyp-1,1}. The picture which emerges from our subsequent analysis is a vigorous generalization of the properties and structures one encounters in the purely rotational case. For example, in this latter case it is known from \cite{MR877637} that the moment map for the action with respect to a polar hyperk\"ahler K\"ahler form (\textit{i.e.}~one of the two preserved by the action) is essentially a K\"ahler potential for all equatorial hyperk\"ahler K\"ahler forms. For a trans-rotational action, we show, equatorial K\"ahler potentials take more generally the form of a Fourier superposition of generalized moment map components with longitude angle-dependent harmonic phase factors (Corollary~\ref{poles-equator}). Moreover, as we have already revealed earlier, just like in the purely rotational case, in the trans-rotational ones we continue to have a hyperholomorphic line bundle over the hyperk\"ahler manifold, with a hyperhermitian connection with curvature given by the cohomologically non-trivial hyper \mbox{$(1,1)$} form. On the twistor space there is a corresponding holomorphic line bundle endowed with a meromorphic connection with poles on the twistor fibers over \mbox{$\zeta = 0$} and $\infty$\,---\,of order $j$, however, rather than the simple poles one has in the purely rotational case (Lemma~\ref{A_V_mero}, Theorem~\ref{main_th}). Of particular practical importance, we should also note, are the equation \eqref{princ_part_N}, detailing the pole structure of the connection at \mbox{$\zeta = 0$}, and the equation \eqref{A_V-manif_mero}, giving us a manifestly meromorphic formula for the connection in certain special twistor coordinates. 
%in terms of some locally defined meromorphic functions one can construct from the generalized moment maps for the trans-rotational action. 

In \textbf{section~\ref{sec:SymmGens}} we return to these coordinates to examine them in closer detail. By fixing a hyperk\"ahler complex structure $I_0$ labeled by \mbox{$\zeta=0$}, choosing a set of local Darboux coordinates for the canonical 2-form on the twistor space around the fiber over \mbox{$\zeta=0$}, and then looking at these coordinate's Taylor expansions in $\zeta$, one can naturally define two systems of local coordinates on any hyperk\"ahler manifold: one holomorphic with respect to $I_0$, and one which we shall call a system of \textit{special coordinates anchored at $I_0$}. Depending on whether a certain Taylor expansion coefficient is real or not, we differentiate between two types of special coordinates, one of which signals the presence of tri-Hamiltonian symmetries. A function $L$ can be also defined, which plays in special coordinates the role of local potential rather similarly to the way a K\"ahler potential does in holomorphic coordinates. In fact, the two potentials are related by a flipped-sign Legendre transform. Local expressions for the hyperk\"ahler metric and 2-forms can be given in terms of the second derivatives of  $L$ (equations \eqref{HK-sym-L} and \eqref{HK-sym-tor}), which must satisfy a set of differential constraints. This generalizes in some sense the Gibbons-Hawking picture, although mostly formally, since in the Gibbons-Hawking case the differential constraints are, rather very specially, linear. However, our interest in these coordinates stems primarily from the fact that the Cauchy-Riemann equations for chains of hyperpotentials with respect to the complex structure $I_0$ may be expressed in terms of them as a generalized moment map-like condition for a certain symplectic gradient-like vector field acting on only half the coordinates (Proposition~\ref{gmm-hyperps}, with $\smash{ X_f }$ defined for any function $f$ by either the formula \eqref{X_f_1} or the formula \eqref{X_f_2}, depending on the type of special coordinates one uses). By exploiting this property we then show that for any twisted action of either trans-tri-Ha\-mil\-to\-ni\-an or trans-rotational type, the equivalence class of the generating vector fields admits a representative whose elements are gradients of the generalized moment maps associated to the action (Theorem~\ref{t-3Ham_grad} and Theorem~\ref{t-rot_grad}, respectively).

In \textbf{section \ref{sec:Trans-rot-Ex}} we deploy the theoretical approach that we have developed in the previous sections to examine in detail several examples of non-compact hyperk\"ahler metrics with rotational and trans-rotational symmetry, and work out explicitly, in particular, the associated  twistor meromorphic connections in the form of some \textit{clutching} constructions over twistor lines. The simplest examples of hyperk\"ahler metrics with non-trivial trans-rotational symmetries are given by the so-called \textit{affine c-map}, or \textit{semi-flat}, metrics, which are determined by a certain holomorphic function called \textit{prepotential}. Higher on the complexity scale, we show that the class of hyperk\"ahler constructions proposed by Gaiotto, Moore, and Neitzke in \cite{MR2672801} in relation to BPS wall-crossing phenomena in certain supersymmetric quantum field theories possess as well a trans-rotational symmetry with \mbox{$j=2$}. These constructions can be understood as quantum instanton corrections to a certain special semi-flat metric with a Seiberg-Witten-type prepotential. This metric's two tri-Hamiltonian Killing symmetries are essentially obliterated by the instanton corrections (with one partial exception, see the case of the Ooguri-Vafa metric in subsection~\ref{ssec:O-V}), while its rotational Killing symmetry only gets  deformed by them into a non-Killing one of trans-rotational type. By virtue of our earlier general considerations, the existence of this twisted symmetry immediately casts a new geometric light on\,---\,as well as allows for a fresh principial derivation of\,---\,a series of results obtained by Neitzke and Alexandrov, Moore, Neitzke, and Pioline in \cite{Neitzke:2011za} and \cite{Alexandrov:2014wca}, respectively.

\section{Hyperk\"ahler spaces} \label{sec:HK-sp}

\subsection{Generalities} \hfill \medskip

A hyperk\"ahler manifold $M$ is a smooth $4m$ real-dimensional manifold endowed with a triplet of symplectic 2-forms $\omega_1$, $\omega_2$, $\omega_3$ which reduce the structure group of the tangent bundle from $GL(4m,\mathbb{R})$ to $Sp(m)$ (or some non-compact form thereof{\,\textemdash\,}the considerations of this paper apply in equal measure to the pseudo-hyperk\"ahler case)~\cite{MR887284}. Regarding \mbox{$\omega_1$, $\omega_2$, $\omega_3$} as sections of \mbox{$\Lambda^2 \hp T^*M \subset \text{Hom}(TM,T^*M)$} we can define
\begin{equation} \label{i=oo}
I_1 = \omega_3^{-1} \omega_2, \ I_2 = \omega_1^{-1} \omega_3, \ I_3 = \omega_2^{-1} \omega_1 \in \text{End}(TM)
\end{equation}
and the structure group condition can be reformulated as the requirement that these satisfy the algebra of imaginary quaternions with respect to the composition law on $\text{End}(TM)$, that is \mbox{$I_1^2 = I_2^2 = I_3^2 = I_1I_2I_3 = - 1$}, with $1$ denoting the identity endomorphism. Complex structures are endomorphisms of the tangent bundle $TM$ and, by duality, for each complex structure we get a corresponding endomorphism on the cotangent bundle $T^*M$ which we continue to denote with the same symbol. To differentiate between its dual roles, in these notes we use the convention that complex structures act on vector fields from the left and on 1-forms from the right.

Any manifold with this structure is automatically Riemannian (or perhaps pseudo-Rie\-mannian), a hyper-Hermitian metric being induced by
\begin{equation}
g(X,Y) = - \pt \omega_1(X,I_1Y) = - \pt \omega_2(X,I_2Y) =  - \pt \omega_3(X,I_3Y)
\end{equation}
for any vector fields \mbox{$X,Y \in TM$}. This is known as the hyperk\"ahler metric.  Note that with this sign choice we have \mbox{$\omega_i(X,Y) = g(X,I_iY)$} for all $i=1,2,3$. 

Each of the three endomorphisms $I_1$, $I_2$, $I_3$ is covariantly constant with respect to the Levi-Civita connection corresponding to the metric $g$ and hence integrable in the sense of complex structures. In fact, one can associate to any point $\mathbf{u} = (x_1,x_2,x_3)$ on the unit 2-sphere in $\mathbb{R}^3$ an integrable complex structure \mbox{$\IU = x_1 I_1 + x_2 I_2 + x_3 I_3$} covariantly constant with respect to the Levi-Civita connection and satisfying $\IU^2 = - 1$. Hyperk\"ahler manifolds possess thus naturally a whole $S^2$ family of integrable complex structures compatible with the hyperk\"ahler metric. 

For later reference let us also record here the fact that for any vector field \mbox{$X \in TM$} the following Lie derivative formula holds:
\begin{equation} \label{HK_formula}
\mathcal{L}_{I_iX} \omega_j = - \, \varepsilon_{ijk} \pt \mathcal{L}_X \omega_k + \delta_{ij} \hp d(\iota_X g) \rlap{.}
\end{equation}
The indices $i,j,k$ run over the values $1,2,3$, $\varepsilon_{ijk}$ is the antisymmetric Levi-Civita symbol, and $\iota_X$ denotes the insertion operator.

\subsection{Holomorphic factorizations of complex subspace projectors}  \label{ssec:proj_hol_fact} \hfill \medskip

There are two basic ways to look at the sphere of complex structures, each emphasizing one of the sides of the isomorphism $S^2 \cong \mathbb{CP}^1$, both reflected in a choice of coordinates: extrinsic global Euclidean $\mathbb{R}^3$-coordinates on one hand, intrinsic local complex coordinates on the other. One way preserves the spherical symmetry but obscures the complex structure, the other breaks the spherical symmetry but renders the complex projective structure of the 2-sphere manifest. For the twistor-theoretic approach, where complex structures play a central role, the second description is the natural choice.

Consider an open covering of $S^2$ with two patches $N$ and $S$ obtained by removing the points $\mathbf{u}_S=(0,0,-1)$ and $\mathbf{u}_N=(0,0,1)$, respectively. On the components of this covering we define complex coordinate charts by means of the stereographic projection
\begin{equation} \label{stereo_map}
S^2 \longrightarrow \mathbb{CP}^1, 
\qquad
\mathbf{u} = (x_1, x_2, x_3) \mapsto
\begin{cases}
\displaystyle{\zeta = - \frac{x_1+ix_2}{1+x_3}} & \textrm{on $N$} \\[2ex]
\displaystyle{\tilde{\zeta} = - \frac{x_1-ix_2}{1-x_3}} & \textrm{on $S$} \rlap{.}
\end{cases}
\end{equation}
On the intersection $N \cap S$ the two complex coordinates are related by the biholomorphic transition relation $\tilde{\zeta} = 1/\zeta$. This exhibits $\mathbb{CP}^1$ as a complex manifold obtained by patching together two copies of $\mathbb{C}$. The antipodal map $\mathbf{u} \mapsto - \mathbf{u}$ on $S^2$ interchanges $N$ and $S$ and induces a fixed point-free anti-holomorphic involution $\zeta \mapsto \zeta^c \coloneqq - 1/\bar{\zeta}$ on $\mathbb{CP}^1$. 

Corresponding to this choice of complex atlas for the twistor 2-sphere we pick for the complexified bundle of hyperk\"ahler complex structures a mirror frame with generators \mbox{$\smash[t]{ I_+ = \frac{1}{2} (I_1 + iI_2) }$}, \mbox{$I_0 = I_3$}, \mbox{$\smash[t]{ I_- = - \frac{1}{2}(I_1 - iI_2) }$}, whose elements satisfy the property $\smash{ \bar{I}_m = (-)^m I_{-m} }$, which we will call an \textit{alternating reality condition}. To each of the two open charts we then associate a holomorphi\-cally-para\-me\-trized family of elements of $\text{End}_{\mathbb{C}}(TM)$ 
\begin{equation}
\begin{aligned}
P_N(\zeta) = P^{\hp 0,1}_{\Io} + i\hp \zeta I_- & \quad \textrm{for $\zeta \in N$} \\
P_S(\tilde{\zeta}) \hspace{1.5pt} = P^{\hp 1,0}_{\Io} - i \hp \tilde{\zeta} I_+ & \quad \textrm{for $\tilde{\zeta} \in S$} \mathrlap{.}
\end{aligned}
\end{equation}
where $\smash{ P^{\hp 1,0}_{\Io} = \frac{1}{2}(1-iI_0) }$ and $\smash{ P^{\hp 0,1}_{\Io} = \frac{1}{2}(1+iI_0) }$ are the complex subspace projectors for the complex structure $I_0$. These are related by \textit{antipodal conjugation}, which in these notes we define as the operation induced by \emph{antipodal mapping composed with complex conjugation}. That is to say, we have $\smash{ \overline{P_N(\zeta^c)} = P_S(\tilde{\zeta}) }$. The quaternionic properties of the complex structures imply that on their respective domains of definition they are idempotent 
\begin{equation}
P_N(\zeta)^2 = P_N(\zeta) 
\qquad\qquad
P_S(\tilde{\zeta})^2 = P_S(\tilde{\zeta}) \label{idempotence}
\end{equation}
(and therefore so are their complements $1 - P_N(\zeta)$ and $1 - P_S(\tilde{\zeta})$) and, in addition, on the intersection $N \cap S$ they satisfy the compatibility relations
\begin{equation} \label{PP=P}
[1 - P_S(\tilde{\zeta})] P_N(\zeta)  = 0  
\qquad\qquad 
[1 - P_N(\zeta)] P_S(\tilde{\zeta}) = 0 \mathrlap{.}
\end{equation}
Note that the equations on each line are interchanged by antipodal conjugation, so it suffices to verify only one in each case. 

Furthermore, for every point $\mathbf{u} \in S^2$ let 
\begin{equation}
P^{1,0}_{\Iu} = \frac{1}{2}[1 - i \IU]
\quad \text{and} \quad 
P^{0,1}_{\Iu} = \frac{1}{2}[1 + i\IU]
\end{equation}
be the $(1,0)$ respectively $(0,1)$ complex subspace projectors corresponding to the hyperk\"ahler complex structure $\IU$. With respect to their eigenvalues the complexified tangent and cotangent bundles admit the direct sum decompositions $T_{\mathbb{C}}M = T^{1,0}_{\Iu}M \oplus T^{0,1}_{\Iu}M$ and, dually, $\smash{ T^*_{\mathbb{C}}M = T^{*1,0}_{\Iu}M \oplus T^{*0,1}_{\Iu}M }$. 

The remarkable feature which arises and which sits at the core of the twistor space approach to hyperk\"ahler geometry is that the choice of complex atlas for $\smash{ \mathbb{CP}^1 }$ translates into certain algebraic decomposition properties of these projectors in terms of the holomorphically-parametrized ones. More precisely, we have
\begin{itemize}
\item[1.] linear decompositions
\begin{equation}
P^{0,1}_{\Iu} = \rho_NP_N(\zeta) + \rho_SP_S(\tilde{\zeta})= \rho_N [\overline{1-P_N(\zeta)\vphantom{\tilde{\zeta}}}] + \rho_S[\overline{1-P_S(\tilde{\zeta})}] \label{lin_decomp_proj}
\end{equation}
\item[2.] holomorphic factorizations
{\allowdisplaybreaks
\begin{equation}
P^{0,1}_{\Iu} = 
	\begin{cases} 
     	\rho_N P_N(\zeta)[\overline{1 - P_N(\zeta)\vphantom{\dot{i}}}] & \textrm{for $\mathbf{u} \in N$} \\[2pt]
	\hspace{0.5pt} \rho_S \hspace{2.8pt} P_S(\tilde{\zeta}) \hspace{1pt} [\overline{1 \hspace{1pt} - \hspace{1pt} P_S(\tilde{\zeta})}] & \textrm{for $\mathbf{u} \in S$} 
	\end{cases} \label{factorization_Ps}
\end{equation}
}
\end{itemize}
where, by definition, \mbox{$\rho_N = (1+|\zeta|^2)^{-1} = \frac{1}{2}(1+x_3)$} and \mbox{$\smash[t]{ \rho_S = (1+|\tilde{\zeta}|^2)^{-1} = \frac{1}{2}(1-x_3) }$}. The linear decomposition formulas follow directly from the definitions. The holomorphic factorization formulas, on the other hand, encode in addition quaternionic algebra properties of the hyperk\"ahler complex structures. 

Let us consider now the projector $P_N(\zeta)$ more closely. Observe that its list of algebraic properties includes the following items: 
{\allowdisplaybreaks
\begin{equation}
\begin{aligned}
P^{1,0}_{\Iu} P_N(\zeta) & = 0 &\qquad\qquad& [1-P_N(\zeta)] P^{0,1}_{\Iu} = 0 \\[0pt] 
P_N(\zeta) P^{1,0}_{\Io} & = 0 &\qquad\qquad& P^{0,1}_{\Io}[1-P_N(\zeta)] = 0 \mathrlap{.}
\end{aligned}
\end{equation}
}%
The ones in the first line are direct corollaries of the first factorization formula (and its complex conjugation), and the remaining ones follow easily from the definitions. Based on these properties and again, crucially, on the first factorization formula we can then see that  the actions of $P_N(\zeta)$ and its complementary projector \mbox{$1-P_N(\zeta)$} on the complexified tangent and cotangent bundles are described by the following two split short exact sequences 
\begin{equation*}
\begin{tikzcd}[column sep=3em]
0 \rar & T^{1,0}_{\Io}M \rar[shift left] \rar[shift right, leftarrow, "1-P_N(\zeta)"'] & T_{\mathbb{C}}M \rar[shift left, "P_N(\zeta)"] \rar[shift right, leftarrow] & T^{0,1}_{\Iu}M \rar & 0 \\[-16pt]
0 \rar & T^{*0,1}_{\Io}\npt M \rar[shift left] \rar[shift right, leftarrow, "P_N(\zeta)"'] & T^*_{\mathbb{C}}M \rar[shift left, "1-P_N(\zeta)"] \rar[shift right, leftarrow] & T^{*1,0}_{\Iu}M \rar & 0 
\end{tikzcd}
\end{equation*}
where, in accordance with our previously stated convention, the two linear operators are viewed alternatively as elements of
$\text{End}(T_{\mathbb{C}}M)$ and of $\text{End}(T^*_{\mathbb{C}}M)$, respectively. 

At the antipodes we have for $P_S(\tilde{\zeta})$ and its complementary projector \mbox{$1-P_S(\tilde{\zeta})$} the analogous pair of split short exact sequences
\begin{equation*}
\begin{tikzcd}[column sep=3em]
0 \rar & T^{0,1}_{\Io}M \rar[shift left] \rar[shift right, leftarrow, "1-P_S(\tilde{\zeta})"'] & T_{\mathbb{C}}M \rar[shift left, "P_S(\tilde{\zeta})"] \rar[shift right, leftarrow] & T^{0,1}_{\Iu}M \rar & 0 \\[-16pt]
0 \rar & T^{*1,0}_{\Io}M \rar[shift left] \rar[shift right, leftarrow, "P_S(\tilde{\zeta})"'] & T^*_{\mathbb{C}}M \rar[shift left, "1-P_S(\tilde{\zeta})"] \rar[shift right, leftarrow] & T^{*1,0}_{\Iu}M \rar & 0 \mathrlap{.}
\end{tikzcd}
\end{equation*}

In addition to the projectors $P_N(\zeta)$ and $P_S(\tilde{\zeta})$ defined on $N$ and $S$ we introduce also a transition element on the intersection $N \cap S$ by
\begin{align} \label{def_I(zeta)}
I(\zeta) & = i\hp[P_S(\tilde{\zeta}) - P_N(\zeta)] \\
& = \frac{I_+}{\zeta\ } + I_0 + \zeta I_- \rlap{.} \nonumber
\end{align}
By resorting for instance to the algebraic identities \eqref{idempotence} and \eqref{PP=P} we promptly see that the new operator is nilpotent, that is, 
\begin{equation}
I(\zeta)^2 = 0
\end{equation}
for any $\zeta \in \mathbb{C}^{\times}$. Remarkably, we can also claim for it a corresponding holomorphic factorization formula, namely,
\begin{equation} \label{factorization_Q}
P^{0,1}_{\Iu}  = - \, \rho \pt I(\zeta)\overline{I(\zeta)} \quad \textrm{for $\mathbf{u} \in N \cap S$}
\end{equation}
with $\rho = \rho_N\rho_S = (1+|\zeta|^2)^{-1}(1+|\tilde{\zeta}|^2)^{-1} = \frac{1}{4}(1-x_3^2)$. The dual actions of $I(\zeta)$ on the complexified tangent and cotangent bundles are described by the two short exact sequences 
\begin{equation*}
\begin{tikzcd}[column sep=normal]
0 \rar & T^{0,1}_{\scriptscriptstyle \Iu}M \rar & T_{\mathbb{C}}M \rar["I(\zeta)"] & T^{0,1}_{\Iu}M \rar & 0 \\[-20pt]
0 \rar & T^{*1,0}_{\scriptscriptstyle \Iu}M \rar & T^*_{\mathbb{C}}M \rar["I(\zeta)"] & T^{*1,0}_{\Iu}M \rar & 0 \mathrlap{.}
\end{tikzcd}
\end{equation*}
%The identity between kernels and images is a reflection of nilpotence. Surjectivity is again a consequence of the factorization formula. 

Finally, let us record also the following useful derivation formula:
\begin{equation} \label{deriv_formula}
\partial_{\hp\smash{\mathbb{CP}^1}} P^{0,1}_{\Iu} = -i \, \overline{I(\zeta)} \, \rho \hp \frac{d\zeta}{\zeta} \rlap{.}
\end{equation}
Observe in particular that $\rho \hp \frac{d\zeta}{\zeta}$ is the $(1,0)$ part of the real-valued 1-form $-\frac{1}{2}\pt dx_3$ on the twistor $\smash{ S^2 }$, with $x_3$ being the height function. The formula can be easily continued analytically to include the ``north" and ``south" poles of $S^2$.

\subsection{The Cauchy-Riemann equations for a generic complex structure} \hfill \medskip

To understand the significance of these factorizations consider now for a moment the Dolbeault operator on $M$ with respect to a generic complex structure $\IU$. Given a function $f \in \mathscr{A}^0(M,\mathbb{C})$, we have by definition $\bar{\partial}_{\Iu} f = df P^{0,1}_{\Iu}$. The function $f$ is holomorphic with respect to $\IU$ on a domain on $M$ if and only if $\bar{\partial}_{\Iu} f = 0$ on that domain. %We have the following 

\begin{proposition}[The Cauchy-Riemann equations on $M$ for the complex structure $\IU$] \label{CR_M}
Let $M$ be a hyperk\"ahler manifold. A function $f \in \mathscr{A}^0(M,\mathbb{C})$ is holomorphic on a domain on $M$ with respect to a complex structure $\IU$ with
\begin{flushleft}
\begin{tabular}{llcr}
\hspace{10pt} & $\mathbf{u} \in N$ & $\Leftrightarrow$ & $df P_N(\zeta) = 0$ \\[2pt]
& $\mathbf{u} \in S$ & $\Leftrightarrow$ & $df P_S(\tilde{\zeta}) = 0$ \\[2.5pt]
& $\mathbf{u} \in N \cap S$ & $\Leftrightarrow$ & $df I(\zeta) = 0$
\end{tabular}
\end{flushleft}
on that domain.
\end{proposition}

\begin{proof}

To prove the direct implication of the first statement consider a point $\mathbf{u} \in N$ and suppose $\bar{\partial}_{\Iu} f = 0$. By the first factorization formula \eqref{factorization_Ps} we have $df P_N(\zeta) \in \text{ker}\hp [\overline{1-P_N(\zeta)}] \allowbreak \cong T^{*1,0}_{\Io}$. On the other hand, by the properties of $P_N(\zeta)$ it is clear that ${df P_N(\zeta) \in \text{im} \hp [P_N(\zeta)]} \allowbreak \cong T^{*0,1}_{\Io}$. The only way in which these two conditions can be simultaneously satisfied is if $df P_N(\zeta) = 0$. The converse implication is an immediate consequence of the same factorization formula \eqref{factorization_Ps}. The second statement follows from a similar argument. 

Suppose now $\mathbf{u} \in N \cap S$. If $\bar{\partial}_{\Iu} f = 0$ then by the first two parts we must have both $df P_N(\zeta) = 0$ and $df P_S(\tilde{\zeta}) = 0$, and so $df I(\zeta) = 0$. Conversely, if this holds, then by the factorization formula \eqref{factorization_Q} $\bar{\partial}_{\Iu} f = 0$ must hold as well. 
\end{proof}

\subsection{Hyperk\"ahler integrability} \hfill \medskip

In a few situations that we will encounter we will find it convenient to work in a local coordinate frame on $M$ holomorphic with respect to $I_0$. The choice of $I_0$ as manifest complex structure must be understood in close connection with our  choice of complex coordinate atlas on $\mathbb{CP}^1$, and should be seen primarily as a practical device and not necessarily as an indication of privileged status among other hyperk\"ahler complex structures (although, as we shall see, in some instances this can well be the case). Our choice of complex atlas for $\mathbb{CP}^1$ favors in (the complexification of) the quaternionic subbundle $\text{Span}(1,I_1,I_2,I_3) \subset \text{End}(TM)$ not just a complex structure, but in fact an entire basis, namely the one generated by \mbox{$\smash{ P^{1,0}_{\Io} }$, $\smash{ P^{0,1}_{\Io} }$, $I_+$, $I_-$}. From the quaternionic algebra we have $\smash{ P^{1,0}_{\Io}I_+ = I_+P^{0,1}_{\Io} = 0 }$, so in such a coordinate frame the elements of this basis take the form
{\allowdisplaybreaks
\begin{equation}
\begin{aligned}
P^{1,0}_{\Io} & = \frac{\partial}{\partial x^{\mu}} \otimes dx^{\mu} &\qquad& I_+ = (I_+)^{\bar{\mu}}{}_{\nu} \pt \frac{\partial}{\partial x^{\bar{\mu}}} \otimes dx^{\nu}  \\
P^{0,1}_{\Io} & = \frac{\partial}{\partial x^{\bar{\mu}}} \otimes dx^{\bar{\mu}} && I_- = (I_-)^{\mu}{}_{\bar{\nu}} \pt \frac{\partial}{\partial x^{\mu}} \otimes dx^{\bar{\nu}} \mathrlap{.}
\end{aligned}
\end{equation}
}% 
From the identity $\smash[b]{ I_-I_+ = P^{1,0}_{\Io} }$ we get, moreover, the algebraic constraint
\begin{equation}
(I_-)^{\mu}{}_{\bar{\rho}} (I_+)^{\bar{\rho}}{}_{\nu} = \delta^{\mu}{}_{\nu} \mathrlap{.} \label{I+I-}
\end{equation}
The covariant constancy property of hyperk\"ahler complex structures imposes additional differential constraints which we can write as follows: 
\begin{equation} \label{HK_integrability}
\begin{aligned}
\partial_{\mu} (I_+)^{\bar{\kappa}}{}_{\nu} & = \partial_{\nu} (I_+)^{\bar{\kappa}}{}_{\mu} \\[3pt]
(I_+)^{\bar{\eta}}{}_{\mu} \pt \partial_{\bar{\eta}}(I_+)^{\bar{\kappa}}{}_{\nu} & = (I_+)^{\bar{\eta}}{}_{\nu} \pt \partial_{\bar{\eta}}(I_+)^{\bar{\kappa}}{}_{\mu} \mathrlap{.}
\end{aligned}
\end{equation}
Indeed, notice that if we replace by hand in these formulas the derivatives with Levi-Civita covariant derivatives $\nabla$ corresponding to the hyperk\"ahler metric we obtain in view of the fact that $\nabla I_+ = 0$ identically true equations. The Christoffel symbols can then be dropped out from these equations due to their Hermiticity and index symmetry properties, which leaves us with the expressions above. 

Given two conjugated holomorphic respectively antiholomorphic coordinate co\-frames on $T^{*1,0}_{\Io}M$ and $T^{*0,1}_{\Io}\npt M$, then for any $\mathbf{u} \in S^2$ each of their elements can be decomposed uniquely into $(1,0)$ and $(0,1)$ components with respect to the complex structure $\IU$: 
\begin{equation} \label{dx_splitting}
dx^{\mu} = \theta^{\mu}_{\plus} + \theta^{\mu}_{\minus}
\qquad\text{and}\qquad
dx^{\bar{\mu}} = \theta^{\bar{\mu}}_{\plus} + \theta^{\bar{\mu}}_{\minus} \mathrlap{.}
\end{equation}
The resulting forms can be interpreted as \textit{soldering forms} providing isomorphisms between the tangent subspaces holomorphic or anti-holomorphic with respect to $\IU$ on one hand, and with respect to $I_0$ on the other. 

To see this, consider for example $\theta^{\mu}_{\plus}$.  From the complex conjugate of the second decomposition formula \eqref{lin_decomp_proj} we have explicitly
\begin{equation} \label{vielbein}
\theta^{\mu}_{\plus} = \rho_N [\hp dx^{\mu} - i\hp \zeta \hp (I_-)^{\mu}{}_{\bar{\nu}} \hp dx^{\bar{\nu}}] \rlap{.}
\end{equation}
By the covariant constancy of the hyperk\"ahler complex structures it follows then that this satisfies the Cartan-Maurer equation
\begin{equation}
d\theta^{\mu}_{\plus} + \Gamma^{\mu}{}_{\nu} \!\wedge \theta^{\nu}_{\plus} = 0
\end{equation}
where $\Gamma^{\mu}{}_{\nu} = \Gamma^{\mu}{}_{\nu\rho} \pt dx^{\rho}$ is the $(1,0)$ part of the complexified Levi-Civita connection (since $M$ endowed with $I_0$ is K\"ahler, all Christoffel symbols of mixed type vanish). Thus, $\theta^{\mu}_{\plus}$ defines an isomorphism
\begin{equation*}
\begin{tikzcd}[row sep=-5pt, column sep=small]
T^{1,0}_{\Iu}M \vert_x \ar[r,"\sim"] & T^{1,0}_{\Io}M \vert_x \\
X \arrow[r,mapsto] & \displaystyle \iota_X\theta^{\mu}_{\plus} \frac{\partial}{\partial x^{\mu}} \rlap{.}
\end{tikzcd}
\end{equation*}
Similar arguments hold in turn for $\theta^{\mu}_{\minus}$, $\theta^{\bar{\mu}}_{\plus}$ and $\theta^{\bar{\mu}}_{\minus}$.

\section{Twistor spaces} \label{sec:Tw-sp}

\subsection{Generalities} \hfill \medskip

The origins of the twistor space theory of hyperk\"ahler manifolds go back to the work of Penrose, who famously showed that the geometry of anti-self-dual four-manifolds can be naturally encoded into the complex geometry of a twistor space of one complex dimension higher \cite{MR0439004}. Penrose's construction was subsequently extended by Salamon \cite{MR0439004,MR860810} and independently by B\'erard-Bergery (see Theorem 14.9 in \cite{MR867684}) to higher-dimensional analogues of anti-self-dual four-manifolds, which turned out to be manifolds with a quaternionic structure. Hyperk\"ahler manifolds form a subclass of these. The corresponding twistor construction was investigated \textit{per se} by Hitchin, Karlhede, Lindstr\"om and Ro\v{c}ek in \cite{MR877637}. 

Let $M$ continue to denote a hyperk\"ahler manifold. The central idea of the twistor approach is to ``unfurl" the $S^2$ family of complex structures of $M$ and incorporate them into a single holomorphic structure on a larger manifold, the twistor space $\mathcal{Z}$. From a purely differential geometric point of view, $\mathcal{Z} = M \times S^2$. An almost complex structure on $\mathcal{Z}$ is defined by combining diagonally on the tangent space $TM |_x \oplus TS^2 |_{\mathbf{u}}$ at any  point $(x,\mathbf{u}) \in \mathcal{Z}$ the action of the hyperk\"ahler complex structure $\IU$ with that of the natural complex structure on $S^2 \cong \mathbb{CP}^1$, $I_{\smash{\mathbb{CP}^1}}$. This almost complex structure can be shown to be integrable, and so $\mathcal{Z}$ is a complex manifold. The endomorphism $(x,\mathbf{u}) \mapsto (x,-\mathbf{u})$ induced by antipodal conjugation on $S^2$ defines moreover on $\mathcal{Z}$ an anti-holomorphic involution or, equivalently, a real structure as it simultaneously inverts the signs of both $\IU$ and $I_{\smash{\mathbb{CP}^1}}$\,---\,and consequently that of the complex structure of $\mathcal{Z}$.  

Let $\pi : \mathcal{Z} \rightarrow \mathbb{CP}^1$ and $p : \mathcal{Z} \rightarrow M$ be the natural projections. The projection onto the $\mathbb{CP}^1$ factor defines a holomorphic fibration whose fibers $\pi^{-1}(\mathbf{u})$ for any $\mathbf{u} \in S^2 \cong \mathbb{CP}^1$ are biholomorphic to copies of $M$ endowed with the complex structure $\IU$. In addition, each fiber carries a complex symplectic structure compatible with its complex structure. To see how that occurs consider first the fibers above the hyperk\"ahler complex structures $I_0$ and $-I_0$ (recall that $I_0$ is the same as $I_3$). If we define the complex-linear combinations
%\begin{equation}
%\omega_+ = \frac{1}{2}(\omega_1 + i \hp \omega_2)
%\qquad
%\omega_0 = \omega_3
%\qquad
%\omega_- = - \frac{1}{2}(\omega_1 - i \hp \omega_2)
%\end{equation}
\mbox{$\smash{ \omega_+ = \frac{1}{2}(\omega_1 + i \hp \omega_2) }$}, \mbox{$\omega_0 = \omega_3$}, \mbox{$\smash{ \omega_- = - \frac{1}{2}(\omega_1 - i \hp \omega_2) }$}
then a set of complex symplectic structures corresponding to these fibers is given by the transversal forms $\omega_+$ and $\omega_-$, respectively. The proof amounts in essence to showing that $\omega_+$ is a closed holomorphic type $(2,0)$ form with respect to $I_0$. Closure is evident. Then based on the quaternionic algebra we have \mbox{$\smash[b]{ \omega_+(X,P^{0,1}_{\Io}Y) = g(X,I_+P^{0,1}_{\Io}Y) = 0 }$} for any vector fields $X, Y \in TM$. Together with antisymmetry and closure this yields the remainder of the statement. The statement for the fiber over $-I_0$ follows by complex conjugation. Finally, for the remaining fibers above complex structures $\IU \neq I_0, -I_0$ such a complex symplectic structure is given by
\begin{equation}
\omega(\zeta) = \frac{\omega_+}{\zeta\ } + \omega_0 + \zeta \hp \omega_-
\end{equation}
where $\zeta \in \mathbb{C}^{\times}$ corresponds to $\mathbf{u} \in N \cap S \subset S^2$ through the stereographic map in the usual way. This is a closed 2-form on $M$ and, moreover, by the factorization formula \eqref{factorization_Q} and nilpotence of $I(\zeta)$ we have $\omega(\zeta)(X,P^{0,1}_{\Iu}Y) = g(X,I(\zeta)P^{0,1}_{\Iu}Y) = 0$ for any vector fields $X,Y \in TM$. Antisymmetry and closure eventually imply that $\omega(\zeta)$ is a holomorphic type $(2,0)$ form with respect to $\IU$. %On the twistor space these arguments imply the 
On the twistor space these various facts are summarized concisely by the statement of
existence of a canonical global holomorphic section of the bundle $\Lambda^2T_F^* \allowbreak \otimes \pi^*\mathcal{O}(2)$ over $\mathcal{Z}$, where $T_F =\text{ker}\pt (d\pi)$ is the holomorphic tangent bundle along the fibers and the second factor is the pullback on $\mathcal{Z}$ of the $\mathcal{O}(2)$ bundle over $\mathbb{CP}^1$. The fibers $\pi^{-1}(\mathbf{u})$ can then be viewed as copies of $M$ endowed with a complex structure $\IU$ and a compatible complex symplectic structure induced through restriction by this global section. 

Unlike the projection $\pi$, the projection $p$ onto the $M$ factor is not holomorphic in general. However, for every point $x \in M$, $p^{-1}(x)$ is a complex analytic submanifold of $\mathcal{Z}$ which is called the horizontal twistor line through $x$. The restricted projection \mbox{$\pi |_{p^{-1}(x)} : p^{-1}(x) \rightarrow \mathbb{CP}^1$} gives a canonical identification of $p^{-1}(x)$ with $\mathbb{CP}^1$. One also has the notion of real twistor line, designating a holomorphic section of $\pi$ which commutes with the real structure on $\mathcal{Z}$. The two notions coincide: a twistor line is horizontal if and only if it is real. The normal bundle to any of these twistor lines is isomorphic to $\mathbb{C}^{2m} \otimes \pi^*\mathcal{O}(1)$. Note, finally, that through every point of $\mathcal{Z}$ pass a unique fiber and horizontal twistor line. 

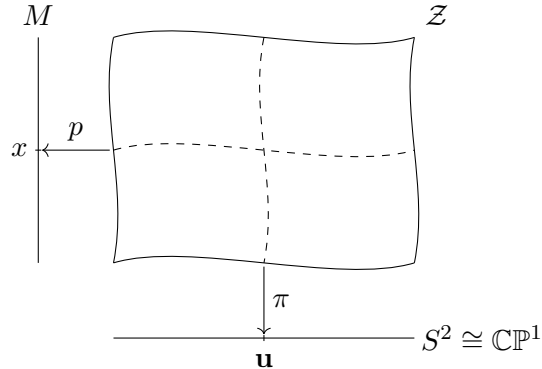
\begin{figure}[ht]
\begin{tikzpicture}[scale = 1]

\draw (-2,2) .. controls (-1,2.3) and (1,1.7) .. (2,2);
\draw (-2,-1) .. controls (-1,-0.7) and (1,-1.3) .. (2,-1);
\draw[dashed] (-2,0.5) .. controls (-1,0.8) and (1,0.2) .. (2,0.5);

\draw (-2,2) .. controls (-2.2,1) and (-1.8,0) .. (-2,-1);
\draw (2,2) .. controls (1.8,1) and (2.2,0) .. (2,-1);
\draw[dashed] (0,2) .. controls (-0.2,1) and (0.2,0) .. (0,-1);

\draw (-2,-2) -- (2,-2);
\draw  (-3,2) -- (-3,-1);

\draw[->] (-2.05,0.5) -- (-2.95,0.5);
\draw[->] (0,-1.05) -- (0,-1.95);

\draw (-3.04,0.5) -- (-2.96,0.5);
\draw (0,-1.96) -- (0,-2.04);

\node at (-3,2.3) {$M$};
\node at (2.9,-2) {$S^2 \cong \mathbb{CP}^1$};
\node at (2.3,2.3) {$\mathcal{Z}$};

\node at (0,-2.3) {$\mathbf{u}$};

\node at (0.23,-1.5) {$\pi$};

\node at (-3.25,0.5) {$x$};

\node at (-2.5,0.75) {$p$};

\end{tikzpicture}
\caption{A schematic representation of the twistor space projections, with the fiber over a point $\smash{ \mathbf{u} \in \mathbb{CP}^1 }$ and the horizontal twistor line through a point $x \in M$ depicted by dotted lines.}
\end{figure}

The complex structure, the real structure, the holomorphic fibration structure, the fiber\-wise-supported holomorphic $(2,0)$ form and the horizontal twistor line normal bundle data form together a complete set of twistor space data. From it, the hyperk\"ahler manifold can be retrieved as the parameter space of horizontal twistor lines. The precise formulation of this statement is given in Theorem 3.3 of \cite{MR877637}. 

The twistor space comes also equipped with a natural metric induced with respect to the product structure by the hyperk\"ahler metric on $M$ and the Fubini-Study metric on $\mathbb{CP}^1$. This combines with the complex structure on $\mathcal{Z}$ to give the $(1,1)$ form $\varpi = \omega_F + \omega_{\hp\smash{\mathbb{CP}^1}}$, where $\omega_F$ and $\omega_{\hp\smash{\mathbb{CP}^1}}$ are the pullbacks on $\mathcal{Z}$ of the 2-form $\omega(\mathbf{u}) = x_1\hp\omega_1 + x_2\hp \omega_2 + x_3\hp \omega_3$ on the fiber $F= \pi^{-1}(\mathbf{u})$ and the Fubini-Study symplectic form on $\mathbb{CP}^1$, respectively. As known since \cite{MR623721}, twistor spaces admit K\"ahler metrics only rather accidentally. Instead, Kaledin and Verbitsky point out in \cite[Proposition 4.5]{MR1669956}, the twistor metric is naturally \textit{balanced} in the sense of Michelsohn \cite{MR688351}. That is, 
\begin{equation}
d_{\mathcal{Z}} (\varpi^{2m}) = 0 
\end{equation}
where $2m = \dim_{\hp \mathbb{C}}\!M$ and $d_{\mathcal{Z}}$ is the exterior derivative on $\mathcal{Z}$. Indeed, in view of the decomposition $d_{\mathcal{Z}} = d_F + d_{\hp\smash{\mathbb{CP}^1}}$, where $d_F$ and $d_{\hp\smash{\mathbb{CP}^1}}$ stand for the exterior derivative along the local fiber (\mbox{$\pt \cong M$}) respectively the local horizontal twistor line (\mbox{$\pt \cong \mathbb{CP}^1$}), we have $d_{\mathcal{Z}}\varpi = d_F\omega_F + d_{\hp\smash{\mathbb{CP}^1}}\omega_F + d_F \omega_{\hp\smash{\mathbb{CP}^1}} + d_{\hp\smash{\mathbb{CP}^1}}\omega_{\hp\smash{\mathbb{CP}^1}}$.  The first term vanishes by the closure of the hyperk\"ahler 2-forms and the last two by the definition of $\omega_{\hp\smash{\mathbb{CP}^1}}$, and we are left with $d_{\mathcal{Z}}\varpi = d_{\hp\smash{\mathbb{CP}^1}}\omega_F$. Since $\omega(\mathbf{u})(X,Y) = g(X,\IU Y)$ for any $X,Y \in TM$, by resorting to the differentiation formula \eqref{deriv_formula} we get immediately
\begin{equation}
\partial_{\hp\smash{\mathbb{CP}^1}} \omega(\mathbf{u}) = - \hp 2 \, \overline{\omega(\zeta)} \wedge \rho \hp \frac{d\zeta}{\zeta}
\end{equation}
which shows in particular that $\smash[t]{ \partial_{\hp\smash{\mathbb{CP}^1}} \omega_F \in \pi^*\mathscr{A}^{1,0}(\mathbb{CP}^1) \otimes p^* \mathscr{A}^{0,2}_{\Iu}(M) }$ if $F=\pi^{-1}(\mathbf{u})$. On another hand, by the definitions, $\smash{ \omega_{\hp\smash{\mathbb{CP}^1}} \in \pi^* \mathscr{A}^{1,1}(\mathbb{CP}^1) }$ and $\smash{ \omega_F \in p^* \mathscr{A}^{1,1}_{\Iu}(M) }$ for $F=\pi^{-1}(\mathbf{u})$. Now let us write
\begin{equation}
d_{\mathcal{Z}} (\varpi^{2m}) = 2m \pt \varpi^{2m-1} \!\wedge d_{\mathcal{Z}}\varpi = 2m \hp (\hp \omega_F + \omega_{\hp\smash{\mathbb{CP}^1}} \npt)^{2m-1} \!\wedge {(\partial_{\hp\smash{\mathbb{CP}^1}} \omega_F + \bar{\partial}_{\hp\smash{\mathbb{CP}^1}} \omega_F)} \rlap{.}
\end{equation}
 % I put some {...} around the last parenthesis to prevent it from breaking
The terms in the binomial expansion of the first factor (not counting the numerical $2m$ factor) containing at least one $\omega_{\hp\smash{\mathbb{CP}^1}}$ vanish either by themselves or when wedged against the last factor due to the oversaturation of the $\mathbb{CP}^1$ degrees, while the $\omega_{\hp\smash{\mathbb{CP}^1}}$-free term, $\smash{ \omega_F^{\hp 2m-1} }$, vanishes when wedged against the last factor due to the oversaturation of the $M$ degrees, and so the claim is proved.

\subsection{Some properties of the twistor space Dolbeault operator} \hfill \medskip

The holomorphic tangent bundle along the fibers, $T_F$, is by definition the kernel of the differential map $d\pi$, that is
\begin{equation*}
\begin{tikzcd}[column sep=scriptsize]
0 \rar & T_F \rar & T_{\mathcal{Z}} \rar["d\pi",shorten >=-0.2em] & \pi^*T_{\hp\smash{\mathbb{CP}^1}} \rar & 0
\end{tikzcd}
\end{equation*}
is a short exact sequence of complex bundles over $\mathcal{Z}$, with dual sequence
\begin{equation*}
\begin{tikzcd}[column sep=scriptsize]
0 \rar & \pi^*T^*_{\hp\smash{\mathbb{CP}^1}} \rar & T^*_{\mathcal{Z}} \rar & T^*_F \rar & 0\mathrlap{.}
\end{tikzcd}
\end{equation*}
Even though $\mathcal{Z} = M \times \mathbb{CP}^1$, the holomorphic cotangent bundle $T^*_{\mathcal{Z}}$ is an extension rather than a direct sum of complex bundles, reflecting the fact that as a complex manifold $\mathcal{Z}$ is non-trivial. Only in a pointwise sense can the holomorphic cotangent space to $\mathcal{Z}$ be expressed as a direct sum,
\begin{equation}
T^*_{\mathcal{Z}} \big|_{(x,\mathbf{u})} \cong p^* T^{*1,0}_{\Iu} M \big|_x \oplus \pi^* T^*_{\hp\smash{\mathbb{CP}^1}} \big|_{\mathbf{u}},
\end{equation}
with the two terms corresponding to the local fiber and the local horizontal twistor line, respectively. Accordingly, any complex differential form $\alpha \in \mathscr{A}^{1,0}(\mathcal{Z})$ decomposes locally into a component along the local fiber and one along the local horizontal twistor line: $\alpha = \alpha_F + \alpha_{\hp\smash{\mathbb{CP}^1}}$. Let \mbox{$\mathscr{A}^{\hp r,s}_F(\mathcal{Z}) = \Gamma(\mathcal{Z};\Lambda^{r,s\hp}T^*_F)$} be the sheaf of $C^{\infty}$-{\hp}sections of forms of type $(r,s)$ supported on the fibers of $\mathcal{Z}$. For \mbox{$F = \pi^{-1}(\mathbf{u})$} we have \mbox{$T^*_F \cong p^* T^{*\smash[t]{1,0}}_{\Iu} M$} and $\mathscr{A}^{\hp r,s}_F(\mathcal{Z}) = p^*\mathscr{A}^{\hp r,s}_{\Iu}(M)$. That is, the component along any given fiber of a $(r,s)$ form on $\mathcal{Z}$ is a $(r,s)$ form on that fiber with respect to its specific hyperk\"ahler complex structure. The de Rham operator on $\mathcal{Z}$ splits also locally along the fibration structure into \mbox{$d_{\mathcal{Z}} = d_F + d_{\hp\smash{\mathbb{CP}^1}}$.}\footnote{\pt Since the fibers are isomorphic to $M$, we will usually denote $d_F$ simply by $d$.} Furthermore, in accordance with the respective complex structures on the local horizontal twistor line and fiber we can split further $d_{\hp\smash{\mathbb{CP}^1}} = \partial_{\hp\smash{\mathbb{CP}^1}} + \bar{\partial}_{\hp\smash{\mathbb{CP}^1}}$ and, for $F = \pi^{-1}(\mathbf{u})$, $d_F = \partial_{\Iu} + \bar{\partial}_{\Iu}$. In contrast, due to the non-trivial character of the holomorphic fibration, the Dolbeault operato $\bar{\partial}_{\mathcal{Z}}$ does not split into a sum of Dolbeault operators along the local fibers and horizontal twistor lines. Not, that is, in general\,---\,because this may nevertheless happen in specific circumstances. The next lemma is illustrative of this dichotomic behavior.  
\begin{lemma} \label{Dolbeault_Z}
Let $M$ be a hyperk\"ahler manifold with twistor space $\mathcal{Z}$. The following formulas hold at generic points of $\mathcal{Z}$ situated on a fiber $F = \pi^{-1}(\mathbf{u})$:
\begin{itemize}
\setlength{\itemsep}{0pt}
\item[1.] For any function $f \in \mathscr{A}^0(\mathcal{Z},\mathbb{C})$, %we have
\begin{gather}
\bar{\partial}_{\mathcal{Z}} f = (\bar{\partial}_{\Iu} + \bar{\partial}_{\hp\smash{\mathbb{CP}^1}} \npt) f \label{del_bar_Z_fct} \\[2pt]
\partial_{\mathcal{Z}}\bar{\partial}_{\mathcal{Z}} f = \partial_{\Iu} \bar{\partial}_{\Iu} f + \partial_{\hp\smash{\mathbb{CP}^1}}\bar{\partial}_{\Iu} f - \bar{\partial}_{\hp\smash{\mathbb{CP}^1}} \partial_{\Iu} f + \partial_{\hp\smash{\mathbb{CP}^1}}\bar{\partial}_{\hp\smash{\mathbb{CP}^1}} f \mathrlap{.} \label{delZ_delZbar}
\end{gather}

\item[2.] For any complex differential form $\alpha \in \mathscr{A}^{1,0}(\mathcal{Z})$, %and with the choice \eqref{stereo_map} of inhomogeneous coordinate on the $\mathbb{CP}^1$ twistor base we have
\begin{equation} \label{d_bar_(1,0)}
\bar{\partial}_{\mathcal{Z}} \alpha = (\bar{\partial}_{\Iu} + \bar{\partial}_{\hp\smash{\mathbb{CP}^1}}) \hp \alpha -  i\pt \alpha_F \cramped{\overline{I(\zeta)} \wedge \rho} \pt \frac{d\zeta}{\zeta}\rlap{.}
\end{equation}
%where $\alpha_F$ denotes the component of $\alpha$ along the fiber.

\end{itemize}
\end{lemma} 

\begin{remark}
In the last formula we have assumed implicitly that $\mathbf{u} \in N \cap S$, but the expression given can be easily continued analytically to include the ``north" and ``south" poles of the twistor $S^2$. Also, to avoid the proliferation of pull-back symbols, here and throughout these notes such formulas are to be understood as being written in a local trivialization of the twistor space.
\end{remark}

\begin{proof}

The first formula of the lemma follows immediately by projecting $d_{\mathcal{Z}}f = d_Ff + d_{\hp\smash{\mathbb{CP}^1}}f$ to $\mathscr{A}^{1,0}(\mathcal{Z})$.

Next, we turn our attention to formula \eqref{d_bar_(1,0)}. Given a form $\alpha \in \mathscr{A}^{1,0}(\mathcal{Z})$, its fiberwise component \mbox{$\alpha_F \in \mathscr{A}_F^{1,0}(\mathcal{Z})$} is a $(1,0)$ form on each fiber with respect to the particular hyperk\"ahler complex structure corresponding to that fiber. Hence, for $F=\pi^{-1}(\mathbf{u})$ we have $\alpha_F P^{1,0}_{\Iu} = \alpha_F$. By differentiating this relation and resorting to the property \eqref{deriv_formula}~we~get
{\allowdisplaybreaks
\begin{equation}
\begin{aligned}
\partial_{\hp\smash{\mathbb{CP}^1}}\alpha_F & = \underbracket[0.4pt][4pt]{(\partial_{\hp\smash{\mathbb{CP}^1}}\alpha_F) P^{1,0}_{\Iu}}_{\text{$_{2,0}$}} -  \underbracket[0.4pt][4pt]{ i \pt \alpha_F\cramped{\overline{I(\zeta)} \wedge \rho} \pt \frac{d\zeta}{\zeta} }_{\text{$_{1,1}$}} \\[-2pt]
\bar{\partial}_{\hp\smash{\mathbb{CP}^1}}\alpha_F & = \underbracket[0.4pt][4pt]{(\bar{\partial}_{\hp\smash{\mathbb{CP}^1}}\alpha_F) P^{1,0}_{\Iu}}_{\text{$_{1,1}$}} - \underbracket[0.4pt][4pt]{ i \pt \alpha_F\cramped{I(\zeta) \wedge \rho} \pt \frac{d\bar{\zeta}}{\bar{\zeta}} }_{\text{vanishes}} \mathrlap{.}
\end{aligned}
\end{equation}
}%
Notice that the last term vanishes as \mbox{$\alpha_FI(\zeta) = \alpha_FP^{1,0}_{\Iu}I(\zeta) = 0$} by the factorization property \eqref{factorization_Q} and the nilpotence of $I(\zeta)$. Underneath each non-vanishing term on the right-hand side we indicate its Hodge type with respect to the complex structure on $\mathcal{Z}$. It is clear then that the $(1,1)$ component of $d_{\hp\smash{\mathbb{CP}^1}} \alpha_F$ is of the form 
\begin{equation}
d_{\hp\smash{\mathbb{CP}^1}} \alpha_F \hp |^{\phantom{+}}_{\text{$(1,1)$ on $\mathcal{Z}$}} = \bar{\partial}_{\hp\smash{\mathbb{CP}^1}}\alpha_F - i \pt \alpha_F\cramped{\overline{I(\zeta)} \wedge \rho} \pt \frac{d\zeta}{\zeta} \mathrlap{.}
\end{equation}
To compute $\bar{\partial}_{\mathcal{Z}}\alpha$ we use the fact that it is the $(1,1)$ component of the 2-form $d_{\mathcal{Z}}\alpha$. We have $d_{\mathcal{Z}}\alpha = (d_F+d_{\hp\smash{\mathbb{CP}^1}})(\alpha_F + \alpha_{\hp\smash{\mathbb{CP}^1}})$ and, moreover, 
{\allowdisplaybreaks
\begin{equation}
\begin{aligned}
d_F\alpha_F \hp |^{\phantom{+}}_{\text{$(1,1)$ on $\mathcal{Z}$}} & = \bar{\partial}_{\Iu} \alpha_F \\
d_F\alpha_{\hp\smash{\mathbb{CP}^1}} \hp |^{\phantom{+}}_{\text{$(1,1)$ on $\mathcal{Z}$}} & = \bar{\partial}_{\Iu} \alpha_{\hp\smash{\mathbb{CP}^1}} \\
d_{\hp\smash{\mathbb{CP}^1}}\alpha_{\hp\smash{\mathbb{CP}^1}} \hp |^{\phantom{+}}_{\text{$(1,1)$ on $\mathcal{Z}$}} & = \bar{\partial}_{\hp\smash{\mathbb{CP}^1}} \alpha_{\hp\smash{\mathbb{CP}^1}} \mathrlap{.}
\end{aligned}
\end{equation}
}%
Formula \eqref{d_bar_(1,0)} of the lemma follows then immediately. 

Finally, to show the remaining formula of the lemma, let us apply the formula we have just proved to the particular case when $\alpha = \partial_{\mathcal{Z}}f$, for an arbitrary function $f \in \mathscr{A}^0(\mathcal{Z},\mathbb{C})$. By the first formula, $\partial_{\mathcal{Z}}f = \partial_{\Iu}f + \partial_{\hp\smash{\mathbb{CP}^1}}f$ on a fiber $F=\pi^{-1}(\mathbf{u})$, and so, at any point on that fiber, 
{\allowdisplaybreaks
\begin{align}
\bar{\partial}_{\mathcal{Z}} \partial_{\mathcal{Z}} f & = (\bar{\partial}_{\Iu} + \bar{\partial}_{\hp\smash{\mathbb{CP}^1}})(\partial_{\Iu} f + \partial_{\hp\smash{\mathbb{CP}^1}} f) - i\pt \partial_{\Iu} f \hp \cramped{\overline{I(\zeta)} \wedge \rho} \pt \frac{d\zeta}{\zeta} \nonumber \\[-4pt]
& = \bar{\partial}_{\Iu}\partial_{\Iu} f + \bar{\partial}_{\Iu}\partial_{\hp\smash{\mathbb{CP}^1}} f + \bar{\partial}_{\hp\smash{\mathbb{CP}^1}}\partial_{\Iu} f + \bar{\partial}_{\hp\smash{\mathbb{CP}^1}}\partial_{\hp\smash{\mathbb{CP}^1}} f - i\pt d f \hp \cramped{\overline{I(\zeta)} \wedge \rho} \pt \frac{d\zeta}{\zeta} \mathrlap{.} \nonumber
\end{align}
}%
To obtain the second line we have used that $\partial_{\Iu} f \hp \overline{I(\zeta)} = df P^{1,0}_{\Iu} \hp \overline{I(\zeta)} = df \hp \overline{I(\zeta)}$. Based on the property \eqref{deriv_formula} one can show, moreover, that
\begin{equation} \label{deriv_commutator}
(\partial_{\hp\smash{\mathbb{CP}^1}}\bar{\partial}_{\Iu} + \bar{\partial}_{\Iu}\partial_{\hp\smash{\mathbb{CP}^1}}) f = i\pt d f \hp \cramped{\overline{I(\zeta)} \wedge \rho} \pt \frac{d\zeta}{\zeta}
\end{equation}
expressing a non-commutativity of derivatives, and so we obtain immediately
\begin{equation}
\bar{\partial}_{\mathcal{Z}} \partial_{\mathcal{Z}} f = \bar{\partial}_{\Iu}\partial_{\Iu} f - \partial_{\hp\smash{\mathbb{CP}^1}}\bar{\partial}_{\Iu} f + \bar{\partial}_{\hp\smash{\mathbb{CP}^1}}\partial_{\Iu} f + \bar{\partial}_{\hp\smash{\mathbb{CP}^1}}\partial_{\hp\smash{\mathbb{CP}^1}} f
\end{equation}
which is clearly equivalent to formula \eqref{delZ_delZbar} of the lemma. 
\end{proof}

\begin{remark} 

%We can use a similar argument to give an alternative proof of the integrability of the twistor complex structure on $\mathcal{Z}$. 
A similar argument can be used to prove the integrability of the complex structure on $\mathcal{Z}$. Note first that we have the following complementary formula to \eqref{d_bar_(1,0)} 
\begin{equation}
\partial_{\mathcal{Z}}\alpha = (\partial_{\Iu} + \partial_{\hp\smash{\mathbb{CP}^1}}) \hp \alpha + i\pt \alpha_F \cramped{\overline{I(\zeta)} \wedge \rho} \pt \frac{d\zeta}{\zeta}
\end{equation}
for any form $\alpha \in \mathscr{A}^{1,0}(\mathcal{Z})$ and at a generic point situated on a fiber $F = \pi^{-1}(\mathbf{u})$, as can be easily checked by verifying that $\partial_{\mathcal{Z}}\alpha + \bar{\partial}_{\mathcal{Z}}\alpha = d_{\mathcal{Z}}\alpha$. Taking in particular $\alpha = \partial_{\mathcal{Z}}f$ for some function $f \in \mathscr{A}^0(\mathcal{Z},\mathbb{C})$ and proceeding as above, we get
{\allowdisplaybreaks
\begin{align}
\partial_{\mathcal{Z}}^2 f & = (\partial_{\Iu} + \partial_{\hp\smash{\mathbb{CP}^1}})(\partial_{\Iu} f + \partial_{\hp\smash{\mathbb{CP}^1}} f) + i\pt \partial_{\Iu}f \hp \cramped{\overline{I(\zeta)} \wedge \rho} \pt \frac{d\zeta}{\zeta} \nonumber \\[-3pt]
& = \partial_{\Iu}\partial_{\hp\smash{\mathbb{CP}^1}} f + \partial_{\hp\smash{\mathbb{CP}^1}}\partial_{\Iu} f + i\pt df \hp \cramped{\overline{I(\zeta)} \wedge \rho} \pt \frac{d\zeta}{\zeta} \nonumber \\[3.5pt]
& = 0 \mathrlap{.} \nonumber
\end{align}
}%
The vanishing of the second line can be checked directly using the property \eqref{deriv_formula} or, alternatively and more easily, it can be inferred from the equation \eqref{deriv_commutator} by observing that $(d_F \partial_{\hp\smash{\mathbb{CP}^1}} + \partial_{\hp\smash{\mathbb{CP}^1}} d_F) f = 0$. Since $f$ is arbitrary, it follows that the twistor complex structure on $\mathcal{Z}$ is integrable.

\end{remark}

In view of formula \eqref{del_bar_Z_fct} of Lemma \ref{Dolbeault_Z}, for a function $f$ to be holomorphic on a domain on $\mathcal{Z}$, both the component along the local fiber and the component along the local horizontal twistor line of $\bar{\partial}_{\mathcal{Z}} f$ must vanish at every point of the domain. Proposition \ref{CR_M} translates then on the twistor space as follows: 

\begin{proposition}[The Cauchy-Riemann equations on the twistor space] \label{CR_Z}
Let $M$ be a hyperk\"ahler manifold with twistor space $\mathcal{Z}$, $f$ a function on $\mathcal{Z}$, and $V_N$, $V_S$, $V_{N \cap \hp S}$ domains on $\mathcal{Z}$ which project down on $\mathbb{CP}^1$ to $N$, $S$ and $N \cap S$  or subdomains of these, respectively. Then $f$ is holomorphic on
\begin{flushleft}
\begin{tabular}{llcr@{\hskip 0pt}l}
\hspace{10pt} & $V_N$ & $\Leftrightarrow$ & $df P_N(\zeta) = 0$ and $\bar{\partial}_{\hp\smash{\mathbb{CP}^1}} f = 0$ & \ on $V_N$; \\[2pt]
& $V_S$ & $\Leftrightarrow$ & $df P_S(\tilde{\zeta}) = 0$ and $\bar{\partial}_{\hp\smash{\mathbb{CP}^1}} f = 0$ & \ on $V_S$; \\[3pt]
& $V_{N \cap \hp S}$ & $\Leftrightarrow$ & $df I(\zeta) = 0$ and $\bar{\partial}_{\hp\smash{\mathbb{CP}^1}} f = 0$ & \ on $V_{N \cap \hp S}$.
\end{tabular}
\end{flushleft}
\end{proposition}

\noindent This shows in particular that holomorphic functions on $\mathcal{Z}$ must necessarily have a holomorphic dependence on the complex $\mathbb{CP}^1$ coordinate. For this reason in what follows we will often indicate this dependence explicitly. 

A particularly important application of formula \eqref{d_bar_(1,0)} of Lemma \ref{Dolbeault_Z} is given by 

\begin{lemma} \label{fiberwise-lemma}
Let \mbox{$\alpha \in \mathscr{A}^{1,0}(\mathcal{Z})$} be a 1-form on the twistor space $\mathcal{Z}$ of type $(1,0)$. The following statements are equivalent:
\begin{itemize}
\setlength\itemsep{0.4em}
\item[1.] $\bar{\partial}_{\mathcal{Z}}\alpha \in \mathscr{A}_F^{1,1}(\mathcal{Z})$, that is, $\bar{\partial}_{\mathcal{Z}}\alpha$ is entirely supported on the fibers.
\item[2.] Locally, $\alpha$ is of the form
\begin{equation}
\alpha = df(\zeta) I(\zeta) - i f(\zeta) \frac{d\zeta}{\zeta}
\end{equation}
for some function $f = f(\zeta) \in \mathscr{A}^0(\mathcal{Z},\mathbb{C})$ satisfying $\bar{\partial}_{\hp\smash{\mathbb{CP}^1}} f = 0$. 
\end{itemize}
\end{lemma}

\begin{proof}
We begin with the converse implication, $2 \Rightarrow 1$. In this case \mbox{$\alpha_F = df(\zeta) I(\zeta)$} and then by formula \eqref{d_bar_(1,0)} we have
%\begin{equation}
%\bar{\partial}_{\mathcal{Z}} \alpha = (\bar{\partial}_{\Iu} + \underbracket[0.4pt][4pt]{\bar{\partial}_{\hp\smash{\mathbb{CP}^1}}}_{\mathclap{\rule[6pt]{0pt}{0pt}\text{vanishes}}}) [df(\zeta)I(\zeta) - i f(\zeta) \frac{d\zeta}{\zeta}] - i \pt \cramped{ \underbracket[0.4pt][4pt]{df(\zeta)I(\zeta)\cramped{\overline{I(\zeta)} \wedge \rho}\rule[-5pt]{0pt}{0pt} }_{\mathrlap{\text{\rule[9pt]{0pt}{0pt}$ \hspace{-41pt} = -\, df(\zeta) P^{0,1}_{\Iu} = - \, \bar{\partial}_{\Iu} f$}}} \frac{d\zeta}{\zeta} } 
%\, = \bar{\partial}_{\Iu} \alpha_F 
%\end{equation}
\begin{equation}
\bar{\partial}_{\mathcal{Z}} \alpha = (\bar{\partial}_{\Iu} + \underbracket[0.4pt][4pt]{\bar{\partial}_{\hp\smash{\mathbb{CP}^1}}}_{\mathclap{\rule[6pt]{0pt}{0pt}\text{vanishes}}}) [df(\zeta)I(\zeta) - i f(\zeta) \frac{d\zeta}{\zeta}] - i \pt \cramped{ \underbracket[0.4pt][4pt]{df(\zeta)I(\zeta)\cramped{\overline{I(\zeta)} \wedge \rho}\rule[-5pt]{0pt}{0pt} }_{\mathrlap{\rule[11pt]{0pt}{0pt} \text{\normalsize{$ \hspace{-41pt} = -\, df(\zeta) P^{0,1}_{\Iu} = - \, \bar{\partial}_{\Iu} f$}}}} \frac{d\zeta}{\zeta} } 
\, = \bar{\partial}_{\Iu} \alpha_F 
\end{equation}
\noindent which is clearly supported on the fiber.

Assume now instead that part 1 holds. As a $(1,0)$ form on $\mathcal{Z}$, $\alpha$ splits uniquely into a component along the local fiber and one along the local horizontal twistor line: \mbox{$\alpha = \alpha_F + \alpha_{\hp\smash{\mathbb{CP}^1}}$.} Without any loss of generality we can write
\begin{equation}
\alpha_{\hp\smash{\mathbb{CP}^1}} = -i f \hp \frac{d\zeta}{\zeta}
\end{equation}
for some function $f$. For the form $\bar{\partial}_{\mathcal{Z}}\alpha$ to be entirely supported on the fibers, both its component along the local horizontal twistor line and any mixed-type components it might have must vanish. By way of formula \eqref{d_bar_(1,0)} these requirements are equivalent to
{\allowdisplaybreaks
\begin{equation}
\begin{aligned}
& \bar{\partial}_{\Iu} f +  \rho \pt\hp \alpha_F \overline{I(\zeta)} = 0 \\
& \bar{\partial}_{\hp\smash{\mathbb{CP}^1}} \alpha_F = 0 \\[0pt]
& \bar{\partial}_{\hp\smash{\mathbb{CP}^1}} f = 0 \mathrlap{.}
\end{aligned}
\end{equation}
}%
Writing $\bar{\partial}_{\Iu} f = df P^{0,1}_{\Iu}$ and then using the factorization formula \eqref{factorization_Q}, the first condition yields $[\hp df I(\zeta) - \alpha_F] \overline{I(\zeta)} = 0$. This implies in particular that \mbox{$df I(\zeta)- \alpha_F \in \mathscr{A}^{0,1}_F(\mathcal{Z})$.} However, both terms plainly belong to $\mathscr{A}^{1,0}_F(\mathcal{Z})$, which leads to a contradiction unless $df I(\zeta)- \alpha_F = 0$. Note that between this and the third condition, the second condition is automatically satisfied. The direct implication of the lemma follows now readily.
\end{proof}

The following statement is an easy consequence of Lemma \ref{fiberwise-lemma}: 

\begin{corollary}
Let $\phi(\zeta)$ be a holomorphic function on a domain on $\mathcal{Z}$. Then
\begin{equation} \label{delZ_hol}
\partial_{\mathcal{Z}}\phi(\zeta) = i\pt d[\hp \zeta\partial_{\zeta}\phi(\zeta)]I(\zeta) + \zeta\partial_{\zeta}\phi(\zeta) \frac{d\zeta}{\zeta}.
\end{equation}
\end{corollary}

\section{Closed hyper $(1,1)$ forms} \label{sec:Hyp-1,1}

\subsection{Hyper $(1,1)$ forms} \hfill \medskip

We will shift now our focus towards the study of a class of \mbox{2-forms} which, we will try to argue, plays a fundamentally important role in hyperk\"ahler geometries. In the beginning of our analysis we will be concerned primarily with local aspects. Further on we will examine also some implications of a global definition. 

%We shift now our focus towards the study of a special class of \mbox{2-forms} on hyperk\"ahler manifolds with a rich structure and of a fundamental geometric importance. In the beginning of our analysis we will be concerned mostly with local aspects. Further on we will examine also some important global implications. 

\begin{definition}
Let $M$ be a hyperk\"ahler manifold. By definition, we call a 2-form on $M$ to be \textit{of hyper $(1,1)$ type} if it is of Hodge type $(1,1)$ with respect to \textit{all} hyperk\"ahler complex structures of $M$ simultaneously. 
\end{definition}

%In four dimensions the condition that a 2-form be of hyper $(1,1)$ type is equivalent to it being anti-self-dual. 
In four dimensions this definition yields the class of anti-self-dual forms on $M$.
Note that hyper $(1,1)$ forms are naturally pulled back to fiber-supported $(1,1)$ forms on the twistor space. 
The condition in the definition is equivalent to requiring that the 2-form be of $(1,1)$ type with respect to each one of the elements of a standard quaternionic frame \mbox{$I_1$, $I_2$, $I_3$} of the bundle of hyperk\"ahler complex structures. That is to say, a 2-form $\sigma$ is of hyper $(1,1)$ type if and only if $\sigma(X,I_iY) = \sigma(Y,I_iX)$ for all $i=1,2,3$ and any vector fields $X,Y \in TM$. In what follows we will use this criterion mostly in the following trivially rephrased form: 

\begin{lemma} \label{criterion_1}
Let $\sigma \in \mathscr{A}^2(M,\mathbb{C})$ be a form of type $(1,1)$ with respect to the complex structure $I_0$. Then $\sigma$ is of hyper $(1,1)$ type if and only if \begin{equation}
\sigma(X,I_+Y) = \sigma(Y,I_+X)
\qquad \text{and} \qquad
\sigma(X,I_-Y) = \sigma(Y,I_-X)
\end{equation}
for any vector fields $X,Y \in TM$.
\end{lemma}

Remark for later benefit that in an arbitrary local coordinate coframe for $T^*_{\mathbb{C}}M$ holomorphic with respect to $I_0$ in which $\smash{ \sigma = \sigma_{\mu\bar{\nu}} \pt dx^{\mu} \!\wedge dx^{\bar{\nu}} }$, the two conditions of the lemma can be equivalently stated as the vanishing of the following two forms:
\begin{equation}
\sigma_{\mu\bar{\eta}} \hp (I_+)^{\bar{\eta}}{}_{\nu} \pt dx^{\mu} \!\wedge dx^{\nu} = 0 
\qquad \text{and} \qquad
\sigma_{\bar{\mu}\eta} \hp (I_-)^{\eta}{}_{\bar{\nu}} \pt dx^{\bar{\mu}} \!\wedge dx^{\bar{\nu}} = 0 \rlap{.}
\end{equation}

The next result was proved by Verbitsky in the compact case \cite[Lemma 2.1]{MR1486984}. In keeping in line with our generic assumptions we give here an alternative proof which dispenses with the compactness requirement.  

\begin{proposition}
Let $\Lambda_{\omega(\mathbf{u})}: \mathscr{A}^2(M) \longrightarrow \mathscr{A}^0(M)$ be the linear Hodge $\Lambda$-operator defined by tracing with the inverse of the K\"ahler form $\omega(\mathbf{u})$. If $\sigma$ is a hyper $(1,1)$ form on $M$ then \mbox{$\Lambda_{\omega(\mathbf{u})}(\sigma) = 0$} for all $\mathbf{u} \in S^2$. 
\end{proposition}

\begin{proof} 

Let $I$ be a complex structure on $M$. Following \cite{MR1486984}, we define the linear left action $\text{ad}(I)$ of $I$ on the bundle of differential forms on $M$ of arbitrary positive degree by extending the usual endomorphic left action of $I$ on the bundle of differential 1-forms on $M$ by means of Leibniz's formula: $\smash{ (\alpha \wedge \beta) \pt \text{ad}(I) = \alpha \wedge (\beta \pt \text{ad}(I)) + (\alpha \pt \text{ad}(I)) \wedge \beta }$. With this definition, for any two complex structures $I$, $J$ on $M$ we have $[\text{ad}(I),\text{ad}(I)] = \text{ad}([I,J])$. Observe in particular that the condition that a 2-form $\sigma$ on $M$ be of $(1,1)$ type with respect to the complex structure $I$ can be equivalently expressed as the requirement that $\sigma$ be in the kernel of the operator $\text{ad}(I)$. 

Choose an arbitrary standard quaternionic frame $I_1$, $I_2$, $I_3$ for the bundle of hyperk\"ahler complex structures and a corresponding frame $\omega_1$, $\omega_2$, $\omega_3$ in the bundle of hyperk\"ahler symplectic forms. From the hyperk\"ahler properties \eqref{i=oo} and the fact that complex structures square to the minus identity a simple calculation shows that for any 2-form $\sigma$
\begin{equation}
\Lambda_{\omega_i} (\sigma \hp \text{ad}(I_j)) = \varepsilon_{ijk} \hp \Lambda_{\omega_k}(\sigma) \rlap{.}
\end{equation}
Clearly then, if $\sigma$ is of hyper $(1,1)$ type\,---\,and so in the kernel of $\text{ad}(I_j)$ for all $j=1,2,3$\,---\,we have $\Lambda_{\omega_k}(\sigma) = 0$ for all $k=1,2,3$. 
Since the frame was chosen arbitrarily, the statement of the Proposition follows. 
\end{proof}

Thus, at least locally, \textit{closed} hyper $(1,1)$ forms satisfy formally the same type of constraints characterizing curvatures of hyper-Hermitian Yang-Mills connections with vanishing slope on complex line bundles.

\subsection{Hyperpotentials} \hfill \medskip

By the local $\partial\bar{\partial}$\pt-\hp lemma, any closed hyper $(1,1)$ form $\sigma$ on $M$ can be locally derived from a potential in any one of the complex structures $\IU$. That is,  for any $\mathbf{u} \in S^2$, every point in $M$ has an open neighborhood on which \mbox{$\sigma = i \pt \partial_{\Iu}\bar{\partial}_{\Iu} \phi(\mathbf{u})$} for some potential $\phi(\mathbf{u})$ defined on that neighborhood (the imaginary factor is conventional and ensures that if $\sigma$ is real-valued the potential can be chosen to be real). This suggests the following

%\begin{definition}
%Let $\IU$ be a complex structure on $M$ with Dolbeault operator $\partial_{\Iu}$. By definition we call a local function $\phi(\mathbf{u})$ such that $i \pt \partial_{\Iu}\bar{\partial}_{\Iu} \phi(\mathbf{u})$ is of hyper $(1,1)$ type a \textit{hyperpotential with respect to $\IU$}. 
%\end{definition}

\begin{definition}
Let $I$ be a complex structure on $M$ with Dolbeault operator $\partial_I$. By definition we call a local function $\phi$ such that $\partial_I\bar{\partial}_I \phi$ is of hyper $(1,1)$ type a \textit{hyperpotential with respect to $I$}. 
\end{definition}

\noindent So in these terms the problem of finding and characterizing closed hyperinvariant forms can be alternatively formulated, when convenient, as the problem of finding and characterizing hyperpotentials. %\smallskip

Let us pick now $I_0 \equiv I_3$ from among the $S^2$ family of hyperk\"ahler complex structures as the \textit{manifest} complex structure. To mark its distinguished status we will henceforth denote the corresponding Dolbeault operator simply by $\partial$. It is important to keep in mind that for the time being this is just an arbitrary choice of perspective which does not reflect or imply anything intrinsically special about $I_0$. With these notation conventions in place, consider the set of five complex second-order partial differential operators on $M$ given by 
\begin{equation} \label{Fs_and_Hs}
\begin{gathered}
F_0(f) = i \hp \partial \bar{\partial} f  \\[0pt]
F_+(f) = \bar{\partial} (\bar{\partial} f \hp I_+) \qquad\qquad H_+(f) = \partial (\bar{\partial} f \hp I_+) \\[1pt]
F_-(f) = \partial (\partial f \hp I_-) \qquad\qquad H_-(f) = \bar{\partial} (\partial f \hp I_-)
\end{gathered}
\end{equation}
for any function $f \in \mathscr{A}^0(M,\mathbb{C})$. Remark that owing to the quaternionic properties of hyperk\"ahler complex structures we can write alternatively 
\begin{equation}
\begin{aligned}
\bar{\partial} f \hp I_+ & = df P^{0,1}_{\Io}I_+ = df I_+ \\
\partial f \hp I_- & = df P^{1,0}_{\Io}I_- = df I_- \mathrlap{.}
\end{aligned}
\end{equation}
These are 1-forms of type $(1,0)$ respectively $(0,1)$ relative to $I_0$, and so the operators of $F$ type take values in $\mathscr{A}^{1,1}_{\Io}(M)$, the $H_+$ operator in $\mathscr{A}^{2,0}_{\Io}(M)$ and the $H_-$ one in $\mathscr{A}^{\hp 0,2}_{\Io}(M)$.  As a consequence of hyperk\"ahler integrability we have

\begin{lemma} \label{F+F-}
$F_+(f)$ and $F_-(f)$ are hyper $(1,1)$ forms for \underline{any} function $f \in \mathscr{A}^0(M,\mathbb{C})$.
\end{lemma}

\begin{proof}

Let us choose an arbitrary coordinate coframe holomorphic with respect to $I_0$. On one hand, by distributing a derivative we obtain
{\allowdisplaybreaks
\begin{align}
F_+(f)_{\mu\bar{\eta}} (I_+)^{\bar{\eta}}{}_{\nu} \pt dx^{\mu} \!\wedge dx^{\nu} & = - \, \partial_{\bar{\eta}} [\hp \partial_{\bar{\kappa}} f \hp (I_+)^{\bar{\kappa}}{}_{\mu}] (I_+)^{\bar{\eta}}{}_{\nu} \pt dx^{\mu} \!\wedge dx^{\nu} \\[2pt]
& = - \, [\hp \partial_{\bar{\eta}}\partial_{\bar{\kappa}} f \hp (I_+)^{\bar{\kappa}}{}_{\mu} (I_+)^{\bar{\eta}}{}_{\nu} +  \partial_{\bar{\kappa}} \pt f \partial_{\bar{\eta}}(I_+)^{\bar{\kappa}}{}_{\mu} (I_+)^{\bar{\eta}}{}_{\nu}] \hp dx^{\mu} \!\wedge dx^{\nu} \nonumber \\[2.5pt]
& = 0 \mathrlap{.} \nonumber
\end{align}
}%
The first term in the second line vanishes by index symmetry considerations due to the differentiability of $f$ and the last one by the second integrability constraint \eqref{HK_integrability}. On the other hand, this time by forcing out a total derivative and then resorting to the property \eqref{I+I-}, we get successively
{\allowdisplaybreaks
\begin{align}
F_+(f)_{\bar{\mu}\eta} (I_-)^{\eta}{}_{\bar{\nu}} \pt dx^{\bar{\mu}} \!\wedge dx^{\bar{\nu}} & = \partial_{\bar{\mu}} [\hp \partial_{\bar{\kappa}} f \hp (I_+)^{\bar{\kappa}}{}_{\eta}] (I_-)^{\eta}{}_{\bar{\nu}} \pt dx^{\bar{\mu}} \!\wedge dx^{\bar{\nu}} \\[2pt]
& = [\partial_{\bar{\mu}} (\hp \partial_{\bar{\kappa}} f \hp (I_+)^{\bar{\kappa}}{}_{\eta} (I_-)^{\eta}{}_{\bar{\nu}}) - \partial_{\bar{\kappa}} f \hp (I_+)^{\bar{\kappa}}{}_{\eta} \pt \partial_{\bar{\mu}} (I_-)^{\eta}{}_{\bar{\nu}}] \hp dx^{\bar{\mu}} \!\wedge dx^{\bar{\nu}} \nonumber \\[2pt]
& = [\partial_{\bar{\mu}}\partial_{\bar{\nu}} f - \partial_{\bar{\kappa}} f \hp (I_+)^{\bar{\kappa}}{}_{\eta} \pt \partial_{\bar{\mu}} (I_-)^{\eta}{}_{\bar{\nu}}] \hp dx^{\bar{\mu}} \!\wedge dx^{\bar{\nu}} \nonumber \\[2.5pt]
& = 0 \mathrlap{.} \nonumber
\end{align}
}%
The first term in the third line vanishes again by index symmetry considerations and the last one by way of the first integrability constraint \eqref{HK_integrability}. Lemma \ref{criterion_1} implies then that $F_+(f)$ is a hyper $(1,1)$ form. A similar argument works for $F_-(f)$.
\end{proof} 

In contrast to $F_+(f)$ and $F_-(f)$, $F_0(f)$ is not of type $(1,1)$ with respect to all hyperk\"ahler complex structures of $M$ at once for generic functions $f$. That, however, does not prevent it to be so for select functions $f$. The next result collects a number of five equivalent maximal criteria for $f$, any one of which, if satisfied, guaranteeing that $F_0(f)$ is of hyper $(1,1)$ type. 

\begin{proposition} \label{criterion_2}
Let $\phi$ be a smooth and possibly local complex function on $M$. The following conditions are equivalent:
\begin{itemize}
\setlength{\itemsep}{4pt}

\item[1.] $F_0(\phi)$ is of hyper $(1,1)$ type,  that is, $\phi$ is a hyperpotential with respect to $I_0$;

\item[2.] $H_+(\phi) = H_-(\phi) = 0$;

\item[3.] Locally there exists a complex function $\phi_+$ such that $\bar{\partial}\phi \hp I_+ = - \hp i\pt \partial \phi_+$;

\item[4.] Locally there exists a complex function $\phi_-$ such that $\partial\phi \hp I_- =  i\pt \bar{\partial} \phi_-$;

\item[5.] The symplectic gradient vector field $X \in T^{1,0}_{\scriptscriptstyle I_0}M$ defined by $\iota_{X} \omega_0 = \bar{\partial} \phi$ admits an alternative characterization as a symplectic gradient with respect to $\omega_+$, that is, $\iota_{X} \omega_+ = \partial \phi_+$ for some local complex function $\phi_+$;

\item[6.] The symplectic gradient vector field $Y \in T^{0,1}_{\scriptscriptstyle I_0}M$ defined by $\iota_{Y} \omega_0 = \partial \phi$ admits an alternative characterization as a symplectic gradient with respect to $\omega_-$, that is, $\iota_{Y} \omega_- = \bar{\partial} \phi_-$ for some local complex function $\phi_-$.

\end{itemize}
\end{proposition}

\begin{proof}

The logical flow of the proof is as follows:
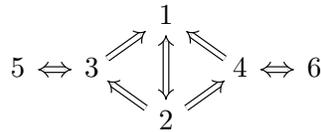
\begin{equation*}
\begin{tikzcd}[column sep=small,row sep=tiny]
					& 									& 1 \arrow[rd, Leftarrow, shorten >=-.2em] 	\arrow[dd, Leftrightarrow]			&	& \\
5 \arrow[r, Leftrightarrow]	& 3 \arrow[ru, Rightarrow, shorten <=-.2em] 	& 																& 4 	& \arrow[l, Leftrightarrow] 6 \\
					&									& 2 \arrow[lu, Rightarrow, shorten >=-.2em] \arrow[ru, Rightarrow, shorten >=-.2em] 	& 	&
\end{tikzcd}
\end{equation*}

\begin{itemize}[leftmargin=\leftmargin-3.5pt]
\setlength{\itemsep}{4pt}

\item[$1 \Leftrightarrow 2$] 

Note that in a generic local coordinate coframe holomorphic with respect to $I_0$, by the definition of the $H$-operators and the first integrability condition \eqref{HK_integrability} we have
{\allowdisplaybreaks
\begin{equation}
\begin{aligned}
H_+(\phi) & = \partial_{\mu} [\hp \partial_{\bar{\rho}}\phi \hp (I_+)^{\bar{\rho}}{}_{\nu}] \hp dx^{\mu} \! \wedge dx^{\nu} = \partial_{\mu} \partial_{\bar{\rho}}\phi \hp (I_+)^{\bar{\rho}}{}_{\nu} \pt dx^{\mu} \! \wedge dx^{\nu} \\[2pt]
H_-(\phi) & = \partial_{\bar{\mu}} [\hp \partial_{\rho}\phi \hp (I_-)^{\rho}{}_{\bar{\nu}}] \hp dx^{\bar{\mu}} \! \wedge dx^{\bar{\nu}} = \partial_{\bar{\mu}}\partial_{\rho}\phi \hp (I_-)^{\rho}{}_{\bar{\nu}} \hp dx^{\bar{\mu}} \! \wedge dx^{\bar{\nu}} \mathrlap{.}
\end{aligned}
\end{equation}
}%
In view of these relations, the equivalence of the two statements follows immediately based on Lemma \ref{criterion_1}. 

\item[$2 \Rightarrow 3$] 

The condition $H_+(\phi) = 0$ means that the form $\bar{\partial}\phi \hp I_+$ of type $(1,0)$ relative to $I_0$ is $\partial$-closed on the domain of definition of $\phi$, and so by the $\partial$\pt-\hp Poincar\'e lemma it must be locally $\partial$-exact. 

\item[$2 \Rightarrow 4$]

Follows from a mirror argument starting from the condition $H_-(\phi) = 0$ and relying on the \mbox{$\bar{\partial}$\pt-\hp Poincar\'e} lemma. 

\item[$3 \Rightarrow 1$]

Using the quaternionic relation $I_+I_- = P^{0,1}_{\Io}$ the equation in part 3 can be equivalently rewritten as $\partial\phi_+I_- = i \hp \bar{\partial}\phi$. Then 
\begin{equation}
F_0(\phi) = i\hp \partial\bar{\partial}\phi =  i\hp \partial(\bar{\partial}\phi)  =  \partial(\partial\phi_+I_-) = F_-(\phi_+)
\end{equation}
which by Lemma \ref{F+F-} is automatically of hyper $(1,1)$ type.

\item[$4 \Rightarrow 1$]

The argument follows a similar route, with the equation in part 4 now equivalently recast in the form $\bar{\partial} \phi_-I_+ = - \hp i \hp \partial\phi$. Hence we can write
\begin{equation}
F_0(\phi) = i \hp \partial\bar{\partial}\phi = - \hp i \hp \bar{\partial}(\partial\phi) = \bar{\partial}(\bar{\partial}\phi_-I_+) = F_+(\phi_-)
\end{equation}
which is again manifestly of hyper $(1,1)$ type. 

\item[$3 \Leftrightarrow 5$]

The key observation underlying the proof is that $I_+ = - \hp i \pt \omega_0^{-1}\omega_+$.

\item[$4 \Leftrightarrow 6$]

Follows similarly based on the conjugate relation. \hfill \ \qedhere  
\end{itemize}
\end{proof}

\subsection{Recursive chains of hyperpotentials} \label{ssec:rec_chains} \hfill \medskip

This Proposition not only lays out a set of criteria for a function to be a hyperpotential, but also, quite remarkably, describes a mechanism through which hyperpotentials produce new hyperpotentials.

\begin{lemma} \label{recurrence_lemma}
If $\phi$ is a hyperpotential with respect to $I_0$, then so are the corresponding functions $\phi_+$ and $\phi_-$.
\end{lemma}

\begin{proof}

Indeed, by way of the equivalences $1 \Leftrightarrow 3$ and $1 \Leftrightarrow 4$ of Proposition \ref{criterion_2}, respectively, we have
\begin{align}
\begin{split}
F_0(\phi_+) & = i\hp \partial\bar{\partial}\phi_+ = - i\hp \bar{\partial}(\partial\phi_+) = \bar{\partial}(\bar{\partial}\phi I_+) = F_+(\phi)  \\[0pt]
F_0(\phi_-) & = i\hp \partial\bar{\partial}\phi_- = \phantom{+} i\hp \partial(\bar{\partial}\phi_-) = \partial(\partial\phi I_-) = F_-(\phi)     
\end{split}
\end{align} 
both of which are of hyper $(1,1)$ type  by Lemma \ref{F+F-}.
\end{proof}

\noindent In what follows we will call triplets of hyperpotentials related in this way \textit{adjacent} and we will represent them by means of the schematic notation 
\begin{equation*}
\begin{tikzcd}
\phi_- \!\! \arrow[dash]{r} & \! \phi \! \arrow[dash]{r} & \! \phi_+
\end{tikzcd}
\end{equation*}
%\begin{equation*}
%\phi_- \,\noarrow\, \phi \,\noarrow\, \phi_+ \rlap{.} \\[1pt]
%\end{equation*}
From the considerations above it is clear that for any such triplet of hyperpotentials we have
\begin{equation} \label{F_shifts}
\begin{gathered}
F_0(\phi) = F_-(\phi_+) = F_+(\phi_-) \\
F_0(\phi_+) = F_+(\phi) \\
F_0(\phi_-) = F_-(\phi) \mathrlap{.}
\end{gathered}
\end{equation}

\noindent Incidentally, the last two properties imply that

\begin{corollary}
If the automatically closed $F_0(\phi)$ is hyper $(1,1)$, then the automatically hyper $(1,1)$ $F_+(\phi)$ and $F_-(\phi)$ are closed. 
\end{corollary}

The most salient and immediately apparent consequence of Lemma \ref{recurrence_lemma} is that hyperpotentials naturally generate new hyperpotentials of the same type in recursive cascades. Let us try to understand how this happens more closely. Suppose we start with a generic hyperpotential $\phi$. In a first step this gives rise via the $\partial$ and \mbox{$\bar{\partial}$\pt-\hp Poincar\'e} lemmas in accordance with the mandates of Proposition \ref{criterion_2} to two new hyperpotentials, $\phi_+$ and $\phi_-$:
\begin{equation*}
\begin{tikzcd}
\phi_{-} \!\! \arrow[dash]{r}  &  \arrow[l, dashed, bend right=45, in=222] \phi  \arrow[r, dashed, bend left=45, in=140] \arrow[dash]{r} & \! \phi_{+}
\end{tikzcd}
\end{equation*}
Through the same mechanism, each of these generates in turn two more hyperpotentials, which we denote in similar fashion by $(\phi_+)_+$, $(\phi_+)_-$ and $(\phi_-)_+$, $(\phi_-)_-$. Note, however, that not all of these are necessarily new. By acting from the right on the equations of parts 3 and 4 of Proposition \ref{criterion_2} with $I_-$ and $I_+$ and making use of the quaternionic identity $I_+I_- = P^{0,1}_{\scriptscriptstyle I_0}$ and its complex conjugate, respectively, one can easily convince oneself that we can actually take $(\phi_+)_- = (\phi_-)_+ = \phi$. In contrast, no such argument can be conceived for the remaining two hyperpotentials, which are therefore genuinely new. Let us denote them by $(\phi_+)_+ = \phi_{++}$ and $(\phi_-)_- = \phi_{--}$. What we have shown then is that the above adjacency relations can be extended to
\begin{equation*}
\begin{tikzcd}
\phi_{--} \!\! \arrow[dash]{r} & \arrow[l, dashed, bend right=40] \! \phi_{-} \!\!  \arrow[r, dashed, bend left=45, in=138] \arrow[dash]{r}  & \phi  \arrow[dash]{r} & \arrow[l, dashed, bend right=45, in=222] \! \phi_{+} \!\! \arrow[r, dashed, bend left=45, in =152] \arrow[dash]{r} & \! \phi_{++}
\end{tikzcd}
\end{equation*}
%\begin{equation*}
%\phi_{--} \,\noarrow\, \phi_{-} \,\noarrow\, \phi \,\noarrow\, \phi_{+} \,\noarrow\, \phi_{++} \\[1pt]
%\end{equation*}
The argument can be repeated recursively again and again, with each successive iteration producing in the same way two new hyperpotentials and falling back onto two old ones. Thus, if $\phi_{+n}$ and $\phi_{-n}$ denote the new hyperpotentials resulting from the $n$-th iteration,\footnote{\pt We encourage the reader to think of the indices interchangeably as both integers and pluses or minuses, in the obvious way.} then the next iteration produces on one hand $(\phi_{+n})_- = \phi_{+(n-1)}$ and $(\phi_{-n})_+ = \phi_{-(n-1)}$, \textit{i.e.}, two lower-level hyperpotentials, and on the other hand two new ones, which we denote similarly by $(\phi_{+n})_+ = \phi_{+(n+1)}$ and $(\phi_{-n})_- = \phi_{-(n+1)}$. Symbolically, we have
\begin{equation*}
\begin{tikzcd}
\phi_{-(n+1)} \!\! \arrow[dash]{r} & \arrow[l, dashed, bend right] \! \phi_{-n} \!\! \arrow[r, dashed, bend left]  \arrow[dash]{r} & \! \phi_{-(n-1)} \!\! \arrow[dash]{r} & {\! \dots \, \phi \, \dots \!} \arrow[dash]{r} & \! \phi_{+(n-1)} \!\! \arrow[dash]{r} & \arrow[l, dashed, bend right] \! \phi_{+n} \!\! \arrow[r, dashed, bend left]  \arrow[dash]{r} & \! \phi_{+(n+1)} 
\end{tikzcd}
\end{equation*}
%\begin{equation*}
%\phi_{-(n+1)} \,\noarrow\, \phi_{-n} \,\noarrow\, \phi_{-(n-1)} \,\noarrow\, \dots  \, \phi \,  \dots \,\noarrow\, \phi_{+(n-1)} \,\noarrow\, \phi_{+n} \,\noarrow\, \phi_{+(n+1)}  
%\\[1pt]
%\end{equation*}
In this way the hyperpotential $\phi$ generates recursively an infinite ordered sequence of hyperpotentials $(\phi_n )_{n \in \mathbb{Z}}$, where we identify, conventionally, $\phi_{\hp 0} = \phi$.
%\begin{equation*}
%\begin{tikzcd}
%\dots \! \arrow[dash]{r} & \! \phi_{-n} \!\! \arrow[dash]{r} & \! \dots \! \arrow[dash]{r} & \! \phi_- \!\! \arrow[dash]{r} & \! \phi \! \arrow[dash]{r} & \! \phi_+ \!\! \arrow[dash]{r} & \! \dots \! \arrow[dash]{r} & \! \phi_{+n} \!\! \arrow[dash]{r} & \! \dots
%\end{tikzcd}
%\end{equation*}
By construction, the elements of the sequence satisfy the right respectively left-moving recursion relations
\begin{equation} \label{chain_recursion}
\begin{aligned}
\bar{\partial}\phi_n I_+ & = - \hp i\hp \partial \phi_{n+1} \\[0pt]
\partial\phi_n I_- & = \phantom{+} \hp i\hp \bar{\partial} \phi_{n-1} 
\end{aligned}
\end{equation}
for all $n \in \mathbb{Z}$. The two recursions are in fact equivalent as they can be obtained from one another by means of the previously mentioned quaternionic identity. Another equivalent condition is
\begin{equation} \label{chain_recursion_2}
d\phi_{n-1}I_+ + d\phi_n I_0 + d\phi_{n+1} I_- = 0
\end{equation}
and, indeed, the two equations \eqref{chain_recursion} can be easily identified as the $(1,0)$ respectively $(0,1)$ parts of this relation with respect to the complex structure $I_0$. The operators $F_+$ and $F_-$ act as step-right and step-left operators in the sense that for any $n \in \mathbb{Z}$ we have
\begin{equation} \label{step-ops}
F_+(\phi_{n-1}) = F_0(\phi_n) = F_-(\phi_{n+1}) \rlap{.}
\end{equation} 
%\begin{equation} \label{step-ops}
%\begin{gathered}
%F_+(\phi_n) = F_0(\phi_{n+1}) \\[1pt]
%F_-(\phi_n) = F_0(\phi_{n-1}) \rlap{.}
%\end{gathered}
%\end{equation}

\begin{definition}
We call an ordered sequence $(\phi_n )_{n \in \mathbb{Z}}$ of functions on $M$ sharing a non-trivial common domain on which they satisfy either one of the three recursion relations from \eqref{chain_recursion} and \eqref{chain_recursion_2} a \textit{recursive chain of hyperpotentials with respect to the complex structure $I_0$}.
\end{definition}

\noindent Note that the term ``hyperpotentials" is quite adequately used in this definition. By the equivalences $1 \Leftrightarrow 3$ and $1 \Leftrightarrow 4$ of Proposition \ref{criterion_2} all the functions making up such a sequence are indeed guaranteed to be hyperpotentials. The considerations above can then be summed up as follows: 

\begin{proposition}
Any hyperpotential with respect to a given complex structure on $M$ gives rise to a recursive chain of hyperpotentials with respect to the same complex structure. 
\end{proposition}

We end this discussion with several remarks. If at some point during a right-moving recursion along a chain we encounter a hyperpotential which is holomorphic with respect to $I_0$ then we can take all the subsequent potentials to its right to be equal to zero. Similarly, if in the course of a left-moving recursion we encounter a hyperpotential which is anti-holomorphic with respect to $I_0$ then we can take all the hyperpotentials to its left to be equal to zero. Thus, chains can be \textit{bounded}, \textit{half-bounded}\,---\,from the right or from the left, or \textit{infinite}. A right-boundary hyperpotential is always holomorphic and a left-boundary one always anti-holomorphic with respect to $I_0$. 

Observe that if $(\pt \phi_n \hp )_{n \in \mathbb{Z}}$ is a chain of hyperpotentials then so is \mbox{$\smash{ ( \pt \phi^c_n \equiv (-)^n \bar{\phi}_{-n} \pt )_{n \in \mathbb{Z}} }$}, where the overhead bar symbolizes complex conjugation. This will be termed the \textit{conjugate chain}. A self-conjugate chain always contains a real hyperpotential and, conversely, a real hyperpotential can always generate a self-conjugate chain. 

Chains are constructed by a recursive application of the $\partial$ and $\bar{\partial}$\pt-\hp Poincar\'e lemmas, and at each iteration the domain of definition of the newly produced hyperpotentials may possibly shrink with respect to the domain of definition of the previous crop of hyperpotentials. Here we will assume that in the infinite iteration limit we are still left with a domain containing a non-empty open subset (this is obviously always the case for bounded chains). %We will call this the \textit{domain of the chain}. 
The intersection domain of all the hyperpotentials in the chain will be called the \textit{domain of the chain}.

\subsection{Chains of hyperpotentials and holomorphic functions} \hfill \medskip

Chains of hyperpotentials on a hyperk\"ahler manifold are closely related to holomorphic functions on its twistor space.

\begin{proposition} \label{hol-hyper(1,1)}
Let $(\phi_n)_{n \in \mathbb{Z}}$ be a sequence of functions on $M$ with a non-trivial intersection domain and
\begin{equation} \label{hol-fct-ser}
\phi(\zeta) = \sum_{n=-\infty}^{\infty} \phi_{n} \hp \zeta^{-n}
\end{equation}
be the associated formal series. \smallskip

$1.$~If $(\phi_n)_{n \in \mathbb{Z}}$ forms a recursive chain of hyperpotentials with respect to the complex structure on $M$ parametrized by $\zeta = 0$ then on the domain of $\mathcal{Z}$ on which the associated series exists and converges, $\phi(\zeta)$ is a holomorphic function. \smallskip

$2.$~Conversely, any function $\phi(\zeta)$ holomorphic on a domain of $\mathcal{Z}$ is, in particular, holomorphic in the $\mathbb{CP}^1$ coordinate, and so if \eqref{hol-fct-ser} is its Laurent expansion around $\zeta=0$, then the coefficients $(\phi_n)_{n \in \mathbb{Z}}$ form a recursive chain of hyperpotentials on $M$ with respect to the complex structure parametrized by $\zeta = 0$.
\end{proposition}

\begin{proof}

If we parametrize the twistor sphere as in \eqref{stereo_map} then the complex structure corresponding to $\zeta =0$ will be $I_0$. From the series expansion of $\phi(\zeta)$ and the second formula \eqref{def_I(zeta)} we get immediately
\begin{equation} \label{dphiQ}
\hspace{38pt}
d\phi(\zeta) I(\zeta) = \sum_{n=-\infty}^{\infty} (d\phi_{n-1}I_+ + d\phi_n I_0 + d\phi_{n+1}I_-) \hp \zeta^{-n} \rlap{.}
\end{equation} 
By the third part of Proposition \ref{CR_Z}, $\phi(\zeta)$ is holomorphic on $\mathcal{Z}$ if and only if $d\phi(\zeta) I(\zeta) = 0$ for all allowed values of $\zeta$. 
This condition is clearly equivalent to the recursion relation \eqref{chain_recursion_2}. If the series has no poles at $\zeta = 0$, a similar argument can be made using $P_N(\zeta)$ instead of $I(\zeta)$ and the second recursion relation \eqref{chain_recursion}. 
\end{proof}

\begin{remark}
Chains of hyperpotentials do not always give rise to holomorphic functions on $\mathcal{Z}$ since the associated series may be nowhere convergent. On the other hand, holomorphic functions on $\mathcal{Z}$ always yield chains of hyperpotentials when Laurent-expanded.
\end{remark}

% \subsection{Transitioning to other complex structures} \hfill \medskip

% \subsection{The action of $\partial_{\Iu}\bar{\partial}_{\Iu}$-operators on hyperpotentials} \hfill \medskip

% \subsection{$\partial\bar{\partial}$\pt-\hp operators and hyperpotentials at different complex structures} \hfill \medskip

\subsection{The action of generic $\partial\bar{\partial}$\pt-\hp operators on hyperpotentials} \hfill \medskip

Even though defined for a given hyperk\"ahler complex structure, recursive chains of hyperpotentials have properties which allow us to represent with ease the actions on them of various differential operators associated to other hyperk\"ahler complex structures. 

Let $(\phi_n)_{n \in \mathbb{Z}}$ be a recursive chain of hyperpotentials with respect to the complex structure $I_0$. For any other complex structure $\IU$, from the first decomposition formula \eqref{lin_decomp_proj} and the recursion relations \eqref{chain_recursion} we obtain  
\begin{equation} \label{dbar_Iu_phi_n}
\bar{\partial}_{\Iu} \phi_n = \rho_S \hp (\partial_{\Io}\phi_n - \partial_{\Io}\phi_{n+1} \zeta^{-1}) \hp + \rho_N (\bar{\partial}_{\Io}\phi_n - \bar{\partial}_{\Io}\phi_{n-1}\zeta \hp) \rlap{.}
\end{equation}
A similar expression holds also for $\partial_{\Iu}\phi_n$. Acting on this equation with the exterior derivative of $M$ yields after a few identifications the formula 
\begin{equation} \label{ddbar-chain}
\partial_{\Iu} \bar{\partial}_{\Iu} \phi_n = \partial_{\Io}\bar{\partial}_{\Io} (\hp x_-\phi_{n+1} + x_0\hp \phi_n + x_+\phi_{n-1}) 
\end{equation}
where $x_+ = \frac{1}{2} (x_1 + i x_2)$, $x_0 = x_3$, $x_- = - \frac{1}{2} (x_1 - i x_2)$ satisfying the alternating reality property $\bar{x}_m = (-)^m x_{-m}$ are the complex spherical-basis components of the position \mbox{$\mathbb{R}^3$-vector} corresponding to $\mathbf{u}$. Thus, we find that the action of the $\partial\bar{\partial}$-operator with respect to $\IU$  on a hyperpotential $\phi_n$ from the recursive chain can be very simply expressed in terms of the actions of the $\partial\bar{\partial}$-operator with respect to $I_0$ on the adjacent triplet of hyperpotentials \mbox{$\phi_{n-1}$, $\phi_n$, $\phi_{n+1}$}.

It is instructive to consider an additional alternative derivation of this remarkable formula which underscores the role of the second-order differential operators defined in \eqref{Fs_and_Hs}. These, it turns out, appear quite naturally when trying to express the $\partial\bar{\partial}$-operator for a complex structure $\IU$ in a coordinate frame holomorphic with respect to $I_0$. In fact, it was because of this reason that we have considered them in the first place. By projecting onto two complex subspaces of the complexified cotangent bundle and its second exterior power, respectively, using the soldering forms defined in equation \eqref{dx_splitting} we obtain the following formulas: 

\begin{lemma}
For any point $\mathbf{u} \in S^2$ and function $f \in \mathscr{A}^0(M,\mathbb{C})$ we have
\begin{equation}
\bar{\partial}_{\Iu} f = \partial_{\mu} f \hp \theta^{\mu}_{\minus} + \partial_{\bar{\mu}} f \hp \theta^{\bar{\mu}}_{\minus}
\end{equation}
%{\allowdisplaybreaks
%\begin{align} \label{deldelbar_at_u}
%i \pt \partial_{\Iu} \bar{\partial}_{\Iu} f = [\hp x_-F_+(f) + x_0 \hp F_0(f) + x_+F_-(f)]_{\mu\bar{\nu}} \pt (&\theta^{\mu}_{\plus} \wedge \theta^{\bar{\nu}}_{\minus} + \theta^{\mu}_{\minus} \wedge \theta^{\bar{\nu}}_{\plus}) \\[0pt]
%+ \hspace{3pt}  x_- \hp H_+(f)_{\mu\nu} \ \pt\hp & \theta^{\mu}_{\plus} \wedge \theta^{\nu}_{\minus} \nonumber \\
%+ \hspace{3.5pt} x_+ \hp H_-(f)_{\bar{\mu}\bar{\nu}} \ \pt & \theta^{\bar{\mu}}_{\plus} \wedge \theta^{\bar{\nu}}_{\minus} \rlap{.} \nonumber
%\end{align}
%}%
{\allowdisplaybreaks
\begin{align} \label{deldelbar_at_u}
i \pt \partial_{\Iu} \bar{\partial}_{\Iu} f = & \ [\hp x_-F_+(f) + x_0 \hp F_0(f) + x_+F_-(f)]_{\mu\bar{\nu}} \pt (\theta^{\mu}_{\plus} \wedge \theta^{\bar{\nu}}_{\minus} + \theta^{\mu}_{\minus} \wedge \theta^{\bar{\nu}}_{\plus}) \\[0pt]
& + x_- H_+(f)_{\mu\nu} \pt \theta^{\mu}_{\plus} \wedge \theta^{\nu}_{\minus} \nonumber \\
& + x_+ H_-(f)_{\bar{\mu}\bar{\nu}} \pt \theta^{\bar{\mu}}_{\plus} \wedge \theta^{\bar{\nu}}_{\minus} \rlap{.} \nonumber
\end{align}
}%
\end{lemma}

\noindent Here, $F_m(f)_{\mu\bar{\nu}}$, $H_+(f)_{\mu\nu}$ and $H_-(f)_{\bar{\mu}\bar{\nu}}$ are the (in the latter two cases, anti-symmetrized and combinatorially normalized) components of $F_m(f)$, $H_+(f)$ and $H_-(f)$ in an  \mbox{$I_0$-adapted} coordinate coframe.

Specialize then to $f = \phi_n$ in formula \eqref{deldelbar_at_u} for some $n \in \mathbb{Z}$. Using the properties \eqref{step-ops} and the linearity of the $F$-operators we may write
\begin{equation}
x_-F_+(\phi_n) + x_0 \hp F_0(\phi_n) + x_+F_-(\phi_n) = F_0(x_-\phi_{n+1} + x_0\hp \phi_n + x_+\phi_{n-1}) \rlap{.}
\end{equation}
The corresponding term in \eqref{deldelbar_at_u} can be further simplified by noting that if $\sigma$ is a hyper $(1,1)$ form then $\sigma_{\mu\bar{\nu}} \pt (\theta^{\mu}_{\plus} \wedge \theta^{\bar{\nu}}_{\minus} + \theta^{\mu}_{\minus} \wedge \theta^{\bar{\nu}}_{\plus}) = \sigma_{\mu\bar{\nu}} \pt dx^{\mu} \!\wedge dx^{\bar{\nu}}$. By the equivalence of parts 1 and 2 of Proposition \ref{criterion_2} the remaining two terms drop out and we retrieve again the result \eqref{ddbar-chain}.

\subsection{The local \mbox{$\partial\bar{\partial}$\pt-\hp lemma} for closed hyper $(1,1)$ forms} \hfill \medskip

Consider now a closed hyper $(1,1)$ form $\sigma$ defined at least locally on $M$. In particular, $\sigma$ must be of type $(1,1)$ with respect to the complex structure $I_0$, and so by the corresponding local \mbox{$\partial\bar{\partial}$\pt-\hp lemma} one may express it locally in the form $\sigma = i \hp \partial\bar{\partial} \phi_0 = F_0(\phi_0)$ in terms of a function $\phi_0$, which is in fact a hyperpotential. As such, this can be multiplied recursively to an entire chain of hyperpotentials, and so for closed hyper $(1,1)$ forms the usual local \mbox{$\partial\bar{\partial}$\pt-\hp lemma} in complex structure $I_0$ expands into the following more elaborate statement: 

\begin{lemma} \label{exp_ddbar_lem}
A hyper $(1,1)$ form $\sigma$ on $M$ is closed if and only if locally there exists a chain of hyperpotentials $(\phi_n)_{n \in \mathbb{Z}}$ with respect to $I_0$ such that
\begin{equation}
\sigma = F_0(\phi_n) = F_+(\phi_{n-1}) = F_-(\phi_{n+1})  
\end{equation} 
for some $n \in \mathbb{Z}$ (in what follows we will take, conventionally, \mbox{$n=0$}). An analogous statement is valid for any hyperk\"ahler complex structure on $M$.
\end{lemma}

\noindent Clearly, such chains are not unique. If $\sigma$ is real-valued the chain can always be chosen to be self-conjugate. More generally, if $\sigma$ is complex-valued then its complex conjugate is also a closed hyper $(1,1)$ form and the two corresponding chains can always be chosen to be mutually conjugate.

\subsection{The structure of hyperpotentials} \label{ssec:hyper_str} \hfill \medskip

There is more insight to be gained if we assume instead a twistor space perspective. The first thing to notice in this respect is that closed hyper $(1,1)$ forms on $M$ are naturally pulled back by the projection map $p: \mathcal{Z} \rightarrow M$ to closed  $(1,1)$ forms on the twistor space $\mathcal{Z}$\,---\,where again, by the corresponding local \mbox{$\partial\bar{\partial}$\pt-\hp lemma}, they can be derived from local potentials. That is, for every point on $\mathcal{Z}$ in the interior of the preimage of the domain of $\sigma$ in $\mathcal{Z}$ there exists a neighborhood $V$ such that
\begin{equation} \label{form-pot}
p^*\sigma |_{V}^{\phantom{I}} = i \pt \partial_{\mathcal{Z}}\bar{\partial}_{\mathcal{Z}} \phi_V(\mathbf{u}) \mathrlap{.}
\end{equation}
Although the potential is a function on $V \subset \mathcal{Z}$, we indicate explicitly only its dependence on the $\mathbf{u}$ variable. From here on until the end of the section we will assume without any essential loss of generality that $\sigma$ is real-valued, in which case we can take $\phi_V(\mathbf{u})$ to~be~real. 
The right-hand side of \eqref{form-pot} can be decomposed in components along the local twistor fibration structure using the formula \eqref{delZ_delZbar}. By virtue of its definition, $p^*\sigma \in \mathscr{A}^{1,1}_F(\mathcal{Z})$. The vanishing of the non-fiberwise supported components imposes the following set of restrictions on the potential:
{\allowdisplaybreaks 
\begin{align}
\bar{\partial}_{\hp\smash{\mathbb{CP}^1}} \partial_{\Iu}\phi_V(\mathbf{u}) & = 0 \nonumber \\
\partial_{\hp\smash{\mathbb{CP}^1}} \bar{\partial}_{\Iu}\phi_V(\mathbf{u}) & = 0 \label{pot_constrs} \\
\partial_{\hp\smash{\mathbb{CP}^1}} \bar{\partial}_{\hp\smash{\mathbb{CP}^1}} \phi_V(\mathbf{u})& = 0 \mathrlap{.} \nonumber
\end{align}
}%
In fact, it suffices to retain only one out of the first two conditions since by the reality assumption for the potential they are mutually complex conjugated. From the remaining fiberwise components, on any non-empty domain $p(V \cap \pt \pi^{-1}(\mathbf{u})) \subset M$ we have
\begin{equation} \label{leeloo}
\sigma 
= i \pt \partial_{\Iu}\bar{\partial}_{\Iu} \phi_V(\mathbf{u}) \rlap{.}
\end{equation}
This means that the twistor space potential $\phi_V(\mathbf{u})$ with $\mathbf{u}$ viewed now as a parameter rather than a variable doubles on $M$ as a hyperpotential with respect to the complex structure $\IU$. 

Let $x$ be an arbitrary point in the interior of the domain of $\sigma$ in $M$. On the associated horizontal twistor line $H_x = p^{-1}(x)$ in $\mathcal{Z}$ choose two antipodally conjugated points and parametrize the twistor $\mathbb{CP}^1$ in such a way that they correspond to $\zeta = 0$ and $\zeta = \infty$ (see the sketch in Figure \ref{cover-fig}). The local \mbox{$\partial\bar{\partial}$\pt-\hp lemma} lemma on $\mathcal{Z}$ guarantees the existence of two open neighborhoods $V_N$ and $V_S$ around these two points and of two corresponding potentials $\phi_{V_N}(\mathbf{u})$ and $\phi_{V_S}(\mathbf{u})$ for the form $p^*\sigma$ defined on them. Through further applications of this lemma around various points situated on $H_x$ and subsequent refinements it is possible to construct more potentials $\phi_V(\mathbf{u})$ for $p^*\sigma$ defined on open subsets $V$ of its domain in $\mathcal{Z}$ forming a collection $\mathscr{V}_x$ with the following properties: 
\begin{itemize}
\setlength{\itemsep}{1pt}

\item[1.] $V_N, V_S \in \mathscr{V}_x$.

\item[2.] Each element of $\mathscr{V}_x$ has a non-empty overlap with $H_x$. Together, these overlaps form an open covering of $H_x$.

\item[3.] Any non-empty intersection of an element of $\mathscr{V}_x$ with a horizontal twistor line is simply-connected. 

\item[4.] Any non-empty intersection of two elements of $\mathscr{V}_x$ is simply-connected. 

\end{itemize}

\begin{figure}[ht]
\begin{tikzpicture}

\draw (-2,2) .. controls (-1,2.3) and (1,1.7) .. (2,2);
\draw (-2,-1) .. controls (-1,-0.7) and (1,-1.3) .. (2,-1);
\draw[dashed] (-2,0.5) .. controls (-1,0.8) and (1,0.2) .. (2,0.5);

\draw (-2,2) .. controls (-2.2,1) and (-1.8,0) .. (-2,-1);
\draw (2,2) .. controls (1.8,1) and (2.2,0) .. (2,-1);
\draw[dashed] (1,1.92) .. controls (0.8,1) and (1.2,0) .. (1,-1.08);
\draw[dashed] (-1,2.09) .. controls (-1.2,1) and (-0.8,0) .. (-1,-0.92);

\draw[fill=lightgray,opacity=0.5] (-1.2,0.6) ellipse (19pt and 15pt);
\draw[fill=lightgray,opacity=0.5] ( 1.2,0.4) ellipse (19pt and 15pt);
\draw[fill=lightgray,opacity=0.5] (-0.35,0.55) ellipse (15pt and 10pt);
\draw[fill=lightgray,opacity=0.5] (0.35,0.48) ellipse (15pt and 10pt);

\draw (-2,-2) -- (2,-2);
\draw  (-3,2) -- (-3,-1);

\draw[->] (-2.05,0.5) -- (-2.95,0.5);
\draw[->] (0,-1.05) -- (0,-1.95);

\draw[dotted,->] (-1,-1) -- (-1,-1.95);
\draw[dotted,->] (1,-1.15) -- (1,-1.95);

\draw (-3.04,0.5) -- (-2.96,0.5);
\draw (-1,-1.96) -- (-1,-2.04);
\draw (1,-1.96) -- (1,-2.04);

\node at (-3,2.3) {$M$};
\node at (2.5,-2) {$\mathbb{CP}^1$};
\node at (2.3,2.3) {$\mathcal{Z}$};

\node at (-1,-2.3) {$\zeta = 0$};
\node at ( 1,-2.3) {$\zeta = \infty$};

\node at (0.23,-1.5) {$\pi$};

\node at (-3.25,0.5) {$x$};

\node at (-2.5,0.75) {$p$};

\node at (-0.7,1.34) {$V_N$};
\node at (1.28,1.22) {$V_S$};

\end{tikzpicture}
\caption{} \label{cover-fig}
\end{figure} 

For any set $V \in \mathscr{V}_x$, the last condition \eqref{pot_constrs} and the third property above imply that the corresponding local potential must be of the form
\begin{equation} \label{zeta_hol_str}
\phi_V(\mathbf{u}) = f_V(\zeta) + \overline{f_V(\zeta)}
\end{equation}
for some complex function $f_V(\zeta)$ on $V$ analytic in the $\mathbb{CP}^1$ variable. Here we will assume additionally that this function has on its domain of definition a Laurent series
\begin{equation}
f_V(\zeta) = \sum_{n = - \infty}^{\infty} f^V_n \zeta^{-n} \rlap{.}
\end{equation}

Note in particular that the requirement that the potentials on $V_N$ and $V_S$ be well-defined at $\zeta = 0$ respectively $\zeta = \infty$ entails that $f^{V_N}_n \!= 0$ for $n>0$ and $f^{V_S}_n \!= 0$ for $n<0$. Moreover, a simple argument shows that these two potentials can always be chosen in such a way that, close enough to the two points around which they are defined, they are interchanged by the action of the antipodal map, modulo a minus sign. In this case, their Laurent coefficients are related by 
%\mbox{$\Re f^{V_S}_0 = - \Re f^{V_N}_0$} and $f^{V_S}_{\hp n} = - (-)^n \overline{f^{V_N}_{-n}}$ for $n \neq 0$. 
\begin{equation}
\begin{aligned}
& \Re f^{V_S}_0 = - \, \Re f^{V_N}_0 && \text{for $n = 0$} \\[-2.5pt]
& f^{V_S}_{\hp n} = - (-)^n \overline{f^{V_N}_{-n}} && \text{for $n \neq 0$\rlap{.}}
\end{aligned}
\end{equation}
If we define
\begin{equation} \label{pot-phi}
\phi_n = 
\begin{cases}
\Re (f^{V_N}_0 \!- f^{V_S}_0) & \text{for $n=0$} \\[2pt]
\phantom{\Re(}f^{V_N}_n \!- f^{V_S}_n & \text{for $n \neq 0$}
\end{cases}
\end{equation}
satisfying as a consequence the self-conjugacy condition $\phi_n = \phi^c_n$ for all $n \in \mathbb{Z}$, then the potentials can be expressed as the following series expansions
{\allowdisplaybreaks
\begin{equation} \label{N&S_pots}
\begin{aligned}
\phi_{V_N}(\mathbf{u}) & = \phantom{+\,} \phi_0 + \sum_{n=1}^{\infty} (\hp \phi_{-n} \zeta^n + \text{c.c.} ) \\[-1.5pt]
\phi_{V_S}(\mathbf{u}) & = - \, \phi_0 - \sum_{n=1}^{\infty} (\hp \phi_{n} \zeta^{-n} + \text{c.c.} ) 
\end{aligned}
\end{equation}
}%
where c.c.~stands for the complex conjugate of the preceding expression. Remarkably, the following result holds:

\begin{lemma} \label{different_version}
The sequence of coefficients $(\phi_n)_{n \in \mathbb{Z}}$ thus defined forms a recursive chain of hyperpotentials with respect to the hyperk\"ahler complex structure on $M$ corresponding to $\zeta = 0$. 
\end{lemma}

\begin{proof}

Let us begin by observing that on any non-empty intersection of two elements of $\mathscr{V}_x$ the corresponding potentials differ by a local pluriharmonic function on $\mathcal{Z}$. That is, if $U,V \in \mathscr{V}_x$ such that $U \cap V \neq \emptyset$ then on $U \cap V$
\begin{equation}
\phi_U(\mathbf{u}) - \phi_V(\mathbf{u}) = \phi_{UV}(\zeta) + \overline{\phi_{UV}(\zeta)}
\end{equation}
for some function $\phi_{UV}(\zeta)$ holomorphic with respect to the complex structure on $\mathcal{Z}$. This function depends in particular holomorphically on the $\mathbb{CP}^1$ coordinate, and by an assumption implicit in our choice of open sets, it has a Laurent expansion
\begin{equation}
\phi_{UV}(\zeta) = \sum_{n=-\infty}^{\infty} \phi^{UV}_n \zeta^{-n} \rlap{.}
\end{equation}
As Laurent coefficients of the $\zeta$-expansion of a holomorphic function on $\mathcal{Z}$, $\phi^{UV}_n$ form a recursive chain of hyperpotentials with respect to the hyperk\"ahler complex structure on $M$ labeled by $\zeta = 0$. Substituting the form \eqref{zeta_hol_str} for the potentials and comparing term by term the Laurent expansions on both sides of the resulting equation yields the relations 
\begin{equation} \label{fs-and-phis}
\begin{aligned}
\Re \hp \phi^{UV}_0 & = \Re (f^U_0 - f^V_0) && \text{for $n=0$} \\
\phi^{UV}_n & = f^U_n - f^V_n && \text{for $n\neq 0$} \rlap{.}
\end{aligned}
\end{equation}

The second fact we will exploit is that any two elements of an open covering of a connected topological space either overlap or are finitely connected. Applying this rule to the restrictions of the sets $V_N$ and $V_S$ to the horizontal twistor line $H_x$, it follows that 
\begin{itemize}
\setlength{\itemsep}{2pt}
\item[1.] either $V_N \cap V_S \cap H_x \neq \emptyset$

\item[2.] or there exists a collection $\{U_i\}_{i=0,\dots,k+1}$ of open sets from $\mathscr{V}_x$ for some positive integer $k$ such that $U_0 = V_N$, $U_{k+1} = V_S$ and $U_i \cap U_{i+1} \cap H_x \neq \emptyset$ for all $i = 0,\dots,k$.
\end{itemize}
%Applying this rule to the restrictions of the sets $V_N$ and $V_S$ to the horizontal twistor line $H_x$ we have consequently the following two alternatives: either $V_N \cap V_S \cap H_x \neq \emptyset$ or there exists a collection $\{U_i\}_{i=0,\dots,k+1}$ of open sets from $\mathscr{V}_x$ for some positive integer $k$ such that $U_0 = V_N$, $U_{k+1} = V_S$ and $U_i \cap U_{i+1} \cap H_x \neq \emptyset$ for all $i = 0,\dots,k$. 
Let us assume that the second, more generic alternative holds. (If the first one were to hold the proof would be similar, only simpler.) By virtue of the argument above we have in this case a set of $k+1$ recursive chains $\phi^{V_NU_1}_n, \phi^{U_1U_2}_n, \dots, \phi^{U_{k}V_S}_n$ with respect to the hyperk\"ahler complex structure on $M$  labeled by $\zeta=0$, corresponding to as many local holomorphic functions on $\mathcal{Z}$. The intersection of their domains is an open set containing at least one point, $x$, and therefore non-empty. On the overlap, by the linearity of the defining recursive relations, $\phi^{V_NV_S}_n \!\coloneqq \phi^{V_NU_1}_n \!+ \phi^{U_1U_2}_n \!+ \dots + \phi^{U_kV_S}_n$ forms a recursive chain as well with respect to the same complex structure. This chain, however, does not necessarily correspond to a holomorphic function on $\mathcal{Z}$ as its associated \mbox{$\zeta$-series} is not guaranteed to be convergent. From the second relation \eqref{fs-and-phis} we get
\begin{equation}
\phi^{V_NV_S}_n =
\begin{cases}
\phi^{V_NU_1}_0 \!+ \dots + \phi^{U_kV_S}_0 & \text{for $n=0$} \\[1pt]
f^{V_N}_n \!- f^{V_S}_n & \text{for $n \neq 0$} \rlap{.}
\end{cases}
\end{equation}
Since the $n=0$ term in the chain is flanked by two mutually conjugated hyperpotentials, one may replace it with its real part without breaching recursiveness. That is, this substitution yields yet another recursive chain, which is moreover self-conjugate. Based on the first relation \eqref{fs-and-phis} we have
\begin{equation}
\Re (\phi^{V_NU_1}_0 \!+ \dots + \phi^{U_kV_S}_0) = \Re (f^{V_N}_0 \!- f^{U_1}_0 \!+ \dots + f^{U_k}_0 \!- f^{V_S}_0) = \Re(f^{V_N}_0 \!- f^{V_S}_0) \rlap{.} \qquad
\end{equation}
In view of the definition \eqref{pot-phi}, this new chain is therefore precisely $\phi_n$. 

The collection of sets connecting $V_N$ and $V_S$ is not unique, in general. Based on the ambiguities in the definitions of the functions $f_V(\mathbf{u})$ and $\phi_{UV}(\mathbf{u})$ one can show that a different collection of connecting sets yields the same coefficients $\phi^{V_NV_S}_n\!$, with the possible exception of the $n=0$ one, which may differ by a purely imaginary constant. This, however, vanishes when the real part is taken, and so in the end the conclusion is the same regardless of which connecting sets one considers.
\end{proof}

The coincidence between the notations of the coefficients in the series expansions \eqref{N&S_pots} and the hyperpotentials in Lemma \ref{exp_ddbar_lem} is not accidental. If, as we have always assumed so far, the complex structure labeled by \mbox{$\zeta = 0$} is $I_0$, then the recursive chain of hyperpotentials of Lemma \ref{different_version} is precisely of the kind featuring in Lemma \ref{exp_ddbar_lem}. That this is so follows immediately from the equation \eqref{leeloo} and the second formula below:

\begin{lemma}
For each $\mathbf{u} \in \pi(V_N)$, on the domain $p(V_N \cap \pi^{-1}(\mathbf{u})) \subset M$ we have
\begin{gather} 
\partial_{\Iu} \phi_{V_N}(\mathbf{u}) = \partial_{\Io} \phi_0 + \sum_{n=1}^{\infty} d \phi_{-n} \zeta^n \label{delphi_at_u} \\ 
\partial_{\Iu}\bar{\partial}_{\Iu} \phi_{V_N}(\mathbf{u}) = \partial_{\Io}\bar{\partial}_{\Io} \phi_0. \label{hyper-pots}
\end{gather}
Analogous formulas hold also for the $V_S$-potential. 
\end{lemma}

\begin{remark}
The fact that $\partial_{\Iu}\phi_{V_N}(\mathbf{u})$ depends on $\zeta$ strictly holomorphically means that the first constraint \eqref{pot_constrs} is automatically satisfied in this case.
\end{remark}

\begin{proof} 

In order to facilitate our manipulations, let us rewrite the first formula \eqref{N&S_pots} in the following condensed form
\begin{equation}
\phi_{V_N}(\mathbf{u}) = \sum_{n=-\infty}^{\infty} a_{-n}\phi_n
\end{equation}
with the coefficients given by $a_n = \zeta^n$ if $n > 0$, $a_0=1$, $a_n = (\zeta^c)^n$ if $n < 0$. Here, $\zeta^c$ denotes the antipodal conjugate of $\zeta$. We have then, successively, 
{\allowdisplaybreaks
\begin{align}
\partial_{\Iu}\phi_{V_N}(\mathbf{u}) & = \sum_{n=-\infty}^{\infty} a_{-n} \pt \partial_{\Iu}\phi_n \\[-2pt]
& = \sum_{n=-\infty}^{\infty} [\pt \rho_N(a_{-n} - \frac{a_{-n+1}}{\zeta^c}) \pt \partial_{\Io}\phi_n + \rho_S \hp (a_{-n} - a_{-n-1}\zeta^c) \pt \bar{\partial}_{\Io}\phi_n \hp] \nonumber \\[-2pt]
& = \sum_{n=-\infty}^{0} \partial_{\Io}\phi_n \hp \zeta^{-n} + \sum_{n=-\infty}^{-1} \bar{\partial}_{\Io}\phi_n \hp \zeta^{-n} \rlap{.} \nonumber
\end{align}
}%
To obtain the second line we used for $\partial_{\Iu}\phi_n$ a formula analogous to formula \eqref{dbar_Iu_phi_n}, which can in fact be retrieved from this  simply by substituting in it $\zeta$ with $\zeta^c$ (note that this entails in particular the interchange of $\rho_N$ and $\rho_S$). The third line follows then easily and is clearly equivalent to the first formula of the Lemma, equation \eqref{delphi_at_u}. Acting on this with the exterior derivative of $M$ yields immediately the second formula of the Lemma, equation \eqref{hyper-pots}. 

Alternatively, we can derive this last result directly by resorting to the formula \eqref{ddbar-chain}, as follows: 
{\allowdisplaybreaks
\begin{align}
\partial_{\Iu} \bar{\partial}_{\Iu} \phi_{V_N}(\mathbf{u}) & = \sum_{n=-\infty}^{\infty} a_{-n} \pt \partial_{\Iu} \bar{\partial}_{\Iu} \phi_n \\[-4pt]
& = \partial_{\Io}\bar{\partial}_{\Io} \! \sum_{n=-\infty}^{\infty} (x_- a_{-n+1} + x_0 \hp a_{-n} + x_+ a_{-n-1}) \hp \phi_n \nonumber \\[2pt]
& = \partial_{\Io}\bar{\partial}_{\Io} \phi_0 \rlap{.} \nonumber
\end{align}
}%
Indeed, one can easily check that $x_- a_{-n+1} + x_0 \hp a_{-n} + x_+ a_{-n-1}$ is equal to 0 for $n \neq 0$ and to 1 for $n=0$.
\end{proof}

While Lemma \ref{exp_ddbar_lem} guarantees the local existence of recursive chains of hyperpotentials with respect to a fixed hyperk\"ahler complex structure given a closed hyper $(1,1)$ form, \mbox{\textit{a priori}} it is not clear whether the corresponding series \eqref{hol-fct-ser} or \eqref{N&S_pots} converge. Our arguments demonstrate that, although the series \eqref{hol-fct-ser} may or may not converge in general, it is always possible to find chains for which the two series \eqref{N&S_pots} are well-defined on small enough but non-vanishing open domains. More precisely, we have proved the following result:

\begin{proposition} \label{lulu} 

Let $\sigma$ be a real-valued closed hyper $(1,1)$ form defined at least locally on $M$. Choose two conjugated hyperk\"ahler complex structures $I_0$ and $-I_0$, and a holomorphic parametrization of the twistor $\mathbb{CP}^1$ in which these correspond to $\zeta = 0$ respectively $\zeta = \infty$. Then for any point $x \in M$ from the domain of definition of $\sigma$ there exists a self-conjugated chain of hyperpotentials $(\phi_n)_{n \in \mathbb{Z}}$ with respect to $I_0$ defined on a neighborhood of $x$ such that the two associated series
{\allowdisplaybreaks
\begin{equation} \label{phi_VNetco}
\begin{aligned}
\phi_{V_N}(\mathbf{u}) & = \phantom{+\,} \phi_0 + \sum_{n=1}^{\infty} (\hp \phi_{-n} \zeta^n + \text{c.c.} ) \\[-1.5pt]
\phi_{V_S}(\mathbf{u}) & = -\, \phi_0 - \sum_{n=1}^{\infty} (\hp \phi_{n} \zeta^{-n} + \text{c.c.} ) 
\end{aligned}
\end{equation}
}%
are convergent and well-defined on non-empty open neighborhoods $V_N$ and $V_S$ of the points \mbox{$(x,\zeta=0)$} and \mbox{$(x, \zeta=\infty)$} on $\mathcal{Z}$. Viewed as functions on $M$ parametrized by $\mathbf{u}$,  they define hyperpotentials for $\sigma$ with respect to $\IU$. That is, for each $\mathbf{u} \in \pi(V_N)$, \mbox{$\sigma = i \pt \partial_{\Iu}\bar{\partial}_{\Iu} \phi_{V_N}\!\pt(\mathbf{u})$} on some neighborhood in $M$\,---\,and similarly on the $V_S$ side. 

If $V_N$ and $V_S$ overlap, then in addition the series 
\begin{equation} 
\phi_{V_NV_S}(\zeta) = \sum_{n=-\infty}^{\infty} \phi_{n} \hp \zeta^{-n}
\end{equation}
converges as well and defines a holomorphic function on $V_N \cap V_S \subset \mathcal{Z}$, on which we have 
\begin{equation}
\phi_{V_N}(\mathbf{u}) - \phi_{V_S}(\mathbf{u}) = \phi_{V_NV_S}(\zeta) + \overline{\phi_{V_NV_S}(\zeta)} \rlap{.}
\end{equation}
\end{proposition}

\subsection{Global issues}  \label{ssec:global_iss} \hfill \medskip

Let us assume now that $\sigma$ is a \textit{globally-defined} real-valued closed hyper $(1,1)$ form on $M$. Then the pullback form $p^*\sigma$ gives a globally-defined real-valued closed $(1,1)$ form on $\mathcal{Z}$. A well-known result states that if a closed 2-form of type $(1,1)$ on a complex manifold $\mathcal{Z}$, divided by $2\pi$, belongs to an \textit{integral} cohomology class\,---\,that is, a cohomology class in the image of the natural morphism $H^2(\mathcal{Z},\mathbb{Z}) \rightarrow H^2(\mathcal{Z},\mathbb{R})$\,---\,then one can regard it as the curvature of a Hermitian connection on a holomorphic line bundle over $\mathcal{Z}$. In the absence of the integrality condition one can still make a number of weaker but nevertheless interesting and useful claims. Let us see what these are.

Repeated applications of the local \mbox{$\partial\bar{\partial}$\pt-\hp lemma} for the form $p^*\sigma$ around various points of $\mathcal{Z}$ can be used to construct an open covering $\mathscr{V}$ of $\mathcal{Z}$ with the property that each element $V \in \mathscr{V}$ has an associated local real potential $\phi_V(\mathbf{u})$ such that
\begin{equation}
p^*\sigma |_{V}^{\phantom{I}} = i \pt \partial_{\mathcal{Z}}\bar{\partial}_{\mathcal{Z}} \phi_V(\mathbf{u}) \mathrlap{.}
\end{equation}
Through appropriate refinements this cover can be chosen so as to have contractible intersections. On non-empty double intersections $U \cap V$, the difference between the potentials associated to the two intersecting sets are local pluriharmonic functions on $\mathcal{Z}$, that is,
\begin{equation} \label{pot_gluing}
\phi_U(\mathbf{u}) - \phi_V(\mathbf{u}) = \phi_{UV}(\zeta) + \overline{\phi_{UV}(\zeta)}
\end{equation}
for some functions $\phi_{UV}(\zeta)$ holomorphic with respect to the twistor complex structure, single-valued (by the assumption of contractibility for $U \cap V$), and defined only up to constant imaginary shifts $\phi_{UV}(\zeta) \mapsto \phi_{UV}(\zeta) + i \pt c_{UV}$.  On non-empty triple intersections $U \cap V \cap W$ the trivial identity $\phi_U(\mathbf{u}) - \phi_V(\mathbf{u}) + \phi_V(\mathbf{u}) - \phi_W(\mathbf{u}) + \phi_W(\mathbf{u}) - \phi_U(\mathbf{u}) = 0$ imposes the constraint
\begin{equation}
\phi_{UV}(\zeta) + \phi_{VW}(\zeta) + \phi_{WU}(\zeta) + \text{c.c.} = 0 \mathrlap{.}
\end{equation}
Together with holomorphicity this implies in turn that additional real constants $c_{UVW}$ exist such that 
\begin{equation} \label{pre-cycle}
\phi_{UV}(\zeta) + \phi_{VW}(\zeta) + \phi_{WU}(\zeta) = i \pt c_{UVW} \mathrlap{.}
\end{equation}
Due to the ambiguity in the definition of the holomorphic functions these constants are defined only modulo shifts $c_{UVW} \mapsto c_{UVW}  + c_{UV} + c_{VW} + c_{WU}$, so they represent in fact a class in $H^2(\mathcal{Z},\mathbb{R})$. 
%\textcolor{blue}{This is of course precisely what one expects from the de Rham theorem.} 
The non-vanishing of this class is the obstruction for $\phi_{UV}(\zeta)$ to form a cocycle. By contrast, $d_{\mathcal{Z}}\phi_{UV}(\zeta)$ always satisfies the cocycle conditions and determines a class in $H^1(\mathcal{Z},d\mathcal{O}_{\mathcal{Z}})$, where $d\mathcal{O}_{\mathcal{Z}}$ denotes the sheaf of germs of closed holomorphic 1-forms on $\mathcal{Z}$ (the notation reflects the fact that closed 1-forms are locally exact). Thus what this argument shows is that to every closed real-valued 2-form on $\mathcal{Z}$ of $(1,1)$ type one can naturally associate a 1-cocycle of closed holomorphic 1-forms, that is to say, an element of $H^1(\mathcal{Z},d\mathcal{O}_{\mathcal{Z}})$. 

Let us look at this sheaf cohomology group more closely. First, note that we have the following morphisms:
\begin{equation}
\begin{tikzcd}
0 \rar & H^1(\mathcal{Z},d\mathcal{O}_{\mathcal{Z}}) \dar \rar & H^1(\mathcal{Z},\Omega_{\mathcal{Z}})  \\
& H^2(\mathcal{Z},\mathbb{C}) &
\end{tikzcd}
\end{equation}
where $\Omega_{\mathcal{Z}}$ stands for the sheaf of germs of holomorphic 1-forms on $\mathcal{Z}$. The vertical morphism is the second connecting morphism from the long exact sequence associated to the short exact sequence of sheaves
$0 \longrightarrow \mathbb{C} \longrightarrow \mathcal{O}_{\mathcal{Z}} \longrightarrow d\mathcal{O}_{\mathcal{Z}} \longrightarrow 0$.
%\begin{tikzcd}
%0 \rar & \mathbb{C} \rar & \mathcal{O}_{\mathcal{Z}} \rar & d\mathcal{O}_{\mathcal{Z}} \rar & 0.
%\end{tikzcd}
The horizontal ones, on the other hand, come from the long exact sequence associated to the short exact sequence of sheaves
$0 \longrightarrow d\mathcal{O}_{\mathcal{Z}} \longrightarrow \Omega_{\mathcal{Z}} \longrightarrow d\Omega_{\mathcal{Z}} \longrightarrow 0$,
%\begin{tikzcd}
%0 \rar & d\mathcal{O}_{\mathcal{Z}} \rar & \Omega_{\mathcal{Z}} \rar & d\Omega_{\mathcal{Z}} \rar & 0,
%\end{tikzcd}
where we took also into account the observation, due to Hitchin \cite{MR3116317}, that for twistor spaces of hyperk\"ahler manifolds one has $H^0(\mathcal{Z},d\Omega_{\mathcal{Z}}) = 0$. (The argument in \cite{MR3116317} goes as follows: the holomorphic cotangent bundle of $\mathcal{Z}$ splits naturally along horizontal twistor lines and their normal bundles into \mbox{$T_{\mathcal{Z}}^* \cong \pi^* T^*_{\hp\smash{\mathbb{CP}^1}} \oplus \mathbb{C}^{2m} \otimes \pi^*\mathcal{O}(-1)$}. Moreover, $T^*_{\hp\smash{\mathbb{CP}^1}} \cong \mathcal{O}(-2)$, so it is clear that there are no non-trivial globally-defined holomorphic forms, and in particular no non-trivial globally-defined \textit{closed} holomorphic forms on $\mathcal{Z}$.) By standard homological algebra arguments (see \textit{e.g.}~\cite{MR0086359}) we have, furthermore,
\begin{equation}
H^1(\mathcal{Z},\Omega_{\mathcal{Z}}) \cong H^1(\mathcal{Z},\Hom(T_{\mathcal{Z}},\mathcal{O}_{\mathcal{Z}})) \cong \Ext^1(T_{\mathcal{Z}},\mathcal{O}_{\mathcal{Z}})
\end{equation}
so we conclude that each element of $H^1(\mathcal{Z},d\mathcal{O}_{\mathcal{Z}})$ determines \textit{uniquely} an extension~class
\begin{equation} \label{Atiyah_seq}
\begin{tikzcd}
0 \rar & \mathcal{O}_{\mathcal{Z}} \rar & E \rar & T_{\mathcal{Z}} \rar & 0
\end{tikzcd}
\end{equation}
characterized by a  class in $H^2(\mathcal{Z},\mathbb{C})$. 

In particular, when this latter class is integral there exists a holomorphic line bundle $\cramped{L_{\mathcal{Z}} \rightarrow \mathcal{Z}}$ for which $E$ is the so-called \textit{Atiyah algebroid} \cite{MR0086359}. (Atiyah's original paper was concerned with complex principal bundles; here we use the closely related concept from the theory of complex vector bundles.) The importance of Atiyah algebroids stems from the following property: the holomorphic line bundle $L_{\mathcal{Z}}$ admits a holomorphic connection if and only if the corresponding short exact sequence \eqref{Atiyah_seq} splits. 

The emergence of the holomorphic line bundle can be seen explicitly in our case as follows: the integrality condition implies that the real constants $c_{UVW}$ can all be chosen to be integer multiples of $2\pi$ simultaneously. So, in view of this, if we define
\begin{equation} \label{hh-trans-fcts}
g_{UV} = \exp [\hp \phi_{UV}(\zeta)]
\end{equation}
by the equation \eqref{pre-cycle} these form a multiplicative 1-cocycle of non-vanishing holomorphic functions, which then through a canonical construction in complex geometry determine a holomorphic line bundle over $\mathcal{Z}$ for which they play the role of transition functions. This line bundle is trivial on each horizontal twistor line, and so by the hyperk\"ahler version of the Atiyah\pt-Ward correspondence it descends to a hyperholomorphic line bundle over M. (By definition, a vector bundle over $M$ is called \textit{hyperholomorphic} if it is holomorphic with respect to every hyperk\"ahler complex structure on $M$.)

\section{Tri-Hamiltonian and Killing tensor symmetries} \label{sec:Tri-Ham}

\subsection{Tri-Hamiltonian vector fields} \label{ssec:tri-Ham} \hfill \medskip

As we move towards applications of this general theory let us begin this section with an easy but very instructive exercise and review briefly a few well-known facts about tri-Hamiltonian vector fields in the language of hyperpotentials. Let $M$ denote again a hyperk\"ahler manifold, assumed now in addition to have vanishing first cohomology group, $H^1(M,\mathbb{R}) = 0$. By definition, a tri-Hamiltonian vector field on $M$ is a vector field preserving each element of the standard basis of hyperk\"ahler symplectic forms of $M$ (and therefore \textit{every} hyperk\"ahler symplectic form of $M$): 
\begin{equation}
\mathcal{L}_X\omega_1 = \mathcal{L}_X\omega_2 = \mathcal{L}_X\omega_3 = 0 \mathrlap{.}
\end{equation}
Such a vector field is automatically Killing. One can associate to it a triplet of momenta \mbox{$\mu_1, \mu_2, \mu_3: M \rightarrow \mathbb{R}$} such that
\begin{equation}
\iota_X\omega_1 = d\mu_1
\qquad\qquad
\iota_X\omega_2 = d\mu_2
\qquad\qquad
\iota_X\omega_3 = d\mu_3 \mathrlap{.}
\end{equation}
Via the expressions \eqref{i=oo} for the hyperk\"ahler complex structures in terms of the symplectic forms we have the relations
$d\mu_1 = d\mu_2 I_3$,
$d\mu_2 = d\mu_3 I_1$,
$d\mu_3 = d\mu_1 I_2$. 
By acting on them from the right with various $I_k$'s and using the quaternionic properties of the complex structures one can generate more relations of this type, in particular the equalities \mbox{$d\mu_1 I_1 = d\mu_2 I_2 = d\mu_3 I_3$}. 

If we introduce, alternatively, the complex spherical components $\mu_{\pm} = \pm\frac{1}{2}(\mu_1 \pm i\mu_2)$ and $\mu_0 = \mu_3$, then from the definition of the Dolbeault operator and the relations satisfied by the momenta we have, successively, 
\begin{equation}
\bar{\partial}_{\Io} \mu_+ = d\mu_+ P^{0,1}_{\Io} = \frac{1}{2}(d\mu_1 +id\mu_2)\frac{1}{2}(1+iI_3) = 0 \rlap{.}
\end{equation}
That is, $\mu_+$ is holomorphic relative to $I_0 \equiv I_3$. Similarly or by complex conjugation, $\mu_-$ is anti-holomorphic relative to $I_0$. Moreover, by virtue of the same relations above,
\begin{equation}
d\mu_-I_+ + d\mu_0 I_0 + d\mu_+I_- = 0 \rlap{.}
\end{equation}
These properties may be rephrased jointly as the statement that 
\begin{equation*}
0 \,\noarrow\, \mu_- \,\noarrow\, \mu_0 \,\noarrow\, \mu_+ \,\noarrow\, 0 \\[1pt]
\end{equation*}
forms a recursive chain of hyperpotentials with respect to $I_0$. 

On the twistor space we can arrive at the same conclusion by means of a holomorphicity argument. Tri-Hamiltonian vector fields are automatically tri-holomorphic on $M$ and lift trivially to holomorphic fiber-supported vector fields on $\mathcal{Z}$ which preserve the canonical holomorphic 2-form, also supported on the fibers. As such, they come with associated holomorphic Hamiltonian functions. More precisely, with our usual parametrization of $\mathbb{CP}^1$ and trivialization of $\mathcal{Z}$, we have
\begin{equation} \label{Ham-zeta}
\mathcal{L}_X \omega(\zeta) = 0
\end{equation}
from which we get that $\iota_X \omega(\zeta) = d\mu(\zeta)$, with
\begin{equation}
\mu(\zeta) = \frac{\mu_+}{\zeta \ } + \mu_0 + \zeta \mu_-
\end{equation}
holomorphic not only on the fiber but in fact on $\mathcal{Z}$, where it can be viewed as the tropical component of a section of the bundle $\pi^*\mathcal{O}(2)$. Proposition \ref{hol-hyper(1,1)} guarantees then that the Laurent coefficients of its \mbox{$\zeta${\hp-\hp}expansion} form a recursive chain of hyperpotentials with respect to $I_0$.

\subsection{The \mbox{$E \otimes\npt H$} formalism} \label{ssec:EH} \hfill \medskip

Before we deepen our incursion into the realm of hyperk\"ahler symmetries, it is useful to take a detour in order to review Salamon's \mbox{$E \otimes\npt H$} formalism \cite{MR664330}. This is a higher-dimensional generalization of the four-dimensional two-component spinor formalism \cite{MR776784}, which emerged originally in connection to quaternionic K\"ahler manifolds, understood in the broader sense which includes rather than excludes hyperk\"ahler manifolds. For the time being we will carry out the discussion in quaternionic K\"ahler terms, although eventually we will specialize to the hyperk\"ahler case. 

A quaternionic K\"ahler manifold $M$ is a Riemannian manifold of real dimension $4m$ whose holonomy group is, for $m>1$, a subgroup of $\smash{ Sp(1)Sp(m) = Sp(1) \times_{\mathbb{Z}_2} Sp(m) }$. The statements to follow hold all the same if one considers alternatively pseudo-Riemannian manifolds and non-compact versions of $Sp(m)$. Assuming no cohomological Marchiafava-Romani obstruction in $H^2(M,\mathbb{Z}_2)$, the complexified tangent bundle of $M$ splits locally~as
\begin{equation} \label{E-H}
T_{\mathbb{C}}M = E \otimes H
\end{equation}
where $E$ and $H$ represent locally-defined complex vector bundles of ranks $2m$ and $2$ underlying the standard representations of $Sp(m)$ and $Sp(1)$ on $\mathbb{H}^m$ and $\mathbb{H}$, respectively. For \mbox{$m=1$} the isomorphism $Sp(1)Sp(1) \cong SO(4)$ renders the holonomy group characterization trivial as it includes all oriented four-dimensional manifolds. The natural analogues of quaternionic K\"ahler manifolds in four dimensions are the Einstein self-dual manifolds, and one takes in this case $E$ and $H$ to be the spinor bundles $S_+$ and $S_-$. 

After choosing local frames on each bundle, the local isomorphism \eqref{E-H} is represented concretely by a so-called vielbein. Suppose for instance that $\frac{\partial}{\partial x^{\alpha}}$ is a local coordinate frame on $TM$. Then one can trade tangent space indices for pairs of indices of the type $AA'$, with $A$ $2m$-valued and $A'$ two-valued, by contracting with a vielbein $e^{\A\A'}{}_{\alpha}$. Dually, cotangent space indices can be similarly converted to a two-index notation by contracting with the inverse vielbein $e^{\alpha}{}_{\A\A'}$. Corresponding to these one has the solder form and its dual 
\begin{equation}
e^{\A\A'} = e^{\A\A'}{}_{\alpha} dx^{\alpha}
\qquad\qquad
e^{\vee}_{\A\A'} = e^{\alpha}{}_{\A\A'} \frac{\partial}{\partial x^{\alpha}} \mathrlap{.}
\end{equation}

The Levi-Civita connection on $TM$ induces on the reduced frame bundle a Cartan connection. The Cartan structure equations for the corresponding connection 1-forms $\theta^{\A\A'}{}_{\B\B'}$ read
\begin{equation}
\begin{aligned}
de^{\A\A'} + \theta^{\A\A'}{}_{\B\B'} \!\wedge e^{\B\B'} & = 0 \\
d\theta^{\A\A'}{}_{\B\B'} + \theta^{\A\A'}{}_{\C\C'} \!\wedge \theta^{\C\C'}{}_{\B\B'} & = R^{\A\A'}{}_{\B\B'} \mathrlap{.}
\end{aligned}
\end{equation}
The Levi-Civita connection preserves the $\cramped{E-H}$ decomposition \eqref{E-H} in all dimensions. Accordingly, the Cartan connection 1-forms take the form
\begin{equation}
\theta^{\A\A'}{}_{\B\B'} = \theta^{\A'}{}_{\B'} \delta^{\A}{}_{\B} + \theta^{\A}{}_{\B} \delta^{\A'}{}_{\B'}
\end{equation}
entailing an analogous local decomposition for the curvature 2-form
\begin{equation}
R^{\A\A'}{}_{\B\B'} = R^{\A'}{}_{\B'} \delta^{\A}{}_{\B} + R^{\A}{}_{\B} \delta^{\A'}{}_{\B'}
\end{equation}
with $R^{\A'}{}_{\B'} = d\theta^{\A'}{}_{\B'} + \theta^{\A'}{}_{\C'} \!\wedge \theta^{\C'}{}_{\B'}$ and similarly for $R^{\A}{}_{\B}$.

The bundles $E$ and $H$ come equipped with natural symplectic metrics\,---\,covariantly constant anti-symmetric tensors $\varepsilon_{\A\B}$ and $\varepsilon_{\A'\B'}$ from $\mathscr{A}^2(E)$ and $\mathscr{A}^2(H)$, respectively. These can be used to raise and lower $E$ and $H$-indices. Here we will use the index-raising and lowering conventions from Penrose and Rindler \cite[p.104 \& ff.]{MR776784}. 

As the group $Sp(1)$ is a double cover of $SO(3)$ and the adjoint representation of $Sp(1)$ is isomorphic to the vector representation of $SO(3)$, we are allowed to trade an $SO(3)$ vector index\,---\,such as for instance the one carried by the almost K\"ahler structures $\omega_k$ (for quaternionic K\"ahler manifolds these form the components of a section of an $SO(3)$-bundle over $M$ on which they provide a frame)\,---\,for two primed $H$-indices. This is achieved in practice by means of the Pauli matrices as follows: $\omega^{\A'}{}_{\B'} = \frac{1}{2} (\sigma_k)^{\A'}{}_{\B'} \hp \omega_k$ (here and throughout this section a summation over the repeated index $k$ is implied) and, conversely, $\omega_k = (\sigma_k)^{\B'}{}_{\A'} \omega^{\A'}{}_{\B'}$. Recall that, besides the Pauli algebra, the Pauli matrices satisfy also the orthogonality and completeness relations 
\begin{equation} \label{Pauli_orthog}
\begin{aligned}
(\sigma_k)^{\A'}{}_{\B'} (\sigma_j)^{\B'}{}_{\A'} & = 2\pt \delta_{kj} \\[1pt]
(\sigma_k)^{\A'}{}_{\B'} (\sigma_k)^{\C'}{}_{\D'} & =  2\pt \delta^{\A'}{}_{\D'}\delta^{\C'}{}_{\B'} - \delta^{\A'}{}_{\B'}\delta^{\C'}{}_{\D'} \mathrlap{.}
\end{aligned}
\end{equation}

In the Cartan frame one has the following further decomposition formulas:
\begin{equation} \label{g_I_o}
\begin{tabular}{lr@{=}l}
metric & $\displaystyle g_{\A\A',\B\B'}\ $ & $\displaystyle \ \varepsilon_{\A\B} \varepsilon_{\A'\B'}$ \\[4pt]
almost complex structures & $\displaystyle (I_k)^{\A\A'}{}_{\B\B'}\ $ & $\displaystyle \ -i \hp (\sigma_k)^{\A'}{}_{\B'} \delta^{\A}{}_{\B}$ \\[5pt]
almost K\"ahler structures & $\displaystyle (\omega_{\C'\npt\D'})_{\A\A',\B\B'} \ $ & $\displaystyle \ i \pt \varepsilon_{\A'(\C'}\varepsilon_{\D')\B'}\varepsilon_{\A\B} \mathrlap{.}$ 
\end{tabular}
\end{equation}
The expressions for the metric and almost complex structures are natural choices satisfying the requisite algebraic and Hermiticity properties. The remaining expression for the almost K\"ahler structures can be derived from the first two. By definition we have $\smash{ (\omega_k)_{\A\A',\B\B'} = g_{\A\A',\C\C'} (I_k)^{\C\C'}{}_{\B\B'} }$, and then the formula follows after some standard $\varepsilon$-tensor gymnastics, upon converting the $k$ index into two primed indices and making use of the second Pauli matrix property above.

Quaternionic K\"ahler manifolds are automatically Einstein. What is more, the $Sp(1)$ component of their curvature is proportional to the almost K\"ahler forms, \textit{i.e.},
\begin{equation}
R^{\A'}{}_{\B'} = s \pt \omega^{\A'}{}_{\B'} \mathrlap{,}
\end{equation}
with the proportionality factor $s$ equal to a positive dimension-dependent fractional multiple of the scalar curvature. Hyperk\"ahler manifolds, the subclass of quaternionic K\"ahler manifolds of interest to us here, are in particular Ricci-flat, and therefore for them $s=0$. That is, the connection induced on the $H$-bundle is in their case locally flat. If $M$ is in addition simply connected then this implies that the $H$-bundle is trivial. From here onwards we will assume $M$ to be again hyperk\"ahler (which in practice means that we can drop the ``almost" modifiers from the statements above) and, moreover, simply connected.

\subsection{Killing tensors} \label{ssec:K-tensors} \hfill \medskip

With these general considerations in place, we focus now on a class of hyperk\"ahler spaces possessing a certain type of hidden symmetry, closely related in some way we will make precise later on to tri-Hamiltonian symmetries, first defined and studied by Dunajski and Mason in \cite{MR1785432, MR2006758}.

\begin{definition}
Given a positive integer $j$, a totally symmetric section $\mu_{\A'_1\cdots\A'_{2j}} = \mu_{(\A'_1\cdots\A'_{2j})}$ of $S^{2j}H$ is termed a \textit{valence $(0,2j)$ Killing tensor} (or \textit{Killing spinor}, in four dimensions) if it satisfies the equation
\begin{equation} \label{(0,2j)_Killing_sp_eq}
\nabla_{\A(\A'} \mu_{\A'_1 \cdots \A'_{2j})} = 0 \mathrlap{.}
\end{equation}
\end{definition}

Observe that if the covariant derivative of $\mu_{\pt \cdots}$ is of the form
\begin{equation} \label{nabla_mu}
\nabla^{\A}{}_{\A'} \mu_{\A'_1\cdots\A'_{2j}} = - i \pt \varepsilon_{\A'(\A'_1} X^{\A}{}_{\A'_2\cdots\A'_{2j})}
\end{equation}
for some valence $(1,2j-1)$ tensor $X^{\hp ^{.}}{}_{\!\cdots}$, then $\mu_{\pt \cdots}$ satisfies automatically the Killing tensor equation, as follows from a simple symmetrization argument. (The imaginary factor has been introduced for subsequent convenience. For the general rules and conventions regarding index symmetrization and anti-symmetrization see \textit{e.g.}~\cite[p.132 \& ff.]{MR776784}.) In fact, the covariant derivative of a valence $(0,2j)$ Killing tensor can \textit{always} be written in this form, with
\begin{equation}
X^{\A}{}_{\A'_2\cdots\A'_{2j}} = - i \hp \frac{2j}{2j+1} \nabla^{\A\A'_1} \mu_{\A'_1\A'_2\cdots \A'_{2j}} \mathrlap{,}
\end{equation}
see \textit{e.g.}~\cite[eqs.~(2.8) through (3.1), with the last one needing an obvious correction]{MR0425802}. 

The integrability conditions for the first-order differential equation \eqref{(0,2j)_Killing_sp_eq} lead to constraints on the covariant derivatives of the tensor $X^{\hp ^{.}}{}_{\!\cdots}$. To see this, note first that the commutator of two covariant derivatives acting on $\mu_{\pt \cdots}$ vanishes due to the vanishing of the curvature component along the $H$-bundle:
\begin{equation}
[\nabla_{\A\A'},\nabla_{\B\B'}] \hp \mu_{\A'_1\cdots\A'_{2j}} = \sum_{k=1}^{2j} (R^{\C'}{}_{\A'_k})_{\A\A',\B\B'} \hp \mu_{\A'_1\cdots\C'\cdots\A'_{2j}} = 0 \mathrlap{.}
\end{equation}
This allows us to write further the second-order differential equations
\begin{equation}
\begin{aligned}
\nabla_{(\A}{}^{\A'} \nabla_{\B)\A'}^{\phantom{CC}} \mu_{\A'_1\cdots\A'_{2j}} & = - \frac{1}{2}\varepsilon^{\A'\B'} [\nabla_{\A\A'}, \nabla_{\B\B'}] \hp \mu_{\A'_1\cdots\A'_{2j}} = 0 \\[0pt]
\nabla_{\A(\A'} \nabla^{\A}{}_{\B')} \mu_{\A'_1\cdots\A'_{2j}} & = \phantom{+} \frac{1}{2} \varepsilon^{\A\B} \hspace{5.5pt} [\nabla_{\A\A'}, \nabla_{\B\B'}] \hp \mu_{\A'_1\cdots\A'_{2j}} = 0 \mathrlap{.}
\end{aligned}
\end{equation}
Using then the relation \eqref{nabla_mu} and, in the second case, the third equation \eqref{g_I_o}, we arrive after some straightforward manipulations to the following two constraints:
\begin{gather}
\nabla^{(\A}{}_{(\A'_1} X^{\B)}{}_{\A'_2 \cdots \A'_{2j})} = 0 \\
(\omega_{\A'\B'})_{\C\C',\A(\A'_1} \nabla^{\C\C'} X^{\A}{}_{\A'_2 \cdots \A'_{2j})} = 0 \mathrlap{.}
\end{gather}
The first constraint is simply the valence $(1,2j-1)$ Killing tensor equation. By definition, we call a tensor satisfying both constraints a \textit{tri-Hamiltonian Killing tensor}. This terminology is justified partly by the fact that, in particular, for $j=1$, the two constraints imply that $X^{\A\A'} e^{\vee}_{\A\A'}$ is a tri-Hamiltonian Killing vector field. We leave the proof of this statement as an exercise for the reader and refer to \cite{MR551471} for some useful guidance. So, we have shown that

\begin{proposition} \label{higher_K_tens}
On a hyperk\"ahler manifold, a valence $(0,2j)$ Killing tensor gives rise to a valence $(1,2j-1)$ tri-Ha\-mil\-to\-ni\-an Killing tensor.
\end{proposition}

For the next step it is convenient to recast the triplet of complex structures $I_k$ in the equivalent representation $I_{\A'\B'}$ by trading the $SO(3)$ index for two primed indices in the manner described above. Thus, similarly to the K\"ahler forms which can be viewed as the components of a section of $S^2H \otimes \mathscr{A}^2(M,\mathbb{C})$, the complex structures can be viewed as the components of a section of $S^2H \otimes \text{End}(T_{\mathbb{C}}M)$. This allows us to formulate the following important observation: 

\begin{lemma} \label{pic}
Let $\mu_{\A'_1\cdots\A'_{2j}}$ be a valence $(0,2j)$ Killing tensor. Then
\begin{equation} \label{nabla_mu_I}
\nabla \mu_{(\A'_1 \cdots \A'_{2j}} I_{\B'\npt\C')} = 0 \mathrlap{.}
\end{equation}
\end{lemma}

\begin{proof}
By definition, we have
{\allowdisplaybreaks
\begin{equation}
\begin{aligned}
\nabla \mu_{\A'_1 \cdots \A'_{2j}} I_{\B'\npt\C'} & = \nabla_{\A\A'} \mu_{\A'_1 \cdots \A'_{2j}} (I_{\B'\npt\C'})^{\A\A'}{}_{\D\D'} e^{\D\D'} \\[2pt]
& = e_{\A(\B'} \varepsilon_{\C')(\A'_1} X^{\A}{}_{\A'_2 \cdots \A'_{2j})} \mathrlap{.}
\end{aligned}
\end{equation}
}The second line is obtained by switching the positions of the double index $AA'$ from lowered to raised and vice versa and then using the relation \eqref{nabla_mu} and the third formula \eqref{g_I_o}. The final expression vanishes at the total symmetrization of the primed indices due to the presence of the $\varepsilon$-tensor.
\end{proof}

This formula plays a key role in bridging the current discussion with our previous considerations regarding hyperpotentials. Assuming that the primed indices take the values $0'$ and $1'$, let us define the \textit{components} of the $(0,2j)$ tensor $\mu_{\A'_1 \cdots \A'_{2j}}$ as follows
{\allowdisplaybreaks
\begin{gather}
\mu_{\hp n} = (-)^n C^{\hp n}_{2j} \, \mu_{\hp\underbrace{\0 \cdots \0}_\text{$j+n$}\underbrace{\1 \cdots \1}_\text{$j-n$}} \\
C^{\hp n}_{2j} = \dbinom{2j}{j+n} \text{ for $n \in [-j,j\hp ]$ and $0$ otherwise.} \nonumber
\end{gather}
}%
The order of the indices is of course irrelevant. Notice that this definition makes sense for any $n \in \mathbb{Z}$: when the values in the combinatorial factor are out of range, the corresponding component is by definition equal to zero. A $(0,2j)$ tensor has $2j+1$ non-zero components. 

Taking exactly $j+1+n$ of the indices in the equation \eqref{nabla_mu_I} to be equal to $0'$, we get 
\begin{equation}
\nabla \mu_{\hp n-1} I_{\0\0} + 2 \nabla \mu_{\hp n} \hp I_{\0\1} + \nabla \mu_{\hp n+1} I_{\1\1} = 0 \mathrlap{.}
\end{equation}
For $n = -j-1, \dots, j+1$ these equations form a recursive system for the components of the $(0,2j)$ tensor. Let us consider now the following particular representation for the Pauli matrices
\begin{equation}
(\sigma_k)^{\A'}{}_{\B'} :
\quad
\sigma_1 = \begin{pmatrix*}[r] \phantom{\,} 0 & 1 \phantom{\,} \\ \phantom{\,} 1 & 0 \phantom{\,} \end{pmatrix*},
\quad
\sigma_2 = \begin{pmatrix*}[r] \phantom{\,} 0 & \!\! -\hp i \phantom{\,} \\ \phantom{\,} i & \!\! 0 \phantom{\,} \end{pmatrix*},
\quad
\sigma_3 = \begin{pmatrix*}[r] \phantom{\,} 1 & \!\! 0 \phantom{\,} \\ \phantom{\,} 0 & \!\! -1 \phantom{\,} \end{pmatrix*} 
\end{equation}
for which we have
\begin{equation}
I_{\A'\npt\B'} = \frac{1}{2} (\sigma_k)_{\A'\B'} I_k \ : 
\quad
\begin{pmatrix*} \phantom{\pt} - I_+ & \!\! \frac{1}{2} \hp I_0 \phantom{\,} \\[1pt] \phantom{\pt} \frac{1}{2} \hp I_0  & \!\! - \pt I_- \phantom{\,} \end{pmatrix*} \mathrlap{.}
\end{equation}
Moreover, let us choose this frame in such a way that the connection 1-forms $\theta^{\A'}{}_{\B'}$ vanish. Since for simply-connected hyperk\"ahler manifolds the $H$-bundle is trivial, this is always possible globally. The covariant derivatives can then be replaced with simple derivatives and the above equation becomes
\begin{equation}
d\mu_{\hp n-1} I_+ + d\mu_{\hp n} \hp I_0 + d\mu_{\hp n+1} I_- = 0 \mathrlap{.}
\end{equation}
This is precisely a recursive system of the type \eqref{chain_recursion_2}, so what we have shown is this: 

\begin{proposition} \label{Kill_tens_chain}
Let $\mu_{\A'_1\cdots\A'_{2j}}$ be a valence $(0,2j)$ Killing tensor and consider its components in a frame chosen as above. Then
\begin{equation*}
0 \, \noarrow \, \mu_{-j} \, \noarrow \, \cdots \, \noarrow \, \mu_0 \, \noarrow \, \cdots \, \noarrow \, \mu_{+j} \, \noarrow \, 0 \\[1pt]
\end{equation*} 
forms a recursive chain of hyperpotentials with respect to the complex structure $I_0$.
\end{proposition}

Recursive chains of hyperpotentials, we have seen, are closely related to holomorphic functions on the twistor space. In particular, \textit{bounded} recursive chains of hyperpotentials always have holomorphic functions associated to them as the defining series is finite and there are no convergence issues to worry about. Accordingly, this Proposition has the immediate corollary that valence $(0,2j)$ Killing tensors give rise to holomorphic sections of the pullback bundle $\pi^*\mathcal{O}(2j)$ over $\mathcal{Z}$.

For $j=1$ this recursive chain of hyperpotentials is the same one that we have encountered in \S\,\ref{ssec:tri-Ham}. There we saw that the associated holomorphic function can be interpreted as a Hamiltonian function with respect to the twisted symplectic form on each twistor fiber. We will now show that a similar statement holds for general values of $j$. To this purpose, it is convenient to view the valence $(1,2j-1)$ tensor field $X^{\hp ^{.}}{}_{\!\cdots}$ as a $S^{2j-2}H$-valued vector field
\begin{equation}
X_{\A'_1\cdots\A'_{2j-2}} = X^{\A\A'}{}_{\A'_1\cdots\A'_{2j-2}} e^{\vee}_{\A\A'} 
\end{equation}
where $e^{\vee}_{\A\A'}$ stands for the vector frame dual to the coframe $e^{\A\A'}$.  

\begin{lemma} \label{poc}
With the definitions and assumptions above, the following formula holds:
\begin{equation} \label{i_X_o_symm}
\iota_{X_{(\A'_1\cdots\A'_{2j-2}}} \omega_{_{\A'_{2j-1}\A'_{2j})}} = \nabla \mu_{\A'_1 \cdots \A'_{2j}} \mathrlap{.}
\end{equation}
\end{lemma}

\begin{proof}

From the definitions, then the third formula \eqref{g_I_o}, and finally the relation \eqref{nabla_mu} we have successively
{\allowdisplaybreaks
\begin{align}
\iota_{X_{(\A'_1\cdots\A'_{2j-2}}} \omega_{_{\A'_{2j-1}\A'_{2j})}} & = X^{\B\B'}{}_{(\A'_1\cdots\A'_{2j-2}} (\omega_{\A'_{2j-1}\A'_{2j})})_{\B\B',\C\C'} e^{\C\C'} \\[1pt]
& = - i \pt \varepsilon_{\C'(\A'_{2j}} X^{\B}{}_{\A'_{2j-1}\A'_1\cdots\A'_{2j-2})} \varepsilon_{\B\C} \pt e^{\C\C'} \nonumber \\[3pt]
& = \nabla_{\C\C'} \mu_{\A'_1 \cdots \A'_{2j}} e^{\C\C'} \nonumber
\end{align}
}which proves the statement.
\end{proof}

Proceeding as before, we introduce now vector field-valued components defined as follows:
{\allowdisplaybreaks
\begin{gather}
X_{n} = (-)^n C^{\hp n}_{2j-2} \pt X_{\hp\underbrace{\0 \cdots \0}_\text{$j\!-\!1\!+\!n$}\underbrace{\1 \cdots \1}_\text{$j\!-\!1\!-\!n$}} \\
C^{\hp n}_{2j-2} = \dbinom{2j-2}{j-1+n} \text{ for $n \in [1-j,j-1]$ and $0$ otherwise.} \nonumber
\end{gather}
}%
The components for which the parameters of the corresponding combinatorial factor are out of range vanish by definition, leaving $2j-1$ non-vanishing components. The component of equation \eqref{i_X_o_symm} having exactly $j+n$ indices equal to $0'$ reads then
\begin{equation} 
\iota_{X_{n-1}} \omega_{\0\0} + 2\hp \iota_{X_{n}} \omega_{\0\1} + \iota_{X_{n+1}} \omega_{\1\1} = \nabla \mu_{\hp n}
\end{equation}
for any $n = -j, \dots, j$. When choosing the same frame as in Proposition~\ref{Kill_tens_chain} this equation becomes
\begin{equation} \label{i_X_M}
\iota_{X_{n-1}} \omega_+ + \iota_{X_{n}} \omega_0 + \iota_{X_{n+1}} \omega_- = d\mu_{\hp n} \mathrlap{.}
\end{equation}
Acting on this relation with an exterior derivative and using the fact that the hyperk\"ahler 2-forms are all closed, we get via Cartan's formula for the Lie derivative
\begin{equation} \label{L_X_M}
\mathcal{L}_{X_{n-1}} \omega_+ + \mathcal{L}_{X_{n}} \omega_0 + \mathcal{L}_{X_{n+1}} \omega_- = 0 \mathrlap{.}
\end{equation}
%\textcolor{blue}{Note that the converse statement also holds. Given a set of vector fields satisfying this last property, based on the closure of the hyperk\"ahler 2-forms one can retrieve locally by means of the Poincar\'e lemma the potentials $\mu_{\hp n}$ and the equations \eqref{i_X_M}. }

These results are best understood if we adopt a twistor space perspective. The vector fields $X_n$ can be viewed as the components of a section of $T_F\mathcal{Z} \otimes \pi^*\mathcal{O}(2j-2)$ over $\mathcal{Z}$. In a certain local trivialization of $\mathcal{Z}$ this section takes the form
\begin{equation} \label{twisted_X}
X(\zeta) = \frac{X_{j-1}}{\zeta^{j-1}} + \cdots + X_0 + \cdots + \zeta^{j-1} X_{-j+1}
\end{equation}
with $\zeta$ an inhomogeneous complex coordinate on the twistor base chosen such that $\zeta = 0$ labels the complex structure $I_0$. Correspondingly, the holomorphic section of $\pi^*\mathcal{O}(2j)$ associated to the recursive chain of hyperpotentials of Proposition \ref{Kill_tens_chain} reads
\begin{equation}
\mu(\zeta) = \frac{\mu_{+j}}{\zeta^j} + \cdots + \mu_0 + \cdots + \zeta^j \mu_{-j} \mathrlap{.}
\end{equation}
The equations \eqref{i_X_M} assemble naturally into the \textit{fiberwise} moment map equation 
\begin{equation} \label{twist_mom-map}
\iota_{X(\zeta)} \omega(\zeta) = d\mu(\zeta) 
\end{equation}
\noindent while the equations \eqref{L_X_M} give similarly the \textit{fiberwise} symplectic invariance condition 
\begin{equation} \label{twist_tri-Ham}
\mathcal{L}_{X(\zeta)} \omega(\zeta) = 0 \mathrlap{.}
\end{equation}

\begin{remark}
For $j=1$ this is the equation \eqref{Ham-zeta} from the tri-Hamiltonian case. We thus see that the ``hidden" symmetries labeled by $j>1$ correspond naturally to a \textit{twisting} of the tri-Ha\-mil\-to\-nian condition. For this reason we say that the symmetries associated to Killing tensors proper belong to a \textit{trans-tri-Ha\-mil\-to\-nian hidden series}.
\end{remark}

Observe that, unlike in the $j=1$ case, in the $j>1$ one the generator $X(\zeta)$ is \textit{not} uniquely defined. In the latter case, a redefinition of $X(\zeta)$ of the form %$X(\zeta) \longmapsto X(\zeta) + I(\zeta)Y(\zeta)$~with 
\begin{equation}
X(\zeta) \longmapsto X(\zeta) + I(\zeta)Y(\zeta)
\quad\text{with}\quad
Y(\zeta) = \sum_{n =2-j}^{j-2} Y_n \hp \zeta^{-n}
\end{equation}
%\begin{equation}
%Y(\zeta) = \sum_{n =2-j}^{j-2} Y_n \hp \zeta^{-n}
%\end{equation}
leaves the formulas \eqref{twist_mom-map} and \eqref{twist_tri-Ham} invariant for any choice of vector fields $Y_n \in T_{\mathbb{C}}M$ subject only to the reality condition $\bar{Y}_{n} = (-)^n Y_{-n} $, required to preserve the reality property of $X(\zeta)$. Since \eqref{twist_tri-Ham} is implied by \eqref{twist_mom-map}, it suffices to show that $\iota_{X(\zeta)}\omega(\zeta)$ is invariant. This follows immediately from the fact that $I(\zeta)Y(\zeta)$ and $\omega(\zeta)$, viewed as a vector field respectively a 2-form on $M$, are of type $(0,1)$ respectively $(2,0)$ relative to $\IU$, and thus their contraction vanishes trivially. Equivalently, the two sets of formulas \eqref{i_X_M} and \eqref{L_X_M}  are invariant at redefinitions \mbox{$X_n \longmapsto X_n + I_+Y_{n-1} + I_0 Y_n + I_- Y_{n+1}$}. %for arbitrary choices of vector fields $Y_n$ subject to the above reality constraint and the requirement that $Y_n = 0$ for $n$ outside the interval $[2-j,j-2]$. 
Rather than consider it a drawback, we can use this freedom to choose a convenient representative for the system of vector fields $X_n$. Thus, for any vector field $X(\zeta)$ of the type \eqref{twisted_X}, performing a transformation with 
\begin{equation}
Y(\zeta) = - \sum_{n=0}^{j-2} ( I_-X_{n+1} \zeta^{-n} + I_+ X_{-n-1} \zeta^n)
\end{equation}
brings it to the form
\begin{equation}
X^{\star}(\zeta) = \sum_{n=0}^{j-1} [P^{1,0}_{\Io} (X_n - i \hp I_-X_{n+1}) \hp \zeta^{-n } + P^{0,1}_{\Io} (X_{-n} + i \hp I_+X_{-n-1}) \hp \zeta^n] \mathrlap{.}
\end{equation}
This shows that

\begin{lemma} \label{canon_pres}
Given a trans-tri-Hamiltonian hidden action, we can always redefine its generators $X_n$ so that they become sections of the bundle $T^{1,0}_{\Io}M$ if \mbox{$n > 0$} and of the bundle $T^{0,1}_{\Io}M$ if \mbox{$n < 0$}.
\end{lemma}

The conclusions of this discussion can be summed up as follows:

\begin{theorem} \label{Kill_spins_tw_sp}
A valence $(0,2j)$ Killing tensor on a simply-connected hyperk\"ahler manifold $M$ with twistor space $\mathcal{Z}$ gives rise to a global section of $T_F\mathcal{Z} \otimes \pi^*\mathcal{O}(2j-2)$ whose fiberwise Lie derivative preserves the symplectic form induced on each fiber by the canonical holomorphic 2-form of $\mathcal{Z}$, and for which the holomorphic $\pi^*\mathcal{O}(2j)$ section associated to the recursive chain of hyperpotentials of Proposition \ref{Kill_tens_chain} plays on each fiber the role of Hamiltonian function.
\end{theorem}

\begin{remark}
Our considerations here can be viewed as an explicit demonstration of how Dunajski and Mason's Killing spinor ideas \cite{MR1785432, MR2006758} connect with Bielawski's notion of \textit{twistor group actions} \cite{MR1848654}. The latter describe situations when there is no group action preserving the hyperk\"ahler structure, but where for each complex structure there is an action of a complex group preserving the corresponding complex symplectic structure. 
\end{remark}

A large class of examples of hyperk\"ahler metrics with trans-tri-Hamiltonian hidden symmetries is given by Lindstr\"om and Ro\v{c}ek's \textit{generalized Legendre transform metrics} \cite{MR929144}. These are metrics for which the holomorphic twistor projection factorizes through a direct sum of line bundles over $\mathbb{CP}^1$ of positive even degree:
\begin{center}
\begin{tikzcd}
\mathcal{Z} \rar & \displaystyle \bigoplus_{\I=1}^{\dim_{\pt \mathbb{H}} M} \! \mathcal{O}(2j_{\I}) \rar & \mathbb{CP}^1 \rlap{.}
\end{tikzcd}
\end{center}
For a recent review including a detailed discussion of their hidden symmetries in a similar vein as here we refer the reader to \cite{MR3681388} (esp.~\S\,3.3).

\section{Rotational and trans-rotational symmetries} \label{sec:Trans-rot}

\subsection{Quasi-rotational vector fields}

\subsubsection{}

According to their action on the 2-sphere of hyperk\"ahler symplectic structures, isometric vector field actions on hyperk\"ahler manifolds fall into one of two categories: 
\begin{itemize}
\setlength{\itemsep}{2pt}
\item[1.] tri-Hamiltonian actions, for which every point of the sphere is a fixed point;
\item[2.] rotational actions, with two antipodally-opposite fixed points.
\end{itemize}
In the previous section we have seen that the first class of actions admits twisted versions forming a trans-tri-Hamiltonian class of hidden symmetries. In this section we will pursue the question of whether similar twistings are possible in the second case as well.

\begin{definition}
Let $M$ be a hyperk\"ahler manifold. We call a vector field $X \in TM$ \textit{rotational} if its Lie action generates a rotation of the 2-sphere of hyperk\"ahler symplectic structures about a fixed axis, as viewed from the perspective of its natural embedding in $\mathbb{R}^3$.
\end{definition}

A vector field $X$ whose Lie action on the standard basis of hyperk\"ahler symplectic structures takes the form 
\begin{equation} \label{axial_rot}
\mathcal{L}_X \omega_1 = \omega_2
\qquad\quad
\mathcal{L}_X \omega_2 = - \, \omega_1
\qquad\quad
\mathcal{L}_X \omega_3 = 0
\end{equation}
is rotational with respect to the 3-axis. The first two conditions imply by way of the third formula \eqref{i=oo} that $\mathcal{L}_XI_3 = 0$, and since $X$ also preserves the symplectic structure $\omega_3$ it follows that $X$ is a Killing vector field for the hyperk\"ahler metric on $M$. This argument carries over for any choice of rotation axis to show that actions generated by rotational vector fields are necessarily isometric. 

Note that in the alternative complexified basis given by $\omega_+$, $\omega_0$, $\omega_-$ the conditions \eqref{axial_rot} assume the diagonal form 
\begin{equation} \label{eigenform}
\begin{aligned}
\mathcal{L}_X\omega_{\pm} & = \mp \hp i \pt \omega_{\pm}  \\
\mathcal{L}_X\omega_0 \, & = 0 \rlap{.}
\end{aligned}
\end{equation}
With this in mind, let us introduce the following relaxation of the notion of rotational vector field:  

\begin{definition}
We call a vector field $X \in TM$ \textit{quasi-rotational (relative to the 3-axis)} if there exists a complex-valued 2-form $\sigma_+$ on $M$ of type $(1,1)$ with respect to the complex structure $I_3 \equiv I_0$ such that
\begin{equation} \label{quasi_rot}
\begin{aligned}
\mathcal{L}_X\omega_{\pm} & = \mp \hp i \hp (\hp \omega_{\pm} - \sigma_{\pm}) \\[0pt]
\mathcal{L}_X\omega_0 \, & = 0 
\end{aligned}
\end{equation}
where, by definition, $\sigma_- = - \, \bar{\sigma}_+$.
\end{definition}

\begin{remark}
This definition was considered by Korman in \cite{MR3625762}, albeit in a form requiring that $\sigma_+$ be of hyper $(1,1)$ type rather than just of simple $(1,1)$ type with respect to a single complex structure. We will see immediately that the less restrictive condition we consider here implies the one in \cite{MR3625762}.
\end{remark}

\begin{examples} \hfill \smallskip

1.~An important class of examples is provided by Korman in \textit{op.~cit.}, who considers what happens to a \textit{spectator} $S^1$-action of rotational type under hyperk\"ahler reduction\,---\,that is, to an $S^1$-action which commutes with the tri-Hamiltonian Lie group action $G$ with respect to which the hyperk\"ahler quotient is being taken. What he finds is that the spectator $S^1$-action does not always descend to a similarly rotational $S^1$-action on the reduced hyperk\"ahler manifold, but rather it descends in general to a deformed version of such an action, of which the above quasi-rotational actions are an example. By observing that the conditions on the generating vector field $X$ for the original $S^1$-action can be equivalently reformulated as a set of conditions for the 1-form $\smash{ \alpha = \iota_X\omega_0 }$, Korman is also able to generalize his arguments to an infinite-dimensional setting where such a vector field might not exist. \hfill \smallskip

2.~\noindent The next example, inspired by considerations in \cite{Lindstrom:2008gs}, shows that, \textit{locally}, vector fields with the quasi-rotational property occur quite naturally on hyperk\"ahler manifolds. For any function \mbox{$f \in \mathscr{A}^0(M)$}, where $M$ represents here a generic hyperk\"ahler manifold, let \mbox{$X_f \in TM$} denote the symplectic gradient of $f$ with respect to the hyperk\"ahler symplectic form $\omega_0$, that is,
\begin{equation} 
\iota_{X_f}\omega_0 = df \mathrlap{.}
\end{equation}
We also have then
\begin{equation} \label{mambo}
\begin{aligned}
& \iota_{X_f}\omega_+ = df \omega_0^{-1}\omega_+ = i \hp df I_+ \\[1pt]
& \iota_{X_f} g = - \, (\iota_{X_f}\omega_0) I_0 = - \, df I_0 \mathrlap{.}
\end{aligned}
\end{equation}
Let $\kappa_0$ be a real K\"ahler potential for $\omega_0$, that is, $\omega_0 = - \hp i \pt \partial\bar{\partial} \kappa_0$ on some open set in $M$. Then the locally-defined symplectic gradient  $X_{-\kappa_0/2}$ \pagebreak is a quasi-rotational vector field relative to the 3-axis. More precisely, the following two formulas hold:%\pt\footnote{\pt The action on $\omega_-$ can be trivially obtained from that on $\omega_+$ by (alternating) complex conjugation.}
\begin{equation} \label{X_k0/2}
\begin{aligned}
\mathcal{L}_{X_{-\kappa_0/2}} \omega_+ & = -i \hp [ \hp \omega_+ + F_+(\kappa_0/2)] \\[1pt]
\mathcal{L}_{X_{-\kappa_0/2}} \omega_0 \, & = 0 \mathrlap{.}
\end{aligned}
\end{equation}
(The action on $\omega_-$ can be trivially obtained from that on $\omega_+$ by alternating complex conjugation.) Indeed, the second formula is a direct consequence of the closure of $\omega_0$ by way of Cartan's homotopy formula:
\begin{equation}
\mathcal{L}_{X_{-\kappa_0/2}} \omega_0 = d(\iota_{X_{-\kappa_0/2}}\omega_0) = - \, d^2 (\kappa_0/2) = 0 \rlap{.}
\end{equation}
On the other hand, by a similar first step and then resorting successively to the first relation \eqref{mambo} and the definitions \eqref{Fs_and_Hs} we get
\begin{equation}
\mathcal{L}_{X_{-\kappa_0/2}} \omega_+ = d(\iota_{X_{-\kappa_0/2}}\omega_+)  = - i \hp d\hp [d(\kappa_0/2)I_+] = -i \hp [H_+(\kappa_0/2) + F_+(\kappa_0/2) ] \mathrlap{.}
\end{equation}
Furthermore, choosing for convenience an arbitrary local coordinate system holomorphic with respect to $I_0$ and using the first integrability condition \eqref{HK_integrability} yields for the first term 
{\allowdisplaybreaks
\begin{align} \label{H+(K)}
H_+(\kappa_0/2) & = \frac{1}{2} \partial(\bar{\partial}\kappa_0 I_+) \\
& = \frac{1}{2} \partial \pt [ \hp \partial_{\bar{\rho}} \kappa_0 (I_+)^{\bar{\rho}}{}_{\nu} \pt dx^{\nu} ] \nonumber \\
& = \frac{1}{2} [ \underbracket[0.4pt][4pt]{\partial_{\mu}\partial_{\bar{\rho}} \kappa_0}_{\text{\normalsize{$g_{\mu\bar{\rho}}$}} \rule[6pt]{0pt}{0pt}} (I_+)^{\bar{\rho}}{}_{\nu} + \partial_{\bar{\rho}} \kappa_0 \underbracket[0.4pt][4pt]{\partial_{\mu}(I_+)^{\bar{\rho}}{}_{\nu}}_{\text{vanishes}}] \pt dx^{\mu} \!\wedge dx^{\nu} \nonumber \\[-8pt]
& = \frac{1}{2} \omega_{+\mu\nu} \pt dx^{\mu} \!\wedge dx^{\nu} \nonumber \\[6pt]
& = \omega_+ \rlap{.} \nonumber
\end{align}
}%
This proves the remaining formula. For reasons to become soon apparent let us also mention an equivalent reformulation of the equations \eqref{X_k0/2}. From the generic hyperk\"ahler formula \eqref{HK_formula} we have: 
\begin{equation} \label{grad_quasi_rot}
\begin{aligned}
\mathcal{L}_{I_0X_{-\kappa_0/2}} \omega_+ & = i \hp \mathcal{L}_{X_{-\kappa_0/2}} \omega_+ = \omega_+ + F_+(\kappa_0/2) \\
\mathcal{L}_{I_0X_{-\kappa_0/2}} \omega_0 \, & = d(\iota_{X_{-\kappa_0/2}}g) = \frac{1}{2} d(d\kappa_0I_0) = \omega_0 \rlap{.}
\end{aligned}
\end{equation}
In the last derivation we made use of the second relation \eqref{mambo}.

\end{examples}

%\smallskip

Returning now to the definition, notice that due to the fact that $\omega_+$ and $\omega_-$ are of type $(2,0)$ respectively $(0,2)$ relative to the complex structure $I_0$\,---\,or equivalently, by the formula \eqref{HK_formula}\,---\,we may write the first equation \eqref{quasi_rot} in the form
\begin{equation} \label{Lolek}
\mathcal{L}_{I_0X}\omega_{\pm} = \omega_{\pm} - \sigma_{\pm} \rlap{.}
\end{equation}
In parallel to this, let us \textit{define} an additional 2-form $\sigma_0$ by 
\begin{equation} \label{Bolek}
\mathcal{L}_{I_0X}\omega_0 = \omega_0 \pt - \pt \sigma_0 \mathrlap{.}
\end{equation}

\begin{remarks} \hfill
\begin{itemize}
\setlength{\itemsep}{3pt}
\item[1.] Any rotational vector field is also quasi-rotational, with
\begin{equation}
\sigma_+ = \sigma_- = 0
\end{equation}
and in general a non-trivial $\sigma_0$.
\item[2.] At the opposite end, in the second example above the formulas \eqref{grad_quasi_rot} show for instance that $X_{-\kappa_0/2}$ is a locally-defined quasi-rotational vector field relative to the 3-axis with  
\begin{equation}
\sigma_{\pm} = - \, F_{\pm}(\kappa_0/2)
\quad\text{and}\quad
\sigma_0 = 0
\rlap{.}
\end{equation}
%\begin{equation}
%\sigma_+ = - F_+(\kappa_0/2)
%\qquad\quad
%\sigma_0 = 0
%\qquad\quad
%\sigma_- = - F_-(\kappa_0/2)
%\rlap{.}
%\end{equation}
\end{itemize}
\end{remarks}
 
\medskip

\noindent The following key property generalizes Proposition 1 of \cite{MR3116317} to the class of quasi-rotational vector fields:

\begin{proposition} \label{quasi_rot_hyperpot}
$\sigma_+$, $\sigma_0$, $\sigma_-$ thus defined are all closed hyper $(1,1)$ forms.
\end{proposition}

\begin{proof}

% Closure is immediately apparent from the definitions and is due to the closure of the hyperk\"ahler symplectic forms.

Let us begin by observing that for any vector field $Y \in TM$ and any $m \in \{-1,0,1\}$ we have $\omega_m(I_0X,Y) = g(I_0X,I_mY) = g(I_mY,I_0X) = \omega_0(I_mY,X) = - \, \omega_0(X,I_mY)$, which may be expressed equivalently as $\iota_{I_0X}\omega_m = -\, (\iota_X\omega_0) I_m$. Using Cartan's formula, the equations \eqref{Lolek}--\eqref{Bolek} yield then the relations
\begin{equation} \label{[sgm]=[omg]}
\sigma_m = \omega_m  + d[(\iota_X\omega_0) I_m] \rlap{.}
\end{equation}
The closure of the $\sigma_m$ follows from the closure of the hyperk\"ahler symplectic forms.

The second equation \eqref{quasi_rot} states that the action of $X$ is symplectic with respect to the hyperk\"ahler symplectic structure $\omega_0$, implying that locally on $M$ there must exist a real-valued function $\mu$ defined up to constant shift such that
\begin{equation} \label{mom_map_3}
\iota_X\omega_0 = d\mu \mathrlap{.}
\end{equation}
If we assume that $H^1(M,\mathbb{R}) = 0$ then the action of $X$ is Hamiltonian and the moment map $\mu$ can be defined globally on $M$. The last two relations imply together that
\begin{equation} \label{sgm-omg}
\sigma_m = \omega_m  + d(d\mu I_m) \rlap{.}
\end{equation}

The $m = 1$ component of this equation has the structure
\begin{equation}
\underbracket[0.4pt][4pt]{\sigma_+}_{\text{$_{1,1}$}} = \underbracket[0.4pt][4pt]{\omega_+}_{\text{$_{2,0}$}} + \, d(\underbracket[0.4pt][4pt]{d\mu \hp I_+}_{\text{$_{1,0}$}}) 
\mathrlap{\hspace{5pt} \underbracket[0pt][7.2pt]{\phantom{}}_{\text{w.r.t. $I_0$}}}
\end{equation}
By matching Hodge types with respect to the complex structure $I_0$ we obtain 
\begin{equation}
\begin{aligned}
\sigma_+ & = \phantom{+\,} \bar{\partial}(\bar{\partial}\mu \hp I_+) = \phantom{+} \hspace{3.8pt} F_+(\mu) \\[1pt]
\omega_+ & = -\, \partial(\bar{\partial}\mu \hp I_+) = -\, H_+(\mu) \mathrlap{.}
\end{aligned}
\end{equation}
The first equation implies immediately, via Lemma \ref{F+F-}, that $\sigma_+$ must be a hyper $(1,1)$ form. On the other hand, from \eqref{H+(K)} we have $H_+(\kappa_0/2) = \omega_+$, which, together with the second equation, gives us $H_+(\kappa_0 + 2 \mu) = 0$. Complex conjugation yields in turn $H_-(\kappa_0 + 2 \mu) = 0$. Note in addition that the $m=0$ component of the equation \eqref{sgm-omg} gives locally
\begin{equation} \label{sgm0_pot}
\sigma_0 = -i \hp \partial\bar{\partial} (\kappa_0 + 2\mu) \rlap{.}
\end{equation}
The equivalence $1 \Leftrightarrow 2$ of Proposition \ref{criterion_2} guarantees then that $\sigma_0$ must be a hyper $(1,1)$ form as well. 
\end{proof}

\begin{remark}
If the vector field $X$ is defined globally on $M$, the relations \eqref{[sgm]=[omg]} imply that, for each value of $m$, $\sigma_m$ and $\omega_m$ belong to the same cohomology class in $H^2(M,\mathbb{C})$.
\end{remark}

\subsubsection{}

According to Lemma \ref{exp_ddbar_lem}, to any closed hyper $(1,1)$ form one can associate locally a recursive chain of hyperpotentials. \pagebreak Thus, to $\sigma_+$, $\sigma_-$ and $\sigma_0$ we can associate respectively two mutually conjugated ``complex" recursive chains and one self-conjugated one. Let us focus now for a moment on $\sigma_+$. As a closed hyper $(1,1)$ form, $\sigma_+$ can be locally derived from a complex hyperpotential $\varphi_+$ with respect to the complex structure $I_0$ through the formula
\begin{equation} \label{ragu}
\sigma_+ = i \hp \partial\bar{\partial} \hp \varphi_+ \mathrlap{.}
\end{equation}
Using this in the first equation \eqref{quasi_rot} yields
\begin{equation}
d(\underbracket[0.4pt][4pt]{\iota_X\omega_+ - \partial\varphi_+}_{\text{$_{1,0}$}}) = -i \npt \underbracket[0.4pt][4pt]{\omega_+}_{\text{$_{2,0}$}} 
\mathrlap{\hspace{5pt} \underbracket[0pt][7.2pt]{\phantom{}}_{\text{w.r.t. $I_0$}}}
\end{equation}
By separating this equality into a $(1,1)$ part and a $(2,0)$ part with respect to $I_0$, we infer that locally there exists a symplectic 1-form potential $\theta_+$ for $\omega_+$ holomorphic relative to $I_0$ (that is, a locally-defined 1-form $\theta_+ \in \mathscr{A}^{1,0}_{\scriptscriptstyle \hp I_0}(M)$ satisfying the two conditions: \mbox{$\omega_+ = d \theta_+$} and $\bar{\partial}\theta_+ = 0$) such that 
\begin{equation} \label{quasi_mom_map_+}
\iota_X\omega_+ = \partial\varphi_+ - i \hp \theta_+ \mathrlap{.}
\end{equation}
This can be regarded as a counterpart to the moment map equation \eqref{mom_map_3}. Together, via the first property \eqref{mambo}, they imply in particular that
\begin{equation} \label{mom-pot}
d\varphi_+ - i \hp \theta_+ = i\hp d\mu I_+ + \bar{\partial} \varphi_+  \mathrlap{.}
\end{equation}
Define now a local vector field $S \in T^{1,0}_{\scriptscriptstyle I_0}M$ holomorphic relative to $I_0$ through the condition \mbox{$\iota_{S \hp} \omega_+ = \theta_+$}. By expressing $\omega_0$ in terms of a local K\"ahler potential $\kappa_0$ and exploiting the holomorphicity property of $S$ we can show that $\iota_{S \hp}\omega_0 = - i \hp \bar{\partial} \hp [S(\kappa_0)]$. These last two relations enable us to recast the pair of equations \eqref{quasi_mom_map_+} and \eqref{mom_map_3} into the form 
\begin{equation} \label{double_grad}
\begin{aligned}
\iota_{P^{1,0}_{\Io}X + i S \pt} \omega_+ & = \partial \varphi_+ \\[-2pt]
\iota_{P^{1,0}_{\Io}X + i S \pt} \omega_0 \, & = \bar{\partial} \varphi_0
\end{aligned}
\end{equation}
where $P^{1,0}_{\Io}X$ is the $(1,0)$ component of $X$ with respect to $I_0$ and, by definition,
\begin{equation}
\varphi_0 = \mu + S(\kappa_0) \mathrlap{.}
\end{equation}
The equivalence $1 \Leftrightarrow 5$ of Proposition \ref{criterion_2} ensures then that $\varphi_0$ and $\varphi_+$ are adjacent hyperpotentials in a recursive chain of hyperpotentials with respect to $I_0$, that is, we have symbolically
\begin{equation*}
\cdots \,\noarrow\, \varphi_0 \,\noarrow\, \varphi_+ \,\noarrow\, \cdots \\[2pt]
\end{equation*}
By the equation \eqref{ragu} and the chain recurrence relations we have $\sigma_+ = F_0(\varphi_+) = F_+(\varphi_0)$, and so this chain is associated to $\sigma_+$ in the sense of Lemma \ref{exp_ddbar_lem}.

\subsection{Twisted rotational actions}

\subsubsection{}

With the help of these concepts we are now ready to tackle the issue of trans-rotational symmetries. To define these, we begin with the observation that in the case when $X$ is a rotational vector field with respect to the 3-axis the corresponding equations \eqref{eigenform} can be assembled on the twistor space into a single formula as follows
\begin{equation}
\Big( \! - i \hp \zeta \frac{\partial}{\partial \zeta} + \mathcal{L}_{X} \Big) \omega(\zeta) = 0 
\end{equation}
with the Lie derivative taken fiberwise. The analogous formula in the tri-Hamiltonian case is equation \eqref{Ham-zeta}. In \S\,\ref{ssec:K-tensors} we have seen that trans-tri-Hamiltonian hidden symmetries are characterized by a twisted version of the tri-Hamiltonian formula, namely, equation \eqref{twist_tri-Ham}. The twist consists in formally replacing $X$ with a multi-component fiberwise-acting vector field $X(\zeta)$ of the form \eqref{twisted_X}. In this analogy, the formula above opens up the intriguing possibility of a similar twisting in the rotational case. In what follows we will show that such twisted actions not only can be defined, but they also lead to interesting, non-trivial, new hyperk\"ahler symmetries forming, in parallel with the trans-tri-Hamiltonian twisted series, a \textit{trans-rotational twisted series}. 

So, in accordance with this heuristic argument, for any given integer $j \geq 1$ let us consider now a condition of the type
\begin{equation} \label{twist_ax_rot_1}
\Big(\! - i \hp \zeta \frac{\partial}{\partial \zeta} + \mathcal{L}_{X(\zeta)} \Big) \omega(\zeta) = 0
\end{equation}
with the Lie derivative being taken along the fiber directions and $X(\zeta)$ the tropical component of a section of the bundle $T_F\mathcal{Z} \otimes \pi^*\mathcal{O}(2j-2)$, assumed real with respect to the real structure on $\mathcal{Z}$. More precisely, in an appropriate local trivialization of $\mathcal{Z}$,
\begin{equation} \label{twist_X}
X(\zeta) = \sum_{n =1-j}^{j-1} X_n \hp \zeta^{-n}
\end{equation}
with the components $X_n$ forming a system of $2j-1$ vector fields in $T_{\mathbb{C}}M$ obeying the alternating reality condition $\bar{X}_{n} = (-)^n X_{-n}$. In purely hyperk\"ahler space terms the condition \eqref{twist_ax_rot_1} is equivalent to requiring that the vector fields $X_n$ satisfy the property
\begin{equation} \label{twist_ax_rot_2}
\mathcal{L}_{X_{n-1}} \omega_+ + \mathcal{L}_{X_{n}} \omega_0 + \mathcal{L}_{X_{n+1}} \omega_- = - \pt i \hp n \pt \omega_n \rlap{.}
\end{equation}
%\begin{equation} \label{twist_ax_rot_2}
%\mathcal{L}_{X_{n-1}} \omega_+ + \mathcal{L}_{X_{n}} \omega_0 + \mathcal{L}_{X_{n+1}} \omega_- =
%\begin{cases}
%\displaystyle{-i \pt \omega_+} & \text{if $n = +1$} \\
%\displaystyle{\phantom{+}i \pt \omega_-} & \text{if $n = -1$} \\
%\ \ \ 0 & \text{otherwise} \mathrlap{.}
%\end{cases}
%\end{equation}
To simplify the appearance of our formulas we adopt the convention that $X_n$ is defined for all $n \in \mathbb{Z}$ but such that $X_n = 0$ when $n$ is outside the symmetric interval $[1-j,j-1]$, and similarly for $\omega_n$, which are assumed to vanish for all $n \in \mathbb{Z}$ outside the interval $[-1,1]$. Note that for $j=1$ the relations \eqref{twist_ax_rot_2} reduce to the rotational equations \eqref{eigenform}.

\begin{definition}
We refer to a system of vector fields on $M$ having this property for some $j > 1$ as generating an \textit{$\mathcal{O}(2j-2)$-twisted rotational action relative to the 3-axis}.
\end{definition}

\begin{remark}
To facilitate a unitary treatment, in what follows we will in effect often extend this definition to include the limit case $j=1$, that is, the rotational case proper.
\end{remark}

By substituting for the Lie derivatives Cartan's formula it is easy to see that due to the closure of the hyperk\"ahler symplectic forms the left-hand side of the equation \eqref{twist_ax_rot_2} is an exact 2-form, and so based on Poincar\'e's lemma we can infer that there exist complex potentials $\varphi_n$, a real-valued potential $\mu$, and a 1-form symplectic potential $\theta_+$ (for $\omega_+$) holomorphic relative to $I_0$ such that 
\begin{equation} \label{gen_mom_maps}
\iota_{X_{n-1}}\omega_+ + \iota_{X_{n}}\omega_0 + \iota_{X_{n+1}}\omega_- =
\begin{cases}
d\varphi_n & \text{if $n = 2,...\, , j$} \\
d\varphi_+ \!- i \hp \theta_+ & \text{if $n = 1$} \\
d\mu & \text{if $n=0$} \rlap{.}
\end{cases}
\end{equation}
The equations for negative values of $n$ can be obtained from the corresponding positive-value equations by complex conjugation. \pagebreak We will assume here that $H^1(M,\mathbb{R}) = 0$, in which case the potentials $\mu$ and $\varphi_n$ with $n=2,\dots,j$ are defined \textit{globally} on $M$. Moreover, note that even though neither $\varphi_+$ nor $\theta_+$ are globally defined, the combination $d\varphi_+ \!- i\hp \theta_+$ is.

\subsubsection{}

An important difference between the purely rotational case and the twisted one is that in the latter case the generator $X(\zeta)$ is not uniquely defined. Indeed, for $j>1$, the defining condition \eqref{twist_X} remains invariant at any redefinition of the type
\begin{equation}
X(\zeta) \longmapsto X(\zeta) + I(\zeta)Y(\zeta)
\quad\text{with}\quad
Y(\zeta) = \sum_{n =2-j}^{j-2} Y_n \hp \zeta^{-n}
\end{equation}
such that $Y_n \in T_{\mathbb{C}}M$ are subject to the reality condition \mbox{$\bar{Y}_{n} = (-)^n Y_{-n}$} but otherwise arbitrary. 
An equivalent statement is that the two sets of formulas \eqref{twist_ax_rot_2} and \eqref{gen_mom_maps}  are invariant at redefinitions \mbox{$X_n \longmapsto X_n + I_+Y_{n-1} + I_0 Y_n + I_- Y_{n+1}$}. This follows by precisely the same argument which precedes Lemma~\ref{canon_pres} from the previous section. In fact, in the same way as there we can show that 

\begin{lemma} \label{canon_pres_rot}
Given a twisted rotational action relative to the 3-axis we can always redefine the generators $X_n$ so that they become sections of the bundle $T^{1,0}_{\Io}M$ if \mbox{$n > 0$} and of the bundle $T^{0,1}_{\Io}M$ if \mbox{$n < 0$}.
\end{lemma}

\begin{definition}
We will henceforth refer to such a choice of representative for the equivalence class of generators as a \textit{canonical presentation} of the action. Generators in a canonical presentation will be marked with a star superscript.
\end{definition} 

Observe then that $\mathcal{L}_{X^{\star}_{n>0}} \omega_- = 0$ and $\mathcal{L}_{X^{\star}_{n<0}} \omega_+ = 0$, with one term in the corresponding Cartan formulas vanishing due to the closure of the forms and the other due to the mismatch between the Hodge types with respect to $I_0$ of the vector fields and forms being contracted. Consequently, when the generators are in the canonical presentation some of the terms in the equations \eqref{twist_ax_rot_2} vanish and we are left with 
{\allowdisplaybreaks
\begin{alignat}{3}
\mathcal{L}_{X^{\star}_{j-1}} \omega_+ && & = 0 \nonumber \\[1pt]
\mathcal{L}_{X^{\star}_{j-2}} \omega_+ &\,+\pt\pt& \mathcal{L}_{X^{\star}_{j-1}} \omega_0 & = 0 \nonumber \\[-1pt]
&& \vdots & \label{red_twist_ax_rot} \\
\mathcal{L}_{X^{\star}_+} \omega_+ &\,+\pt\pt& \mathcal{L}_{X^{\star}_{++}} \omega_0 & = 0 \nonumber \\[1pt]
\mathcal{L}_{X^{\star}_0} \omega_+ &\,+\pt\pt& \mathcal{L}_{X^{\star}_+} \omega_0 & = -i \hp \omega_+ \nonumber \\[1pt]
&& \mathcal{L}_{X^{\star}_0} \omega_0 & = 0 \mathrlap{.} \nonumber
\end{alignat}
}%
Again, the remaining non-trivial equations can be obtained from these by conjugation. In the spirit of our previous conventions we toggle freely between the numerical and the \mbox{$+/-$} notation for the indices. Let us concentrate for a moment on the structure of the second to last equation:
\begin{equation}
\underbracket[0.4pt][2.3pt]{\mathcal{L}_{X^{\star}_0} \omega_+}_{\text{$_{2,0 \pt + \pt 1,1}$}} + 
\underbracket[0.4pt][1.4pt]{\mathcal{L}_{X^{\star}_+} \omega_0}_{\text{$_{1,1 \pt + \pt 0,2}$}} = - i \npt 
\underbracket[0.4pt][4pt]{\omega_+}_{\text{$_{2,0}$}} 
\mathrlap{\hspace{5pt} \underbracket[0pt][7.2pt]{\phantom{}}_{\text{w.r.t. $I_0$}}}
\end{equation}
From matching Hodge types with respect to the complex structure $I_0$ it is clear that the $(0,2)$ component of $\mathcal{L}_{X^{\star}_+} \omega_0$ must vanish, which then implies that this is purely of type $(1,1)$. Comparing now with this in mind the last two equations in the list with the two equations \eqref{quasi_rot} we arrive at the following crucial observation: 

\begin{proposition} \label{middle_q-rot}
The $X^{\star}_0$ generator in the canonical presentation of a twisted rotational action relative to the 3-axis is a vector field of quasi-rotational type relative to the same axis, with
\begin{equation}
\sigma_+ = i\hp \mathcal{L}_{X^{\star}_+} \omega_0 = i\hp d(\iota_{X^{\star}_+}\omega_0) \mathrlap{.}
\end{equation}
\end{proposition}

\noindent  By Proposition \ref{quasi_rot_hyperpot} this is then automatically a closed hyper $(1,1)$ form on $M$. Actually, since the formula holds globally, $\sigma_+$ is in this case not only closed but also exact. In light of equation \eqref{[sgm]=[omg]} we hence have

\begin{corollary} \label{transv-triv}
A hyperk\"ahler manifold may possess a twisted rotational action relative to the 3-axis only if the cohomology classes of both $\omega_1$ and $\omega_2$ are trivial.
\end{corollary}

\noindent By the same Proposition, in addition to the complex-valued closed hyper $(1,1)$ form $\sigma_+$ and its conjugate form $\sigma_-$ we can also define a third real-valued closed hyper $(1,1)$ form $\sigma_0$. Furthermore, by Lemma~\ref{exp_ddbar_lem}, locally these forms have associated recursive chains of hyperpotentials. In the next part of our discussion we will be concerned with \mbox{characterizing these}. 

Similarly to above we have $\smash{ \iota_{X^{\star}_{n>0}} \omega_- = 0 }$ and $\smash{ \iota_{X^{\star}_{n<0}} \omega_+ = 0 }$, and so if we take the generators of the twisted action to be in the canonical presentation, the moment map equations \eqref{gen_mom_maps} assume the reduced form
{\allowdisplaybreaks
\begin{alignat}{3}
\iota_{X^{\star}_{j-1}} \omega_+ && & = d\varphi_j \nonumber \\[1pt]
\iota_{X^{\star}_{j-2}} \omega_+ &\,+\pt\pt& \iota_{X^{\star}_{j-1}} \omega_0 & = d\varphi_{j-1} \nonumber\\[-1pt]
&&  \vdots & \label{red_eqs} \\
\iota_{X^{\star}_+} \omega_+ &\,+\pt\pt& \iota_{X^{\star}_{++}} \omega_0 & = d\varphi_{++} \nonumber\\[1pt]
\iota_{X^{\star}_0} \omega_+ &\,+\pt\pt& \iota_{X^{\star}_+} \omega_0  & = d\varphi_+ - i \hp \theta_+ \nonumber\\[1pt]
&& \iota_{X^{\star}_0} \omega_0 & = d\mu \mathrlap{.} \nonumber
\end{alignat}
}%
By separating each of these equations into $(1,0)$ and $(0,1)$ components with respect to $I_0$ we obtain the equivalent system
{\allowdisplaybreaks
\begin{equation} \label{sys_pairs}
\begin{array}{clcl}
\bullet & \displaystyle \bar{\partial} \varphi_j = 0 & \multicolumn{2}{l}{\text{\textit{i.e.} $\varphi_j$ is holomorphic with respect to $I_0$}} \\[7pt] \allowbreak
\bullet & \displaystyle \iota_{X^{\star}_n} \omega_+ = \partial \varphi_{n+1} & \multirow{2}{*}[-0.3em]{$n = 1,...\pt , j-1$} & \\[5pt] 
& \displaystyle \iota_{X^{\star}_n} \omega_0 \hspace{2.3pt} = \bar{\partial} \varphi_n & & \\[7pt] \allowbreak
\bullet & \displaystyle  \iota_{X^{\star}_0 \pt} \omega_+ = \partial \varphi_+ - i\hp \theta_+ & \multirow{2}{*}[-0.3em]{$\Longleftrightarrow$}& \displaystyle \iota_{P^{1,0}_{\Io}X^{\star}_0 + i S \pt} \omega_+ = \partial \varphi_+ \\[5pt]
& \displaystyle \iota_{X^{\star}_0 \pt} \omega_0 \hspace{2.3pt} = d \mu & & \displaystyle \iota_{P^{1,0}_{\Io}X^{\star}_0 + i S \pt} \omega_0 \hspace{2.3pt} = \bar{\partial} \varphi_0 \mathrlap{.}
\end{array}
\end{equation}
}%
The proof of the last equivalence goes along the same lines as the proof of the equivalence between the equations \eqref{quasi_mom_map_+} and \eqref{mom_map_3} on one hand and the equations \eqref{double_grad} on the other. Thus, what we obtain is a recursive system of pairs of symplectic gradient equations of the type featuring in part 5 of Proposition~\ref{criterion_2}, which we can then immediately interpret in the sense that
\begin{equation*}
\cdots \,\noarrow\, \varphi_0 \,\noarrow\, \varphi_{+} \,\noarrow\, \varphi_{++} \,\noarrow\, \cdots \,\noarrow\, \varphi_{j-1} \,\noarrow\, \varphi_{j} \,\noarrow\, 0 \\[1pt]
\end{equation*}
forms a right-bounded recursive chain of hyperpotentials with respect to $I_0$. \pagebreak By noticing that on the domain of definition of $\varphi_+$ we have $\sigma_+ = i \hp \partial\bar{\partial} \hp \varphi_+$, we can identify this as a recursive chain locally associated to the closed hyper $(1,1)$ form $\sigma_+$ in the sense of the extended \mbox{$\partial\bar{\partial}$\pt-\hp Lemma}~\ref{exp_ddbar_lem}. 

%Such recursive chains always exist to which we can associate functions
%\begin{equation}
%\varphi_{\hp V_N}\npt(\zeta) = \sum_{n=-\infty}^j \varphi_n \hp \zeta^{-n}
%\quad \text{and} \quad
%\varphi_{\hp V_S}\npt(\zeta) = - \sum_{n=-j}^{\infty} \varphi^c_n \hp \zeta^{-n}
%\end{equation}
%which are holomorphic in a twistor space sense on open sets $V_N, V_S \subset \mathcal{Z}$\,---\,outside, that is, of their intersections with the fibers over $\zeta = 0$ and $\infty$, respectively, where they are singular.

For the next step we need the following generic technical result: 

\begin{lemma} 
For any right-bounded chain of hyperpotentials with respect to $I_0$ of the form
\begin{equation*}
\cdots \,\noarrow\, \varphi_{+} \,\noarrow\, \varphi_{++} \,\noarrow\, \cdots \,\noarrow\, \varphi_{j-1} \,\noarrow\, \varphi_{j} \,\noarrow\, 0 \\[1pt]
\end{equation*}
the following identities hold:
{\allowdisplaybreaks
\begin{equation} \label{hb_chains}
\begin{aligned}
\bar{\partial} \varphi_+ + \sum_{n=2}^j \frac{d\varphi_{n}}{\zeta^{n-1}} & = - i \hp d \Big(\sum_{n=2}^j \frac{n-1}{\zeta^{n-1}} \varphi_{n} \Big) I(\zeta) \\
\bar{\partial} \varphi_+ + \sum_{n=2}^j \frac{1}{2} \Big(\hp \frac{1}{\zeta^{n-1}} + (-\bar{\zeta})^{n-1} \Big) d\varphi_{n} & = d \Big[\sum_{n=2}^j \frac{1}{2i} \Big(\hp \frac{1}{\zeta^{n-1}} - (-\bar{\zeta})^{n-1} \Big) \varphi_{n} \Big] \IU \mathrlap{.}
\end{aligned}
\end{equation}
}%
\end{lemma}

\begin{proof} 

Consider a point $\mathbf{u} \in S^2$ with complex inhomogeneous coordinate $\zeta \neq 0,\infty$. For any set of functions $\{a_n \}_{n = 2, \dots, j}$ assumed to depend solely on $\mathbf{u}$, using the second expression \eqref{def_I(zeta)} for $I(\zeta)$ and the chain's recursion relations we have 
{\allowdisplaybreaks
\begin{align}
d\Big(\sum_{n=2}^j a_n \varphi_n \Big) I(\zeta) & = 
i \sum_{n=2}^j [ (a_n - \frac{a_{n-1}}{\zeta}) \hp \partial \varphi_n + (a_{n+1}\zeta - a_n) \hp \bar{\partial}\varphi_n ] \\
& \phantom{=} + i \hp a_2 \hp \zeta \hp \bar{\partial} \varphi_+ + i \hp \frac{a_1}{\zeta \,} \partial \varphi_{++}  \rlap{.} \nonumber
\end{align}
}%
If we want the coefficients of $\partial\varphi_n$ and $\bar{\partial}\varphi_n$ in the sum on the right-hand side to be equal\,---\,so that they yield together a multiple of $d\varphi_n$\,---\,we need to have
\begin{equation}
\frac{a_{n-1}}{\zeta} - 2 \hp a_n + a_{n+1}\zeta = 0 \mathrlap{.}
\end{equation}
This is a second-order linear recurrence relation. To determine it we need in addition two initial conditions. Guided by the last two terms above we choose as initial conditions $a_1 = 0$ and $a_2 = - i/ \zeta$. The solution of the recurrence is then
\begin{equation}
a_n = - i \hp \frac{n-1}{\zeta^{n-1}}
\end{equation}
and the first identity of the Lemma follows promptly. 

The second identity can be derived in much the same way. For any choice of coefficients $\{b_n\}_{n = 2,\dots,j}$ depending at most on $\mathbf{u}$, writing $\IU = x_0 I_0 - 2 \hp x_-I_+ \!- 2 \hp x_+I_-$ and resorting to the chain's recurrence relations gives us
{\allowdisplaybreaks
\begin{align}
d\Big(\sum_{n=2}^j b_n \varphi_n \Big) \IU & = i \sum_{n=2}^j [(x_0b_n + 2x_-b_{n-1}) \partial\varphi_n - (x_0b_n + 2x_+b_{n+1})\bar{\partial}\varphi_n] \\
& \phantom{=} - 2i \hp x_+ b_2 \hp \bar{\partial} \varphi_+ - 2i \hp x_-b_1 \partial\varphi_{++} \mathrlap{.} \nonumber
\end{align}
}%
By following a similar rationale as above we require that
\begin{equation}
x_+ b_{n+1} + x_0 b_n + x_- b_{n-1} = 0
\end{equation}
with initial conditions $b_1 = 0$ and $b_2 = i/ (2 x_+)$. This new recurrence has the solution 
\begin{equation}
b_n = \frac{1}{2i} \Big(\hp \frac{1}{\zeta^{n-1}} - (-\bar{\zeta})^{n-1} \Big) \mathrlap{.}
\end{equation}
Substituting it back into the equation yields the desired identity.
\end{proof}

Returning to our discussion, let us assemble now the relations \eqref{sgm-omg} corresponding to the case when $X = X^{\star}_0$ into the following two equations:
\begin{equation}
\begin{aligned}
\frac{\sigma_+}{\zeta\ } + \sigma_0 + \zeta \hp \sigma_- & = \omega(\zeta) \hp + d[d\mu \hp I(\zeta)] \\
x_1\hp\sigma_1 + x_2 \hp\sigma_2 + x_3 \hp\sigma_3 & = \omega(\nhp\mathbf{u}) + d[d\mu \IU] \mathrlap{.}
\end{aligned}
\end{equation} 
The real components $\sigma_1$, $\sigma_2$, $\sigma_3$ are related to the complex ``spherical" ones $\sigma_+$, $\sigma_0$, $\sigma_-$ in the usual way. Acting with the exterior derivative of $M$ on each of the two identities of the previous Lemma and recalling that $\sigma_+ = i \hp \partial\bar{\partial}\hp \varphi_+$ gives us two formulas for $\sigma_+$\,---\,and, through antipodal conjugation, two corresponding ones for $\sigma_-$. By substituting then in each equation one of the two matching sets of formulas for $\sigma_+$ and $\sigma_-$ (after rewriting the left-hand side of the second equation as $x_0 \hp \sigma_0 - 2 \hp x_- \sigma_+ \!- 2 \hp x_+\sigma_-$) we can cast them in the form
\begin{equation} \label{mojo}
\begin{aligned}
\sigma_0 & = \omega(\zeta) \hp + d[d\mu(\zeta) \hp I(\zeta)]   \\[2pt]
x_0 \hp \sigma_0 & = \omega(\nhp\mathbf{u}) + d[d\mu(\hspace{-0.6pt}\mathbf{u} \hspace{-0.2pt}) \IU] 
\end{aligned}
\end{equation}
where, by definition,
\begin{equation} \label{mu-zeta}
\mu(\zeta) = \mu - \sum_{n=2}^j (n-1) ( \varphi_{n} \hp \zeta^{-n} + \varphi^c_{-n} \zeta^n )
\end{equation}
and
{\allowdisplaybreaks
\begin{align}
\mu(\mathbf{u}) & = \mu + \sum_{n=2}^j \frac{|\zeta|^2 + (-)^n|\zeta|^{2n}}{1+|\zeta|^2} ( \varphi_{n} \hp \zeta^{-n} + \text{c.c.}) \\
& = \mu - \sum_{n=2}^j \Big(\pt \sum_{k=1}^{n-1} (-)^k |\zeta|^{2k} \pt\Big) ( \varphi_{n} \hp \zeta^{-n} + \text{c.c.}) \mathrlap{.} \nonumber
\end{align}
}%
The two equations \eqref{mojo} will play a central role in our analysis. They hold for any $\zeta \neq 0, \infty$ respectively any $\mathbf{u} \neq \mathbf{u}_N, \mathbf{u}_S$. The two \mbox{$\mu$\pt-\hp functions}, regarded as functions on the twistor space, are defined everywhere on $\mathcal{Z}$ except on the fibers over these two points.  

\begin{lemma} \label{qmm_twist_mu}
The function $\mu(\zeta)$ thus defined satisfies the following quasi-moment map equation on $M$:
\begin{equation}
\iota_{X(\zeta)}\omega(\zeta) = d\mu(\zeta) - i \hp \zeta \partial_{\zeta}[d\mu(\zeta)I(\zeta)] \mathrlap{.}
\end{equation}
\end{lemma}

\begin{proof}

By subtracting from the first identity \eqref{hb_chains}, divided by $\zeta$, its antipodal conjugate, comparing the result with the definition of $\mu(\zeta)$, and then using in the  terms of order 1 and $-1$ in $\zeta$ the relation \eqref{mom-pot} and its conjugate, we get
%\begin{equation} \label{d_mu_I}
%i \hp d\mu(\zeta)I(\zeta) = i \hp d\mu I_0 + (d\varphi_+ \!- i\hp \theta_+)\zeta^{-1} \!- (d\varphi^c_- \!+ i \hp \theta_-) \hp \zeta + \sum_{n=2}^j(d\varphi_n\zeta^{-n} - d\varphi^c_{-n}\zeta^n) \rlap{.}
%\end{equation}
%\begin{equation} \label{d_mu_I}
%d\mu(\zeta)I(\zeta) = d\mu I_0 - i \hp (d\varphi_+ \!- i\hp \theta_+)\zeta^{-1} \!+ i \hp (d\varphi^c_- \!+ i \hp \theta_-) \hp \zeta - i \sum_{n=2}^j(d\varphi_n\zeta^{-n} - d\varphi^c_{-n}\zeta^n) \rlap{.}
%\end{equation}
\begin{equation} \label{d_mu_I}
d\mu(\zeta)I(\zeta) = d\mu I_0 - \frac{\theta_+\! + i \hp d\varphi_+}{\zeta} - (\theta_-\! - i \hp d\varphi^c_-) \hp \zeta - i \sum_{n=2}^j\Big(\frac{d\varphi_n}{\zeta^{n}} - d\varphi^c_{-n}\zeta^n\Big) \rlap{.}
\end{equation}
From this we deduce successively \smallskip
{\allowdisplaybreaks
\begin{align}
& d\mu(\zeta) - i \hp \zeta \partial_{\zeta}[d\mu(\zeta)I(\zeta)] = \\
& \quad = (d\varphi_+ \!- i\hp \theta_+)\zeta^{-1} \!+ d\mu +(d\varphi^c_- \!+ i \hp \theta_-) \hp \zeta + \sum_{n=2}^j(d\varphi_n\zeta^{-n} \! + d\varphi^c_{-n}\zeta^n) \nonumber \\[-7pt]
& \quad = \sum_{n=-j}^j (\iota_{X_{n-1}}\omega_+ + \iota_{X_{n}}\omega_0 + \iota_{X_{n+1}}\omega_-) \zeta^{-n} \nonumber \\[1pt]
& \quad = \iota_{X(\zeta)} \omega(\zeta) \mathrlap{.} \nonumber
\end{align}
}%
The third line  follows from the equations \eqref{gen_mom_maps}, and the last one is an immediate consequence of the definitions of $X(\zeta)$ and $\omega(\zeta)$.
\end{proof}

The following table lists the three closed hyper $(1,1)$ forms which characterize the twisted rotational action together with a corresponding set of local recursive chains of hyperpotentials with respect to the complex structure $I_0$: 
\begin{equation*}
\begin{array}{c|c|c|c}
\text{Hyper $(1,1)$ form} & \text{Associated local chain}  &  \text{Globally-defined elements} & \text{Related by} \\[1pt] \hline
\rule[13pt]{0pt}{0pt}
\sigma_+  & \{\varphi^{\phantom{c}}_n\}^{\phantom{c}}_{n \leq j\phantom{+}} & \varphi_{j}, \dots, \varphi_2  & \ \sigma_+ = i \hp \partial\bar{\partial} \varphi_+  \\[4pt]
\sigma_0\phantom{\hspace{2.7pt}} & \{\phi^{\phantom{c}}_n = \phi^c_n\}^{\phantom{c}}_{n \in \mathbb{Z}} & & \ \hspace{-2.7pt} \sigma_0 \hspace{2.5pt} = i \hp \partial\bar{\partial} \phi_{\hp 0} \\[4pt]
\sigma_- 	& \{\varphi^c_n\}^{\phantom{c}}_{n \geq -j} & \varphi^c_{-j}, \dots, \varphi^c_{-2} & \ \sigma_- = i \hp \partial\bar{\partial} \varphi^c_- \\[4pt] \hline
\end{array}
\vspace{2.5pt}
\end{equation*}
In light of the equation \eqref{sgm0_pot} we may identify $\phi_0 = - (\kappa_0 + 2\mu)$, and so the recursive chain associated to $\sigma_0$ can be thought of as the (self-conjugated) chain generated locally by this hyperpotential. It is important to remember that since the chains of hyperpotentials arise recursively by means of existence arguments based on the holomorphic Poincar\'e lemma, none of their elements, with the exception of the globally-defined ones, are uniquely defined. In fact, even these are defined only up to possible constant shifts. 

Let us consider the local aspects a little more carefully. Implicit in the table above is the fact that (most of) the hyperpotentials listed are local, and as such they should really be labeled by the corresponding open set on which they are defined. As we have explained in \mbox{\S\,\ref{ssec:global_iss}}, to $\sigma_0$ one can associate an open cover $\mathscr{V}$ of $\mathcal{Z}$ with each element $V \in \mathscr{V}$ supporting a real-valued potential $\phi_V(\mathbf{u})$ such that  \mbox{$\sigma_0 = i \hp \partial_{\Iu} \bar{\partial}_{\Iu} \phi_V(\mathbf{u})$} on some local domain on $M$ for each $\mathbf{u} \in \pi(V)$. The covering $\mathscr{V}$ may always be chosen to be invariant under the action of the antipodal conjugation-induced real structure on $\mathcal{Z}$. Let \mbox{$\mathscr{V}_N = \{ V \in \mathscr{V} \, | \, \mathbf{u}_N \in \pi(V) \}$} and \mbox{$\mathscr{V}_S = \{ V \in \mathscr{V} \, | \, \mathbf{u}_S \in \pi(V) \}$} denote the collections of elements of the cover which project down to arctic and antarctic regions on the twistor sphere, respectively. As before, we will denote generic elements of $\mathscr{V}_N$ and $\mathscr{V}_S$ by $V_N$ and $V_S$. For any such pair of open sets $V_N$ and $V_S$ interchanged by the real structure of $\mathcal{Z}$ Proposition~\ref{lulu} guarantees that a local recursive chain $\{\phi_n\}_{n \in \mathbb{Z}}$ exists such that $\phi_{V_N}\npt(\mathbf{u})$ and $\phi_{V_S}\npt(\mathbf{u})$ are given in terms of its elements by the two series \eqref{phi_VNetco}. Similarly, a local recursive chain $\{\varphi_n\}_{n \leq j}$ associated to $\sigma_+$  can always be chosen such that the two series 
\begin{equation} \label{varphi_NS}
\varphi_{\hp V_N}\npt(\zeta) = \sum_{n=-\infty}^j \varphi_n \hp \zeta^{-n}
\quad \text{and} \quad
\varphi_{\hp V_S}\npt(\zeta) = - \sum_{n=-j}^{\infty} \varphi^c_n \hp \zeta^{-n}
\end{equation}
converge on $V_N$ and $V_S$\,---\,outside, that is, of their intersections with the fibers over $\zeta = 0$ respectively $\infty$ where they are singular\,---\,thereby defining meromorphic functions in a twistor space sense by Proposition~\ref{hol-hyper(1,1)}. (In order for $V_N$ and $V_S$ to be able to accommodate the additional functions one might need to shrink them through a refinement of the cover.) We then extend this collection of polar meromorphic functions trivially to the whole cover $\mathscr{V}$ by setting
\begin{equation}
\varphi_V(\zeta) =
\begin{cases}
\varphi_{V_N\npt}(\zeta) & \text{if $V=V_N \in \mathscr{V}_N$} \\[1pt]
\varphi_{V_S}(\zeta) & \text{if $V=V_S \, \in \mathscr{V}_S$} \\[1pt]
0 & \text{otherwise} \rlap{.}
\end{cases}
\end{equation} 
For quick reference we summarize this discussion in yet another table:
\begin{equation*}
\begin{array}{c|cl|l}
\text{Hyper $(1,1)$ form(s)} & \text{Associated local functions}  &  & \text{Polar components} \\[1pt] \hline
\rule[13pt]{0pt}{0pt}
\sigma_+, \, \sigma_-  &  \{ \varphi_V(\zeta) \}_{V \in \mathscr{V}} & \text{meromorphic} & \qquad \text{eqs.\,\eqref{varphi_NS}} \\[4pt]
\sigma_0\phantom{\hspace{2.7pt}} &  \{ \phi_V(\mathbf{u}) \}_{V \in \mathscr{V}} & \text{real potentials} & \qquad \text{eqs.\,\eqref{phi_VNetco}}  \\[4pt] \hline
\end{array}
\vspace{7pt}
% \bigskip
\end{equation*}

\subsubsection{}

Next, we turn our attention to examining the \textit{local} implications of the two equations \eqref{mojo}. Let us begin with the second one:
\begin{equation} 
x_0 \hspace{-1.2pt} \underbracket[0.4pt][5.2pt]{\sigma_0}_{\text{$_{1,1}$}} = \underbracket[0.4pt][4pt]{\omega(\mathbf{u})}_{\text{$_{1,1}$}} + \underbracket[0.4pt][4pt]{d[d\mu(\mathbf{u}) \IU]}_{\text{$_{1,1}$}}
\mathrlap{\hspace{5pt} \underbracket[0pt][7.2pt]{\phantom{}}_{\text{w.r.t. $\IU$}}}
\end{equation}
Each constituent term is a closed form of type $(1,1)$ with respect to the complex structure $\IU$, and as such can be derived by virtue of the corresponding \mbox{$\partial\bar{\partial}$\pt-\hp lemma} from a local potential. It thus follows immediately that
\begin{equation} \label{K-pot}
\kappa_V\nhp(\mathbf{u}) = -\, 2 \mu(\mathbf{u}) - x_0 \hp \phi_V(\mathbf{u})
\end{equation}
defines a local K\"ahler potential in the complex structure $\IU$. As before, $x_0$ denotes the component of $\mathbf{u}$ along the rotation axis (\textit{i.e.} the height function).

Observe, however, that at the \textit{north} and \textit{south poles} of the twistor sphere this formula is not well-defined due to the fact that $\mu(\mathbf{u})$ is singular there. This problem can be rectified by a redefinition of the K\"ahler potential in the arctic and antarctic neighborhoods through the addition of a local pluriharmonic counter\-term which cancels the singular part while preserving the K\"ahler potential property: 
\begin{equation}
\kappa^{\text{reg}}_V\nhp(\mathbf{u}) = \kappa^{\phantom{g}}_V\nhp(\mathbf{u}) - (x^{\phantom{g}}_0 - x^{\phantom{g}}_V) [\varphi^{\phantom{g}}_{V}(\zeta) + \hp \overline{\varphi^{\phantom{.}}_{V}(\zeta)}]
\end{equation}
where $x_{V_N} = 1$ and $x_{V_S} = -1$ are the height functions of $\mathbf{u}_N$ and $\mathbf{u}_S$ for any $V_N \in \mathscr{V}_N$ and $V_S \in \mathscr{V}_S$. Noting that $\phi_{\hp V_N}\npt(\mathbf{u}_N) = \phi_0 = -(\kappa_0 + 2\mu)$, one can check that with this new definition we have $\kappa^{\text{reg}}_{V_N}\npt(\mathbf{u}) \mapsto \kappa_0$ as $\mathbf{u} \mapsto \mathbf{u}_N$. Similarly, $\kappa^{\text{reg}}_{V_S}\npt(\mathbf{u}) \mapsto \kappa_0$ as $\mathbf{u} \mapsto \mathbf{u}_S$, and so the corrected potential is indeed regular at the poles. 
 
On another hand, from the first formula it is clear that at the \textit{equator} of the twistor sphere (defined by $x_0 = 0$), $-2 \hp \mu(\mathbf{u})$ is a K\"ahler potential for $\omega(\mathbf{u})$. On the basis of this observation we have: 

\begin{corollary} \label{poles-equator}
If a hyperk\"ahler manifold $M$ with $H^1(M,\mathbb{R}) = 0$ has an $\mathcal{O}(2j-2)$-twisted rotational action relative to the 3-axis, then for any $\gamma \in [0,2\pi)$ the Fourier superposition 
\begin{equation}
\kappa_{\hp \textup{equator}}(\gamma) = - \pt 2\mu - 2 \sum_{\substack{n=2 \\ \textup{even}}}^j (\varphi_n e^{-i n \gamma} + \bar{\varphi}_n \hp e^{i n \gamma})
\end{equation}
with $\mu$ and $\varphi_n$ Hamiltonian functions characterizing the action via the equations \eqref{gen_mom_maps} gives a K\"ahler potential for the K\"ahler form $\omega(\mathbf{u})$ at the point $\mathbf{u}$ on the equator of the twistor 2-sphere with longitude angle $\gamma$.
\end{corollary}

\begin{figure}[ht]
\vskip-12pt
\begin{tikzpicture}[scale=0.9]

\fill[ball color=white] (0,0) circle [radius=2]; % or lightgray!20

\newcommand\latitude[1]{%
  \draw (#1:2) arc (0:-180:{2*cos(#1)} and {0.6*cos(#1)});
  \draw[dashed] (#1:2) arc (0:180:{2*cos(#1)} and {0.6*cos(#1)});
}
\latitude{0};

\draw [->] (0,0) ++(-135:0.2) arc (-135:-20:0.28 and 0.05);

\draw[darkgray] (0,0) -- ++(-20:1.35);

\draw [->,darkgray] (0,0) -- (0,2.6);
\draw [->,darkgray] (0,0) -- (2.7,0);
\draw [->,darkgray] (0,0) -- (-1.05,-1.05);

\draw[fill=black] (0,1.8) circle (1.2pt);
\draw[fill=black] (0,0) ++(-20:1.35) circle (1.2pt);

\node at (1.35,-0.72) {$\mathbf{u}$};
\node at (0.07,-0.38) {$\gamma$};
\node at (-1.38,-0.90) {$x_1$};
\node at (2.6,0.4) {$x_2$};
\node at (-0.35,2.6) {$x_3$};

\end{tikzpicture}
\caption{The longitude angle $\gamma$.}
\vskip-10pt
\end{figure}
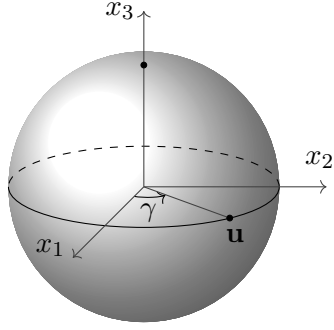 

\begin{remark}
This result generalizes to the twisted case an observation made in \cite{MR877637} in the case of a hyperk\"ahler manifold with a purely rotational action, where it was shown that the moment map with respect to either one of the two K\"ahler forms preserved by the action and corresponding to an axis of the twistor sphere plays (up to a possible multiplicative constant depending on one's normalization conventions) the role of K\"ahler potential for any K\"ahler form in the equatorial plane normal to the axis.
\end{remark}

\subsubsection{} \label{ssec:sofp}

In contrast to the second equation \eqref{mojo}, the first one has a less homogeneous structure:
\begin{equation}
\underbracket[0.4pt][4pt]{\sigma_0}_{\text{$_{1,1}$}} = \underbracket[0.4pt][4pt]{\omega(\zeta)}_{\text{$_{2,0}$}} + \, d[\underbracket[0.4pt][4pt]{d\mu(\zeta) \hp I(\zeta)}_{\text{$_{1,0}$}}] 
\mathrlap{\hspace{5pt} \underbracket[0pt][7.2pt]{\phantom{}}_{\text{w.r.t. $\IU$}}}
\end{equation}
By filtering according to Hodge type with respect to the complex structure $\IU$ we can split it into two homogeneous pieces:
\begin{equation} \label{too_pieces}
\begin{aligned}
\partial_{\Iu}[d\mu(\zeta) \hp I(\zeta)]  & = -\, \omega(\zeta) \\[2pt]
\bar{\partial}_{\Iu}[d\mu(\zeta) \hp I(\zeta)] & = \sigma_0 \mathrlap{.}
\end{aligned}
\end{equation}
Expressing $\sigma_0$ in terms of a local hyperpotential with respect to $\IU$ allows us to recast these equations in the form
\begin{equation}
\begin{aligned}
\partial_{\Iu}[d\mu(\zeta) \hp I(\zeta) + i \hp \partial_{\Iu} \phi_V\nhp(\mathbf{u})]  & = -\, \omega(\zeta) \\[2pt]
\bar{\partial}_{\Iu}[d\mu(\zeta) \hp I(\zeta) + i \hp \partial_{\Iu} \phi_V\nhp(\mathbf{u})] & = 0 
\end{aligned}
\end{equation}
valid for any $V \in \mathscr{V}$, from which it follows that
\begin{equation} \label{s-pot}
\theta_V(\zeta) = -\, d\mu(\zeta)I(\zeta) - i \hp \partial_{\Iu} \phi_V\nhp(\mathbf{u})
\end{equation}
provides a local symplectic 1-form potential for the twisted symplectic form $\omega(\zeta)$, holomorphic relative to $\IU$ (that is, a 1-form $\theta_V(\zeta) \in \mathscr{A}^{\smash[t]{1,0}}_{\Iu}(M)$ satisfying the two properties: $\omega(\zeta) = d\theta_V(\zeta)$ and $\bar{\partial}_{\Iu}\theta_V(\zeta) = 0$). By the first property \eqref{pot_constrs}, $\theta_V(\zeta)$ depends indeed on $\zeta$ holomorphically.

Note that this is singular at $\zeta = 0$, and thus in fact meromorphic, with the singular part dictated entirely by the first term, as can be seen from a cursory inspection of the equation \eqref{delphi_at_u}. By the equation \eqref{d_mu_I} this is of the form
%\begin{equation} \label{spur_sings}
%\theta_{V_N}\npt(\zeta) = i \sum_{n=2}^j \frac{d\varphi_n}{\zeta^n} + \frac{i \hp d\varphi_{+} \nhp + \theta_+}{\zeta\ } + \mathcal{O}(\zeta^0) \rlap{.}
%\end{equation}
\begin{equation} \label{spur_sings}
\theta_{V_N}\npt(\zeta) = i \sum_{n=1}^j \frac{d\varphi_n}{\zeta^n} + \frac{\theta_+}{\zeta\ } + \mathcal{O}(\zeta^0) \rlap{.}
\end{equation}
An antipodally conjugated version of this expansion holds for the corresponding neighborhood $V_S$ from $\mathscr{V}_S$. 

As we will discuss in greater detail later on in \mbox{\S\,\ref{ssec:cof}}, since $\omega(\zeta)$ defines a complex symplectic 2-form on the twistor fiber above $\zeta$ we can introduce complex canonical coordinates, and in particular we can try to identify the symplectic potential \eqref{s-pot} with the corresponding complex canonical 1-form.  However, a straightforward attempt to do that fails in essence due to the discrepancy between the singularity structure of $\omega(\zeta)$, which has a simple pole at $\zeta = 0$, and that of $\theta_V(\zeta)$, which has a pole of order $j$. Just as in the case of the K\"ahler potentials this problem can be solved by exploiting the inherent ambiguity in the definition of symplectic 1-form potentials. Thus, if to the meromorphic symplectic 1-form potential \eqref{s-pot} we add trivial local meromorphic counterterms as follows
\begin{equation} \label{s-pot-reg}
\theta^{\text{reg}}_V(\zeta) = \theta^{\phantom{g}}_V(\zeta) - i \hp d\varphi^{\phantom{g}}_V(\zeta)
\end{equation}
then the result will not only be another meromorphic symplectic 1-form potential for $\omega(\zeta)$, but also one which has now a simple pole at $\zeta=0$:
\begin{equation} \label{vph^reg_N_ser}
\theta^{\text{reg}}_{V_N}\npt(\zeta) = \frac{\theta_+}{\zeta \ } + \mathcal{O}(\zeta^0) \rlap{.}
\end{equation}
This and not the representative \eqref{s-pot} should therefore be the symplectic potential which should be identified with the canonical 1-form.

\subsubsection{}

As we have seen in Corollary \ref{transv-triv}, two of the three closed hyper $(1,1)$ forms associated to a twisted rotational action, namely $\sigma_+$ and $\sigma_-$, are exact, and so their cohomology classes are trivial. Their role is important but rather discrete. On the other hand, the cohomology class of the remaining closed hyper $(1,1)$ form, $\sigma_0$, is in general non-trivial. In particular, when this cohomology class is, up to a $2\pi$ factor, \textit{integral} (which by formula \eqref{[sgm]=[omg]} is equivalent to the cohomology class of  $\omega_0$ being up to a $2\pi$ factor integral), we can interpret $\sigma_0$ as the curvature form of a hyperhermitian connection on a hyperholomorphic line bundle over $M$. Over the twistor space $\mathcal{Z}$, by virtue of the Atiyah-Ward correspondence, we have in this case a corresponding holomorphic line bundle $L_{\mathcal{Z}}$ trivial along each horizontal twistor line. In what follows we will argue that this line bundle is moreover endowed with a natural meromorphic connection. In fact, we will actually prove that a more comprehensive result holds, one which generalizes this paradigm and does not need any additional assumptions such as the integrality condition or the existence of a line bundle. 

Let us associate to each open set $V \in \mathscr{V}$ a local $(1,0)$ form on $\mathcal{Z}$ defined in a corresponding local trivialization as follows:
\begin{equation} \label{A_V-def}
\mathcal{A}_V = - i\hp\partial_{\mathcal{Z}} \phi_V\nhp(\mathbf{u}) - d\mu(\zeta)I(\zeta) + i \hp \mu(\zeta) \frac{d\zeta}{\zeta} \mathrlap{.}
\end{equation}
The key properties of these forms are summarized in the lemma below:

\begin{lemma} \label{A_V_mero}
The 1-forms $\mathcal{A}_V$ are meromorphic with respect to the twistor complex structure on $\mathcal{Z}$.  They are well-defined everywhere on their domains except on the fibers over $\zeta=0$ and $\infty$ for $V \in \mathscr{V}_N$ and $\mathscr{V}_S$, respectively, where they have poles of order $j$, with the residues of all orders from 1 to $j$ globally defined on the fibers and canonically determined by the data of the action.
\end{lemma}

\begin{proof}
Observe that the component of $\mathcal{A}_V$ along the twistor fiber is precisely the meromorphic symplectic 1-form potential $\theta_V(\zeta)$ defined above. From the $\zeta$-expansion \eqref{spur_sings} and the $\zeta$-series formulas \eqref{mu-zeta} and \eqref{phi_VNetco} we get that 
{\allowdisplaybreaks
\begin{equation} \label{princ_part_N}
\mathcal{A}_{V_N} = \underbracket[0.4pt][4pt]{ i \sum_{n=1}^j \frac{d\varphi_n}{\zeta^n} + \frac{\theta_+}{\zeta\ } }_{\stackrel{\rotatebox{90}{$\scriptscriptstyle{\, =\, }$}}{\text{\normalsize{$\displaystyle \mathclap{ i \sum_{n=1}^{\smash[t]{j}} \frac{\iota_{X_{n-1}}\omega_+ + \iota_{X_{n}}\omega_0 + \iota_{X_{n+1}}\omega_-}{\zeta^n} }$}}}} + \mathcal{O}(\zeta^0) - i \Big( \sum_{n=2}^j (n-1) \frac{\varphi_n}{\zeta^n} - \mu + \mathcal{O}(\zeta) \Big) \frac{d\zeta}{\zeta}
\end{equation}
}%
This expansion formula together with its antipodal conjugate provide the substance of the second part of the statement.

To prove mermorphicity, let $\alpha = d\mu(\zeta)I(\zeta) - i \hp \mu(\zeta) \pt d\zeta/\zeta$, and then in terms of this we have 
\begin{equation}
\bar{\partial}_{\mathcal{Z}} \mathcal{A}_V = - i \hp \bar{\partial}_{\mathcal{Z}}\partial_{\mathcal{Z}} \phi_V\nhp(\mathbf{u}) -  \bar{\partial}_{\mathcal{Z}} \alpha \mathrlap{.}
\end{equation}
By the property \eqref{form-pot}, in the local trivialization considered the first term is equal to $\sigma_0$. On the other hand, from Lemma \ref{fiberwise-lemma} and then the second equation \eqref{too_pieces} we obtain 
\begin{equation}
\bar{\partial}_{\mathcal{Z}} \alpha = \bar{\partial}_{\Iu} \alpha_F = \bar{\partial}_{\Iu} [d\mu(\zeta)I(\zeta)] = \sigma_0 \mathrlap{.}
\end{equation}
The two terms thus cancel against each other to yield $\bar{\partial}_{\mathcal{Z}} \mathcal{A}_V = 0$.
\end{proof}

\smallskip 

Should $\sigma_0/2\pi$ belong to an integral cohomology class, the forms $\mathcal{A}_V$ would play the role of local connection 1-forms for a meromorphic connection on the holomorphic line bundle $L_{\mathcal{Z}}$. A more invariant characterization of the connection would be in this case given by the curvature 2-form, also known as field strength. Splitting successively $d_{\mathcal{Z}} = \partial_{\mathcal{Z}} + \bar{\partial}_{\mathcal{Z}}$ and then $d_{\mathcal{Z}} = d + d_{\hp\smash{\mathbb{CP}^1}}$, we get 
%{\allowdisplaybreaks
%\begin{align}
%\mathcal{F} = d_{\mathcal{Z}}\mathcal{A}_V & = \underbracket[0.4pt][4pt]{i \hp \partial_{\mathcal{Z}}\bar{\partial}_{\mathcal{Z}} \phi_V\nhp(\mathbf{u})}_{\text{$\sigma_0$}} - \, d_{\mathcal{Z}} \alpha \\[-5pt]
%& = \underbracket[0.4pt][4pt]{\sigma_0 - d[d\mu(\zeta) I(\zeta)]}_{\text{\rule[7pt]{0pt}{0pt} $\omega(\zeta)$}} - \, d_{\hp\smash{\mathbb{CP}^1}}[d\mu(\zeta) I(\zeta)] + i \hp d\mu(\zeta) \wedge \frac{d\zeta}{\zeta} \nonumber \\[-7pt]
%& = \omega(\zeta) + [\zeta \partial_{\zeta}(d\mu(\zeta)I(\zeta)) + i \hp d\mu(\zeta)] \wedge \frac{d\zeta}{\zeta} \nonumber \\
%& = \omega(\zeta) + i \pt \iota_{X(\zeta)}\omega(\zeta) \wedge \frac{d\zeta}{\zeta} \mathrlap{.} \nonumber
%\end{align}
%}%
{\allowdisplaybreaks
\begin{align}
\mathcal{F} = d_{\mathcal{Z}}\mathcal{A}_V & = \underbracket[0.4pt][4pt]{i \hp \partial_{\mathcal{Z}}\bar{\partial}_{\mathcal{Z}} \phi_V\nhp(\mathbf{u})}_{\text{\normalsize{$\displaystyle \sigma_0$}}\rule[7pt]{0pt}{0pt}} - \, d_{\mathcal{Z}} \alpha \\[-5pt]
& = \underbracket[0.4pt][4pt]{\sigma_0 - d[d\mu(\zeta) I(\zeta)]}_{\text{\normalsize{$\omega(\zeta)$}}\rule[9pt]{0pt}{0pt}} - \, d_{\hp\smash{\mathbb{CP}^1}}[d\mu(\zeta) I(\zeta)] + i \hp d\mu(\zeta) \wedge \frac{d\zeta}{\zeta} \nonumber \\[-7pt]
& = \omega(\zeta) + [\zeta \partial_{\zeta}(d\mu(\zeta)I(\zeta)) + i \hp d\mu(\zeta)] \wedge \frac{d\zeta}{\zeta} \nonumber \\
& = \omega(\zeta) + i \pt \iota_{X(\zeta)}\omega(\zeta) \wedge \frac{d\zeta}{\zeta} \mathrlap{.} \nonumber
\end{align}
}%
The last line follows from Lemma \ref{qmm_twist_mu}.  As expected from a field strength, this is defined globally on $\mathcal{Z}$, albeit with the divisor $\pi^{-1}(0,\infty)$ removed. Observe, moreover, that if we lift the generator $\smash{- \, i \hp \zeta \frac{\partial}{\partial \zeta}}$ of rotations around the 3-axis of the twistor $S^2$ to the vector field $\smash{- \, i \hp \zeta \frac{\partial}{\partial \zeta} + X(\zeta)}$ on $\mathcal{Z}$ considered here in an appropriate trivialization, this will then be a null vector for $\mathcal{F}$. 

From the defining formula \eqref{A_V-def} it is clear that on intersections $U \cap V$ in $\mathcal{Z}$ with $U,V \in \mathscr{V}$ we have the gluing relations
\begin{equation} \label{glue}
\mathcal{A}_U - \mathcal{A}_V = - \pt i \hp d_{\mathcal{Z}} \phi_{UV}(\zeta)
\end{equation}
where $\phi_{UV}(\zeta)$ are holomorphic gluing functions for the potentials $\phi_U\nhp(\mathbf{u})$ and $\phi_V\nhp(\mathbf{u})$ of the closed hyper $(1,1)$ form $\sigma_0$, defined as in the equation \eqref{pot_gluing}. In general these do not satisfy cocycle conditions, but the holomorphic 1-forms $d_{\mathcal{Z}}\phi_{UV}(\zeta)$ do. They determine, as we have explained in \mbox{\S\,\ref{ssec:global_iss}}, an element in the sheaf cohomology group $H^1(\mathcal{Z},d\mathcal{O}_{\mathcal{Z}})$, which in turn determines uniquely an extension class
\begin{equation}
\begin{tikzcd}
0 \rar & \mathcal{O}_{\mathcal{Z}} \rar & E \rar & T_{\mathcal{Z}} \rar & 0 \mathrlap{.}
\end{tikzcd}
\end{equation}
%The gluing relation \eqref{glue} asserts that, if one restricts to $\mathcal{Z} \smallsetminus \pi^{-1}(0,\infty)$ to avoid singularities, this 1-cocycle is in fact a 1-coboundary and therefore trivial, meaning that the correspondingly restricted short exact sequence above \textit{splits}. 
The gluing relation \eqref{glue} asserts that this 1-cocycle is in fact a 1-coboundary and therefore trivial, meaning that this short exact sequence \textit{splits}. Note that if one wants to avoid working with singular objects one needs to restrict the base to $\mathcal{Z} \smallsetminus \pi^{-1}(0,\infty)$. When the  cohomology class of $\sigma_0/2\pi$ is integral, which happens if and only if the cohomology class of $\omega_0/2\pi$ is integral, then $E$ is isomorphic to the Atiyah algebroid of the holomorphic line bundle $L_{\mathcal{Z}}$ with holomorphic transition functions $g_{UV} = \exp[\phi_{UV}(\zeta)]$. 

To sum up, we have proven the following generalization to the twisted rotational case of Proposition 5 from \cite{MR3116317}, corresponding to the purely rotational case:

\begin{theorem} \label{main_th}
Let $M$ be a hyperk\"ahler manifold with $H^1(M,\mathbb{R}) = 0$ and such that $\omega_0/2\pi$ belongs to an integral cohomology class in $H^2(M,\mathbb{R})$. If $M$ has an $\mathcal{O}(2j-2)$-twisted rotational action relative to the 3-axis with twisted generator $X(\zeta)$, then over its twistor space $\mathcal{Z}$ there exists a holomorphic line bundle $L_{\mathcal{Z}}$ trivial on each horizontal twistor line, admitting a meromorphic connection with poles of order $j$ on the divisor $\pi^{-1}(0,\infty)$ and meromorphic curvature form given in our usual trivialization of $\mathcal{Z}$ by the formula
\begin{equation}
\mathcal{F} = \omega(\zeta) + i \pt \iota_{X(\zeta)}\omega(\zeta) \wedge \frac{d\zeta}{\zeta} \mathrlap{.}
\end{equation}
Over $M$ itself, by the Atiyah\pt-Ward correspondence one has a corresponding hyperholomorphic line bundle endowed with a hyperhermitian connection with curvature 2-form equal to the closed hyper $(1,1)$ form $\sigma_0$ canonically associated to the action.
\smallskip
\end{theorem}

\subsubsection{} \label{ssec:cof}

The holomorphic twistor space 2-form $\omega$ induces on each fiber of the holomorphic twistor fibration a complex symplectic structure compatible with the complex structure on the fiber. By a complex version of Darboux's theorem, for a suitable covering $\mathscr{V}$ of $\mathcal{Z}$, on each open set $V \in \mathscr{V}$ we can then introduce a system of holomorphic canonical symplectic coordinates. In particular, given a pair of open sets $V_N \in \mathscr{V}_N$ and $V_S \in \mathscr{V}_S$ we may write 
\begin{equation}
\begin{aligned}
\omega_{\hp V_N} & = d\tilde{\eta}^{\phantom{'}}_{\A,V_N} \!\wedge d\eta^{\A}_{\hp V_N} \\
\omega_{\hp V_S} \hspace{1.5pt} & = d\tilde{\eta}^{\phantom{'}}_{\A,V_S}  \wedge d\eta^{\A}_{\hp V_S} \rlap{.}
\end{aligned}
\end{equation}
The $A$ indices run from 1 to the quaternionic dimension of $M$ and a summation convention over repeated indices is understood. The coordinates $\smash{ \tilde{\eta}^{\phantom{'}}_{\A,V_N} }$, $\smash{ \eta^{\A}_{\hp V_N} }$ and $\smash{ \tilde{\eta}^{\phantom{'}}_{\A,V_S} }$, $\smash{ \eta^{\A}_{\hp V_S} }$ are well-defined local holomorphic functions on $V_N$ and $V_S$, respectively. The ones defined on $V_N$ are in particular well-defined at and around $\zeta = 0$, with Taylor expansions in $\zeta$ of the form
\begin{equation}
\begin{aligned}
\eta^{\A}_{\hp V_N} & = z^{\A} + v^{\A}\zeta + \mathcal{O}(\zeta^2) \\
\tilde{\eta}_{\A,V_N} & = u_{\A} + \mathcal{O}(\zeta) \rlap{.}
\end{aligned}
\end{equation}
If $V_N$ and $V_S$ are interchanged by the action of the real structure on $\mathcal{Z}$ then the Taylor expansion formulas for the $V_S$ canonical coordinates around $\zeta = \infty$ follow from these by antipodal conjugation. 

Earlier on, in \mbox{\S\,\ref{ssec:sofp}}, we have argued that the modified meromorphic symplectic \mbox{1-form} potential $\smash{ \theta^{\text{reg}}_V(\zeta) }$ is the natural candidate for identification with the canonical 1-form corresponding to this canonical system of coordinates. A more precise statement is that, for reasons having to do with the particularities of our definitions, this identification can be made only up to a twisting factor. Thus, for $V_N$ and $V_S$ the correct assignment is 
\begin{equation} \label{sofp-cof}
\begin{aligned}
\theta^{\text{reg}}_{V_N}(\zeta) & = \frac{1}{\zeta} \, \tilde{\eta}^{\phantom{'}}_{\A,V_{\npt N \npt}}\hp d\eta^{\A}_{\hp V_{\nhp N \npt}}  \\
\theta^{\text{reg}}_{V_S}(\zeta) & = \hspace{1.6pt} \zeta \hspace{2.8pt} \tilde{\eta}^{\phantom{'}}_{\A,V_S}\hp d\eta^{\A}_{\hp V_S} \rlap{.}
\end{aligned}
\end{equation}
Let also $\vartheta^{\text{reg}}_V(\zeta)$ denote the pullback of $\theta^{\text{reg}}_V(\zeta)$ to $\mathcal{Z}$ by the non-holomorphic twistor projection. This is a meromorphic 1-form on $V \subset \mathcal{Z}$. In our usual trivialization the pullbacks $\smash{ \vartheta^{\text{reg}}_{V_N}(\zeta) }$ and $\smash{ \vartheta^{\text{reg}}_{V_S}(\zeta) }$ of the forms \eqref{sofp-cof} can be obtained simply by replacing in their expressions $d$ with $d_{\mathcal{Z}}$. 

On non-empty intersections $U \cap V$, with $U,V \in \mathscr{V}$, the meromorphic symplectic 1-form potentials are patched together by means of holomorphic gluing functions $\theta_{UV}(\zeta)$, that is, \mbox{$\smash{\theta^{\text{reg}}_{U}(\zeta) - \theta^{\text{reg}}_{V}(\zeta) = d\theta^{\phantom{g}}_{UV}(\zeta)}$}. One can then argue that yet another set of holomorphic functions $t_{UV}(\zeta)$ must exist such that
\begin{equation} \label{trpbf}
\vartheta^{\text{reg}}_U(\zeta) - \vartheta^{\text{reg}}_V(\zeta) = d^{\phantom{g}}_{\mathcal{Z}} \theta^{\phantom{g}}_{UV}(\zeta) - t^{\phantom{g}}_{UV}(\zeta) \frac{d\zeta}{\zeta} \rlap{.}
\end{equation}

\begin{example}
To illustrate this concretely let us consider a situation when $V_N \cap V_S \neq \emptyset$. Since $\omega$ is a section of a bundle over $\mathcal{Z}$ twisted by $\mathcal{O}(2)$, over the intersection $V_N \cap V_S$ its components are patched together by means of the $\mathcal{O}(2)$ transition function, \textit{i.e.}
\begin{equation} \label{tw-gl}
\omega_{\hp V_S} = \frac{1\pt}{\zeta^2} \, \omega_{\hp V_N} \rlap{.}
\end{equation}
The patching may be alternatively characterized by what is known as a \textit{canonical generating function}. Let us consider specifically a canonical generating function $G_{V_SV_N}$ of type II,  which is to say, assumed to depend (holomorphically) on the variables $\smash{ \eta^{\A}_{\hp V_N} }$, $\smash{ \tilde{\eta}^{\phantom{g}}_{\A,V_S} }$ and also on $\zeta$. In terms of it, the transition between the two patches is described implicitly by the \textit{twisted} canonical transformation
{\allowdisplaybreaks
\begin{equation} 
\begin{aligned}
\tilde{\eta}_{\A,V_N} = \zeta \frac{\partial \hp G_{V_SV_N}}{\partial \eta^{\A}_{\hp V_N}} \\
\eta^{\A}_{\hp V_S} = \frac{1}{\zeta} \frac{\partial \hp G_{V_SV_N}}{\partial \tilde{\eta}_{\A,V_S}}
\end{aligned}
\end{equation}
}%
for any solutions of these equations satisfy automatically the twisted symplectic gluing relation \eqref{tw-gl}. A straightforward calculation then shows that the transition property for $\smash{ \vartheta^{\text{reg}}_{V_N}(\zeta) }$ and $\smash{ \vartheta^{\text{reg}}_{V_S}(\zeta) }$ is indeed of the form \eqref{trpbf}, with the two holomorphic gluing functions, vertical and horizontal, given by 
{\allowdisplaybreaks
\begin{equation}
\begin{aligned}
\theta^{\phantom{g}}_{V_SV_N}(\zeta) & = \tilde{\eta}_{\A,V_S} \frac{\partial \hp G_{V_SV_N}}{\partial \tilde{\eta}_{\A,V_S}} - G_{V_SV_N} \\
t_{V_SV_N}(\zeta) & =  \tilde{\eta}_{\A,V_S} \frac{\partial \hp G_{V_SV_N}}{\partial \tilde{\eta}_{\A,V_S}} - \zeta \frac{\partial G_{V_SV_N}}{\partial \zeta} \rlap{.}
\end{aligned}
\end{equation}
}%
\end{example}

\subsubsection{}

Returning to the generic case, note that from the equation \eqref{s-pot} we get, after absorbing possible integration constants into the definition of $\theta_{UV}(\zeta)$, that the holomorphic gluing functions for the meromorphic connection 1-forms $\mathcal{A}_V$ are related to those of the symplectic 1-form potentials $\theta^{\textup{reg}}_V(\zeta)$ by  
\begin{equation} \label{gfcts}
\phi^{\phantom{g}}_{UV}(\zeta) =  i \hp \theta^{\phantom{g}}_{UV}(\zeta) - \varphi^{\phantom{g}}_U(\zeta) + \varphi^{\phantom{g}}_V(\zeta) \rlap{.}
\end{equation}

Recall that the fiberwise component of $\mathcal{A}_V$ is given by $\smash[b]{ \theta^{\phantom{g}}_V(\zeta) = \theta^{\text{reg}}_V(\zeta) + i \hp d\varphi^{\phantom{g}}_V(\zeta) }$. Its pullback to $\mathcal{Z}$ is naturally meromorphic, and since $\mathcal{A}_V$ must also be meromorphic it follows that on every $V \in \mathscr{V}$ there exists a local meromorphic function $t^{\phantom{g}}_V(\zeta)$ such that
\begin{equation} \label{A_V-too}
\mathcal{A}^{\phantom{g}}_V = \vartheta^{\text{reg}}_V(\zeta) + i\hp d^{\phantom{g}}_{\mathcal{Z}}\varphi^{\phantom{g}}_V(\zeta) + t^{\phantom{g}}_V(\zeta) \frac{d\zeta}{\zeta} \rlap{.}
\end{equation}
In view of the relation \eqref{gfcts}, in order for the gluing formula \eqref{glue} to hold we need to have
\begin{equation} \label{Cech-bndry}
t_U(\zeta) - t_V(\zeta) = t_{UV}(\zeta)
\end{equation}
for all $U,V \in \mathscr{V}$. An immediate corollary is that for a hyperk\"ahler manifold to possess a twisted rotational action the horizontal transition functions $t_{UV}(\zeta)$ must form a \v{C}ech coboundary. 

In addition to this there are also singularity requirements at the poles. By matching the horizontal components of the formulas \eqref{A_V-too} and \eqref{A_V-def} for $\mathcal{A}_V$, and especially the singularity structures around $\zeta = 0$ and $\infty$, we can show that the functions $t_V(\zeta)$ must be of the form
\begin{equation} \label{t_V-sing-str}
t^{\phantom{g}}_V(\zeta) =  i \pt x^{\phantom{g}}_V\varphi^{\phantom{g}}_V(\zeta) + t^{\text{reg}}_V(\zeta)
\end{equation}
for any $V \in \mathscr{V}$, with $t^{\text{reg}}_V(\zeta)$ local holomorphic functions with no singularities. Their singularity structure is in other words completely determined by the meromorphic functions $\varphi_V(\zeta)$. What is more, for the above choice of $\smash{ \vartheta^{\text{reg}}_V(\zeta) }$ we find that
\begin{equation}
t^{\text{reg}}_{V_N}(\zeta) =  i \hp (\mu - \varphi_0) - u_{\A}v^{\A}  + \mathcal{O}(\zeta) 
\end{equation}
with a corresponding expansion formula for $t^{\text{reg}}_{V_S}(\zeta)$ around \mbox{$\zeta = \infty$} dictated by antipodal conjugation. 

It is important to understand that the functions $\varphi_V(\zeta)$ are not uniquely defined, rather we are free to choose them subject only to the constraint that they are meromorphic, with antipodally paired singularities at \mbox{$\zeta = 0$} and \mbox{$\zeta = \infty$}, where their principal parts have a finite number of terms\,---\,the ones of order $\mathcal{O}(\zeta^{-2})$ and lower, respectively of order $\mathcal{O}(\zeta^2)$ and higher\,---\,whose residues are prescribed by the twisted action. Moreover, recalling that \mbox{$\theta_+ + i \hp d\varphi_+$} is globally defined and also fixed by the action, the residues of the terms of order $\mathcal{O}(\zeta^{-1})$ and $\mathcal{O}(\zeta)$, respectively, are determined by our choice of symplectic 1-form potential, which entails that $\theta_+ = u_{\A} dz^{\A}$. This leaves the terms of order $\mathcal{O}(\zeta^0)$ and higher, respectively lower, largely unconstrained beyond the requirement of holomorphicity. For our current discussion the upshot of this is that, for all $V \in \mathscr{V}_N \cup \mathscr{V}_S$, we are basically allowed to absorb the holomorphic functions $\smash{t^{\text{reg}}_V(\zeta)}$ into redefinitions of the meromorphic functions $\varphi_V(\zeta)$.  That is to say

\begin{lemma} \label{A_V-sp-gauge}
There exists a choice of functions $\varphi_V(\zeta)$ for which the connection 1-forms $\mathcal{A}_V$ take the form
\begin{equation} \label{A_V-manif_mero}
\mathcal{A}^{\phantom{g}}_V = \vartheta^{\textup{reg}}_V(\zeta) + i\hp d^{\phantom{g}}_{\mathcal{Z}}\varphi^{\phantom{g}}_V(\zeta) + i \pt x^{\phantom{g}}_V\varphi^{\phantom{g}}_V(\zeta) \frac{d\zeta}{\zeta}
\end{equation}
for all $V \in \mathscr{V}_N \cup \mathscr{V}_S$. In this special gauge we have
\begin{equation} \label{varphi0-mu}
\varphi_0 = \mu + i \hp u_{\A}v^{\A} \rlap{.}
\end{equation}
\end{lemma}

\begin{remark}
These considerations are particularly useful when the hyperk\"ahler geometric data is given to us in the form of holomorphic symplectic patching data. Then the splitting condition \eqref{Cech-bndry} together with the singularity structure at the poles \eqref{t_V-sing-str} can be used to determine a set of functions $\varphi_V(\zeta)$. These functions can be further adjusted by the addition of suitable meromorphic terms with at most a simple pole at $\zeta =0$ and $\infty$ so as to satisfy the special gauge condition \eqref{varphi0-mu} with, crucially, a real-valued $\mu$. The remaining moment maps can be extracted from the Laurent expansion of $\varphi_{V_N}(\zeta)$ around $\zeta = 0$ for any $V_N \in \mathscr{V}_N$. Failure of any step in this procedure to work signals that the hyperk\"ahler manifold in question does not possess in fact a twisted rotational symmetry. 
\end{remark}

\subsubsection{} 

We end this section with a discussion of how our approach here relates to Neitzke's approach to hyperholomorphicity in \cite{Neitzke:2011za}. 

Remark that the first equation \eqref{mojo} implies by way of the local Poincar\'e lemma that for any open set $V \in \mathscr{V}$ and any real-valued local symplectic 1-form potential $A_{0,V}$ for $\sigma_0$ (\textit{i.e.}, \mbox{$\sigma_0 = dA_{0,V}$} on the image $p(V)$ of $V$ through the non-holomorphic twistor projection) there exists a complex function $s_V(\zeta)$ such that
%\begin{equation} \label{Trump_is_an_idiot}
%A^{\phantom{g}}_{0,V} + d\hp s^{\phantom{g}}_V(\zeta) = \underbracket[0.4pt][4.2pt]{\theta^{\text{reg}}_V(\zeta)}_{\text{$_{1,0}$}} + \underbracket[0.4pt][5.2pt]{d\mu(\zeta)I(\zeta)}_{\text{$_{1,0}$}} \rlap{.}
%\mathrlap{\hspace{5pt} \underbracket[0pt][8pt]{\phantom{}}_{\text{w.r.t. $\IU$}}}
%\end{equation}
\begin{equation} \label{Trump_is_an_idiot}
A^{\phantom{g}}_{0,V} + d\hp s^{\phantom{g}}_V(\zeta) = \theta^{\text{reg}}_V(\zeta) + d\mu(\zeta)I(\zeta) \rlap{.}
\end{equation}
From the $m=0$ component of the equation \eqref{sgm-omg} it is clear that \mbox{$A_{0,V} - d\mu I_0 \equiv \theta_{0,V}$} provides a symplectic 1-form potential on $p(V)$ for $\omega_0$. Resorting to the relation \eqref{d_mu_I} and rearranging gives us then the formula
\begin{equation} \label{symp_shift}
\cramped{ \theta^{\text{reg}}_V(\zeta)  = \frac{\theta_+\! + i \hp d\varphi_+}{\zeta} + \theta_{0,V} + (\theta_-\! - i \hp d\varphi^c_-) \hp \zeta +  d \Big[ s^{\phantom{g}}_V(\zeta) + i \sum_{n=2}^j(\varphi_n\zeta^{-n} \!- \varphi^c_{-n}\zeta^n) \Big] } \rlap{.}
\end{equation}
If we compare this with the Laurent expansion \eqref{vph^reg_N_ser} around $\zeta = 0$ and with its antipodal counterpart around $\zeta = \infty$, we see that the function $s_V(\zeta)$ must be of the form
\begin{equation} \label{s_V-sing-str}
s^{\phantom{g}}_V(\zeta) =  - \pt i \hp \varphi^{\phantom{g}}_V(\zeta) + s^{\text{reg}}_V(\zeta)
\end{equation}
with $s^{\text{reg}}_V(\zeta)$ regular everywhere on its domain. 

Observing that both terms on the right hand side of the equation \eqref{Trump_is_an_idiot} are of type $(1,0)$ with respect to the complex structure $\IU$, it follows that the $(0,1)$ component of $A_{0,V}$ relative to the same complex structure is equal to $- \, \bar{\partial}_{\Iu} s_V(\zeta)$. Since $A_{0,V}$ is real-valued we may write 
%\begin{align}
%A_{0,V} & = -\, \partial_{\Iu} \overline{s_V(\zeta)} - \bar{\partial}_{\Iu} s_V(\zeta) \\[3pt]
%& = -\, \partial_{\Iu} \overline{s^{\text{reg}}_V(\zeta)} - \bar{\partial}_{\Iu} s^{\text{reg}}_V(\zeta) \nonumber
%\end{align}
\begin{equation}
A_{0,V} = - \, \partial_{\Iu} \overline{s_V(\zeta)} - \bar{\partial}_{\Iu} s_V(\zeta) 
= -\, \partial_{\Iu} \overline{s^{\text{reg}}_V(\zeta)} - \bar{\partial}_{\Iu} s^{\text{reg}}_V(\zeta) 
\end{equation}
where we have used the meromorphicity of the functions $\varphi_V(\zeta)$ to substitute under the Dolbeault operators each $s_V(\zeta)$ with its regular part. Acting further with an exterior derivative and recalling the analytic structure formula \eqref{zeta_hol_str} we then get that 
\begin{equation}
\phi^{\phantom{g}}_V(\mathbf{u}) = -\pt 2 \hp \Im \hp s^{\text{reg}}_V(\zeta) 
\end{equation}
is a local hyperpotential with respect to $\IU$ for the hyper $(1,1)$ form $\sigma_0$. This and the previous equation above may be equivalently recast as  the equality
\begin{equation}
d\hp s^{\text{reg}}_V(\zeta) + A^{\phantom{g}}_{0,V} = - \hp i \pt \partial_{\Iu} \phi^{\phantom{g}}_V\nhp(\mathbf{u}) 
\end{equation}
between two symplectic 1-form potentials for $\sigma_0$. Since the right hand side is a form of type $(1,0)$ with respect to the complex structure $\IU$, a projection argument using the holomorphic subspace projectors defined in \mbox{\S\,\ref{ssec:proj_hol_fact}} implies that the following equations must hold:
\begin{equation}
\begin{aligned}
& [(d+ A_{0,V}) \hp e^{\hp s^{\text{reg}}_V(\zeta)}] P_N(\zeta) = 0 \\
& [(d+ A_{0,V}) \hp e^{\hp s^{\text{reg}}_V(\zeta)}] P_{\pt S\pt}(\zeta) = 0
\end{aligned}
\end{equation}
for $\zeta \in \pi(V) \cap N$ and $\pi(V) \cap S$, respectively. 
%Note that the equations continue to hold if one replaces $s^{\text{reg}}_V(\zeta)$ by $s^{\phantom{g}}_V(\zeta)$. 
These are essentially the hyperholomorphicity conditions employed by Neitzke in \cite{Neitzke:2011za} (see esp. \mbox{\S\,4.5}). 

On intersections $U \cap V$, denoting with $\theta_{0,UV}$ the real-valued gluing functions for the symplectic 1-form potentials $\theta_{0,V}$ (\textit{i.e.}, $\theta_{0,U} - \theta_{0,V} = d\theta_{0,UV}$), we have
\begin{equation} \label{s^reg_V-glue}
s^{\text{reg}}_U(\zeta) - s^{\text{reg}}_V(\zeta) = - \pt i \hp \phi^{\phantom{g}}_{UV}(\zeta) - \theta^{\phantom{g}}_{0,UV} 
\end{equation}
up to possible integration constants which can be absorbed into a redefinition of $\phi_{UV}(\zeta)$. These gluing equations together with the regularity requirement at the poles and the real structure allow one in principle to determine the sections $\smash{ s^{\text{reg}}_V(\zeta) }$ from the holomorphic symplectic patching data of the hyperk\"ahler manifold.

\section{Universal expressions for the twisted symmetry generators in special coordinates} \label{sec:SymmGens}

\subsubsection{}

We have seen earlier that the generators of a trans-tri-Hamiltonian or a trans-rotational action on a hyperk\"ahler manifold come in equivalence classes and admit a special choice of representatives (Lemmas \ref{canon_pres} and \ref{canon_pres_rot}, respectively). We will now show that another special and very useful choice is possible. Namely, we will show that one has a certain natural coordinate frame in which the generators of the action admit canonical representatives given by universal symplectic gradient-type formulas, with potentials given by moment maps and the relevant gradient field acting on only half of the coordinates. As a by-product of the proof of this fact we will derive a set of generic formulas for the hyperk\"ahler metric and symplectic forms in this coordinate frame of possible independent interest.

Consider a hyperk\"ahler manifold $M$ and let $\smash{ \omega_{\hp V_N} }$ be the component of the canonical holomorphic 2-form section of the twistor space $\mathcal{Z}$ of $M$ corresponding to an open subset \mbox{$V_N \subset \mathcal{Z}$}. On each twistor fiber intersecting this subset the fiber-supported $\omega_{\hp V_N}$ can be viewed as a holomorphic symplectic form, and so by the complex version of Darboux's theorem there exists a local system of holomorphic coordinates $\smash{ \eta^{\A}_{\hp V_N} }$ and $\tilde{\eta}_{\A,V_N}$ in terms of which $\smash{ \omega_{\hp V_N} = d\tilde{\eta}_{\A,V_N} \wedge d\eta^{\A}_{\hp V_N} }$ (we use here the same index notation conventions as in \mbox{\S\,\ref{ssec:cof}}). These coordinates are in particular  well-defined at and around $\zeta=0$, labeling the complex structure $I_0$, with Taylor expansions in $\zeta$ of the form
\begin{equation}
\begin{aligned}
\eta^{\A}_{\hp V_N} & = z^{\A} + v^{\A} \zeta + \mathcal{O}(\zeta^2) \\
\tilde{\eta}_{\A,V_N} & = u_{\A} + w_{\A} \zeta + \mathcal{O}(\zeta^2) 
\end{aligned}
\end{equation}
But since $\omega$ is a global section of the $\mathcal{O}(2)$-twisted bundle $\Lambda^2T_F^* \allowbreak \otimes \pi^*\mathcal{O}(2)$ over $\mathcal{Z}$, we have at the same time $\omega_{\hp V_N} = \omega_+ + \omega_0\hp \zeta + \omega_-\zeta^2$. We can then identify 
\begin{equation} \label{oaloo}
\begin{aligned}
\omega_+ & = du_{\A} \wedge dz^{\A} \\
\omega_0 \hspace{2pt} & =  du_{\A} \wedge dv^{\A} + dw_{\A} \wedge dz^{\A} \rlap{.}
\end{aligned}
\end{equation}
For these 2-forms, or rather for their so-called Euclidean components $\omega_1$, $\omega_2$, $\omega_3$ defined by $\smash{ \omega_+ = \frac{1}{2}(\omega_1 + i\hp \omega_2) }$ and \mbox{$\omega_0 = \omega_3$}, to represent the symplectic forms of a hyperk\"ahler structure they need to satisfy a number of conditions:
\begin{itemize}
\setlength{\itemsep}{3pt}
\item[1.] \textit{Closure}: This condition is clearly automatically satisfied. 

\item[2.] \textit{Reality}: Their Euclidean components must be real, which is to say we must have $\bar{\omega}_{m} = (-)^m \omega_{-m}$.

\item[3.] \textit{Quaternionicity}: The tangent bundle endomorphisms $I_1$, $I_2$, $I_3$ built out of them as in \eqref{i=oo} must satisfy the imaginary quaternionic algebra.
\end{itemize}
These conditions impose a number of constraints on the Laurent coefficients $z^{\A}$, $v^{\A}$, $u_{\A}$ and $w_{\A}$.  To determine them explicitly we need to choose first a coordinate system. The leading coefficients $u_{\A}$ and $z^{\A}$ provide a local system of coordinates on the hyperk\"ahler manifold holomorphic with respect to the complex structure $I_0$. However, for our purposes it will be convenient to choose a different set of coordinates. To describe these one is naturally led to distinguish between several cases based on whether the coordinates $\smash{ \eta^{\A}_{\hp V_N} }$ are real (with respect to antipodal conjugation) sections of the bundle $\pi^*\mathcal{O}(2)$ or not. In the following two subsections we will discuss the two extremal situations when either \textit{none} of the $v^{\A}$ is real, or \textit{all} of the $v^{\A}$ are real. Intermediate cases when only \textit{some} of the $v^{\A}$ are real can be treated in principle through a hybrid approach. 

\subsubsection{}

Let us begin by considering the case when none of the $v^{\A}$ variables are real. In this case we can choose on the hyperk\"ahler manifold a local system of coordinates formed by $z^{\A}$, $v^{\A}$ and their complex conjugates, and in what follows we will refer to these as \textit{special coordinates anchored at $I_0$}. With this choice, the conditions that $\omega_0$ be real read as follows:
\begin{equation}
\begin{aligned}
& \frac{\partial u_{\A}}{\partial v^{\B}} = \frac{\partial u_{\B}}{\partial v^{\A}} 
&\quad\qquad& \frac{\partial w_{\A}}{\partial z^{\B}} = \frac{\partial w_{\B}}{\partial z^{\A}} 
&\quad\qquad& \frac{\partial w_{\A}}{\partial v^{\B}} = \frac{\partial u_{\B}}{\partial z^{\A}} \\
& \frac{\partial u_{\A}}{\partial \bar{v}^{\B}} = - \frac{\partial \bar{u}_{\B}}{\partial v^{\A}} 
&& \frac{\partial w_{\A}}{\partial \bar{z}^{\B}} = - \frac{\partial \bar{w}_{\B}}{\partial z^{\A}} 
&& \frac{\partial w_{\A}}{\partial \bar{v}^{\B}} = - \frac{\partial \bar{u}_{\B}}{\partial z^{\A}} \rlap{.}
\end{aligned}
\end{equation}

For any differentiable function \mbox{$f: M \rightarrow \mathbb{C}$} let us define the symplectic gradient-like vector field 
\begin{equation} \label{X_f_1}
X_f = i \hp \rho^{\A\B} \bigg( \frac{\partial f}{\partial \bar{v}^{\A}} \frac{\partial }{\partial v^{\B}} - \frac{\partial f}{\partial v^{\B}} \frac{\partial }{\partial \bar{v}^{\A}} \bigg)
\end{equation}
with $\rho^{\A\B}$ denoting the matrix inverse of the Hermitian bilinear form
\begin{equation}
\rho_{\A\B} = - i \hp \frac{\partial u_{\A}}{\partial \bar{v}^{\B}} = i \hp \frac{\partial \bar{u}_{\B}}{\partial v^{\A}} \rlap{.}
\end{equation}
By resorting to the above reality conditions we can show with a little bit of effort that, for any $f$,
{\allowdisplaybreaks
\begin{equation} \label{sym_f}
\begin{aligned}
\iota_{X_f} \omega_0 \hspace{2pt} & = df - Z_{A}(f) \hp dz^{\A} - \bar{Z}_{\A}(f) \hp d\bar{z}^{\A} \\[2pt]
\iota_{X_f} \omega_+ & = \phantom{+} \pt \bar{Y}_{\A}(f) \hp dz^{\A} \\[2pt]
\iota_{X_f} \omega_- & = - \pt Y_{\A}(f) \hp d\bar{z}^{\A}
\end{aligned}
\end{equation}
}%
where, by definition, 
{\allowdisplaybreaks
\begin{equation} \label{YZ_1}
\begin{aligned}
Y_{\A} & = \frac{\partial }{\partial \bar{v}^{\A}} - i \hp \frac{\partial \bar{u}_{\C}}{\partial \bar{v}^{\A}} \rho^{\C\D} \frac{\partial}{\partial v^{\D}} \\
Z_{\A} & = \frac{\partial }{\partial z^{\A}} - i \hp \frac{\partial \bar{u}_{\C}}{\partial z^{\A}} \rho^{\C\D} \frac{\partial}{\partial v^{\D}} \rlap{.}
\end{aligned}
\end{equation}
}%
form two sets of local complex-valued vector fields on $M$. Note, incidentally, that the first set of fields is of the above symplectic gradient type: $Y_{\A} = X_{-\bar{u}_{\A}}$. Given three different differentiable functions $f_{n-1}$, $f_n$ and $f_{n+1}$ on $M$ we then have 
\begin{gather} \label{tripl_grad}
\iota_{X_{f_{n-1}}}\omega_+ + \iota_{X_{f_n}}\omega_0 + \iota_{X_{f_{n+1}}}\omega_-  = \\[2pt]
df_n - [Z_{\A}(f_n) - \bar{Y}_{\A}(f_{n-1})] \hp dz^{\A} - [\bar{Z}_{\A}(f_n) + Y_{\A}(f_{n+1})] \hp d\bar{z}^{\A} \rlap{.} \nonumber
\end{gather}

The vector fields $Y_{\A}$ and $Z_{\A}$ form a local frame for the holomorphic tangent bundle $T^{1,0}_{I_0}M$. This follows from the observation that their action on a complete set of coordinates anti-holomorphic with respect to $I_0$ vanishes, namely: $Y_{\A}(\bar{u}_{\B}) = 0$, $Y_{\A}(\bar{z}^{\B}) = 0$ and $Z_{\A}(\bar{u}_{\B}) = 0$, $Z_{\A}(\bar{z}^{\B}) = 0$. To construct the dual coframe in \mbox{$T^{*1,0}_{I_0} M$} we introduce the 1-forms
{\allowdisplaybreaks
\begin{align}
h_{\A} & = du_{\A} - Z_{\B}(u_{\A}) \hp dz^{\B} \\
& = i \hp \rho^{\C\D} \bigg[ \frac{\partial u_{\D}}{\partial v^{\A}} \bigg( \frac{\partial \bar{u}_{\C}}{\partial z^{\B}}dz^{\B} + \frac{\partial \bar{u}_{\C}}{\partial v^{\B}}dv^{\B} \bigg) + \frac{\partial \bar{u}_{\C}}{\partial v^{\A}} \bigg( \frac{\partial u_{\D}}{\partial \bar{z}^{\B}}d\bar{z}^{\B} + \frac{\partial u_{\D}}{\partial \bar{v}^{\B}}d\bar{v}^{\B} \bigg) \npt \bigg]  \rlap{.} \nonumber
\end{align}
}%
From the definitions of the vector fields and the vanishing properties it is straightforward to see that the coframe dual to the frame spanned by $Y_{\A}$ and $Z_{\A}$ is given by the 1-forms $h^{\A} \equiv i \hp \phi^{\A\B} h_{\B}$ and $dz^{\A}$, where by definition $\phi^{\A\B}$ is the matrix inverse of
\begin{equation}
\phi_{\A\B} = i \hp Y_{\B}(u_{\A}) = \rho^{\C\D} \bigg( \pt \frac{\partial \bar{u}_{\B}}{\partial \bar{v}^{\C}} \frac{\partial u_{\A}}{\partial v^{\D}} - \frac{\partial \bar{u}_{\B}}{\partial v^{\D}} \frac{\partial u_{\A}}{\partial \bar{v}^{\C}} \bigg) \rlap{.}
\end{equation}

Notice, furthermore, that the 2-forms \eqref{oaloo} can be expressed in this coframe as follows:
\begin{equation}
\begin{aligned}
\omega_+ & = h_{\A} \wedge dz^{\A} + Z_{\B}(u_{\A})\pt dz^{\B} \npt \wedge dz^{\A} \\
\omega_0 \hspace{2pt} & =  \bigg( \! - \frac{\partial w_{\A}}{\partial \bar{z}^{\B}} + i \hp \rho^{\C\D} \frac{\partial \bar{u}_{\C}}{\partial z^{\A}} \frac{\partial u_{\D}}{\partial \bar{z}^{\B}}  \bigg) dz^{\A} \npt \wedge d\bar{z}^{\B} - i \hp \phi^{\B\A} h_{\A} \wedge \bar{h}_{\B}
\end{aligned}
\end{equation}
This representation is particularly convenient for approaching the remaining issue of qua\-ter\-ni\-o\-ni\-ci\-ty. Thus, imposing the quaternionicity conditions as formulated above yields on one hand a set of constraints on the first derivatives of the $u_{\A}$-functions, namely
\begin{equation} \label{u-con}
Z_{\B}(u_{\A}) = Z_{\A}(u_{\B})
\end{equation}
(note that these are trivially satisfied for $\dim_{\hp \mathbb{H}} \npt M = 1$) and on the other fixes
\begin{equation} \label{w-con}
\frac{\partial w_{\A}}{\partial \bar{z}^{\B}} 
=   i \hp \rho^{\C\D} \frac{\partial \bar{u}_{\C}}{\partial z^{\A}} \frac{\partial u_{\D}}{\partial \bar{z}^{\B}} + i \hp \phi_{\A\B}
= i \hp \rho^{\C\D} \bigg( \pt \frac{\partial \bar{u}_{\C}}{\partial z^{\A}} \frac{\partial u_{\D}}{\partial \bar{z}^{\B}}  + \frac{\partial \bar{u}_{\C}}{\partial \bar{v}^{\B}} \frac{\partial u_{\D}}{\partial v^{\A}} - \frac{\partial \bar{u}_{\C}}{\partial v^{\A}} \frac{\partial u_{\D}}{\partial \bar{v}^{\B}} \bigg)
\end{equation}
to give us eventually the following manifestly quaternionic expressions for the hyperk\"ahler symplectic forms:
\begin{equation} \label{HK_sfs}
\begin{aligned}
\omega_+ & = h_{\A} \wedge dz^{\A} \\
\omega_0 \hspace{2pt} & = -\pt i \hp (\hp \phi_{\A\B} \hp dz^{\A} \npt \wedge d\bar{z}^{\B} + \phi^{\B\A} h_{\A} \wedge \bar{h}_{\B}) \rlap{.}
\end{aligned}
\end{equation}
Corresponding to these we have the hyperk\"ahler metric
\begin{equation} \label{HK_m}
g = \phi_{\A\B} \hp dz^{\A} d\bar{z}^{\B} + \phi^{\B\A} h_{\A} \bar{h}_{\B} \rlap{.}
\end{equation}

\subsubsection{} \label{real-vA}

Let us turn to the case when all the $v^{\A}$ variables are real. To indicate this we will use in what follows the special notation $v^{\A} = \smash{\psi^{\A}}$. In this case it is natural to take the twistor space \mbox{$\eta^{\A}${-\pt}coordinates} to be sections of the $\pi^*\mathcal{O}(2)$ bundle over $\mathcal{Z}$ real with respect to antipodal conjugation, in which case 
%their Taylor expansion around $\zeta =0$ terminates rather fast and 
we have simply
\begin{equation}
\eta^{\A}_{\hp V_N} = z^{\A} + \psi^{\A}\zeta - \bar{z}^{\A} \zeta^2 \rlap{.}
\end{equation}
As $v^{\A}$ and $\bar{v}^{\A}$ now coincide, the previously chosen set of variables is no longer sufficient to coordinatize the hyperk\"ahler manifold. Nevertheless, there is a natural way to supplement this set with the requisite number of variables. The crucial observation is that, when this assumption holds, neither the reality nor the quaternionicity conditions constrain the real parts of the $u_{\A}$ coefficients, which means that we are essentially allowed to cast these in the role of independent variables. In practice it will actually be convenient to retain some freedom of choice by allowing for an ambiguity in the definition, and so we introduce the closely related variables \mbox{$\smash{\tilde{\psi}_{\A}} = \Re u_{\A} - \varrho_{\A}$}, where $\varrho_{\A}$ represent a set of arbitrary real-valued functions, independent of $\smash{\tilde{\psi}_{\A}}$, which we can in principle fix or constrain at a subsequent stage in accordance with more specific needs. The hyperk\"ahler manifold will then be properly coordinatized by $z^{\A}$, $\bar{z}^{\A}$, $\smash{\psi^{\A}}$ and $\smash{\tilde{\psi}_{\A}}$, and these are going to be what we are going to call \textit{special coordinates anchored at $I_0$} in this case.

The requirement that $\omega_0$ be real imposes now a different set of constraints: 
{\allowdisplaybreaks
\begin{equation}
\begin{aligned}
& \frac{\partial (u_{\A} - \bar{u}_{\A})}{\partial \psi^{\B}} = \frac{\partial (u_{\B} - \bar{u}_{\B})}{\partial \psi^{\A}} 
&\quad\qquad& \frac{\partial w_{\A}}{\partial z^{\B}} = \frac{\partial w_{\B}}{\partial z^{\A}} \\
& \frac{\partial (u_{\A} - \bar{u}_{\A})}{\partial \psitilde_{\B}} = 0 
&\quad\qquad& \frac{\partial w_{\A}}{\partial \psitilde_{\B}} = 0 \\
& \frac{\partial (u_{\A} - \bar{u}_{\A})}{\partial z^{\B}} = \frac{\partial w_{\B}}{\partial \psi^{\A}}
&\quad\qquad& \frac{\partial w_{\A}}{\partial \bar{z}^{\B}} = - \frac{\partial \bar{w}_{\B}}{\partial z^{\A}} \rlap{.}
\end{aligned}
\end{equation}
}%

To any differentiable function $f : M \rightarrow \mathbb{C}$ let us associate the symplectic gradient-like vector field 
\begin{equation} \label{X_f_2}
X_f = \frac{\partial f}{\partial \psi^{\A}} \frac{\partial}{\partial \psitilde_{\A}} - \frac{\partial f}{\partial \psitilde_{\A}} \frac{\partial}{\partial \psi^{\A}} \rlap{.}
\end{equation}
Then, provided that we choose the ambiguous shift terms $\varrho_{\A}$ such that
\begin{equation} \label{shift-constr}
\!\!\! \frac{\partial \varrho_{\A}}{\partial \psi^{\B}} = \frac{\partial \varrho_{\B}}{\partial \psi^{\A}}
\qquad\quad
\frac{\partial \varrho_{\A}}{\partial \psitilde_{\B}} = 0  \rlap{,}
\end{equation}
meaning that two of the above reality conditions can be replaced with
\begin{equation}
\frac{\partial u_{\A}}{\partial \psi^{\B}} = \frac{\partial u_{\B}}{\partial \psi^{\A}} 
\qquad\quad
\frac{\partial u_{\A}}{\partial \psitilde_{\B}} = \delta^{\B}_{\A} \rlap{,}
\end{equation}
we can show that this vector field satisfies a set of properties formally identical to the properties \eqref{sym_f} and \eqref{tripl_grad} in the previous case, with the vector fields $Y_{\A}$ and $Z_{\A}$ given now by the expressions
\begin{equation} \label{YZ_2}
\begin{aligned}
Y_{\A} & = \frac{\partial }{\partial \psi^{\A}} - \frac{\partial \bar{u}_{\B}}{\partial \psi^{\A}} \frac{\partial}{\partial \psitilde_{\B}} \\
Z_{\A} & = \frac{\partial }{\partial z^{\A}} - \frac{\partial \bar{u}_{\B}}{\partial z^{\A}} \frac{\partial}{\partial \psitilde_{\B}}  \rlap{.}
\end{aligned}
\end{equation}
By the same token as before these vector fields form a local frame for  the holomorphic tangent bundle $T^{1,0}_{I_0}M$. The corresponding dual coframe for $T^{*1,0}_{I_0}M$ is constructed following the same principle, starting from
{\allowdisplaybreaks
\begin{align} \label{h_A_2}
h_{\A} & = du_{\A} - Z_{\B}(u_{\A}) \hp dz^{\B} \\[3pt]
& = d\tilde{\psi}_{\A} + \frac{\partial \bar{u}_{\A}}{\partial z^{\B}} dz^{\B} + \frac{\partial u_{\A}}{\partial \bar{z}^{\B}} d\bar{z}^{\B} + \frac{\partial u_{\A}}{\partial \psi^{\B}} d\psi^{\B} \nonumber %\\
% & = \frac{\partial u_{\A}}{\partial \psitilde_{\C}} \bigg( \frac{\partial \bar{u}_{\C}}{\partial z^{\B}}dz^{\B} + \frac{\partial \bar{u}_{\C}}{\partial \psitilde_{\B}}d\tilde{\psi}_{\B} \bigg) + \frac{\partial \bar{u}_{\A}}{\partial \psitilde_{\C}} \bigg( \frac{\partial u_{\C}}{\partial \bar{z}^{\B}}d\bar{z}^{\B} + \frac{\partial u_{\C}}{\partial \psi^{\B}}d\psi^{\B} \bigg)  \nonumber
\end{align}
}%
and with
\begin{equation} \label{phi_AB_2}
\phi_{\A\B} = i \hp Y_{\B}(u_{\A}) = i \hp \frac{\partial (u_{\A} - \bar{u}_{\A})}{\partial \psi^{\B}} \rlap{.}
\end{equation}
real and symmetric. The 2-forms \eqref{oaloo} assume in this coframe the form
\begin{equation}
\begin{aligned}
\omega_+ & = h_{\A} \wedge dz^{\A} + Z_{\B}(u_{\A})\pt dz^{\B} \npt \wedge dz^{\A} \\
\omega_0 \hspace{2pt} & =  - \frac{\partial w_{\A}}{\partial \bar{z}^{\B}} \hp dz^{\A} \npt \wedge d\bar{z}^{\B} - i \hp \phi^{\B\A} h_{\A} \wedge \bar{h}_{\B} \rlap{.}
\end{aligned}
\end{equation}
The quaternionicity conditions impose then on one hand the same constraint \eqref{u-con}, and on the other the constraint
\begin{equation}
\frac{\partial w_{\A}}{\partial \bar{z}^{\B}} = i \hp \phi_{\A\B} \rlap{,}
\end{equation}
which, in this context, translate to the conditions
\begin{equation} \label{quatern-con}
\frac{\partial (u_{\A} - \bar{u}_{\A})}{\partial z^{\B}} =  \frac{\partial (u_{\B} - \bar{u}_{\B})}{\partial z^{\A}}
\qquad\qquad
\frac{\partial w_{\A}}{\partial \bar{z}^{\B}} = - \frac{\partial (u_{\A} - \bar{u}_{\A})}{\partial \psi^{\B}} \rlap{,}
\end{equation}
to eventually yield formulas formally identical to the ones in equations \eqref{HK_sfs} and \eqref{HK_m} for the hyperk\"ahler metric and symplectic forms.

In this case, these formulas represent an equivalent rephrasing of a more familiar set of formulas. To see this, introduce real components by \mbox{$z^{\A} = \frac{1}{2} (x^{\A}_1 + i \hp x^{\A}_2)$}, \mbox{$\psi^{\A} = x^{\A}_3$} and assemble them into $\mathbb{R}^3$-vectors $\vec{r}^{\,\A} = x^{\A}_1 \pt \pmb{i} + x^{\A}_2 \pt \pmb{j} + x^{\A}_3 \pt \pmb{k}$. If we define, moreover, the real-valued fields
{\allowdisplaybreaks
\begin{equation}
\begin{aligned}
\Phi_{\A\B} & = \frac{1}{2} \phi_{\A\B} = \frac{i}{2} \frac{\partial (u_{\A} - \bar{u}_{\A})}{\partial \psi^{\B}} \\
A_{\A} & = \frac{\partial \bar{u}_{\A}}{\partial z^{\B}} dz^{\B} + \frac{\partial u_{\A}}{\partial \bar{z}^{\B}} d\bar{z}^{\B} + \frac{1}{2} \frac{\partial (u_{\A} - \bar{u}_{A})}{\partial \psi^{\B}} d\psi^{\B}
\end{aligned}
\end{equation}
}%
by construction independent of the $\smash{\tilde{\psi}_{\A}}$ coordinates, then we have $h_{\A} = d\smash{\tilde{\psi}_{\A}} + A_{\A} - i \hp \Phi_{\A\B} \hp d\psi^{\B}$, and then a straightforward exercise shows that the above formulas for the hyperk\"ahler symplectic 2-forms and metric may be recast in terms of these fields in the extended Gibbons-Hawking form \cite{Gibbons:1979zt,MR953820}
{\allowdisplaybreaks
\begin{equation}
\begin{aligned}
\vec{\omega} & = - \pt \frac{1}{2} \Phi_{\A\B} \pt d\vec{r}^{\,\A} \npt\nhp \wedge d\vec{r}^{\,\B} - d\vec{r}^{\,\A} \npt\nhp \wedge (d\tilde{\psi}_{\A} + A_{\A}) \\
g & = \frac{1}{2} \Phi_{\A\B}\pt d\vec{r}^{\,\A} \!\cdot d\vec{r}^{\,\B} + \frac{1}{2} \Phi^{\A\B}(d\tilde{\psi}_{\A} + A_{\A}) (d\tilde{\psi}_{\B} + A_{\B})
\end{aligned}
\end{equation}
}%
where we similarly assemble the real components of the \mbox{2-forms} $\omega_+$, $\omega_0$, $\omega_-$ into an $\mathbb{R}^3$-vector $\vec{\omega}$. The wedge product between $\mathbb{R}^3$-vector-valued 1-forms is defined consistently with the rules of the cross product. In addition, the extended Bogomolny equations
\begin{equation} 
dA_{\C} = \star^{\B} d \hp \Phi_{\C\B} 
\qquad \text{and} \qquad
\vec{\partial}_{\A} \Phi_{\C\B} = \vec{\partial}_{\B} \Phi_{\C\A}
\end{equation}
hold, insuring the closure of the 2-forms. The Euclidean gradient vector fields $\vec{\partial}_{\A}$ form a coordinate frame on \mbox{$\mathbb{R}^m \otimes \mathbb{R}^3$} (where \mbox{$m = \dim_{\mathbb{H}} \npt M$}), and the action of the linear star operators on the dual coordinate frame is given by \mbox{$\smash{ \star^{\A}d\vec{r}^{\,\B} = \frac{1}{2} \pt d\vec{r}^{\,\A} \!\wedge d\vec{r}^{\,\B} }$}.
Thus, the case when all of the $v^{\A}$ variables are real describes the local geometry of what we shall refer to as \textit{manifolds of extended Gibbons-Hawking type}, that is, of hyperk\"ahler manifolds possessing a continuous abelian tri-Hamiltonian symmetry of rank equal to their quaternionic dimension.

\subsubsection{}

Before we proceed to the next step in the argument we open a parenthesis to make an observation which, while not immediately necessary for our current discussion, will become relevant later on. 

Let us begin with the first case we have considered above, when none of the $v^{\A}$ are real. Remark that in this case the reality requirement for $\omega_0$ may be equivalently restated as the condition that the 1-form \mbox{$\alpha = u_{\A} dv^{\A} - \bar{u}_{\A} d\bar{v}^{\A} + w_{\A} dz^{\A} - \bar{w}_{\A} d\bar{z}^{\A}$} be closed. By the local Poincar\'e lemma it follows that locally there exists a real-valued function $L$ of class at least $C^2$ depending on the variables $z^{\A}$, $\bar{z}^{\A}$, $v^{\A}$, $\bar{v}^{\A}$ such that
\begin{equation} \label{uw-L_1}
\begin{aligned}
u_{\A} & = i \hp L_{v^{A}} \\[0pt]
w_{\A} & = i \hp L_{z^{A}}
\end{aligned}
\end{equation} 
where the indices of $L$ represent derivatives. Substituting back into the formulas \eqref{oaloo} gives us for the hyperk\"ahler symplectic forms the expressions
\begin{equation} \label{HK-sym-L}
\begin{aligned}
\omega_+ & = i \pt dL_{v^{\A}} \npt \wedge dz^{\A}\\
\omega_0 \hspace{1.9pt} & = \frac{i}{2} (\pt dL_{v^{\A}} \npt \wedge dv^{\A} - dL_{\bar{v}^{\A}} \npt \wedge d\bar{v}^{\A} + dL_{z^{\A}} \npt \wedge dz^{\A} - dL_{\bar{z}^{\A}} \npt \wedge d\bar{z}^{\A}) \rlap{.}
\end{aligned}
\end{equation}
with $\omega_0$ now manifestly real  (note that in order to arrive at this real expression one must add following the substitution a trivially vanishing term $- \hp i/2\pt d^2L$). The quaternionicity conditions \eqref{u-con} and \eqref{w-con} impose further a set of highly non-linear differential constraints on $L$. Clearly, any function $L$ satisfying these constraints determines a hyperk\"ahler metric and, conversely, any hyperk\"ahler metric with no tri-Hamiltonian isometries can be retrieved locally from a function of this type. In the special coordinate system given by \mbox{$z^{\A}$, $\bar{z}^{\A}$, $v^{\A}$, $\bar{v}^{\A}$} the function $L$ is a natural choice for the role of local potential, rather similarly to the way K\"ahler potentials are for holomorphic coordinate systems. The two descriptions are in fact related by a Legendre transform. More precisely, the (flipped-sign) Legendre transform of $L$ with respect to the $v^{\A}$ and $\bar{v}^{\A}$ variables, that is, 
%The two descriptions are in fact related by a Legendre transform. More precisely, the Legendre transform of $L$ with respect to the variables $v^{\A}$ and their complex conjugates $\bar{v}^{\A}$, that is
\begin{equation}
\kappa_0(z,u,\bar{z},\bar{u}) = L(z,v,\bar{z},\bar{v}) + i \hp (u_{\A}v^{\A} - \bar{u}_{\A}\bar{v}^{\A}) \rlap{,}
\end{equation}
with the first equation \eqref{uw-L_1} and its complex conjugate playing the role of Legendre conditions, gives a local K\"ahler potential for the hyperk\"ahler metric on $M$ in the coordinates \mbox{$z^{\A}$, $u_{\A}$} holomorphic with respect to $I_0$. 

\begin{remark}
For the class of generalized Legendre transform metrics we have mentioned at the end of \mbox{section \ref{sec:Tri-Ham}}\,---\,for which all of the Darboux coordinates $\eta^{\A}_{\hp V}$ can be viewed as components of sections of bundles $\pi^*\mathcal{O}(2j_{\A})$ over $\mathcal{Z}$ for some positive integers $j_{\A} \geq 2$\,---\,Lindstr\"om and Ro\v{c}ek showed in \cite{MR929144} that the function $L$ can be obtained through extremization from another function of more variables, but in terms of which the differential constraints can be expressed \textit{linearly}. 
\end{remark}

In the other case, when all of the $v^{\A}$ variables are real, denoting again \mbox{$v^{\A} = \psi^{\A}$} one can similarly see that the condition that $\omega_0$ be real is equivalent to the condition that the \mbox{1-form} \mbox{$\smash{\alpha = (u_{\A} - \bar{u}_{\A}) \hp d\psi^{\A} + w_{\A} dz^{\A} - \bar{w}_{\A} d\bar{z}^{\A}}$} be closed. A $C^2$-function $L$ depending on the variables $z^{\A}$, $\bar{z}^{\A}$, $\psi^{\A}$ thus exists such that 
\begin{equation} \label{uw-L_2}
\begin{aligned}
u_{\A} - \bar{u}_{\A} & = i \hp L_{\psi^{A}} \\[0pt]
w_{\A} & = i \hp L_{z^{A}} \rlap{.}
\end{aligned}
\end{equation}
Recalling our earlier considerations regarding the real part of the $u_{\A}$ variables, we may replace the first equation with
\begin{equation} \label{u-tor}
u_{\A} = \tilde{\psi}_{\A} + \varrho_{\A} + \frac{i}{2} L_{\psi^{\A}} \rlap{.}
\end{equation}
Substituting these expressions for $u_{\A}$ and $w_{\A}$ into the formulas \eqref{oaloo} for the hyperk\"ahler symplectic forms we get
\begin{equation} \label{HK-sym-tor}
\begin{aligned}
\omega_+ & = d(\tilde{\psi}_{\A} + \varrho_{\A})\wedge dz^{\A}  + \frac{i}{2}dL_{\psi^{\A}} \wedge dz^{\A} \\
\omega_0 \hspace{1.9pt} & = d(\tilde{\psi}_{\A} + \varrho_{\A}) \wedge d\psi^{\A} + \frac{i}{2} \pt dL_{z^{\A}} \wedge dz^{\A} - \frac{i}{2} \pt dL_{\bar{z}^{\A}} \wedge d\bar{z}^{\A} \rlap{.}
\end{aligned}
\end{equation}
where again, to obtain a manifestly real $\omega_0$, we add a trivially vanishing term $- \hp i/2\pt d^2L$. The quaternionicity conditions \eqref{quatern-con} impose further on $L$, in a drastic simplification with respect to the previous case, the \textit{linear} differential constraints 
\begin{equation} \label{L-PDEs_1}
L_{\psi^{\A}z^{\B}} = L_{\psi^{\B}z^{\A}}
\qquad \text{and} \qquad
L_{\psi^{\A}\psi^{\B}} + L_{z^{\A}\bar{z}^{\B}} = 0 \rlap{.}
\end{equation}
If we work instead with $\mathbb{R}^3$-vector-valued coordinates, that is, if we consider $L$ to be a function of the $\vec{r}^{\,\A}$ variables rather than of the $z^{\A}$, $\bar{z}^{\A}$, $\psi^{\A}$ ones, these constraints can be equivalently expressed with the usual notation conventions from three-dimensional vector calculus in the form
\begin{equation} \label{L-PDEs_2}
\vec{\partial}_{\A} \cdot \pt \vec{\partial}_{\B} \pt L = 0 
\qquad \text{and} \qquad 
\vec{\partial}_{\A} \times \vec{\partial}_{\B} \pt L = 0 
\end{equation}
which one may view as a natural generalization of the Laplace condition on $\mathbb{R}^3$ to $\mathbb{R}^m \otimes \mathbb{R}^{3}$. In terms of the function $L$, the Gibbons-Hawking potential and connection 1-form read
\begin{equation} \label{LT_eqs}
\begin{aligned}
\Phi_{\A\B} & = - \frac{1}{2} L_{\psi^{\A}\psi^{\B}} \\[1pt]
A_{\A} & = \Im (L_{\psi^{A}z^{B}}dz^{\B}) + d\varrho_{\A} \rlap{.}
\end{aligned}
\end{equation}
Any solution of the equations \eqref{L-PDEs_1} or, alternatively, the equations \eqref{L-PDEs_2} determines a hyperk\"ahler metric of extended Gibbons-Hawking type and, conversely, any such metric can be locally derived from such a function $L$. Similarly as above one can check that the (flipped-sign) Legendre transform of $L$ with respect to the $\psi^{\A}$ variables, 
\begin{equation}
\kappa_0(z,u,\bar{z},\bar{u}) = L(z,\bar{z},\psi) + i \hp (u_{\A} - \bar{u}_{\A}) \hp \psi^{\A} \rlap{,}
\end{equation}
for which the Legendre condition coincides with the first equation \eqref{uw-L_2}, gives a K\"ahler potential in the coordinates $z^{\A}$, $u_{\A}$ holomorphic with respect to $I_0$ \cite{MR877637}.

\subsubsection{}

From here on we continue the discussion by treating the two cases together, with the implicit understanding that the fields $X_f$ are defined by the equation \eqref{X_f_1} in the first case respectively \eqref{X_f_2} in the second, the fields $Y_{\A}$ and $Z_{\A}$ by the equations \eqref{YZ_1} respectively \eqref{YZ_2}, and so on. 

The following lemma collects for easy reference one of the important formulas proven above in each case: 

\begin{lemma} \label{tri-sym-formula}
Given any three differentiable functions $f_{n-1}$, $f_n$ and $f_{n+1}$ on $M$ we have
\begin{gather} 
\iota_{X_{f_{n-1}}}\omega_+ + \iota_{X_{f_n}}\omega_0 + \iota_{X_{f_{n+1}}}\omega_- =  \\[2pt]
df_n - [Z_{\A}(f_n) - \bar{Y}_{\A}(f_{n-1})] \hp dz^{\A} - [\bar{Z}_{\A}(f_n) + Y_{\A}(f_{n+1})] \hp d\bar{z}^{\A} \rlap{.} \nonumber
\end{gather}
\end{lemma}

The hyperk\"ahler symplectic 2-forms and metric take in both cases the form \eqref{HK_sfs}--\eqref{HK_m}. Knowledge of the symplectic forms determines the complex structures, and so in particular we get
{\allowdisplaybreaks
\begin{equation} \label{HK-c-strs}
\begin{aligned}
I_+ & = i \hp \bar{Z}_{\A} \otimes h^{\A} - i \pt \bar{Y}_{\A} \otimes dz^{\A} \\
P^{1,0}_{I_0} & = Z_{\A} \otimes dz^{\A} + Y_{\A} \otimes h^{\A}  \rlap{.}
\end{aligned}
\end{equation}
}%
(By definition, $h^{\A} = i \hp \phi^{\A\B} h_{\B}$. Recall also that, in either case above, $Z_{\A}$, $Y_{\A}$ and their complex conjugates on one hand, and $dz^{\A}$, $h^{\A}$ and their complex conjugates on the other give dual local frames for the complexified tangent and cotangent bundles of $M$, respectively.) For any two functions $f_{n}$ and $f_{n+1}$ on $M$ we then have
\begin{equation}
df_{n+1} P^{1,0}_{I_0} - i \hp df_n I_+ = [Z_{\A}(f_{n+1}) - \bar{Y}_{\A}(f_n)] \hp dz^{\A} + [\bar{Z}_{\A}(f_n) + Y_{\A}(f_{n+1})] \hp h^{\A} \rlap{.}
\end{equation}
By the equivalence $1 \Leftrightarrow 3$ of Proposition \ref{criterion_2} it follows that 

\begin{lemma} \label{YZ-hyperpot}
Two functions $f_{n}$ and $f_{n+1}$ on $M$ are adjacent hyperpotentials with respect to the complex structure $I_0$ if and only if they satisfy the system of first-order linear PDEs
\begin{equation}
\begin{aligned}
Z_{\A}(f_{n+1}) - \bar{Y}_{\A}(f_n) & = 0 \\
\bar{Z}_{\A}(f_n) + Y_{\A}(f_{n+1}) & = 0 \rlap{.}
\end{aligned}
\end{equation}
\end{lemma}

Notice that since $\eta^{\A}_{\hp V_N}$ and $\tilde{\eta}_{\A,V_N}^{\phantom{'}}$ are locally-defined holomorphic functions on the twistor space we already have at our disposal two chains of hyperpotentials with respect to $I_0$ on which we can verify this assertion, namely
\begin{equation*}
\begin{aligned}
\cdots \,\noarrow\, v^{\A} \,\noarrow\, z^{\A} \,\noarrow\, 0  & \\[2pt]
\cdots \,\noarrow\, w_{\A} \,\noarrow\, u_{\A} \,\noarrow\, 0 & \rlap{.} \\[1pt]
\end{aligned}
\end{equation*}
And indeed, the equations
\begin{equation}
\begin{aligned}
Z_{\A}(z^{\B}) - \bar{Y}_{\A}(v^{\B}) & = 0 \\
\bar{Z}_{\A}(v^{\B}) + Y_{\A}(z^{\B}) & = 0
\end{aligned}
\end{equation}
follow trivially from the definitions of $Y_{\A}$ and $Z_{\A}$. On the other hand, the equations
\begin{equation}
\begin{aligned}
Z_{\A}(u_{\B}) - \bar{Y}_{\A}(w_{\B}) & = 0 \\
\bar{Z}_{\A}(w_{\B}) + Y_{\A}(u_{\B}) & = 0
\end{aligned}
\end{equation}
are only satisfied provided that the quaternionicity conditions hold, and in fact one may consider them to be equivalent to these.

Lemmas \ref{tri-sym-formula} and \ref{YZ-hyperpot} have the following important direct consequence: 

\begin{proposition} \label{gmm-hyperps}
A sequence of functions $\{f_n\}_{n \in \mathbb{Z}}$ on $M$ forms a chain of hyperpotentials with respect to the complex structure $I_0$ if and only if
\begin{equation}
\iota_{X_{f_{n-1}}}\omega_+ + \iota_{X_{f_n}}\omega_0 + \iota_{X_{f_{n+1}}}\omega_- = df_n
\end{equation}
for all $n \in \mathbb{Z}$.
\end{proposition}

\noindent This equation is essentially a Cauchy-Riemann condition for chains of hyperpotentials, in special coordinates.

\subsubsection{}

The stage is now set to merge these considerations with our previous discussion of twisted symmetries, both of trans-tri-Hamiltonian and of trans-rotational type. 

\begin{theorem}[The trans-tri-Hamiltonian case] \label{t-3Ham_grad}

Consider a hyperk\"ahler manifold $M$ possessing a trans-tri-Hamiltonian symmetry generated by a (non-unique) set of vector fields $\{X_n\}_{n=1-j,\dots,j-1}$, with moment map functions $\{\mu_n\}_{n=-j,\dots,j}$ satisfying the generalized moment map equations\,\footnote{\pt Here as elsewhere we assume implicitly the convention that objects with off-range indices vanish.}
\begin{equation}
\iota_{X_{n-1}} \omega_+ + \iota_{X_{n}} \omega_0 + \iota_{X_{n+1}} \omega_- = d\mu_{\hp n} \mathrlap{.}
\end{equation}
Choose as above a Darboux trivialization of the canonical 2-form on the twistor space $\mathcal{Z}$ of $M$ in such a way that one of the complex canonical coordinates is given by the holomorphic section associated to the generalized moment map, that is, such that for some index $A_0$ we have $\smash{ \eta^{\A_0}_{\hp V_N} = \mu_{j} + \mu_{j-1}\zeta + \cdots + \mu_{-j}\zeta^{2j} }$. Then in the special coordinates on $M$ associated to this trivialization a representative set of the generators is given explicitly by 
\begin{equation} \label{Xn-Xmu}
X_n = X_{\mu_n}
\end{equation}
for all values of $n$.

\end{theorem}

\begin{proof}
By Proposition \ref{Kill_tens_chain}, the moment map functions $\{\mu_n\}_{n=-j,\dots,j}$ form a bounded chain of hyperpotentials with respect to the complex structure $I_0$. The theorem follows then directly from Proposition \ref{gmm-hyperps}. The only thing we still need to check is whether we have the requisite number of generators. This is ensured by our particular choice of Darboux trivialization. For the singled out index $A_0$ we have \mbox{$z^{\A_0} = \mu_j$} and \mbox{$\bar{z}^{\A_0} = (-)^j \mu_{-j}$}, from which it follows immediately that \mbox{$X_{\mu_j} \npt = X_{\mu_{-j}} \npt = 0$}, leaving us as needed with only $2j-1$ non-vanishing generators. 
\end{proof}

\noindent When $j=1$ we have only one generator, and it is straightforward to check that the formula \eqref{Xn-Xmu} reduces in this case to
\begin{equation}
X_0 = \frac{\partial}{\partial \psitilde_{\A_{\mathrlap{0}}}} \rlap{.}
\end{equation}
This shows once again that trans-tri-Hamiltonian symmetries, as we have defined them, are natural generalizations of the tri-holomorphic isometries of Gibbons and Hawking.

\begin{theorem}[The trans-rotational case] \label{t-rot_grad}

Consider a hyperk\"ahler manifold $M$ possessing a trans-rotational symmetry generated by a (non-unique) set of vector fields $\{X_n\}_{n=1-j,\dots,j-1}$ satisfying the reality condition $\bar{X}_n = (-)^n X_{-n}$, such that
\begin{equation} \label{tr-rot-mom-maps}
\iota_{X_{n-1}}\omega_+ + \iota_{X_{n}}\omega_0 + \iota_{X_{n+1}}\omega_- =
\begin{cases}
d\varphi_n & \text{if $n = 2,...\, , j$} \\
d\varphi_+ \!- i \hp \theta_+ & \text{if $n = 1$} \\
d\mu & \text{if $n=0$} \rlap{.}
\end{cases}
\end{equation}
Choose a Darboux trivialization of the canonical 2-form on the twistor space and a corresponding system of special coordinates on $M$ anchored at $I_0$. Let $\{\varphi_n\}_{n \leq j}$ be a right-bounded chain of hyperpotentials with respect to $I_0$ associated to the action, set in a gauge in which \mbox{$\varphi_0 = \mu + i \hp u_{\A} v^{\A}$} (such a gauge exists by Lemma \ref{A_V-sp-gauge}; moreover, in this case $\theta_+ = u_{\A} dz^{\A}$). Then a representative set of the generators is given explicitly by the formulas
\begin{equation} \label{X_n-sg}
X_n = 
\begin{cases}
X_{\varphi_n} & n = 1, \dots, j \\[1pt]
X_{\mu} & n=0
\end{cases}
\end{equation}
with the remaining generators determined by an alternating reality condition. 

\end{theorem}

\begin{proof}
If we take $X_n = X_{\varphi_n}$ for $n=1,\dots,j$, then the equations \eqref{tr-rot-mom-maps} with $n=2,\dots,j$ hold automatically by virtue of Proposition \ref{gmm-hyperps}. Taking further $X_0$ to be equal to $X_{\varphi_0}$ fails however to deliver the $n=1$ equation, and anyway this choice does not satisfy the required reality condition ($X_0$ needs to be real). So let us take instead $X_0 = X_{\mu}$, which is, to begin with, real. By the formula of Lemma \ref{tri-sym-formula},
\begin{gather}
\iota_{X_{\mu}}\omega_+ + \iota_{X_{\varphi_+}}\omega_0 + \iota_{X_{\varphi_{++}}}\omega_-  = \\[2pt]
d\varphi_+ - [Z_{\A}(\varphi_+) - \bar{Y}_{\A}(\mu)] \hp dz^{\A} - [\bar{Z}_{\A}(\varphi_+) + Y_{\A}(\varphi_{++})] \hp d\bar{z}^{\A} \rlap{.} \nonumber
\end{gather}
Recalling that $u_{\B}$ is in the kernel of $\bar{Y}_{\A}$, and since $\varphi_0$, $\varphi_+$, $\varphi_{++}$ are adjacent hyperpotentials with respect to $I_0$, we have
\begin{equation}
\begin{aligned}
& Z_{\A}(\varphi_+) - \bar{Y}_{\A}(\mu) = Z_{\A}(\varphi_+) - \bar{Y}_{\A}(\varphi_0) + i \hp u_{\B} \bar{Y}_{\A}(v^{\B}) = i \hp u_{\A} \\
& \bar{Z}_{\A}(\varphi_+) + Y_{\A}(\varphi_{++}) = 0 
\end{aligned}
\end{equation}
which gives us in the end
\begin{equation}
\iota_{X_{\mu}} \omega_+ + \iota_{X_{\varphi_+}} \omega_0 + \iota_{X_{\varphi_{++}}} \omega_- = d\varphi_+ - i \hp u_{\A} dz^{\A} \rlap{.}
\end{equation}
Thus, this choice for $X_0$ gives us as wanted the $n=1$ component of the system \eqref{tr-rot-mom-maps}, with \mbox{$\theta_+ = u_{\A} dz^{\A}$}, clearly a symplectic 1-form potential for $\omega_+$ holomorphic with respect to $I_0$. 

Let us verify now whether the $n=0$ component holds as well. Similarly as above,
\begin{gather}
\iota_{X_{-\bar{\varphi}_{+}}} \omega_+ + \iota_{X_{\mu}} \omega_0 + \iota_{X_{\varphi_{+}}} \omega_- = \\
d\mu -  [Z_{\A}(\mu) + \bar{Y}_{\A}(\bar{\varphi}_+)] \hp dz^{\A} - [\bar{Z}_{\A}(\mu) + Y_{\A}(\varphi_+)] \hp d\bar{z}^{\A} \nonumber
\end{gather}
and since both $u_{\B}$ and $v^{\B}$ are in the kernel of $\bar{Z}_{\A}$ and $\varphi_0$ and $\varphi_+$ are adjacent hyperpotentials, we have
\begin{equation}
\bar{Z}_{\A}(\mu) + Y_{\A}(\varphi_+) = \bar{Z}_{\A}(\varphi_0) + Y_{\A}(\varphi_+) = 0
\end{equation}
leaving us eventually with 
\begin{equation}
\iota_{\raisebox{0pt}[0pt]{\scriptsize $- \overline{X_{\varphi_{\mathrlap{+}}}}$}} \hspace{5pt} \omega_+ + \iota_{X_{\mu}} \omega_0 + \iota_{X_{\varphi_+}} \omega_- = d\mu \rlap{.} \qedhere
\end{equation}
\end{proof}

\smallskip

\noindent This argument makes it clear that the existence on a hyperk\"ahler manifold of a half-bounded chain of hyperpotentials with one of the elements of the form \eqref{varphi0-mu} is an almost sufficient condition for the existence of a trans-rotational symmetry structure. 
%almost sufficiently determines a trans-rotational symmetry structure. 
More precisely, we have

\begin{lemma} \label{trapp}
Let $M$ be a hyperk\"ahler manifold endowed a system of special coordinates anchored at $I_0$ derived from a choice of Darboux trivialization for the canonical 2-form on its twistor space. Suppose that on $M$ there exists a local right-bounded chain of hyperpotentials $\{\varphi_n\}_{n \leq j}$ with respect to $I_0$ such that $\varphi_0 = \mu + i \hp u_{\A} v^{\A}$, with $\mu$ a real-valued function. Then, if we define \mbox{$2j-1$} vector fields $X_n$ by the symplectic gradient formulas \eqref{X_n-sg} and the requirement that $\bar{X}_n = (-)^n X_{-n}$, these will satisfy the moment map equations \eqref{tr-rot-mom-maps} with $\theta_+ = u_{\A} dz^{\A}$, and therefore at least locally the Hamiltonian properties expected from the generators of a trans-rotational symmetry.
\end{lemma}

\noindent For a trans-rotational symmetry to properly exist in this case the vector fields $X_n$ need of course to be defined and satisfy these properties globally.

\subsubsection{}

Having focused hitherto mostly on symmetry generators and hyperpotentials, we close up this section with an observation regarding hyper $(1,1)$ forms. This exploits the fact that knowledge of the hyperk\"ahler complex structures in a specific frame imposes strict algebraic constraints on the form which a hyper $(1,1)$ form can take in that frame. Thus, Lemma \ref{criterion_1} in conjunction with the expressions \eqref{HK-c-strs} for the hyperk\"ahler complex structures gives us promptly the following Ansatz:

\begin{lemma}
A possibly complex-valued 2-form $\sigma$ on $M$ is of hyper $(1,1)$ type if and only if in the coframe of $T_{\mathbb{C}}^*M$ given by $dz^{\A}$, $h^{\A}$ and their complex conjugates it takes the form
\begin{equation} \label{hyp-(1,1)}
\sigma = i \pt a_{\A\B} (dz^{\A} \npt \wedge d\bar{z}^{\B} - h^{\B} \npt \wedge \bar{h}^{\A}) + b_{\A\B} \hp dz^{\A} \npt \wedge \bar{h}^{\B} + c_{\A\B} \hp h^{\A} \npt \wedge d\bar{z}^{\B} 
\end{equation}
for some arbitrary coefficient matrix $a_{\A\B}$ and symmetric coefficient matrices $b_{\A\B}$ and $c_{\A\B}$.
\end{lemma}

\noindent Note for comparison that in this coframe the hyperk\"ahler symplectic forms take in either case studied the form
\begin{equation} \label{HK-quat-fr}
\begin{aligned}
\omega_+ & =  - \, i \pt \phi_{\A\B} \pt h^{\B} \npt \wedge dz^{\A} \\
\omega_0 \hspace{1.9pt} & =  - \, i \pt \phi_{\A\B} \pt (dz^{\A} \npt \wedge d\bar{z}^{\B} + h^{\B} \npt \wedge \bar{h}^{\A}) \rlap{.}
\end{aligned}
\end{equation}

\section{Rotational and trans-rotational symmetries: examples} \label{sec:Trans-rot-Ex}

In this section we will discuss several examples of hyperk\"ahler metrics with trans-rotational symmetry and investigate in detail the associated structures. The first few ones are all going to be of extended Gibbons-Hawking type. In contrast, the last example\,---\,that of the hyperk\"ahler constructions of Gaiotto, Moore and Neitzke\,---\,will not involve any continuous Killing symmetries at all.

\subsection{Rotation-invariant metrics of extended Gibbons-Hawking type} \label{ssec:rot-inv-L} \hfill \medskip

Consider a generic hyperk\"ahler manifold $M$ of extended Gibbons-Hawking type. In the moment map coordinate parametrization of $M$ one has a natural action of the multiplicative group $\mathbb{H}^{\times}$ of non-zero quaternions on $\mathbb{R}^m \otimes \mathbb{R}^{3}$ generated by 
\begin{equation}
\vec{\mathfrak{L}} = - \, \vec{r}^{\,\A} \!\times \vec{\partial}_{\A}
\qquad \text{and} \qquad
\mathfrak{L_0} = \vec{r}^{\,\A} \!\cdot \vec{\partial}_{\A} \rlap{.}
\end{equation}
In general, these vector fields do not induce symmetries of the metric. In particular circumstances, however, such a possibility is not excluded. The narrower question we are interested in here is whether it is possible to lift one of these generators to a genuine rotational symmetry of $M$, and under what conditions. 

Let us focus our inquiry specifically on the generator of collective rotations around the third axis in $\mathbb{R}^3$. In special coordinates, this reads
\begin{equation}
\mathfrak{L}_3 =  - \pt i \Big( z^{\A}\frac{\partial}{\partial {z^{\A}}} - \bar{z}^{\A}\frac{\partial}{\partial {\bar{z}^{\A}}} \Big) \rlap{.}
\end{equation}
Inserting the vector field $\mathfrak{L}_3$ into the expressions \eqref{HK-sym-tor} for the hyperk\"ahler symplectic forms (in our assumed coordinate trivialization it is rather natural to denote the induced vector field on $M$ by the same symbol as the one on $\mathbb{R}^m \otimes \mathbb{R}^{3}$) gives us after a direct calculation the following formulas:
\begin{equation} \label{baloo}
\begin{aligned}
\iota_{\mathfrak{L}_3} \omega_+ & =  d\varphi_+ - i \hp \theta_+ + \alpha_+ \\[1pt]
\iota_{\mathfrak{L}_3} \omega_0 \hspace{2pt}  & =  d\mu + \alpha_0
\end{aligned}
\end{equation}
where we denote $\mu = - \, \Re(z^{\A}L_{z^{\A}})$
%\begin{equation} 
%\mu = - \, \Re(z^{\A}L_{z^{\A}}) 
%\end{equation}
and
\begin{equation} \label{alphas}
\begin{aligned}
\alpha_+ & = \mathfrak{L}_3(\varrho_{\A}) \hp dz^{\A} + \frac{i}{2}\partial_{\psi^{\A}}\mathfrak{L}_3(L) \hp dz^{\A} \\
\alpha_0 \hspace{2pt} & = \mathfrak{L}_3(\varrho_{\A}) \hp d\psi^{\A}  + \frac{i}{2}\partial_{z^{\A}}\mathfrak{L}_3(L) \hp dz^{\A} - \frac{i}{2}\partial_{\bar{z}^{\A}}\mathfrak{L}_3(L) \hp d\bar{z}^{\A} \rlap{.}
\end{aligned}
\end{equation}
Also, corresponding to the choice $\theta_+ = u_{\A} dz^{\A}$, suggested by considerations in the previous two sections, we have $\varphi_+ = i \hp u_{\A}z^{\A}$. 

These formulas make it immediately clear that our question may be answered as follows:

\begin{lemma} \label{L3-inv}
If the function $L$ is such that $\mathfrak{L}_3(L) = 0$ then the vector field
\begin{equation}
X^{\textup{rot}}_0 = \mathfrak{L}_3 - \mathfrak{L}_3(\varrho_{\A}) \frac{\partial}{\partial \psitilde_{\A}}
\end{equation}
generates a rotational symmetry characterized by the Hamiltonian function $\mu$.
\end{lemma}

\noindent Indeed, when this invariance condition holds the equations \eqref{baloo} become
\begin{equation} 
\begin{aligned}
\iota_{X^{\textrm{rot}}_0} \omega_+ & = d\varphi_+ - i\hp\theta_+ \\
\iota_{X^{\textrm{rot}}_0} \omega_0 \hspace{1.9pt} & = d\mu 
\end{aligned}
\end{equation}
and the statement of the lemma is promptly implied. 

\begin{example}
In \cite{MR877637} it was observed that, quite generally, the $L$-potentials of manifolds of extended Gibbons-Hawking type can be naturally derived from a holomorphic function (with possible singularities and branch cuts) through contour integration:
\begin{equation} \label{L_c-int}
L = \frac{1}{2\pi i} \int_{\mathcal{C}} \frac{d\zeta}{\zeta} H(\eta,\zeta)
\end{equation}
for some contour $\mathcal{C}$ invariant under antipodal conjugation so as to ensure a real result for the integral, and where
\begin{equation} \label{O2-eta}
\eta^{\A} = \frac{z^{\A}}{\zeta\,} + \psi^{\A} - \bar{z}^{\A} \zeta
\end{equation}
are the tropical components of the $\mathcal{O}(2)$ sections mentioned at the beginning of \S\,\ref{real-vA}. Indeed, provided that the integral exists and the derivatives commute with the integral, one can easily see that the linear second-order hyperk\"ahler differential constraints \eqref{L-PDEs_1} are automatically satisfied in this case. 
%\textcolor{blue}{In what follows we will refer to this type of construction as an $\mathcal{O}(2)$-based construction.}

\begin{remark}
Such contour integral representations for $L$-potentials have appeared originally in a physics setting in the guise of Lagrangian density functions as  part of an approach to supersymmetric quantum field theories with eight supergenerators which became known as the \textit{projective superspace} approach \cite{Lindstrom:1983rt,Gates:1984nk, MR877637}.
\end{remark}

Observe that the tropical components $\eta^{\A}$ satisfy the following first order partial linear differential equation:
\begin{equation} \label{L3-rot-eta}
\Big(\! - i \hp \zeta \frac{\partial}{\partial \zeta} + \mathfrak{L}_3 \Big) \eta^{\A} = 0 \rlap{.}
\end{equation}
Let us assume now that the function $H$ does not have an explicit dependence on $\zeta$, that is to say, it depends on $\zeta$ only through the $\eta^{\A}$ variables. By commuting a differential operator with the integral and using the above differential equation we can then write successively
\begin{equation}
\mathfrak{L}_3(L) 
= \frac{1}{2\pi i} \int_{\mathcal{C}} \frac{d\zeta}{\zeta} \frac{\partial H(\eta)}{\partial \eta^{\A}} \mathfrak{L}_3(\eta^{\A})
= \frac{1}{2\pi} \int_{\mathcal{C}} d\zeta \frac{d}{d\zeta} H(\eta) = \frac{1}{2\pi} H(\eta) \Big|_{\partial\mathcal{C}} \rlap{.}
\end{equation}
Thus, the condition of Lemma \ref{L3-inv} is satisfied in this case provided that the boundary term vanishes. This is in particular automatic when the integration contour $\mathcal{C}$ is closed. This simple criterion gives us a large class of rotation-invariant extended Gibbons-Hawking metrics. However, not all rotation-invariant extended Gibbons-Hawking metrics are of this type, and we will soon encounter a counterexample (see Lemma \ref{rot-inv-Fs-lemm} later on). 
\end{example}

The existence of such an action gives rise, as shown by Haydys in \cite{MR2394039} and Hitchin in \cite{MR3116317}, to a hyperholomorphic line bundle over $M$ and a corresponding holomorphic line bundle over the twistor space of $M$ endowed with a meromorphic connection with poles of order one on the fibers over $\zeta = 0$ and $\infty$ (more rigorously put, it determines a splitting of the Atiyah-like short exact sequence \eqref{Atiyah_seq} defined in \S\,\ref{ssec:global_iss}). Let us describe now these structures in the case at hand using the theory we have developed. 

For the arctic meromorphic potential $\smash{ \varphi_{V_N}(\zeta)}$ the above expression for $\varphi_+$ suggests we should take
\begin{equation}
\varphi_{V_N}(\zeta) = i \pt \frac{\tilde{\eta}^{\phantom{'}}_{\A,V_N} \eta^{\A}_{\hp V_N}}{\zeta} \rlap{.}
\end{equation}
By design this has a pole of order one at $\zeta = 0$, with residue $\varphi_+$. Furthermore, its Laurent coefficient of order zero is \mbox{$\varphi_0 = i \hp (w_{\A}z^{\A} + u_{\A}\psi^{\A})$}, and so in view of the second formula \eqref{uw-L_2} it is clear that when the invariance condition $\mathfrak{L}_3(L) = 0$ holds we have
\begin{equation} 
\varphi_0 =\mu + i \hp u_{\A}\psi^{\A} \rlap{.}
\end{equation}
That is to say, this meromorphic potential is in the special gauge of Lemma \ref{A_V-sp-gauge} (recall that in the extended Gibbons-Hawking case \mbox{$v^{\A} = \psi^{\A}$}). Meromorphicity implies that the Laurent coefficients of its $\zeta$-expansion form a chain of hyperpotentials with respect to the hyperk\"ahler complex structure labeled by $\zeta=0$. By Lemma \ref{trapp} it follows that if we let 
\begin{equation}
X_+ = X_{\varphi_+}
\qquad\qquad
X_0 = X_{\mu}
\end{equation}
(where we duly assume the definition \eqref{X_f_2} for the fields on the r.h.s.) and also \mbox{$X_- = - \bar{X}_+$}, then these vector fields satisfy the moment map equations
\begin{equation}
\iota_{X_{n-1}}\omega_+ + \iota_{X_{n}}\omega_0 + \iota_{X_{n+1}}\omega_- =
\begin{cases}
d\varphi_+ \!- i \hp \theta_+ & \text{if $n = 1$} \\
d\mu & \text{if $n=0$} \rlap{.}
\end{cases}
\end{equation}
As per our usual convention we take the fields with out-of-range indices to vanish. 

So we seem to have two different actions characterized by the same moment map functions $\mu$ and $\varphi_+$: on one hand, a rotational action generated by the vector field $X^{\textrm{rot}}_0$, and on the other, a twisted rotational action generated by the triplet of vector fields \mbox{$X_+$, $X_0$, $X_-$}. The apparent discrepancy is resolved if we observe that these two sets of generators belong to the same equivalence class. That is, letting $X(\zeta) = X_+\zeta^{-1} \nhp + X_0 + X_-\zeta$, one can verify that we have
\begin{equation}
X(\zeta) = X^{\textrm{rot}}_0 + I(\zeta)Y_0 
\qquad\text{with}\qquad
Y_0 = z^{\A}Z_{\A} + \bar{z}^{\A}\bar{Z}_{\A}
\end{equation}
the vector fields $Z_{\A}$ here being those defined by the second equation \eqref{YZ_2}. Thus, the triple generators provide an equivalent presentation of the \textit{same} purely rotational action. 

Knowledge of the moment map function $\mu$ allows us to compute, by way of formula \eqref{sgm-omg}, the closed hyper $(1,1)$ forms associated to the action. By making use of the linear invariance condition $\smash{ \mathfrak{L}_3(L) = 0 }$ and of the second-order differential hyperk\"ahler constraints \eqref{L-PDEs_1} we can check directly that in this case, as expected from generic arguments, \mbox{$\sigma_+ = \sigma_- = 0$}. On the other hand, in a similar manner we obtain
\begin{align} \label{hyp(1,1)-ri}
\sigma_0 & =  i \hp (\phi_{\A\B} + \phi_{\A\D} \hp \partial_{\psi^{\B}} (\phi^{\D\C}\mu_{\psi^{\C}}) + \phi_{\B\D} \hp \partial_{\psi^{\A}} (\phi^{\D\C}\mu_{\psi^{\C}}))(dz^{\A} \npt \wedge d\bar{z}^{\B} - h^{\B} \npt \wedge \bar{h}^{\A}) \\[1pt]
& \hspace{12pt} -  2 \hp i \pt \phi_{\B\D} \hp \partial_{z^{\A}} (\phi^{\D\C}\mu_{\psi^{\C}}) \pt dz^{\A} \npt \wedge \bar{h}^{\B} 
+ 2 \hp i \pt \phi_{\B\D} \hp \partial_{\bar{z}^{\A}} (\phi^{\D\C}\mu_{\psi^{\C}}) \pt d\bar{z}^{\A} \npt \wedge h^{\B}
\nonumber
\end{align}
where the indices of $\mu$ indicate derivatives and $\smash{ \phi_{\A\B} = - \, L_{\psi^{\A}\psi^{\B}} }$. This is conspicuously of the hyper $(1,1)$ form \eqref{hyp-(1,1)}, with the symmetry of the coefficients of the mixed-type terms following from the first property \eqref{L-PDEs_1}.

\subsection{The c-map} \label{ssec:c-map} \hfill \medskip

Affine special K\"ahler manifolds are K\"ahler manifolds whose geometry is determined by a single holomorphic function (with possible singularities and branch cuts) called \textit{prepotential}. They first emerged in physics, where they were defined in terms of certain special coordinates \cite{Gates:1983py,Sierra:1983cc}, and were subsequently given a coordinate-free geometric description in \cite{MR1695113}. The interesting thing about them, for us, is that their cotangent bundle carries a natural hyperk\"ahler structure called a \textit{c-map}. The corresponding hyperk\"ahler metrics, of extended Gibbons-Hawking type, known for reasons to become soon clear as \textit{semi-flat} metrics, admit a simple, explicit and very instructive twistor construction, which we will now review briefly. 

In the contour-integral description \eqref{L_c-int}, c-maps with holomorphic prepotential $F$ are characterized by the meromorphic function \cite{Gates:1999zv}
\begin{equation}
H^{\hp \text{sf}}(\eta,\zeta) = i \hp  \frac{F(\eta\hp\zeta)}{\zeta^2} - i \hp \zeta^2 \bar{F}(-\eta/\zeta)
\end{equation}
with the integration contour consisting of a counter-clockwise oriented loop around $\zeta=0$ for the first term and  its antipodal projection around $\zeta =\infty$ for the second. The integral can be computed explicitly, yielding the $L$-potential
\begin{equation} 
L^{\text{sf}} = 2 \hp \Im (\bar{z}^{\A}F_{\A}) - \Im F_{\A\B} \pt \psi^{\A}\psi^{\B}
\end{equation}
where $F_{\A}$, $F_{\A\B}$ denote derivatives of the prepotential $F(z)$ with respect to the $z^{\A}$ variables. If in equation \eqref{u-tor} we choose the shifts $\smash{ \varrho_{\A}^{\hp \text{sf}} = - \hp \Re F_{\A\B} \hp \psi^{\B} }$ (note that this choice satisfies the constraints \eqref{shift-constr}) we get $\smash{ u_{\A} = \tilde{\psi}_{\A} - F_{\A\B} \psi^{\B} }$. From this, by way of the second equalities \eqref{h_A_2} and \eqref{phi_AB_2} we obtain the coframe element $\smash{ h_{\A} = d\tilde{\psi}_{\A} - F_{\A\B} \hp d\psi^{\B} }$ and the bilinear form $\smash{ \phi_{\A\B} = 2 \hp \Im F_{\A\B} }$. In terms of these, the formulas  \eqref{HK_sfs} and \eqref{HK_m} give us then the hyperk\"ahler symplectic forms and metric explicitly (recall that these formulas hold all the same in the extended Gibbons-Hawking case). The emerging structure is that of a flat torus fibration over an affine special K\"ahler base, which is why one calls the metric \textit{semi-flat}.

The hyperk\"ahler symplectic forms can be equivalently recast in the form
\begin{equation} \label{HK_sym_sf}
\begin{aligned}
\omega_+ & = d\tilde{\psi}_{\A} \wedge dz^{\A} + dF_{\A} \npt\wedge d\psi^{\A} \\
\omega_0 \hspace{2pt} & = dz^{\A} \wedge d\bar{F}_{\A} + d\bar{z}^{\A} \wedge dF_{\A} + d\tilde{\psi}_{\A} \wedge d\psi^{\A} 
\end{aligned}
\end{equation} 
Semi-flat metrics have double the number of continuous tri-Hamiltonian symmetries compared to regular extended Gibbons-Hawking metrics of same dimension: to the usual invariance at shifts in the coordinates $\smash{ \tilde{\psi}_{\A} }$ one can now add invariance at shifts in the coordinates $\smash{ \psi^{\A} }$. The symmetry doubling is due to the exceptional presence of a discrete \textit{duality symmetry}. The above hyperk\"ahler symplectic forms remain invariant if we replace the prepotential $F(z)$ and the variables $z^{\A}$ with the (flipped-sign) Legendre transformed prepotential $\smash{ \tilde{F}(\tilde{z}) = F(z) - z^{\A}\tilde{z}_{\A} }$ and the conjugate variables $\tilde{z}^{\A} = F_{\A}$, while at the same time taking $\smash{ \psi^{\A} \mapsto \tilde{\psi}_{\A} }$ and $\smash{ \tilde{\psi}_{\A} \mapsto - \psi^{\A} }$. 

This dual symmetry structure is more apparent on the twistor space where the formulas \eqref{HK_sym_sf} take the compact Darboux form
\begin{equation}
\omega(\zeta) = d\tilde{\eta}^{\phantom{'}}_{\A} \wedge d\eta^{\A}
\end{equation}
with $\eta^{\A}$ given by equation \eqref{O2-eta} and
\begin{equation} \label{etatilde-sf}
\tilde{\eta}_{\A} = \frac{F_{\A}}{\zeta} + \tilde{\psi}_{\A} - \bar{F}_{\A} \zeta \rlap{.}
\end{equation}
The twistor space can be in this case covered by as few as two open sets, $V_N$ and $V_S$, whose projections on the twistor sphere exclude the polar points $\zeta = \infty$ and $\zeta = 0$, respectively.  Each comes with a chart of complex canonical coordinates for its corresponding component of the $\mathcal{O}(2)$-twisted canonical 2-form section. But the canonical coordinates for $\omega(\zeta)$, which by the observation above are $\tilde{\eta}_{\A}$ and $\eta^{\A}$, belong to neither\,---\,rather, they live naturally on the tropical region $V_T = V_N \cap V_S$ whose projection on the twistor sphere does not contain either polar point. If we refine the minimal two-set cover by including in it the intersection $V_T$, then the transition between, say, $V_N$ and $V_T$ is given by the appropriately twisted symplectomorphism
\begin{equation} \label{N-T-trans}
\begin{aligned}
\eta^{\A} & = \frac{\eta^{\A}_{\hp V_N}}{\zeta \, } \\[-5pt]
\tilde{\eta}_{\A} & = \tilde{\eta}_{\A,V_N} + \frac{F_{\A}(\eta_{\hp V_N})}{\zeta}  \rlap{.}
\end{aligned}
\end{equation}
Note that this form ensures in particular that the canonical coordinates on $V_N$, $\tilde{\eta}_{\A,V_N}$ and $\eta^{\A}_{\hp V_N}$, have no singularities at $\zeta = 0$. The corresponding formulas for the transition between $V_S$ and $V_T$ follow by antipodal conjugation.

Let us show now that in addition to the symmetries that we have already counted, semi-flat metrics admit as well a twisted rotational symmetry. As in the first example, we will follow a two-pronged approach to this problem. Let $f(z) = z^{\A}F_{\A}(z) - 2F(z)$ be a function measuring the departure of the prepotential from being a degree-two homogeneous function of its variables, and consider the three vector fields
\begin{equation} \label{X-rot-sf}
X^{\text{rot}}_+ = i f_{\A} \frac{\partial}{\partial \psitilde_{\A}}
\quad\qquad
X^{\text{rot}}_0 = \mathfrak{L}_3
\quad\qquad
X^{\text{rot}}_- = i \bar{f}_{\A} \frac{\partial}{\partial \psitilde_{\A}}
\end{equation}
where $f_{\A}$ denotes the derivative of $f(z)$ with respect to $z^{\A}$. One can then check by a direct computation that these satisfy the following moment map equations
\begin{equation} \label{X-rot-sf-mm}
\iota_{X^{\text{rot}}_{n-1}}\omega_+ + \iota_{X^{\text{rot}}_{n}}\omega_0 + \iota_{X^{\text{rot}}_{n+1}}\omega_- =
\begin{cases}
d\varphi_{++}^{\text{sf}}  & \text{if $n = 2$} \\
d\varphi_+^{\text{sf}} - i \hp \theta_+ & \text{if $n = 1$} \\
d\mu^{\text{sf}} & \text{if $n=0$} 
\end{cases}
\end{equation}
with potentials
\begin{equation}
\varphi_{++}^{\text{sf}} = i f(z),
\quad\quad
\varphi_+^{\text{sf}} = i \hp (z^{\A}\tilde{\psi}_{\A} - F_{\A}\psi^{\A})
\qquad\text{and}\qquad
\mu^{\text{sf}} = - \hp 2 \hp \Im (\bar{z}^{\A}F_{\A}) \rlap{.}
\end{equation}
The second potential is predicated again on choosing $\theta_+ = u_{\A} dz^{\A}$. Alternatively, we can arrive at these equations in a somewhat more roundabout\,---\,but useful later on\,---\,way by resorting to the formulas \eqref{baloo}. In particular, we get
\begin{equation} \label{alphas_sf}
\begin{aligned}
\alpha_+^{\hp \text{sf}} & = i \hp \psi^{\A} \hp df_{\A}  \\
\alpha_0^{\hp \text{sf}} & = i f_{\A} \hp d\bar{z}^{\A} - i \bar{f}_{\A} \hp dz^{\A} + d [\Im(\bar{z}^{\A}f_{\A}) - \frac{1}{2} \Im f_{\A\B} \pt \psi^{\A}\psi^{\B} ] \rlap{.}
\end{aligned}
\end{equation}
In this approach the last two potentials emerge in the form
\begin{equation}
\varphi_+^{\text{sf}} = i \hp ( f_{\A} \psi^{\A} +  u_{\A}z^{\A} )
\quad\text{and}\quad
\mu^{\text{sf}} = - \pt \Re(z^{\A}L^{\text{sf}}_{z^{\A}}) + \Im(\bar{z}^{\A}f_{\A}) - \frac{1}{2} \Im f_{\A\B} \pt \psi^{\A}\psi^{\B}
\end{equation}
which, one can easily verify, is equivalent to the above one.

The latter set of expressions for the $\varphi\pt$-potentials suggest we should pick
\begin{equation} \label{arct-mer-pot-sf}
\varphi_{V_N}(\zeta) = i \pt \frac{f(\eta_{\hp V_N})}{\zeta^2} + i \pt \frac{\tilde{\eta}^{\phantom{'}}_{\A,V_N} \eta^{\A}_{\hp V_N}}{\zeta} 
\end{equation}
as arctic meromorphic potential associated to this twisted action. This has a pole of order two at $\zeta=0$, with $\smash{ \varphi_+^{\text{sf}} }$ and $\smash{ \varphi_{++}^{\text{sf}} }$ as residues of order one and two, respectively. Moreover, the next term in its $\zeta$-expansion may be shown to be of the form $\smash{ \varphi_0^{\text{sf}} = \mu^{\text{sf}} + i\hp u_{\A}\psi^{\A} }$. That is, this meromorphic potential is in the special gauge of Lemma \ref{A_V-sp-gauge}, which allows us to apply the powerful Lemma \ref{trapp}. Thus, if we take the vector fields
\begin{equation} 
X_{++} \npt = X_{\smash{ \varphi_{++}^{\hp \text{sf}} }} \!= 0
\qquad
X_+ \npt = X_{\smash{ \varphi_+^{\hp \text{sf}} }} \! \nhp = - \hp i \Big(z^{\A}\frac{\partial}{\partial \psi^{\A}} + F_{\A} \frac{\partial}{\partial \psitilde_{\A}}\Big)
\qquad
X_0 = X_{\smash{ \mu^{\text{sf}} }} = 0
\end{equation}
where we assume the definition \eqref{X_f_2} for the symplectic gradients $X_f$, and also define negative-indexed ones by alternating conjugation, then these satisfy the same moment map equations
\begin{equation} \label{X-sf-mm}
\iota_{X_{n-1}}\omega_+ + \iota_{X_{n}}\omega_0 + \iota_{X_{n+1}}\omega_- =
\begin{cases}
d\varphi_{++}^{\text{sf}}  & \text{if $n = 2$} \\
d\varphi_+^{\text{sf}} - i \hp \theta_+ & \text{if $n = 1$} \\
d\mu^{\text{sf}} & \text{if $n=0$} 
\end{cases}
\end{equation}
with the same potentials as above. The coincidence between the potentials for the two sets of generating vector fields means that they are actually different presentations of one and the same action. Indeed, if we assemble in the usual way the graded components into the twisted vector fields $X^{\textrm{rot}}(\zeta)$ and $X(\zeta)$, we have
\begin{equation} \label{X-twist-equiv}
X(\zeta) = X^{\textrm{rot}}(\zeta) + I(\zeta)Y_0 
\qquad\text{with}\qquad
Y_0 = z^{\A}Z_{\A} + \bar{z}^{\A}\bar{Z}_{\A}
\end{equation}
where in this case $Z_{\A}$ is simply $\partial_{z^{\A}}$. 

In fact, as first observed by Gaiotto, Moore, Neitzke and Pioline (see footnote 8 of \cite{MR2672801}), an even stronger result holds in this case: all the tropical canonical coordinates satisfy the same first order partial linear differential equation, namely
{\allowdisplaybreaks
\begin{equation} \label{GMN_diff_eqs}
\begin{aligned}
& \Big(\! - i \hp \zeta \frac{\partial}{\partial \zeta} + X(\zeta) \Big) \eta^{\A} = 0 \\
& \Big(\! - i \hp \zeta \frac{\partial}{\partial \zeta} + X(\zeta) \Big) \tilde{\eta}_{\A} = 0 \rlap{.}
\end{aligned}
\end{equation}
}%
Indeed, note that this by itself is a sufficient condition for a twisted action with twisted generators $X(\zeta)$ to exist. Holomorphicity ensures that the equations remain valid if one replaces $X(\zeta)$ with $\smash{ X^{\text{rot}}(\zeta) }$, as well as with any other vector field from its equivalence class. In particular, one can thus see that the first differential equation is equivalent to the differential equation \eqref{L3-rot-eta}. 

Using the formula \eqref{sgm-omg} we can compute explicitly the closed hyper $(1,1)$ forms associated to the twisted action. We obtain
{\allowdisplaybreaks
\begin{equation} \label{hyp(1,1)-sf}
\begin{aligned}
\sigma_+^{\hp \text{sf}} & = - \, d \hp (\hp f_{\A}\bar{h}^{\A})  \\
& = - \, \phi_{\B\D} \hp \partial_{z^{\A}} (\phi^{\D\C}f_{\C}) \pt dz^{\A} \npt \wedge \bar{h}^{\B} - \phi_{\A\D} \hp \partial_{\bar{z}^{\B}} (\phi^{\D\C}f_{\C}) \pt h^{\A} \npt \wedge d\bar{z}^{\B} \\[2pt]
\sigma_0^{\hp \text{sf}} & = i \hp \phi_{\A\B} \hp (\hp dz^{\A} \npt \wedge d\bar{z}^{\B} - h^{\B} \npt \wedge \bar{h}^{\A}) 
\end{aligned}
\end{equation}
}%
where, recall, \mbox{$h^{\A} = i \hp \phi^{\A\B} h_{\B}$}. These are evidently of the hyper $(1,1)$ form \eqref{hyp-(1,1)}, and it is a simple exercise to see that the coefficients of the mixed-type terms of $\smash{ \sigma_+^{\hp \text{sf}} }$ are symmetric, as required. Note, moreover, that $\smash{ \sigma_+^{\hp \text{sf}} }$ is an exact form, in agreement with the general prescription of Corollary \ref{transv-triv}. 

By the main results of section \ref{sec:Trans-rot} the twisted action gives rise to a hyperholomorphic line bundle over the hyperk\"ahler manifold endowed with a hyperhermitian connection with curvature 2-form $\smash{ \sigma_0^{\hp \text{sf}} }$. Corresponding to this, on the twistor space $\mathcal{Z}$ one has a holomorphic line bundle endowed with a meromorphic connection with poles of order two on the fibers over $\zeta = 0$ and $\infty$. For the minimal cover of $\mathcal{Z}$ given by $V_N$ and $V_S$ explicit connection 1-forms are provided in this case by Lemma \ref{A_V-sp-gauge}:
{\allowdisplaybreaks
\begin{equation}
\begin{aligned}
\mathcal{A}^{\phantom{'}}_{V_N} & = \frac{1}{\zeta} \pt \tilde{\eta}^{\phantom{'}}_{\A,V_N} d^{\phantom{'}}_{\mathcal{Z}} \eta^{\A}_{\hp V_N} + i\hp d^{\phantom{'}}_{\mathcal{Z}}\varphi^{\phantom{'}}_{V_N}(\zeta) + i \pt \varphi^{\phantom{'}}_{V_N}(\zeta) \frac{d\zeta}{\zeta} \\
\mathcal{A}^{\phantom{'}}_{V_S} \, & = \, \zeta \pt \, \tilde{\eta}^{\phantom{'}}_{\A,V_S} \, d^{\phantom{'}}_{\mathcal{Z}} \eta^{\A}_{\hp V_S} + i\hp d^{\phantom{'}}_{\mathcal{Z}}\varphi^{\phantom{'}}_{V_S}(\zeta) \, - \, i \pt \varphi^{\phantom{'}}_{V_S}(\zeta) \frac{d\zeta}{\zeta} \rlap{.}
\end{aligned}
\end{equation}
}%
A corresponding gluing function is
\begin{equation}
\phi_{V_SV_N}(\zeta) = - \pt H^{\hp \text{sf}}(\eta,\zeta) \rlap{.}
\end{equation}
If we consider the refined cover instead then we must supplement these two connection 1-forms with
\begin{equation}
\mathcal{A}_{V_T} = \tilde{\eta}^{\phantom{'}}_{\A} d^{\phantom{'}}_{\mathcal{Z}} \eta^{\A}
\end{equation}
and the gluing function with 
\begin{equation}
\begin{aligned}
\phi_{V_TV_N}(\zeta) & = \varphi_{V_N}(\zeta) +  i \hp \frac{F(\eta\hp\zeta)}{\zeta^2} \\
\phi_{V_TV_S}(\zeta) \hspace{1.5pt} & = \varphi_{V_S}(\zeta) \hspace{1.5pt} + i \hp \zeta^2 \bar{F}(-\eta/\zeta)
\end{aligned}
\end{equation}
which one can derive based on the twisted symplectic transition rules \eqref{N-T-trans} and their antipodal conjugates. 

As we have seen in the preceding example, because generators come in equivalence classes purely rotational actions admit descriptions which at first sight might seem to belong to the realm of twisted actions. It is of interest therefore to ask for what class of prepotentials the twisted action is truly twisted or simply equivalent to a purely rotational action. Concerning this question we have the following result: 

\begin{lemma} \label{rot-inv-Fs-lemm}
We have $\smash{ \mathfrak{L}_3(L^{\textup{sf}}) = 0 }$ and so a purely rotational action if and only if the prepotential is of the form 
\begin{equation} \label{rot-inv-Fs}
F(z) = i \hp c_{\A\B} z^{\A}z^{\B} \ln (d_{\C\D} z^{\C}z^{\D}) + b_{\A\B} z^{\A}z^{\B} + c
\end{equation}
for any constant $c$ and symmetric constants $b_{\A\B}$, $c_{\A\B}$, $d_{\A\B}$ such that $c_{\A\B}$ is real-valued.
\end{lemma}

\noindent In other words,  the rotational-invariant class includes prepotentials homogeneous of degree two and, rather less obviously, prepotentials known in physics as of perturbative or one-loop Seiberg-Witten type. Note that the latter give rise to singular metrics. 

\begin{proof}

Begin by observing that we may write
\begin{equation} \label{L3L^sf}
\mathfrak{L}_3(L^{\text{sf}}) = -\, 2 \hp \Re (\bar{z}^{\A}f_{\A}) + \Re f_{\A\B} \pt \psi^{\A}\psi^{\B} \rlap{.}
\end{equation}
Let us remark in a parenthesis that, interestingly, the dependence on the function $f(z)$ can be understood as a vestige of a differential property of the meromorphic function $H^{\text{sf}}(\eta,\zeta)$. Using the equation \eqref{L3-rot-eta} one can see that this satisfies the inhomogeneous first order linear differential equation
\begin{equation}
\Big(\! - i \hp \zeta \frac{\partial}{\partial \zeta} + \mathfrak{L}_3 \Big) H^{\text{sf}} = \frac{f(\eta \hp \zeta)}{\zeta^2} + \text{a.c.}
\end{equation}
Integrating both sides on the closed contour specific to c-maps we have then
\begin{equation}
\mathfrak{L}_3(L^{\text{sf}}) = \frac{1}{2\pi \hp i} \oint_{\mathcal{C}} \frac{d\zeta}{\zeta} \mathfrak{L}_3(H^{\text{sf}}) = \frac{1}{2\pi \hp i} \oint_{\mathcal{C}} \frac{d\zeta}{\zeta} \Big( \frac{f(\eta \hp \zeta)}{\zeta^2} + \text{a.c.} \Big)
\end{equation}
from which the formula \eqref{L3L^sf} follows immediately by way of Cauchy's residue theorem. 

Returning to our problem, it is clear that the necessary and sufficient condition for $\smash{ \mathfrak{L}_3(L^{\text{sf}}) }$ to vanish is that both $\Re (\bar{z}^{\A}f_{\A})$ and $ \Re f_{\A\B}$ vanish. Exploiting holomorphicity, the second condition implies that there must exist a symmetric real-valued constant matrix $c_{\A\B}$ such that $\smash{ f_{\A\B} = 4 \hp i \pt c_{\A\B} }$, with the numerical factor chosen for ulterior convenience. Then $\smash{ f(z) = 2 \hp i \pt c_{\A\B} z^{\A}z^{\B} + c_{\A} z^{\A} - 2\hp c }$ for some complex integration constants $c_{\A}$ and $c$. Inserting this expression into the first condition yields the additional constraint $\smash{ \Re(c_{\A} \bar{z}^{\A}) = 0 }$, which is satisfied if and only if all the constants $c_{\A}$ vanish. Finally, replacing $f(z)$ with its definition leads to a first order differential equation for the prepotential $F(z)$, with the generic solution \eqref{rot-inv-Fs}. 
\end{proof}

We end this section with a remark concerning the symbolic representation of the symplectic invariance which underlies this construction. In order to render symplectic invariance more manifest it is convenient to organize the various quantities we work with into symplectic pairs, for example $\smash{  (\tilde{z}_{\A},z^{\A}) \equiv Z_{\gamma^a} }$, $\smash{ (\tilde{\psi}_{\A},\psi^{\A}) \equiv \psi_{\gamma^a} }$,  $\smash{ (\tilde{\eta}_{\A},\eta^{\A}) \equiv \eta_{\gamma^a}}$. Thus, with this notation, the two formulas \eqref{O2-eta} and \eqref{etatilde-sf} may be concisely condensed into the single formula
\begin{equation} \label{eta^sf}
\eta^{\hp \text{sf}}_{\gamma^a} = \frac{Z_{\gamma^a}}{\zeta} + \psi_{\gamma^a} - \bar{Z}_{\gamma^a} \zeta \rlap{.}
\end{equation}
To be able to rewrite some of the formulas above in a manifestly invariant form we need to introduce, moreover, a constant anti-symmetric symplectic metric $\varepsilon_{ab}$ expressed in a corresponding symplectic frame. We then have
\begin{equation}
\omega(\zeta) = \frac{1}{2} \pt \varepsilon_{ab} \, d\eta^{\hp \text{sf}}_{\gamma^a} \npt \wedge d\eta^{\hp \text{sf}}_{\gamma^{\smash b}},
\qquad\quad
\varphi_+^{\hp \text{sf}} = i \pt \varepsilon_{ab} \pt \psi_{\gamma^{\smash a}} Z_{\gamma^{\smash b}},
\qquad\quad
\mu^{\text{sf}} = i \pt \varepsilon_{ab} \pt Z_{\gamma^{\smash a}} \bar{Z}_{\gamma^{\smash b}}
\end{equation}
and so on.

\subsection{Rotation-invariant deformations of c-maps} \hfill \medskip

Hyperk\"ahler metrics of extended Gibbons-Hawking type enjoy a powerful yet highly specific property: they can be linearly superposed to obtain new metrics of the same type. More precisely, owing to the linearity of the equations \eqref{L-PDEs_1}, given two $L$-potentials corresponding to two extended Gibbons-Hawking metrics, any linear combination of them is automatically an $L$-potential for another extended Gibbons-Hawking metric. We will now take advantage of this feature to construct rotation-invariant deformations of c-map metrics. Thus, we consider $L$-potentials of the form $\smash{ L = L^{\text{sf}} + L^{\text{ri}} }$ such that $\smash{ \mathfrak{L}_3 (L^{\text{ri}}) = 0 }$. 

Let us choose for convenience non-vanishing shifts $\smash{ \varrho^{\phantom{g}}_{\A} = \varrho_{\A}^{\hp \text{sf}} }$, with the latter defined as above and $\smash{ \varrho_{\A}^{\hp \text{ri}} = 0 }$. We claim that a number of formulas from the c-map case still hold, \textit{mutatis mutandis}, in this case. Thus, if we consider precisely the same set of vector fields $\smash{ X^{\text{rot}}_n }$ defined by the formulas \eqref{X-rot-sf}, then these satisfy the equivalent of the moment map equations \eqref{X-rot-sf-mm}, with potentials
{\allowdisplaybreaks
\begin{equation}
\begin{aligned}
\varphi_{++} & = i f(z) \\[2.5pt]
\varphi_+ & = i \hp ( f_{\A} \psi^{\A} +  u_{\A}z^{\A} ) \\[-3pt]
\mu & = - \pt \Re (z^{\A}L_{z^{\A}}) + \Im(\bar{z}^{\A}f_{\A}) - \frac{1}{2} \Im f_{\A\B} \pt \psi^{\A}\psi^{\B} \rlap{.}
\end{aligned}
\end{equation}
}%
The form of the second potential is premised as usual on the choice $\theta_+ = u_{\A}dz^{\A}$. This can be easily seen if we notice that in the equations \eqref{alphas} linearity implies a split $\smash{ \alpha^{\phantom{g}}_m = \alpha^{\hp \text{sf}}_m + \alpha^{\text{ri}}_m }$. Due to our assumptions we have $\smash{ \alpha^{\text{ri}}_m = 0 }$ and then using the formulas \eqref{alphas_sf} the result follows from essentially the same argument as in the c-map case. Further manipulations give us for the potentials the alternative forms
{\allowdisplaybreaks
\begin{equation} \label{pot-sf+inv}
\begin{aligned}
\varphi_{++}^{\phantom{f}} & = \varphi_{++}^{\text{sf}} \\[2pt]
\varphi_{+}^{\phantom{f}} & = \varphi_{+}^{\text{sf}} - \frac{1}{2} z^{\A} L^{\text{ri}}_{\psi^{\A}} \\[-1pt]
\mu & = \mu^{\text{sf}} - \frac{1}{2}(z^{\A}L^{\text{ri}}_{z^{\A}} + \bar{z}^{\A}L^{\text{ri}}_{\bar{z}^{\A}}) \rlap{.}
\end{aligned}
\end{equation}
}%

For the arctic meromorphic potential  the same formula \eqref{arct-mer-pot-sf}
%\begin{equation}
%\varphi_{V_N}(\zeta) = i \pt \frac{f(\eta_{\hp V_N})}{\zeta^2} + i \pt \frac{\tilde{\eta}^{\phantom{'}}_{\A,V_N} \eta^{\A}_{\hp V_N}}{\zeta} 
%\end{equation}
gives in this case too the correct residues of both order two and order one at $\zeta = 0$ and, moreover, a zero-order term conveniently of the form $\smash{ \varphi_0 = \mu + i\hp u_{\A}\psi^{\A} }$. We can then apply Lemma \ref{trapp}, by which it follows that the vector fields
\begin{equation} \label{twisted-Xs}
X_{++}  = X_{\varphi_{++}} = 0
\qquad\quad
X_+ = X_{\varphi_+}
\qquad\quad
X_0 = X_{\mu} 
\end{equation}
and their alternating conjugates satisfy the present equivalent of the moment map equations \eqref{X-sf-mm}, with the same set of potentials as above. The equivalence between the two sets of generators reads formally the same as the c-map relation \eqref{X-twist-equiv}.

However, there are also some differences with respect to the c-map case. Based on either formula for the moment map function $\mu$, using the relation \eqref{sgm-omg} we obtain now for the closed hyper $(1,1)$ forms associated to the twisted action the more complicated expressions
{\allowdisplaybreaks
\begin{align}
\sigma_+ & =  - \, d \hp (\hp f_{\A}\bar{h}^{\A}) \\
& = \phi_{\B\D} \hp \partial_{\psi^{\A}} (\phi^{\D\C}f_{\C}) \hp (\hp dz^{\A} \npt \wedge d\bar{z}^{\B} - h^{\B} \npt \wedge \bar{h}^{\A}) \nonumber \\
& \hspace{12pt} - \, \phi_{\B\D} \hp \partial_{z^{\A}} (\phi^{\D\C}f_{\C}) \pt dz^{\A} \npt \wedge \bar{h}^{\B} - \phi_{\A\D} \hp \partial_{\bar{z}^{\B}} (\phi^{\D\C}f_{\C}) \pt h^{\A} \npt \wedge d\bar{z}^{\B} \nonumber \\[2pt]
\sigma_0 \hspace{1.9pt} & =  i \hp (\phi_{\A\B} + \phi_{\A\D} \hp \partial_{\psi^{\B}} (\phi^{\D\C}\mu_{\psi^{\C}}) + \phi_{\B\D} \hp \partial_{\psi^{\A}} (\phi^{\D\C}\mu_{\psi^{\C}}))(dz^{\A} \npt \wedge d\bar{z}^{\B} - h^{\B} \npt \wedge \bar{h}^{\A}) \\[1pt]
& \hspace{12pt} -  2 \hp i \pt \phi_{\B\D} \hp \partial_{z^{\A}} (\phi^{\D\C}\mu_{\psi^{\C}}) \pt dz^{\A} \npt \wedge \bar{h}^{\B} + 2 \hp i \pt \phi_{\B\D} \hp \partial_{\bar{z}^{\A}} (\phi^{\D\C}\mu_{\psi^{\C}}) \pt d\bar{z}^{\A} \npt \wedge h^{\B}  \nonumber
\end{align}
}%
with $\smash{ \phi_{\A\B}^{\phantom{f}} = - \pt L_{\psi^{\A}\psi^{B}} = \phi_{\A\B}^{\pt \text{sf}} + \phi_{\A\B}^{\pt \text{ri}} }$. These are as expected of the form \eqref{hyp-(1,1)}.  In this case we could not rely in our calculations on linearity in a straightforward way. Nevertheless, notice that the result coincides with both the formula \eqref{hyp(1,1)-ri} from the rotation-invariant case and the formulas \eqref{hyp(1,1)-sf} from the semi-flat case in the limit when $\smash{ L^{\text{sf}} }$ respectively $\smash{ L^{\text{ri}} }$ goes to zero. The above expressions can be recast in a more invariant form by noticing that $\smash{ f_{\C} = -\hp i \hp Z_{\C}(\varphi_{++}) = - \hp i \hp \bar{Y}_{\C}(\varphi_+) }$ (see Lemma \ref{YZ-hyperpot}), and also that the remaining derivatives with respect to $z^{\A}$, $\bar{z}^{\A}$ and $\psi^{\A}$ can be replaced by $Z_{\A}$, $\bar{Z}_{\A}$ and $Y_{\A}$ or $\bar{Y}_{\A}$ respectively (for the relevant definitions refer to \S\,\ref{real-vA}).

\subsection{The Ooguri\hp-Vafa metric} \label{ssec:O-V} \hfill \medskip

We will now discuss the Ooguri-Vafa metric as an example of this construction \cite{Ooguri:1996me,MR2672801}. This is a non-compact four-dimensional hyperk\"ahler metric of Gibbons-Hawking type characterized by the following $L$-potential:
\begin{equation}
L = \frac{1}{2\pi i} \oint_{\mathcal{C}} \frac{d\zeta}{\zeta} \Big( i \hp  \frac{F(q \hp \eta \hp \zeta)}{\zeta^2} + \text{a.c.} \nhp \Big) 
- \frac{1}{2\pi} \npt \int_{\ell_+} \! \frac{d\zeta}{\zeta} \Li_2(e^{i \hp q \hp \eta})
- \frac{1}{2\pi} \npt \int_{\ell_-} \! \frac{d\zeta}{\zeta} \Li_2(e^{- i \hp q \hp \eta}) \rlap{.}
\end{equation}
In this formula $q$ is a positive integer,
\begin{equation}
\eta = \frac{z}{\zeta} + \psi - \bar{z}\hp \zeta
\end{equation}
is the tropical component of the only $\mathcal{O}(2)$ section of type \eqref{O2-eta} one has in quaternionic dimension one, the prepotential is of the form
\begin{equation}
F(z) = \frac{i}{8\pi} \hp z^2 (\pt \ln z^2 - 3 \pt )
\end{equation}
belonging to the special rotation-invariant class \eqref{rot-inv-Fs}, the integration contour $\mathcal{C}$ is the standard c-map one, and $\ell_+$, $\ell_-$ are two oppositely-oriented rays emerging from the origin of the complex $\zeta$-plane defined by
\begin{equation}
\ell_{\pm} = \{ \zeta \in \mathbb{C}^{\times} \ | \ \frac{qz}{\zeta} = \pm \hp i \pt \mathbb{R}_+ \} \rlap{.}
\end{equation}
As the monodromy of the dilogarithm can affect $L$ by terms at most linear in $\psi$ and so does not affect the metric we may assume that the dilogarithm function is represented by its principal branch. The non-semi-flat part of $L$ consisting of the last two integrals will be denoted here by $\smash{ L^{\text{inst}} }$ rather than $\smash{ L^{\text{ri}} }$, since in the original physics context one can interpret it as being due to mutually local BPS instanton corrections to the semi-flat metric. Its rotation invariance follows immediately by the criterion laid out in the Example in \S\,\ref{ssec:rot-inv-L}. 

The semi-flat part of $L$ has been evaluated explicitly in the general case in \S\,\ref{ssec:c-map}. To evaluate the instanton part we substitute into the two integrals \mbox{$\smash{ \zeta = - i \hp s \hp z/|z| }$} and \mbox{$\smash{ \zeta = i \hp s \hp z/|z| }$} respectively, with \mbox{$s \in (0,\infty)$}, and then use the standard power series expansion of the dilogarithm to obtain
\begin{equation}
L^{\text{inst}} = - \pt \frac{2}{\pi} \sum_{k=1}^{\infty} \frac{1}{k^2} \pt K_0(2\hp k\hp |qz|) \cos(k\hp q\hp \psi)
\end{equation}
with $K_0$ a modified Bessel function of second kind. Putting together the two parts we get from the first formula \eqref{LT_eqs} the Gibbons-Hawking potential
\begin{equation}
\Phi = \frac{q^2}{4\pi} \big[ \ln |qz|^2 - 4 \sum_{k=1}^{\infty} K_0(2\hp k\hp |qz|) \cos(k\hp q\hp \psi) \hp \big] \rlap{.}
\end{equation}
This can be Poisson-resummed to yield the alternative and rather more revealing expression
\begin{equation}
\Phi = - \pt \frac{q^2}{2\pt} \sum_{n=-\infty}^{\infty} \bigg(\frac{1}{\sqrt{(q\psi + 2\pi n)^2 + 4|qz|^2}} - c_{|n|} \bigg)
\end{equation}
where the $c_n$ represent a sequence of constants given by \mbox{$c_n = 1/(2\pi n)$} for \mbox{$n \neq 0$}. Thus, the potential is a harmonic function with $q$ singularities in $\mathbb{R}^2 \times S^1$. 

The mutually local BPS instanton corrections break the continuous dual tri-Hamiltonian symmetry of the semi-flat metric down to a discrete $\mathbb{Z}$-action, but improve in the process its singular behavior. The singularity at \mbox{$z=0$} in the semi-flat metric is replaced by conical singularities of $A_{q-1}$ type at each point with \mbox{$z=0$} and \mbox{$\psi = - \hp 2\pi n/q$} in the instanton-corrected metric. In the simplest case with \mbox{$q=1$} the singularity is in fact completely smoothed out. 

What is not broken by the instanton corrections is the Gibbons-Hawking tri-Hamiltonian symmetry and, more importantly for us, a rotational symmetry. The fact that the original and the corrected metrics both have a rotational symmetry is evident at the level of their $L$-potentials: $\smash{ L^{\text{sf}} }$ is invariant by Lemma \ref{rot-inv-Fs-lemm} and $\smash{ L^{\text{inst}} }$ by the argument put forth above. The rotational invariance of the Ooguri-Vafa metric is also apparent at the level of the Gibbons-Hawking potential, with either formula above manifestly invariant at $\smash{ z \mapsto e^{i \hp \gamma} z}$ for any angle \mbox{$\gamma \in [\hp 0,2\pi)$}. In \S\,\ref{ssec:rot-inv-L} we saw that rotational symmetries admit rather unusual but nevertheless perfectly equivalent descriptions in terms of twisted generators with multiple components. In this case just such a generator is the one with components $X_+$, $X_0$ and \mbox{$X_- = - \bar{X}_+$} defined by the symplectic gradient formulas \eqref{twisted-Xs} in terms of the potentials $\varphi_+$ and $\mu$. 

\begin{remark}
One can check as an exercise that our symplectic gradient formulas for these vector fields reproduce, up to convention-harmonizing redefinitions, the corresponding formulas in \cite{MR2672801}, more precisely the three component fields in the equation (4.70), defined piecewise by the formulas (4.63) through (4.65) and (4.68). 
\end{remark}

Using the expressions \eqref{pot-sf+inv} in conjunction with the integral definition of \mbox{$\smash{ L^{\text{ri}} = L^{\text{inst}} }$} gives us immediately the following contour integral representations for the two potentials:
{\allowdisplaybreaks
\begin{gather}
\varphi_+ = \varphi_+^{\text{sf}} 
+ \frac{qz}{4\pi i} \npt \int_{\ell_+} \! \frac{d\zeta}{\zeta} \ln (1 - e^{i \hp q \hp \eta})
- \frac{qz}{4\pi i} \npt \int_{\ell_-} \! \frac{d\zeta}{\zeta} \ln (1 - e^{-i \hp q \hp \eta}) \label{OV-pot-1} \\ 
\mu = \mu^{\text{sf}} 
+ \frac{q}{4\pi i} \npt \int_{\ell_+} \! \frac{d\zeta}{\zeta} \Big(\pt \frac{z}{\zeta} - \bar{z}\hp \zeta \Big) \npt \ln (1 - e^{i \hp q \hp \eta})
- \frac{q}{4\pi i} \npt \int_{\ell_-} \! \frac{d\zeta}{\zeta} \Big(\pt \frac{z}{\zeta} - \bar{z}\hp \zeta \Big) \npt \ln (1 - e^{-i \hp q \hp \eta}) \rlap{.} \label{OV-pot-2}
\end{gather}
}%
The integrals can be evaluated by methods similar to the ones employed for the evaluation of $L$. For the second one for instance we get in this way, up to a possible constant shift,
\begin{equation}
\mu = \frac{1}{2} \sum_{n=-\infty}^{\infty} \big(\sqrt{(q\psi + 2\pi n)^2 + 4|qz|^2} - 2|qz|^2c_{|n|}^{\phantom{'}} - c_{\smash{|n|}}^{-1} \hp \big) \rlap{.}
\end{equation}

In \cite{MR2672801}, the twistor canonical coordinate conjugate to $\eta$ was shown to be given in terms of the latter by the contour integral formula\pt\footnote{\pt To get a more concrete sense of why and how such integrals arise naturally in twisted symplectic gluing problems on twistor spaces, see \textit{e.g.} \cite{MR2485482}.}
%admit the contour integral representation
\begin{equation} \label{GMN-int-OV}
\tilde{\eta}(\zeta) = \tilde{\eta}^{\hp \text{sf}}(\zeta) 
- \frac{q}{4\pi} \npt \int_{\ell_+} \! \frac{d\zeta'}{\zeta'} \frac{\zeta'+\zeta}{\zeta'-\zeta} \ln (1 - e^{i \hp q \hp \eta(\zeta')})
+ \frac{q}{4\pi} \npt \int_{\ell_-} \! \frac{d\zeta'}{\zeta'} \frac{\zeta'+\zeta}{\zeta'-\zeta} \ln (1 - e^{-i \hp q \hp \eta(\zeta')}) \rlap{.}
\end{equation}
Note first that the choice of integration contours and kernel ensures that the result obeys the usual reality condition required of such a coordinate, namely $\smash{ \overline{\tilde\eta(\zeta^c)} = \tilde\eta(\zeta) }$. Moreover, while the $\eta$ coordinate is analytic with poles of order one at \mbox{$\zeta = 0$} and $\infty$, the crucial feature of the $\tilde{\eta}$ coordinate as given by this formula is that it is analytic in $\zeta$ only away from the contours $\ell_+$ and $\ell_-$. 

The formula is best understood as providing a solution to a certain Rie\-mann-Hil\-bert problem on the twistor sphere. To formulate this problem, consider an open covering of the twistor space consisting of the following elements:
\begin{itemize}
\setlength{\itemsep}{0pt}

\item[$-$] two open sets $V_N$ and $V_S$ whose every non-vanishing intersection with a horizontal twistor line projects down to an arctic respectively antarctic cap on the twistor sphere.
\item[$-$] two open sets $V_E$ and $V_W$ whose every non-vanishing intersection with a horizontal twistor line projects down to the two hemispheres on the twistor sphere delimited by $\ell_+$ and $\ell_-$.
\end{itemize}

\begin{figure}[ht]

% \vskip-10pt

\begin{tikzpicture}[scale=0.9] % THE GLOBE

\def\R{2} % sphere radius
\def\angEl{15} % elevation angle
\filldraw[ball color=white] (0,0) circle (\R); % or lightgray!20

%\foreach \t in {-80,-60,...,80} { \DrawLatitudeCircle[\R]{\t} }
%\foreach \t in {-5,-35,...,-175} { \DrawLongitudeCircle[\R]{\t} }

\DrawLatitudeCircle[\R]{ 40}
\DrawLatitudeCircle[\R]{-40}

\DrawLongitudeCircle[\R]{-60}

\node at ( 0.00, 2.50) {$\pi(V_N)$};
\node at ( 0.00,-2.50) {$\pi(V_S)$};
\node at ( 2.80, 0.00) {$\pi(V_E)$};
\node at (-2.80, 0.00) {$\pi(V_W)$};

\node at ( 1.35, 0.00) {$\ell_+$};
\node at (-1.25, 0.00) {$\ell_-$};

\end{tikzpicture}
\caption{The twistor sphere for the Ooguri-Vafa metric.}
\end{figure}
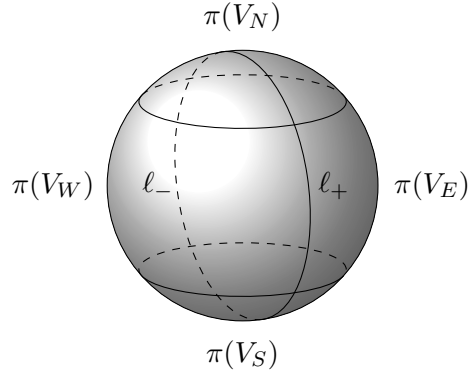

\noindent The form of the $\eta\pt$-coordinate on each set and its transition properties between different sets are in this case completely determined by the requirement that they are local components of a global $\mathcal{O}(2)$ section over the twistor space. For the $\tilde\eta\pt$-coordinates, on the other hand, the constraints take rather the form, as we have said, of a Riemann-Hilbert problem. Thus, we look for local functions $\tilde\eta_{\hp V}$ associated to each one of the sets $V$ above and analytic in $\zeta$ on them with the following properties: 
\begin{itemize}
\setlength{\itemsep}{3pt}

\item[1.] For each $V_T \in \{V_E,V_W\}$ the transition between $V_N$ and $V_T$ respectively $V_S$ and $V_T$ is given by the twisted symplectomorphisms
\begin{equation}
\begin{array}{lcl} 
\displaystyle \eta_{\hp V_T} = \frac{\eta_{\hp V_N}}{\zeta \, } & \multirow{2}{*}{\qquad and {\vrule width 0pt height 5ex depth 0ex} \qquad} & \displaystyle \eta_{\hp V_T} = \zeta \hp \eta_{\hp V_S} \\[5pt]
\displaystyle \tilde{\eta}_{\hp V_T} = \tilde{\eta}_{\hp V_N} + \frac{F' \nhp (\eta_{\hp V_N})}{\zeta} & & \displaystyle \tilde{\eta}_{\hp V_T} = \tilde{\eta}_{\hp V_S} - \zeta \hp \bar{F}' \nhp (-\eta_{\hp V_S})
\end{array}
\end{equation}
Note that $\smash{ \eta_{\hp V_T} }$ and $\smash{ \tilde\eta_{\hp V_T} }$ correspond simply to $\eta$ and $\tilde\eta$, respectively, in our previous notation. This condition is known as the ``asymptotic condition" because it regulates the behavior at $\zeta=0$ and $\infty$.

\item[2.] Over the rays $\ell_+$ and $\ell_-$ one has a discontinuity jump given by the pure (because tropical) symplectomorphisms
\begin{equation} \label{KS-OV}
\begin{aligned}
e^{i \hp \tilde{\eta}^+_{\ell_+}} & = e^{i \hp \tilde{\eta}^{- \phantom{j}}_{\ell_+}} [1 - (e^{i \hp \eta})^{+q} ]^{+q} \\[-1pt]
e^{i \hp \tilde{\eta}^+_{\ell_-}} & = e^{i \hp \tilde{\eta}^{- \phantom{j}}_{\ell_-}} [1 - (e^{i \hp \eta})^{-q}]^{-q}
\end{aligned}
\end{equation}
where $\tilde\eta{}^+_{\hp \ell}$ and $\tilde\eta{}^-_{\hp \ell}$ denote the limits of $\tilde\eta$ when $\zeta$ approaches the ray $\ell$ from the clockwise respectively counterclockwise direction.
\end{itemize}

\noindent The link between this Riemann-Hilbert problem and the integral representation \eqref{GMN-int-OV} is provided by the Sokhotski-Plemelj theorem. 

The tropical canonical coordinates $\eta$ and $\tilde\eta$ continue to satisfy in this case differential equations of the type \eqref{GMN_diff_eqs}. Let us sketch briefly the outline of a proof. First, remark that it suffices to show that both coordinates belong to the kernel of the operator
\begin{equation} \label{rot-op}
- \hp i \hp \zeta \frac{\partial}{\partial \zeta} + X^{\text{rot}}(\zeta)
\end{equation}
with the components of $\smash{ X^{\text{rot}}(\zeta) }$ defined in accordance with the considerations of the previous subsection by the quaternionic dimension one version of the c-map formulas \eqref{X-rot-sf}. Holomorphicity allows us then as before to replace the vector field $\smash{ X^{\text{rot}}(\zeta) }$ by any vector field $X(\zeta)$ from its equivalence class. The equation for $\eta$ follows from the fact that its moduli do not depend on the $\smash{ \tilde{\psi} }$ variable. Both $\smash{ X_+^{\text{rot}} }$ and $\smash{ X_-^{\text{rot}} }$ then drop out and this~reduces~to
\begin{equation} \label{eta-PDE}
\Big( \! - \hp i \hp \zeta \frac{\partial}{\partial \zeta} + \mathfrak{L}_3 \Big) \eta = 0\rlap{.}
\end{equation}
To obtain the equation for $\tilde\eta$ we act with the operator \eqref{rot-op} on the formula \eqref{GMN-int-OV}. The semi-flat term is clearly annihilated by the corresponding result from the pure c-map case.  Inside the two integrals the operators $\smash{ X_+^{\text{rot}} }$ and $\smash{ X_-^{\text{rot}} }$ can be again dropped out as the only object on which they can possibly act is $\eta(\zeta')$. On another hand, when acting on the kernel $\smash{ \frac{\zeta'+\zeta}{\zeta'-\zeta} }$, the operator $\smash{ \zeta \frac{\partial}{\partial \zeta} }$ can be exchanged for $\smash{ - \zeta' \frac{\partial}{\partial \zeta'} }$, and then an integration by parts can be performed to move this on to the right. After some further manipulations under the integrals the result follows by an obvious application of the differential equation for $\eta(\zeta')$ in the form \eqref{eta-PDE}. 

The presence of a rotational symmetry entails with it the existence of a hyperholomorphic line bundle over the Ooguri-Vafa space, and of a corresponding holomorphic line bundle over its twistor space endowed with a meromorphic connection. The connection 1-forms and transition functions can be described explicitly. However, since the Ooguri-Vafa metric is a particular case of the class of metrics we examine in the next subsection, we prefer to subsume their description to the corresponding discussion in that more general context.

\subsection{The Gaiotto\hp-Moore\hp-Neitzke hyperk\"ahler constructions} 

\subsubsection{} 

One of the key observations made in \cite{MR2672801} was that the two symplectomorphisms \eqref{KS-OV} are particular examples of what are known as Kontsevich-Soibelman transformations. We will now define what these are and then briefly describe the general features of certain extremely powerful identities that they satisfy\,---\,the classical Kontsevich-Soibelman \textit{wall-crossing formulas} \cite{Kontsevich:2008fj, MR2681792}. 

Consider a lattice $\Gamma$ of even rank equipped with a non-degenerate symplectic pairing \mbox{$\langle\cdot\pt,\cdot\rangle: \Gamma \times \Gamma \rightarrow \mathbb{Z}$} and let $\smash{ \{\gamma^a\}_{a=1,\dots,\text{rk}(\Gamma)} }$ be a choice of basis for it. To this basis we associate a set of algebraically independent indeterminates $\smash{ \{X_{\gamma^a}\}_{a=1,\dots,\text{rk}(\Gamma)} }$ which we multiply according to the commutative and associative rule
\begin{equation} \label{twist-alg-torus}
X_{\gamma}X_{\gamma'} = (-)^{\langle \gamma,\gamma' \rangle} X_{\gamma + \gamma'} \rlap{.}
\end{equation}
This is a twist on the usual group algebra for an algebraic torus. Note in particular that corresponding to the zero element of the lattice we have the multiplicative unit, which we denote by 1. For each $\gamma = \sum_a q_a \gamma^a \in \Gamma$ let $\sigma_{\gamma}$ be the sign uniquely defined by the equality
\begin{equation}
X_{\gamma} = \sigma_{\gamma} \prod_a (X_{\gamma^a})^{q_a} \rlap{.}
\end{equation}
Remark that this satisfies the property $\smash{ \sigma_{\gamma} \sigma_{\gamma'} = (-)^{\langle \gamma,\gamma' \rangle} \sigma_{\gamma+\gamma'} }$. 

Then for each $\gamma' \in \Gamma$ the corresponding Kontsevich-Soibelman transformation defined~by
\begin{equation}
T_{\gamma'} : X_{\gamma} \ \longmapsto \ X_{\gamma}(1-X_{\gamma'})^{\langle\gamma',\gamma\hp\rangle}
\end{equation}
preserves the product structure as well as the symplectic form 
\begin{equation}
\frac{1}{2} \hp \varepsilon_{ab} \, d\ln X_{\gamma^{\smash a}} \npt \wedge d\ln X_{\gamma^{\smash b}}
\end{equation}
where $\varepsilon_{ab}$ is  the inverse of $\varepsilon^{ab} = \langle \gamma^a,\gamma^b \hp \rangle$. Rather than check directly the second part of the statement one derives more benefits if one checks instead that under the action of $T_{\gamma'}$ we have 
\begin{equation} \label{KS/1-f-pot}
\frac{1}{2} \hp \varepsilon_{ab} \ln X_{\gamma^{\smash a}} d\ln X_{\gamma^{\smash b}} 
\ \longmapsto \ 
\frac{1}{2} \hp \varepsilon_{ab} \ln X_{\gamma^{\smash a}} d\ln X_{\gamma^{\smash b}} - d L_{\sigma_{\gamma'}}\npt(X_{\gamma'})
\end{equation}
where 
\begin{equation}
L_{\sigma}(z) = \Li_2(z) + \frac{1}{2} \ln (\sigma^{-1}z) \ln(1-z)
\end{equation}
is the Alexandrov-Persson-Pioline variant of the Rogers dilogarithm \cite{MR2935636}. 

As we have mentioned above, the Kontsevich-Soibelman transformations satisfy a remarkable set of identities known as wall-crossing formulas. The concrete-minded reader is urged at this point to consult for example the references \cite{Kontsevich:2008fj} and \cite{MR2672801} for a number of illuminating examples in rank 2 and keep them in mind throughout the coming discussion. To describe wall-crossing formulas we need to introduce some additional data. First, we need a group homomorphism $Z : \Gamma \rightarrow \mathbb{C}$ known as \textit{central charge} or \textit{stability function}. Thus, corresponding to the above frame of $\Gamma$ we have a set of complex numbers $Z_{\gamma^a}$ such~that
\begin{equation}
Z : \ \gamma = \sum_a q_a \gamma^a 
\ \longmapsto \ 
Z_{\gamma} = \sum_a q_a Z_{\gamma^a} \rlap{.}
\end{equation}
Given a stability function $Z$, for every non-vanishing $\gamma \in \Gamma$ let
\begin{equation}
\ell_{\gamma} = \{ \zeta \in \mathbb{C}^{\times} \, | \, \frac{Z_{\gamma}}{\zeta} = i \pt \mathbb{R}_+ \}
\end{equation}
be a corresponding ray emerging from the origin of the complex plane and going off to infinity. Obviously, every non-vanishing element of $\Gamma$ lying on the half-line starting at the zero element of $\Gamma$ and passing through $\gamma$ determines the same ray. 

Wall-crossing formulas involve ordered products of the type
\begin{equation}
T_{\hp \Lambda} = \! \prod^{\curvearrowright}_{\gamma \in Z^{-1}(\Lambda)} \! T_{\gamma}^{\pt \Omega(\gamma,Z)}
\end{equation}
with respect to the composition law, where:
\begin{itemize}[leftmargin=\leftmargin-14pt]
\setlength{\itemsep}{2pt}

\item the ordering of the factors corresponds to the clockwise ordering of the rays $\ell_{\gamma}$. 

\item $\Lambda$ is a strict half-plane in $\mathbb{C}$, that is $\Lambda  = \{ \hp z \in \mathbb{C}^{\times} \, | \, \text{Arg}\hp (z) \in (\alpha - \pi,\alpha] \}$ for some angle $\alpha \in [0,\pi]$. 

\item $\Omega(\,\cdot\,,Z)$ is a non-negative integer-valued function on $Z^{-1}(\Lambda) \subset \Gamma$ depending on $Z$ in a piecewise constant way.

\item $Z$ is assumed 
\begin{itemize}
\setlength{\itemsep}{2pt}
\item to be ``away from walls of marginal stability": \smallskip

For every $\gamma \in Z^{-1}(\Lambda)$ such that $\Omega(\gamma,Z) \neq 0$, if $\gamma' \in Z^{-1}(\ell_{\gamma})$ with $\Omega(\gamma',Z) \neq 0$ then $\gamma'$ must lie on the half-line  starting at the zero element of $\Gamma$ and passing through $\gamma$. 

\item to satisfy the following ``support property": \smallskip

There exists $K>0$ such that if we pick a positive-definite norm on $\Gamma$ we have $|Z_{\gamma}| > K \|\gamma\|$ for all $\gamma \in Z^{-1}(\Lambda)$ with $\Omega(\gamma,Z) \neq 0$. This property does not depend on the choice of norm. 
\end{itemize}
\end{itemize}

\noindent For each ray $\ell \subset \Lambda$ let
\begin{equation}
T_{\hp \ell} = \! \prod_{\gamma \in Z^{-1}(\ell) } \! T_{\gamma}^{\pt \Omega(\gamma,Z)} 
\end{equation}
be the slice of $T_{\Lambda}$ supported on its inverse image through $Z$. Rays $\ell$ for which at least one of the factors in the product is non-trivial are called \textit{BPS rays}. Note that if there are more than one non-trivial factors one needs not worry about the product order since by the first assumption about $Z$ above they all commute. The product $T_{\Lambda}$ can then be thought of as a clockwise-ordered product of BPS ray factors. 

So far we have kept the stability function $Z$ fixed, but now let us think what happens when we vary it. The BPS rays start then to rotate in the complex plane. During their rotations it may happen that all or a set of them coalesce and try to cross each other. When the cyclic ordering of the BPS rays changes one says that one has reached a \textit{wall of marginal stability}. As the name says, the question the wall-crossing formula addresses is what happens when one crosses a wall of marginal stability. And the answer is: the integer powers $\Omega(\gamma,Z)$ in the product conspire to jump at the wall in such a way as to leave the whole product $T_{\Lambda}$ unchanged! \cite{Kontsevich:2008fj}

The wall-crossing formulas which we consider here are intimately linked to certain functional identities satisfied by the classical dilogarithm function and its variants, some of which, like Abel's pentagon identity, have been known for a long time. To see how these emerge, let us consider a generic wall-crossing formula, each side of which corresponds to a stability function $Z$ and $Z'$, respectively. The products on the two sides contain either a finite number of or countably many factors. If, for simplicity, we choose an ordering from the right to the left for each product,\footnote{\pt For BPS rays with multiple non-trivial factors there is of course more than one way to assign an ordering since the factors commute, so in any choice of ordering in such a case there is an inherent arbitrariness. However, this does not significantly affect our conclusions.} then the wall-crossing formula may be represented in the form
\begin{equation}
\prod^{\longleftarrow}_{\mathclap{ r = 1, 2, \dots }} \ T_{\smash{\gamma^{\phantom{'}}_r}}^{\pt \Omega(\gamma^{\phantom{'}}_r,Z)} 
=
\ \pt \prod^{\longleftarrow}_{\mathclap{ s = 1,2, \dots }} \ T_{\gamma'_{\smash{s}}}^{\pt \smash{\Omega(\gamma'_{\smash{s}},Z')}} \rlap{.}
\end{equation}
Let us denote the images of a generic element $X_{\gamma}$ under successive Kontsevich-Soibelman transformations as follows
\begin{equation}
\begin{tikzcd}[column sep = large]
X_{\gamma} \equiv X^{(1)}_{\gamma} \rar[mapsto, "T^{\pt \Omega(\gamma_1,Z)}_{\gamma_1}"] 
& X^{(2)}_{\gamma} \rar[mapsto, "T^{\pt \Omega(\gamma_2,Z)}_{\gamma_2}"]
& X^{(3)}_{\gamma} \cdots
\end{tikzcd}
\end{equation}
and similarly for the primed versions starting again with the same $X_{\gamma}$, which this time we denote by $\smash{ X^{\prime (1)}_{\gamma} }$. Then in view of the transformation law of the symplectic 1-form potential \eqref{KS/1-f-pot} the wall-crossing formula implies the dilogarithm functional identity \cite{MR2935636} 
\begin{equation}
\sum_{\mathclap{ r = 1, 2, \dots }} \ \Omega(\gamma^{\phantom{'}}_r,Z) L_{\sigma_{\smash{\gamma^{\phantom{'}}_r}}} \npt (X^{(r)}_{\smash{\gamma^{\phantom{'}}_r}})
-
\ \sum_{\mathclap{ s = 1, 2, \dots }} \ \Omega(\gamma'_{\smash{s}},Z') L_{\sigma_{\smash{\gamma'_{\hp \smash{s}}}}} \npt (X^{\prime (s)}_{\pt \smash{\gamma'_{\smash{s}}}})
= \text{const.}
\end{equation}
This relation's invariance under monodromies $\smash{ M_{\gamma'} : X_{\gamma} \longmapsto e^{2\pi i \langle\gamma,\gamma'\rangle}X_{\gamma} }$ requires that the following $\Gamma$-valued functional identity also holds:
\begin{equation} \label{WC-id-log}
\sum_{\mathclap{ r = 1, 2, \dots }} \ \gamma^{\phantom{'}}_r \pt \Omega(\gamma^{\phantom{'}}_r,Z) \ln (1 - X^{(r)}_{\smash{\gamma^{\phantom{'}}_r}})
=
\ \sum_{\mathclap{ s = 1, 2, \dots }} \ \gamma'_s \pt \Omega(\gamma'_{\smash{s}},Z') \ln (1 - X^{\prime (s)}_{\pt \smash{\gamma'_{\smash{s}}}}) \rlap{.}
\end{equation}

Let us mention finally that for each wall-crossing formula corresponding to a strict half-plane $\Lambda$ one has a mirror wall-crossing formula corresponding to the complement of $\Lambda$ in $\mathbb{C}$. As a result, the definition of $\Omega(\,\cdot\,,Z)$ can be extended to the whole of $\Gamma$ and one has simply $\Omega(-\gamma,Z) = \Omega(\gamma,Z)$ for all $\gamma \in \Gamma$. This property is used in particular to implement CPT invariance for the physical theories in \cite{MR2672801}.

\subsubsection{}

The surprising occurrence of Kontsevich-Soibelman symplectomorphisms in the twis\-tor space description of the Ooguri-Vafa metric led Gaiotto, Moore and Neitzke in \cite{MR2672801} to propose a generalization of that construction founded on the Kontsevich-Soibelman wall-crossing formulas (for a review, see \cite{MR3330790}). Specifically, they aim to construct complete hyperk\"ahler metrics on the total spaces of certain complex integrable systems. The integrable system data consists of the following items:
\begin{itemize}
\setlength{\itemsep}{3pt}

\item a complex manifold $\mathcal{B}$ (the ``Coulomb branch") containing a divisor $D$.

\item a local system of lattices $\Gamma$ over $\mathcal{B}' = \mathcal{B} \setminus D$ with fiber $\Gamma_z$ over a generic point $z \in \mathcal{B}'$ an even-rank lattice (the ``charge lattice") endowed with a non-degenerate symplectic integer-valued bilinear pairing~$\langle\cdot\pt,\cdot\rangle$, having non-trivial monodromy around~$D$. Let~us pick, for explicitness, a generic frame of $\Gamma_z$ induced by local sections $\gamma^a$ of $\Gamma$ and let $\varepsilon^{ab} = \langle \gamma^a, \gamma^b\pt\rangle$ be the antisymmetric lattice metric in this frame. The frame may be moreover chosen to be symplectic. Assuming the symplectic pairing has determinant 1, a symplectic frame is a frame \mbox{$\gamma^a = (\tilde\gamma_{\A},\gamma^{\A})_{\A =1,\dots,\text{rk}(\Gamma)/2}$} such that 
\begin{equation}
\langle\tilde\gamma_{\A},\tilde\gamma_{\B}\rangle = \langle\gamma^{\A},\gamma^{\B}\rangle = 0
\qquad\text{and}\qquad
\langle\gamma^{\A},\tilde\gamma_{\B}\rangle = \delta_{\B}^{\A} \rlap{.}
\end{equation}
Given a section $s_{\gamma}$ of $\Gamma^* \otimes_{\mathbb{Z}} \mathbb{C}$, we will typically denote its symplectic components $\smash{ s_{\smash{\gamma^{\A}}} = s^{\A} }$ and $\smash{ s_{\smash{\tilde{\gamma}_{\A}}} = \tilde{s}_{\A} }$. Given two such sections $s_{\gamma}$, $s^{\mathrlap{\prime}}{}_{\smash{\gamma}}$ and letting $\varepsilon_{ab}$ denote the symplectic metric on the dual lattice ($\pt =$ the inverse of $\varepsilon^{ab}$) we then have with these conventions $\smash{ \varepsilon_{ab} \pt s_{\gamma^{\smash a}}s^{\mathrlap{\prime}}{}_{\smash{\gamma^{\smash b}}} =  \tilde{s}_{\A}s^{\prime\A} - s^{\A}\tilde{s}^{\mathrlap{\hp \prime}}{}_{\A} }$.

\item a stability homomorphism $Z : \Gamma \rightarrow \mathbb{C}$ (the ``central charge") varying holomorphically around $\mathcal{B}'$. This is assumed to be locally derivable from a potential, meaning that one requires that
\begin{equation}
\varepsilon_{ab} \pt dZ_{\gamma^{\smash a}} \nhp \wedge dZ_{\gamma^{\smash b}} = 0 \rlap{.}
\end{equation}
If we let, in a slight departure from the rule stated above, $(\tilde{z}_{\A},z^{\A})$ be the symplectic components of $Z_{\gamma^a}$, then these form a system of local holomorphic functions on $\mathcal{B}'$, and we demand that the holomorphic cotangent space $\smash{ \mathcal{T}^*_z \mathcal{B}' }$ be spanned by the~$dz^{\A}$. The condition above may be written as $dz^{\A} \npt \wedge d\tilde{z}_{\A} = 0$, from which it follows that locally there exists a holomorphic function $F(z)$ such that $\smash{ \tilde{z}_{\A} = F_{\A} }$. Thus, $\mathcal{B}'$ is assumed to be locally identifiable around a point $z$ with a complex Lagrangian submanifold of $\Gamma^*_z \otimes_{\mathbb{Z}}\mathbb{C}$. 

\item a torus bundle $M'$ fibered over $\mathcal{B}'$ associated to the local system $\Gamma$, on the total space of which the hyperk\"ahler metric is to be constructed.
\end{itemize}

From our discussion so far it is clear that from this data one can construct rather naturally on appropriately considered local patches of $M'$ certain explicit hyperk\"ahler metrics, namely, the semi-flat metrics with prepotentials $F(z)$. These patches can then be glued together to give a smooth metric on $M'$. Such a metric, however, has a number of problems: it is incomplete,\footnote{\pt See \textit{e.g.} the closely-related considerations in \cite{MR1686939} regarding affine special K\"ahler manifolds.} and there is no clear way how it can be extended over $D$. Taking this metric as their starting point, Gaiotto, Moore and Neitzke proposed a class of hyperk\"ahler constructions which are believed to overcome these problems. Physically, they add instanton corrections to the semi-flat metric to construct metrics on $M'$ which can in principle be extended to complete metrics on a space $M$ obtained from $M'$ by adding nodal torus fibers over the points of D. In particular, one expects the metrics near singular locus points to be modeled after the Ooguri-Vafa metric. 

In the Gaiotto-Moore-Neitzke approach the instanton corrections are obtained by solving the following explicit functional integral equation generalizing the equation \eqref{GMN-int-OV} from the Ooguri-Vafa case:
\begin{equation} \label{GMN-int}
\mathcal{X}_{\gamma}(\zeta) = \mathcal{X}^{\text{sf}}_{\gamma}(\zeta) \exp \npt \bigg[ \frac{i}{4\pi}\sum_{\gamma'} \langle\gamma,\gamma'\rangle \pt\Omega(\gamma') \! \int_{\ell_{\gamma'}} \! \frac{d\zeta'}{\zeta'} \frac{\zeta'+\zeta}{\zeta'-\zeta} \ln (1 - \mathcal{X}_{\gamma'}(\zeta')) \bigg] \rlap{.}
\end{equation}
The complex twistor variables $\mathcal{X}_{\gamma}(\zeta)$ are assumed to be of the form 
\begin{equation} \label{Xcal-eta}
\text{$\mathcal{X}_{\gamma}(\zeta) = \sigma_{\gamma} \hp \exp \hp [\hp i \hp\eta_{\gamma}(\zeta)]$ \quad 
with \quad $\eta_{\gamma}(\zeta) = \sum_a q_a\hp \eta_{\gamma^a}(\zeta)$}
\end{equation}
for any sections $\gamma = \sum_a q_a\gamma^a$ of the local system $\Gamma$. The presence of the $\sigma_{\gamma}$ factor in this formula endows the space of $\mathcal{X}$-variables with a twisted torus algebra structure: 
\begin{equation}
\mathcal{X}_{\gamma} \mathcal{X}_{\gamma'} = (-)^{\langle\gamma,\gamma'\rangle}\mathcal{X}_{\gamma+ \gamma'}  \rlap{.}
\end{equation}
A corresponding exponential expression holds for the semi-flat piece $\smash{ \mathcal{X}^{\text{sf}}_{\gamma}(\zeta) }$, too, with functions $\smash{ \eta^{\hp\text{sf}}_{\gamma^a}(\zeta) }$ given as in the c-map formula \eqref{eta^sf}. Although these may look like tropical components of $\mathcal{O}(2)$ sections, the full $\eta_{\gamma}$ functions are not in general of this type.\footnote{\pt The Ooguri-Vafa case, where half of the $\eta_{\gamma^a}$ are sections of $\mathcal{O}(2)$, is rather exceptional in this sense.} The residues $Z_{\gamma^a}$ of the poles of $\smash{ \eta_{\gamma^a} }$ at $\zeta = 0$  define a stability function $Z$. The factors $\Omega(\gamma)$ are to be identified with the powers in the Kontsevich-Soibelman wall-crossing formulas, and as such are implicitly assumed to depend in a piecewise constant way on the stability function. Finally, one requires that $\eta_{\gamma}$ be real with respect to antipodal conjugation. The CPT property of $\Omega(\gamma)$ together with the choice of integral kernel ensure that this condition is compatible with the functional integral equation. 

\begin{remark}
In \cite{MR2672801} the authors introduce a positive scale $R$ (compactification radius) and devise an iterative procedure to solve the functional integral equation \eqref{GMN-int} by successive approximations in the large $R$ limit, which shows that at least in this limit a solution exists. In \cite{MR3003931} the same authors give an alternative construction of $\mathcal{X}_{\gamma}(\zeta)$ for a class of theories associated to meromorphic Hitchin systems which does not rely on a functional integral equation and is valid for all $R$. The hyperk\"ahler metrics of \cite{MR3003931} are expected to fall into the category of complete hyperk\"ahler metrics constructed rigorously by Biquard and Boalch in \cite{MR2004129}. 
\end{remark}

So then the claim is that if a solution $\mathcal{X}_{\gamma}$ of the equation \eqref{GMN-int} exists then 
\begin{equation} \label{GMN-om(zeta)}
\omega(\zeta) = - \pt \frac{1}{2} \hp \varepsilon_{ab} \pt d\ln \mathcal{X}_{\gamma^{\smash a}} \npt \wedge d\ln \mathcal{X}_{\gamma^{\smash b}}
\end{equation}
represents the tropical component of the $\mathcal{O}(2)$-twisted 2-form $\omega$ of a hyperk\"ahler metric having the properties specified above. Note that in terms of the symplectic components $\smash{ (\tilde\eta_{\A},\eta^{\A}) }$ of $\eta_{\gamma^a}$  this takes the usual canonical form 
\begin{equation} \label{omega_T-canon}
\omega(\zeta) = d\tilde{\eta}_{\A} \nhp \wedge d\eta^{\A} \rlap{.}
\end{equation}

Just as in the Ooguri-Vafa case, the functional integral equation is to be regarded as providing a solution to a certain Riemann-Hilbert problem. To formulate this, consider a covering of the twistor space $\mathcal{Z}$ consisting of the following elements:
\begin{itemize}
\setlength{\itemsep}{0pt}

\item[$-$] two open sets $V_N$ and $V_S$ whose every non-vanishing intersection with a horizontal twistor line projects down to an arctic respectively antarctic cap on the twistor sphere;
\item[$-$] for each pair of consecutive BPS rays choose a corresponding set $V_T$ in $\mathcal{Z}$ such that its every non-vanishing intersection with a horizontal twistor line projects down on the twistor sphere to the strict sector between the BPS rays, defined as the smaller angle sector excluding the first ray and including the second, when the rays are counted in a counterclockwise order. 
\end{itemize}

\begin{figure}[ht]
\begin{tikzpicture}[scale=0.9] % THE GLOBE

\def\R{2} % sphere radius
\def\angEl{15} % elevation angle
\filldraw[ball color=white] (0,0) circle (\R); % or lightgray!20

%\foreach \t in {-80,-60,...,80} { \DrawLatitudeCircle[\R]{\t} }
%\foreach \t in {-5,-35,...,-175} { \DrawLongitudeCircle[\R]{\t} }

\DrawLatitudeCircle[\R]{ 40}
\DrawLatitudeCircle[\R]{-40}

\DrawLongitudeCircle[\R]{-40}
\DrawLongitudeCircle[\R]{-80}

\node at (0.00, 2.50) {$\pi(V_N)$};
\node at (0.00,-2.50) {$\pi(V_S)$};
\node at (2.80,-0.10) {$\pi(V_T)$};

\draw (2.20,-0.10) edge[out= 160,in=-10,->] (0.90,-0.10);

\end{tikzpicture}
\caption{The twistor sphere for the Gaiotto-Moore-Neitzke hyperk\"ahler construction, with two consecutive BPS rays and their antipodal counterparts depicted as meridians.}
\end{figure}
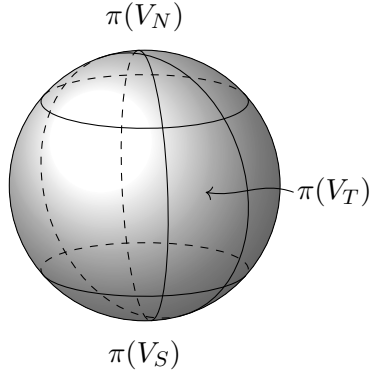

\noindent Then we seek a set of local functions $\eta_{\gamma,V}$ associated to each set $V$ above and analytic in $\zeta$ on it, obeying the following transition rules:
\begin{itemize}
\setlength{\itemsep}{3pt}

\item[1.] For each set $V_T$ the transition between $V_N$ and $V_T$ on one hand, and $V_S$ and $V_T$ on the other is given by the twisted symplectomorphisms
\begin{equation} \label{GMN-asympt-cond}
\begin{array}{r@{}lcr@{}l} 
\displaystyle \eta^{\A}_{\hp V_T} \hspace{2pt} \  & \displaystyle = \frac{\eta^{\A}_{\hp V_N}}{\zeta \, } & \multirow{2}{*}{\qquad and {\vrule width 0pt height 5ex depth 0ex} \qquad} 
& \displaystyle \eta^{\A}_{\hp V_T} \hspace{2pt} \  & \displaystyle = \zeta \hp \eta^{\A}_{\hp V_S} \\[5pt]
\displaystyle \tilde{\eta}_{\A,V_T} \hspace{2pt} \ & \displaystyle = \tilde{\eta}_{\A, V_N} + \frac{F_{\A}(\eta_{\hp V_N})}{\zeta} & & \displaystyle \tilde{\eta}_{\A,V_T} \hspace{2pt} \ & \displaystyle = \tilde{\eta}_{\A,V_S} - \zeta \hp \bar{F}_{\A}(-\eta_{\hp V_S})
\end{array}
\end{equation}
respectively. Note that in our previous notation the tropical functions~$\smash{ \eta^{\A}_{\hp \smash{ V_T }} }$ and~$\smash{\tilde\eta_{\smash{ \A,V_T }}^{\phantom{A}}}$ correspond simply to $\eta^{\A}$ and $\tilde\eta_{\A}$. 

\item[2.] Over each BPS ray $\ell$ (\textit{i.e.}, between the two sets of type $V_T$ whose projections on the twistor sphere border on the two sides of the ray $\ell$) there is a discontinuity jump given by the untwisted (because tropical) symplectomorphism
\begin{equation} \label{GMN-ray-jump}
(\mathcal{X}_{\gamma})_{\ell}^+ = T_{\ell}^{\phantom{j}} (\mathcal{X}_{\gamma})_{\ell}^-
\end{equation}
where $\smash{ (\mathcal{X}_{\gamma})_{\ell}^+ }$ and $\smash{ (\mathcal{X}_{\gamma})_{\ell}^- }$ denote the limits of $\mathcal{X}_{\gamma}$ when $\zeta$ approaches the ray $\ell$ from the clockwise respectively counterclockwise direction.
\end{itemize}

Going back now to the functional integral equation, observe that around $\zeta = 0$ it gives the Laurent series expansion
\begin{equation} \label{GMN-eta^N}
\eta_{\gamma}(\zeta) = \eta^{\text{sf}}_{\gamma}(\zeta) + i \Big(\hp \frac{1}{2} \pt\mathcal{I}_{\gamma,0} + \sum_{n=1}^{\infty} \mathcal{I}_{\gamma,-n} \zeta^n \Big)
\end{equation}
where, by definition,
\begin{equation} \label{I_g,n-int}
\mathcal{I}_{\gamma,n} = \frac{1}{2\pi i} \sum_{\gamma'} \langle\gamma,\gamma'\rangle \pt\Omega(\gamma') \!\int_{\ell_{\gamma'}} \! \frac{d\zeta}{\zeta} \pt \zeta^{n} \ln (1-\mathcal{X}_{\gamma'}(\zeta)) 
\rlap{.}
\end{equation}
These quantities satisfy the alternating reality property $\smash{ \bar{\mathcal{I}}_{\gamma,n} = (-)^n \mathcal{I}_{\gamma,-n} }$. Moreover, due to the functional identity \eqref{WC-id-log}  they vary smoothly across walls of marginal stability provided that the factors $\Omega(\gamma)$ jump across them in accordance with the corresponding Kontsevich-Soibelman wall-crossing formula \cite{Alexandrov:2014wca}. By way of the first set of formulas \eqref{GMN-asympt-cond}, this guarantees the important fact that the coordinates $\smash{ \eta_{\gamma,V_N} }$ are well-defined and smooth on $V_N$. Indeed, as we know, the first Taylor coefficients in their $\zeta$-expansions determine the hyperk\"ahler symplectic forms and, through them, the metric. Alternatively and more directly, one can obtain the hyperk\"ahler symplectic forms $\omega_+$ and $\omega_0$ by inserting the expansion \eqref{GMN-eta^N} into the formula \eqref{omega_T-canon} and collecting the terms of order $-1$ and $0$ in $\zeta$, respectively. We thus get
{\allowdisplaybreaks
\begin{equation}
\begin{aligned}
\omega_+^{\phantom{f}} & = \omega_+^{\text{sf}} + \frac{i}{2} \hp d\pt\tilde{\mathcal{I}}_{\A,0} \nhp \wedge dz^{\A} - \frac{i}{2} \hp d\pt\mathcal{I}^{\A}_{\hp 0} \wedge dF_{\A} \\
\omega_0^{\phantom{f}} \hspace{1.3pt} & = \omega_0^{\text{sf}} 
- \frac{1}{4}\hp d\pt\tilde{\mathcal{I}}_{\A,0}  \wedge d\pt\mathcal{I}^{\A}_{\hp 0} \\
& \hspace{6.1ex} + \hp \frac{i}{2}\hp (\pt d\pt\tilde{\mathcal{I}}_{\A,0} \wedge d\psi^{\A} \!-  d\pt\mathcal{I}^{\A}_{\hp 0} \wedge d\tilde{\psi}_{\A})
+  i\pt(\pt d\pt\tilde{\mathcal{I}}_{\A,-} \npt \wedge dz^{\A} \!- d\pt\mathcal{I}^{\A}_- \wedge dF_{\A}) 
\end{aligned}
\end{equation}
}%
with the semi-flat terms defined as in the formulas \eqref{HK_sym_sf}. The remaining hyperk\"ahler symplectic form $\omega_-$ can be obtained from $\omega_+$ by (alternating) complex conjugation.  The smoothness of the coefficients $\mathcal{I}^{\A}_{\hp 0}$, $\mathcal{I}^{\A}_-$ and their tilded counterparts guarantees then the smoothness of the hyperk\"ahler metric when walls of marginal stability are crossed. 

Observe, incidentally, that the above expressions for the hyperk\"ahler forms can be recast in the form \eqref{oaloo} if we take
{\allowdisplaybreaks
\begin{equation}
\begin{aligned}
v^{\A} & = \psi^{\A} + \frac{i}{2} \hp \mathcal{I}^{\A}_0 \\[-2pt]
u_{\A} & = \tilde{\psi}_{\A} - F_{\A\B}v^{\B} + \frac{i}{2} \hp \tilde{\mathcal{I}}_{\A,0} \\[-2pt]
w_{\A} & = \partial_{\smash{z^{\A}}}(\bar{z}^{\B}F_{\B} - z^{\B}\bar{F}_{\B}) - \frac{1}{2} F_{\A\B\C} v^{\B}v^{\C} +  i \hp (\tilde{\mathcal{I}}_{\A,-} \nhp - F_{\A\B} \pt \mathcal{I}^{\B}_-)
\rlap{.}
\end{aligned}
\end{equation}
}%
In particular, it is clear from this that we are in general in the situation with all variables $v^{\A}$ complex.\footnote{\pt The exception to this rule is once again the Ooguri-Vafa case, where \mbox{$\mathcal{I}^{\A}_n = 0$} for all $n$ and all the variables $v^{\A}$ are real.} But rather than apply right away the results from the corresponding part of section \ref{sec:SymmGens} we prefer in this case to choose a different set of special coordinates, one better adapted to the particularities of this problem, and then follow a similar route as there. 

So, first, let us remark that the condition that the above expression for $\omega_0$ be real can be equivalently phrased as the closure of the 1-form
\begin{equation}
\tilde{\mathcal{I}}_{\A,0} \hp d\psi^{\A}
- \mathcal{I}^{\A}_{\hp 0} \hp d\tilde{\psi}_{\A}
+ 2 \hp \Re (\pt \tilde{\mathcal{I}}_{\A,-} dz^{\A} - \mathcal{I}^{\A}_- \hp dF_{\A} )
\rlap{.}
\end{equation}
Therefore, locally there exists a real-valued function $\smash{ L^{\text{inst}} }$ on $M$ depending on the variables \mbox{$z^{\A}$, $\bar{z}^{\A}$, $\psi^{\A}$, $\tilde\psi_{\A}$} such that
\begin{equation}
\tilde{\mathcal{I}}_{\A,0} = L^{\text{inst}}_{\psi^{\A}} 
\qquad\quad
\mathcal{I}^{\A}_{\hp 0} = - \pt L^{\text{inst}}_{\psitildeindx_{\A}}
\qquad\quad
\tilde{\mathcal{I}}_{\A,-} \nhp - F_{\A\B} \pt \mathcal{I}^{\B}_- \npt = L^{\text{inst}}_{z^{\A}}
\end{equation}
where we indicate derivatives with indices in the usual way.  A simple exercise allows us to re-express these equations in the form
{\allowdisplaybreaks
\begin{align}
L^{\text{inst}}_{z^{\A}} & = - \pt \frac{1}{2\pi i} \sum_{\gamma} \Omega(\gamma) \frac{\partial Z_{\gamma}}{\partial z^{\A}} \int_{\ell_{\gamma}} \npt \frac{d\zeta}{\zeta^2} \ln (1 - \mathcal{X}_{\gamma}(\zeta)) \nonumber \\[0pt]
L^{\text{inst}}_{\psi_{\gamma^a}} & = - \pt \frac{1}{2\pi i} \sum_{\gamma} \Omega(\gamma) \frac{\partial \psi_{\gamma}}{\partial \psi_{\gamma^{\mathrlap{a}}}} \pt\pt \int_{\ell_{\gamma}} \! \frac{d\zeta}{\zeta} \ln (1 - \mathcal{X}_{\gamma}(\zeta)) \\[-2pt]
L^{\text{inst}}_{\bar{z}^{\A}} & = \phantom{+} \pt \frac{1}{2\pi i} \sum_{\gamma} \Omega(\gamma) \frac{\partial \bar{Z}_{\gamma}}{\partial \bar{z}^{\A}} \int_{\ell_{\gamma}}  d\zeta \hspace{2pt} \ln (1 - \mathcal{X}_{\gamma}(\zeta))
\rlap{.} \nonumber
\end{align}
}%
In terms of this new function the preceding formulas for the hyperk\"ahler symplectic forms become
\begin{equation}
\begin{aligned}
\omega_+^{\phantom{f}} & = \omega^{\text{sf}}_+ + \frac{i}{2} \hp dL^{\text{inst}}_{\psi^{\A}} \npt \wedge dz^{\A} + \frac{i}{2} \hp dL^{\text{inst}}_{\psitildeindx_{\A}} \! \wedge dF_{\A} \\
\omega_0^{\phantom{f}} \hspace{1.3pt} & =  \omega^{\text{sf}}_0 \hp + \frac{i}{2} \hp dL^{\text{inst}}_{z^{\A}} \npt \wedge dz^{\A} - \frac{i}{2} \hp dL^{\text{inst}}_{\bar{z}^{\A}} \npt \wedge d\bar{z}^{\A} +  \frac{1}{4} \hp dL^{\text{inst}}_{\psi^{\A}} \npt \wedge dL^{\text{inst}}_{\psitildeindx_{\A}} 
\end{aligned}
\end{equation}
with $\omega_0$ now manifestly real. The quaternionicity condition formulated in section \ref{sec:SymmGens} imposes further non-linear differential constraints on $\smash{ L^{\text{inst}} }$. 

Note in particular that the term in $\omega_0$ quadratic in the $d\psi_{\gamma^a}$ differentials is of the form $\frac{1}{2} \pt \hat{\varepsilon}_{ab} \pt d\psi_{\gamma^{\smash a}} \npt \wedge d\psi_{\gamma^{\smash b}}$, with
\begin{equation}
\hat{\varepsilon}_{ab} = \varepsilon_{ab} + \frac{1}{4} L^{\text{inst}}_{\psi_{\gamma^{\smash a}}\psi_{\gamma^{\smash c}}} \varepsilon^{cd} \pt L^{\text{inst}}_{\psi_{\gamma^{\smash d}}\psi_{\gamma^{\smash b}}}
\rlap{.}
\end{equation}
Let $\hat{\varepsilon}^{\hp ab}$ be the inverse of $\hat{\varepsilon}_{ab}$.  For any $C^1$-function $f$ on $M'$ let us define the symplectic gradient vector field 
\begin{equation} \label{GMN-Xf}
X_f = \hat{\varepsilon}^{\hp ab} \hp \frac{\partial f}{\partial \psi_{\smash{\gamma^{\smash{\mathrlap{a}}}}}} \pt\hp
\frac{\partial }{\partial \psi_{\smash{\gamma^{\mathrlap{b}}}}}
\rlap{.}
\end{equation}
Then the following Cauchy-Riemann type result similar to the one we have derived earlier in section \ref{sec:SymmGens} (\textit{i.e.}, Proposition \ref{gmm-hyperps}) holds:

\begin{proposition} \label{GMN-gmm-hyperps}
A sequence of functions $\{f_n\}_{n \in \mathbb{Z}}$ on $M'$ forms a chain of hyperpotentials with respect to the complex structure $I_0$ if and only if
\begin{equation}
\iota_{X_{f_{n-1}}}\omega_+ + \iota_{X_{f_n}}\omega_0 + \iota_{X_{f_{n+1}}}\omega_- = df_n
\end{equation}
for all $n \in \mathbb{Z}$.
\end{proposition}

Let us recall now a key argument from \cite{MR2672801} which shows that the first order partial linear differential equation for the tropical twistor canonical coordinates $\eta_{\gamma}$ encountered in the previous examples is a rather generic phenomenon. Consider the $\zeta$-dependent vector field defined, in local coordinates, by
\begin{equation}
X(\zeta) = i\hp \zeta \frac{\partial \hp \eta_{\hp \gamma^a}}{\partial \zeta\ } (\mathcal{J}^{-1})_a{}^b \frac{\partial}{\partial \psi_{\smash{\gamma^{\mathrlap{b}}}}}
\qquad\text{with}\qquad
\mathcal{J}_a{}^b = \frac{\partial \hp \eta_{\hp \gamma^b}}{\partial \psi_{\smash{\gamma^{\mathrlap{a}}}}}
\end{equation}
representing the Jacobian matrix of the partial mapping $\psi_{\gamma^a} \longmapsto \eta_{\gamma^a}$.  The discontinuities of the $\eta_{\gamma^a}$ functions along the BPS rays cancel out in $X(\zeta)$ but their poles do not, and so as a result $X(\zeta)$ is genuinely analytic in $\zeta$ except at \mbox{$\zeta = 0$} and $\infty$, where it has a pole of order one.
Thus, this is of the form
\begin{equation} \label{O2-twisted-X}
X(\zeta) = \frac{X_+}{\zeta \ } + X_0 + \zeta X_-
\end{equation}
with components $X_m$ vector fields on $M'$ acting in the directions spanned by $\smash{ \frac{\partial}{\partial \psi_{\smash{\mathrlap{\gamma^a}}}} \hspace{1ex} }$ and satisfying the alternating reality condition $\bar{X}_m = (-)^mX_{-m}$ (this follows from the reality properties of the functions $\eta_{\gamma^a}$ with respect to antipodal conjugation). At the same time, from the definition of $X(\zeta)$ it is clear that we have
\begin{equation} \label{GMN-diff-eq}
\Big(\! - i \hp \zeta \frac{\partial}{\partial \zeta} + X(\zeta) \Big) \eta_{\gamma^a} = 0 
\end{equation}
for all $a = 1, \dots \text{rk}(\Gamma)$.

Very importantly for us, this differential property is enough to guarantee the existence of a twisted rotational action. To see this, let us rewrite for a moment the Laurent expansion \eqref{GMN-eta^N} for $\eta_{\gamma}$ around $\zeta = 0$ in the following generic form
\begin{equation}
\eta_{\gamma} = \sum_{n=-1}^{\infty} \eta_{\gamma,-n}\zeta^n \rlap{.}
\end{equation}
Remark, on one hand, that the differential equation \eqref{GMN-diff-eq} is equivalent to the set of relations
\begin{equation}
X_+(\eta_{\gamma^a,n-1}) + X_0(\eta_{\gamma^a,n}) + X_-(\eta_{\gamma^a,n+1}) = - \hp i \hp n \pt \eta_{\gamma^a,n}
\end{equation}
for all applicable \mbox{$n \in \mathbb{Z}$}. On the other hand, inserting the Laurent expansion into the formula \eqref{GMN-om(zeta)} for $\omega(\zeta)$ yields for the hyperk\"ahler symplectic forms the expressions
\begin{equation}
\omega_m = \frac{1}{2} \hp \varepsilon_{ab} \sum_{k} d\eta_{\gamma^{\smash a},k} \nhp \wedge d\eta_{\gamma^{\smash b}, m-k}
\rlap{.}
\end{equation}
This formula is in fact valid for all $m \in \mathbb{Z}$, in which case we have of course $\omega_m = 0$ for all $m$ except $m=-1,0,1$. By using these two relations one can then prove explicitly that the equation \eqref{twist_ax_rot_2} with $j=2$ holds, which shows that we do have indeed in this case a twisted rotational action. 

In a similar way we obtain in particular that
\begin{equation} \label{mmap-mu_N}
\iota_{X_-}\omega_+ + \iota_{X_0}\omega_0 + \iota_{X_+}\omega_- = d\mu_N
\end{equation}
with the function
\begin{equation}
\mu_N = - \hp \frac{i}{2} \hp \varepsilon_{ab} \! \sum_{m=-1}^{1} \! m \, \eta_{\gamma^{\smash a},m} \hp \eta_{\gamma^{\smash b},-m} 
\rlap{.}
\end{equation}
Reintroducing the specific forms of the Laurent coefficients from equation \eqref{GMN-eta^N} gives us for this the expression $\smash{ \mu_N = \mu^{\text{sf}} + \varepsilon_{ab} Z_{\gamma^a} \mathcal{I}_{\gamma^{\smash b},-} }$. Note that while the left-hand side of equation \eqref{mmap-mu_N} is manifestly real\,---\,due to the (alternating) reality properties of its component fields, the function $\mu_N$ on the right-hand side is not. This can only be the case if the imaginary part of $\mu_N$ is constant. Another way to argue this, if we forget for instance the reality property of the system of vector fields $X_m$, is by repeating the argument but for the Laurent expansion around $\zeta = \infty$. On the right-hand side of  equation \eqref{mmap-mu_N} one now gets a function $\mu_S$ which, in view of the reality property of $\eta_{\gamma}$, turns out to be the complex conjugate of $\mu_N$. In either case, using also the integral representation \eqref{I_g,n-int}, we get the consistency condition
\begin{equation} \label{GMN-L3L}
\Im \mu_N = - \pt \frac{1}{4\pi} \sum_{\gamma} \Omega(\gamma) \! \int_{\ell_{\gamma}} \! \frac{d\zeta}{\zeta} \Big(\frac{Z_{\gamma}}{\zeta} + \bar{Z}_{\gamma}\zeta \Big) \ln (1 - \mathcal{X}_{\gamma}(\zeta)) = C
\end{equation}
for some real, at least locally defined and possibly vanishing constant $C$. This condition can be thought of as the twisted case equivalent of the invariance constraint of Lemma~\ref{L3-inv}, where only a simple rotating action was present. Letting then \mbox{$\mu = \Re \mu_N$} we have the moment map equation
\begin{equation}
\iota_{X_-}\omega_+ + \iota_{X_0}\omega_0 + \iota_{X_+}\omega_- = d\mu
\end{equation}
with the integral representation for the real-valued moment map function
\begin{equation}
\mu = \mu^{\text{sf}} + \frac{1}{4\pi i} \sum_{\gamma} \Omega(\gamma) \! \int_{\ell_{\gamma}} \! \frac{d\zeta}{\zeta} \Big(\frac{Z_{\gamma}}{\zeta} - \bar{Z}_{\gamma}\zeta \Big) \ln (1 - \mathcal{X}_{\gamma}(\zeta)) \rlap{.}
\end{equation}
This is an obvious generalization of the Ooguri-Vafa metric formula \eqref{OV-pot-2}. 

\begin{remark}
This function, which appears here in the guise of a moment map of a twisted rotational action, is the main object of interest in \cite{Alexandrov:2014wca} where it was put forward as a candidate for a Witten-type index of a certain class of supersymmetric quantum field theories. The instrumental property there was instead the one expressed by the $m=1$ component of  equation \eqref{sgm-omg}. 
\end{remark}

By the general results of section \ref{sec:Trans-rot}, the presence of $\mathcal{O}(2)$-twisted rotational actions on the Gaiotto-Moore-Neitzke hyperk\"ahler spaces implies the existence of hyperholomorphic line bundles over them, and of corresponding holomorphic line bundles over their twistor spaces equipped with meromorphic connections with poles of order two on the fibers over \mbox{$\zeta = 0$} and $\infty$. The existence of these bundles was discovered by Neitzke in \cite{Neitzke:2011za}, and our considerations here can be read as a geometric interpretation of his finds. 

For the arctic meromorphic potential associated to the twisted rotational action we consider again
\begin{equation}
\varphi_{V_N}(\zeta) = i \pt \frac{f(\eta_{\hp V_N})}{\zeta^2} + i \pt \frac{\tilde{\eta}^{\phantom{'}}_{\A,V_N}\eta^{\A}_{\hp V_N}}{\zeta} 
\end{equation}
where, recall, \mbox{$f(z) = z^{\A}F_{\A}(z) - 2F(z)$}. This is as required a meromorphic function defined on $V_N$ with a pole of order two at $\zeta = 0$. The crucial test it needs to pass is whether it satisfies the gauge condition of Lemma \ref{A_V-sp-gauge}. This follows from the consistency constraint \eqref{GMN-L3L}. Assuming that the constant $C$ vanishes (otherwise redefine $\varphi_{V_N}(\zeta)$ by subtracting from it the imaginary constant term $i \hp C$; this affects neither its meromorphicity nor its pole structure) we find indeed that its zero-order Laurent coefficient is of the form
\begin{equation}
\varphi_0 = \mu + i\hp u_{\A} v^{\A} \rlap{.}
\end{equation}
Notice, moreover, that the second-order residue is simply
\begin{equation}
\varphi_{++}^{\phantom{f}} = \varphi_{++}^{\hp \text{sf}}
\end{equation}
and that for the first-order one the formula yields $\smash{ \varphi_+^{\phantom{f}} = \varphi_+^{\hp \text{sf}} + \frac{1}{2} \varepsilon_{ab} Z_{\gamma^a} \mathcal{I}_{\gamma^{\smash b},0} }$, which can then be represented in the form
\begin{equation}
\varphi_+^{\phantom{f}} = \varphi_+^{\hp \text{sf}} + \frac{1}{4\pi i} \sum_{\gamma} \Omega(\gamma) \hp Z_{\gamma} \! \int_{\ell_{\gamma}} \! \frac{d\zeta}{\zeta} \ln (1 - \mathcal{X}_{\gamma}(\zeta))
\end{equation}
(compare with the formula \eqref{OV-pot-1} from the Ooguri-Vafa case).  Yet another way in which this can be written is $\smash[t]{ \varphi_+^{\phantom{f}} = \varphi_+^{\hp \text{sf}} - \frac{1}{2}Z_{\gamma^a} \npt L^{\text{inst}}_{\smash{\psi_{\gamma^a}}} }$.

So we can apply Lemma \ref{trapp}\,---\,only this time we want to do it for the symplectic gradient vector field \eqref{GMN-Xf}. This is possible due to Proposition \ref{GMN-gmm-hyperps}. Thus, if we take
\begin{equation} \label{GMN-X_pp}
X_{++}  = X_{\varphi_{++}} = 0
\qquad\quad
X_+ = X_{\varphi_+}
\qquad\quad
X_0 = X_{\mu} 
\end{equation}
with negative-indexed counterparts defined in the usual way by requiring alternating conjugation, then these vector fields satisfy the moment map equations
\begin{equation} 
\iota_{X_{n-1}}\omega_+ + \iota_{X_{n}}\omega_0 + \iota_{X_{n+1}}\omega_- =
\begin{cases}
d\varphi_{++}  & \text{if $n = 2$} \\
d\varphi_+ - i \hp \theta_+ & \text{if $n = 1$} \\
d\mu & \text{if $n=0$} 
\end{cases}
\end{equation}
where as usual we have $\theta_+ = u_{\A}dz^{\A}$. The formulas \eqref{GMN-X_pp} provide then a concrete representation for the component vector fields of the equation \eqref{O2-twisted-X}. 

We are now in a position to write down a set of explicit meromorphic connection 1-forms and holomorphic gluing functions corresponding to the cover of $\mathcal{Z}$ constructed above.\footnote{\pt Recall that the transition functions for the hyperholomorphic line bundle can be obtained from the gluing functions simply by exponentiation, as in the equation \eqref{hh-trans-fcts}.} Lemma \ref{A_V-sp-gauge} gives us for the polar elements of the cover the connection 1-forms
\begin{equation}
\begin{aligned}
\mathcal{A}^{\phantom{'}}_{V_N} & = \frac{1}{\zeta} \pt \tilde{\eta}^{\phantom{'}}_{\A,V_N}d^{\phantom{'}}_{\mathcal{Z}}\eta^{\A}_{\hp V_N} + i\hp d^{\phantom{'}}_{\mathcal{Z}}\varphi^{\phantom{'}}_{V_N}(\zeta) + i \pt \varphi^{\phantom{'}}_{V_N}(\zeta) \frac{d\zeta}{\zeta} \\
\mathcal{A}^{\phantom{'}}_{V_S} \, & = \, \zeta \pt \, \tilde{\eta}^{\phantom{'}}_{\A,V_S}\,d^{\phantom{'}}_{\mathcal{Z}}\eta^{\A}_{\hp V_S} + i\hp d^{\phantom{'}}_{\mathcal{Z}}\varphi^{\phantom{'}}_{V_S}(\zeta) \, - \, i \pt \varphi^{\phantom{'}}_{V_S}(\zeta) \frac{d\zeta}{\zeta} \rlap{.}
\end{aligned}
\end{equation}
From these, we can transition towards the tropical elements using the patching relations \eqref{GMN-asympt-cond}. We find in this way the connection 1-forms 
\begin{equation}
\mathcal{A}_{V_T} = \tilde{\eta}^{\phantom{'}}_{\A,V_T} d^{\phantom{'}}_{\mathcal{Z}}\eta^{\A}_{\hp V_T}
\end{equation}
and the gluing functions
\begin{equation}
\begin{aligned}
\phi_{V_TV_N}(\zeta) & = \varphi_{V_N}(\zeta) +  i \hp \frac{F(\eta_{\hp V_N})}{\zeta^2} \\
\phi_{V_TV_S}(\zeta) \hspace{1.5pt} & = \varphi_{V_S}(\zeta) \hspace{1.5pt} + i \hp \zeta^2 \bar{F}(-\eta_{\hp V_S}) \rlap{.}
\end{aligned}
\end{equation}
Finally, the gluing functions between tropical sets whose projections on the twistor sphere are adjacent to the same BPS line can be derived from the jump condition \eqref{GMN-ray-jump} by way of formula \eqref{KS/1-f-pot}. Thus, for any BPS line $\ell$ and sets $V_{T,\ell}^+ $ and $V_{T,\ell}^-$  whose projections on the twistor sphere are adjacent to it from the clockwise and counterclockwise directions, respectively, we get 
\begin{equation}
\phi_{V_{T,\ell}^+ V_{T,\ell}^-} (\zeta) =  i \!\!\!\!\! \sum_{\gamma \in Z^{-1}(\ell) } \!\! \Omega(\gamma)\hp L_{\sigma_{\gamma}}((\mathcal{X}_{\gamma})^-_{\ell}) + \frac{i}{2} \hp [  (\tilde{\eta}_{\A})^+_{\ell} (\eta^{\A})^+_{\ell} - (\tilde{\eta}_{\A})^-_{\ell} (\eta^{\A})^-_{\ell}] \rlap{.}
\end{equation}
The non-dilogarithmic term on the right stems from the discrepancy between the antisymmetrized symplectic 1-form potential considered in formula \eqref{KS/1-f-pot} versus the canonical one appearing in the expression above for the connection 1-forms $\mathcal{A}_{V_T}$.  

\bigskip\bigskip

\noindent \textbf{Acknowledgements.}
The author is deeply indebted to Martin Ro\v{c}ek for the insightful discussions which guided him over the years through many of the ideas on which these considerations rely.

\bibliographystyle{abbrv}
\bibliography{Hyperkahler}

\begin{thebibliography}{10}

\bibitem{Alexandrov:2014wca}
S.~Alexandrov, G.~W. Moore, A.~Neitzke, and B.~Pioline.
\newblock {$\mathbb R^3$ Index for Four-Dimensional $N=2$ Field Theories}.
\newblock {\em Phys. Rev. Lett.}, 114:121601, 2015.

\bibitem{MR2935636}
S.~Alexandrov, D.~Persson, and B.~Pioline.
\newblock Wall-crossing, {R}ogers dilogarithm, and the {QK}/{HK}
  correspondence.
\newblock {\em J. High Energy Phys.}, (12):027, i, 64, 2011.

\bibitem{MR2485482}
S.~Alexandrov, B.~Pioline, F.~Saueressig, and S.~Vandoren.
\newblock Linear perturbations of hyperk\"ahler metrics.
\newblock {\em Lett. Math. Phys.}, 87(3):225--265, 2009.

\bibitem{MR0086359}
M.~F. Atiyah.
\newblock Complex analytic connections in fibre bundles.
\newblock {\em Trans. Amer. Math. Soc.}, 85:181--207, 1957.

\bibitem{MR867684}
A.~L. Besse.
\newblock {\em Einstein manifolds}, volume~10 of {\em Ergebnisse der Mathematik
  und ihrer Grenzgebiete (3) [Results in Mathematics and Related Areas (3)]}.
\newblock Springer-Verlag, Berlin, 1987.

\bibitem{MR1848654}
R.~Bielawski.
\newblock Twistor quotients of hyperk\"ahler manifolds.
\newblock In {\em Quaternionic structures in mathematics and physics ({R}ome,
  1999)}, pages 7--21 (electronic). Univ. Studi Roma ``La Sapienza'', Rome,
  1999.

\bibitem{MR2004129}
O.~Biquard and P.~Boalch.
\newblock Wild non-abelian {H}odge theory on curves.
\newblock {\em Compos. Math.}, 140(1):179--204, 2004.

\bibitem{MR1785432}
M.~Dunajski and L.~J. Mason.
\newblock Hyper-{K}\"ahler hierarchies and their twistor theory.
\newblock {\em Comm. Math. Phys.}, 213(3):641--672, 2000.

\bibitem{MR2006758}
M.~Dunajski and L.~J. Mason.
\newblock Twistor theory of hyper-{K}\"ahler metrics with hidden symmetries.
\newblock {\em J. Math. Phys.}, 44(8):3430--3454, 2003.
\newblock Integrability, topological solitons and beyond.

\bibitem{MR1695113}
D.~S. Freed.
\newblock Special {K}\"ahler manifolds.
\newblock {\em Comm. Math. Phys.}, 203(1):31--52, 1999.

\bibitem{MR2672801}
D.~Gaiotto, G.~W. Moore, and A.~Neitzke.
\newblock Four-dimensional wall-crossing via three-dimensional field theory.
\newblock {\em Comm. Math. Phys.}, 299(1):163--224, 2010.

\bibitem{MR3003931}
D.~Gaiotto, G.~W. Moore, and A.~Neitzke.
\newblock Wall-crossing, {H}itchin systems, and the {WKB} approximation.
\newblock {\em Adv. Math.}, 234:239--403, 2013.

\bibitem{Gates:1983py}
S.~J. Gates, Jr.
\newblock {Superspace Formulation of New Nonlinear Sigma Models}.
\newblock {\em Nucl. Phys.}, B238:349--366, 1984.

\bibitem{Gates:1999zv}
S.~J. Gates, Jr., T.~H{\"u}bsch, and S.~M. Kuzenko.
\newblock {CNM models, holomorphic functions and projective superspace C-maps}.
\newblock {\em Nucl. Phys.}, B557:443--458, 1999.

\bibitem{Gates:1984nk}
S.~J. Gates, Jr., C.~M. Hull, and M.~Ro\v{c}ek.
\newblock {Twisted Multiplets and New Supersymmetric Nonlinear Sigma Models}.
\newblock {\em Nucl. Phys.}, B248:157--186, 1984.

\bibitem{Gibbons:1979zt}
G.~W. Gibbons and S.~W. Hawking.
\newblock {Gravitational Multi-Instantons}.
\newblock {\em Phys. Lett.}, 78B:430, 1978.

\bibitem{MR0425802}
S.~Hacyan and J.~Pleba{\'n}ski.
\newblock Some basic properties of {K}illing spinors.
\newblock {\em J. Mathematical Phys.}, 17(12):2203--2206, 1976.

\bibitem{MR2394039}
A.~Haydys.
\newblock Hyper{K}\"ahler and quaternionic {K}\"ahler manifolds with
  {$S^1$}-symmetries.
\newblock {\em J. Geom. Phys.}, 58(3):293--306, 2008.

\bibitem{MR623721}
N.~J. Hitchin.
\newblock K\"ahlerian twistor spaces.
\newblock {\em Proc. London Math. Soc. (3)}, 43(1):133--150, 1981.

\bibitem{MR887284}
N.~J. Hitchin.
\newblock The self-duality equations on a {R}iemann surface.
\newblock {\em Proc. London Math. Soc. (3)}, 55(1):59--126, 1987.

\bibitem{MR3116317}
N.~J. Hitchin.
\newblock On the hyperk\"ahler/quaternion {K}\"ahler correspondence.
\newblock {\em Comm. Math. Phys.}, 324(1):77--106, 2013.

\bibitem{MR877637}
N.~J. Hitchin, A.~Karlhede, U.~Lindstr{\"o}m, and M.~Ro{\v{c}}ek.
\newblock Hyper-{K}\"ahler metrics and supersymmetry.
\newblock {\em Comm. Math. Phys.}, 108(4):535--589, 1987.

\bibitem{MR3681388}
R.~A. Iona{\c{s}}.
\newblock Gravitational instantons of type {$D_k$} and a generalization of the
  {G}ibbons-{H}awking {A}nsatz.
\newblock {\em Comm. Math. Phys.}, 355(2):691--740, 2017.

\bibitem{MR1669956}
D.~Kaledin and M.~Verbitsky.
\newblock Non-{H}ermitian {Y}ang-{M}ills connections.
\newblock {\em Selecta Math. (N.S.)}, 4(2):279--320, 1998.

\bibitem{Kontsevich:2008fj}
M.~Kontsevich and Y.~Soibelman.
\newblock {Stability structures, motivic Donaldson-Thomas invariants and
  cluster transformations}.
\newblock 2008.

\bibitem{MR2681792}
M.~Kontsevich and Y.~Soibelman.
\newblock Motivic {D}onaldson-{T}homas invariants: summary of results.
\newblock In {\em Mirror symmetry and tropical geometry}, volume 527 of {\em
  Contemp. Math.}, pages 55--89. Amer. Math. Soc., Providence, RI, 2010.

\bibitem{MR3625762}
E.~O. Korman.
\newblock A hyperholomorphic line bundle on certain hyperk\"ahler manifolds not
  admitting an {$S^1$}-symmetry.
\newblock {\em Differential Geom. Appl.}, 51:76--101, 2017.

\bibitem{Lindstrom:1983rt}
U.~Lindstr{\"o}m and M.~Ro{\v{c}}ek.
\newblock {Scalar Tensor Duality and N=1, N=2 Nonlinear Sigma Models}.
\newblock {\em Nucl. Phys.}, B222:285--308, 1983.

\bibitem{MR929144}
U.~Lindstr{\"o}m and M.~Ro{\v{c}}ek.
\newblock New hyper-{K}\"ahler metrics and new supermultiplets.
\newblock {\em Comm. Math. Phys.}, 115(1):21--29, 1988.

\bibitem{Lindstrom:2008gs}
U.~Lindstr{\"o}m and M.~Ro\v{c}ek.
\newblock {Properties of hyperk{\"a}hler manifolds and their twistor spaces}.
\newblock {\em Commun. Math. Phys.}, 293:257--278, 2010.

\bibitem{MR1686939}
Z.~Lu.
\newblock A note on special {K}\"ahler manifolds.
\newblock {\em Math. Ann.}, 313(4):711--713, 1999.

\bibitem{MR688351}
M.~L. Michelsohn.
\newblock On the existence of special metrics in complex geometry.
\newblock {\em Acta Math.}, 149(3-4):261--295, 1982.

\bibitem{Neitzke:2011za}
A.~Neitzke.
\newblock {On a hyperholomorphic line bundle over the Coulomb branch}.
\newblock 2011.

\bibitem{MR3330790}
A.~Neitzke.
\newblock Notes on a new construction of hyperkahler metrics.
\newblock In {\em Homological mirror symmetry and tropical geometry}, volume~15
  of {\em Lect. Notes Unione Mat. Ital.}, pages 351--375. Springer, Cham, 2014.

\bibitem{Ooguri:1996me}
H.~Ooguri and C.~Vafa.
\newblock {Summing up D-instantons}.
\newblock {\em Phys. Rev. Lett.}, 77:3296--3298, 1996.

\bibitem{MR953820}
H.~Pedersen and Y.~S. Poon.
\newblock Hyper-{K}\"ahler metrics and a generalization of the {B}ogomolny
  equations.
\newblock {\em Comm. Math. Phys.}, 117(4):569--580, 1988.

\bibitem{MR0439004}
R.~Penrose.
\newblock Nonlinear gravitons and curved twistor theory.
\newblock {\em General Relativity and Gravitation}, 7(1):31--52, 1976.
\newblock The riddle of gravitation -- on the occasion of the 60th birthday of
  Peter G. Bergmann (Proc. Conf., Syracuse Univ., Syracuse, N. Y., 1975).

\bibitem{MR776784}
R.~Penrose and W.~Rindler.
\newblock {\em Spinors and space-time. {V}ol. 1}.
\newblock Cambridge Monographs on Mathematical Physics. Cambridge University
  Press, Cambridge, 1984.
\newblock Two-spinor calculus and relativistic fields.

\bibitem{MR664330}
S.~M. Salamon.
\newblock Quaternionic {K}\"ahler manifolds.
\newblock {\em Invent. Math.}, 67(1):143--171, 1982.

\bibitem{MR860810}
S.~M. Salamon.
\newblock Differential geometry of quaternionic manifolds.
\newblock {\em Ann. Sci. \'Ecole Norm. Sup. (4)}, 19(1):31--55, 1986.

\bibitem{Sierra:1983cc}
G.~Sierra and P.~K. Townsend.
\newblock {An introduction to $N=2$ rigid supersymmetry}.
\newblock In {\em {19th Winter School and Workshop on Theoretical Physics:
  Supersymmetry and Supergravity -- Karpacz, Poland, February 14-26, 1983}},
  pages 396--430, 1983.

\bibitem{MR551471}
K.~P. Tod and R.~S. Ward.
\newblock Self-dual metrics with self-dual {K}illing vectors.
\newblock {\em Proc. Roy. Soc. London Ser. A}, 368(1734):411--427, 1979.

\bibitem{MR1486984}
M.~Verbitsky.
\newblock Hyperholomorphic bundles over a hyper-{K}\"ahler manifold.
\newblock {\em J. Algebraic Geom.}, 5(4):633--669, 1996.

\end{thebibliography}

\Address

\end{document}